\documentclass[reqno]{amsart}
\usepackage{amsfonts,amsmath,amssymb}

\numberwithin{equation}{section}

\def\be{{\beta}}

\def\eps{\epsilon}
\def\ka{\kappa}

\def\si{\sigma}
\def\phii{\widetilde{\varphi}}

\def\ka{\kappa}

\def\al{\alpha}

\newtheorem{theorem}{Theorem}[section]
\newtheorem{lemma}[theorem]{Lemma}

\newtheorem{proposition}[theorem]{Proposition}
\newtheorem{definition}[theorem]{Definition}
\newtheorem{remark}[theorem]{Remark}

\begin{document}

\title[The Euler--Maxwell two-fluid system in 3D]{Global solutions of the Euler--Maxwell two-fluid system in 3D}

\author{Yan Guo}
\address{Brown University}
\email{guoy@cfm.brown.edu}

\author{Alexandru D. Ionescu}
\address{Princeton University}
\email{aionescu@math.princeton.edu}

\author{Benoit Pausader}
\address{LAGA, Universit\'e Paris 13 (UMR 7539)}
\email{pausader@math.univ-paris13.fr}

\thanks{The first author is supported in part by NSF grant \#1209437 and a
Chinese NSF grant. The second author is supported in part by a Packard
Fellowship and NSF grant DMS-1065710. The third author is supported in part
by NSF grant DMS-1 142293. Both the first and the third authors thank the
support of Beijing International Mathematical Research Center.}

\begin{abstract}
The fundamental \textquotedblleft two-fluid\textquotedblright\ model for describing plasma dynamics is given by the Euler--Maxwell system, in which compressible ion and electron fluids interact with their own self-consistent electromagnetic field.
We prove global stability of a constant neutral background, in the sense that irrotational, smooth and localized perturbations of a constant background with small amplitude lead to global smooth solutions in three space dimensions for the Euler-Maxwell system. Our construction applies equally well to other plasma models such as the Euler-Poisson system for two-fluids and a relativistic Euler-Maxwell system for two fluids. Our solutions appear to be the first nontrivial global smooth solutions in all of these models.
\end{abstract}

\maketitle

\tableofcontents

\section{Introduction}\label{intro}

\subsection{Statement of the main result}

A plasma is a collection of fast-moving charged particles. It is believed that more than 90\% of the matter in the universe is in the form of plasma, from sparse intergalactic plasma, to the interior of stars to neon signs. In addition, understanding of the instability formation in plasma is one of the main challenges for nuclear fusion, in which charged particles are accelerated at high speed to create energy. We refer to \cite{Bit,Che,DelBer,Stu,Sw} for physics references in book form.

At high temperature and velocity, ions and electrons in a plasma tend to become two separate fluids due to their different physical properties (inertia, charge). One of the basic fluid models for describing plasma dynamics is the so-called \textquotedblleft two-fluid\textquotedblright\ model, in which two compressible ion and electron fluids interact with their own self-consistent electromagnetic field. Such an Euler-Maxwell system describes the dynamical evolution of the functions $n_{e},n_{i}:\mathbb{R}^{3}\rightarrow \mathbb{R}$, $v_{e},v_{i},E,B:\mathbb{R}^{3}\rightarrow  \mathbb{R}^{3}$, which evolve according to the quasi-linear coupled system, 
\begin{equation}
\begin{split}
& \partial _{t}n_{e}+\hbox{div}(n_{e}v_{e})=0, \\
& n_{e}m_{e}\left[ \partial _{t}v_{e}+v_{e}\cdot \nabla v_{e}\right] +\nabla p_{e}=-n_{e}e\left[ E+\frac{v_{e}}{c}\times B\right] , \\
& \partial _{t}n_{i}+\hbox{div}(n_{i}v_{i})=0, \\
& n_{i}M_{i}\left[ \partial _{t}v_{i}+v_{i}\cdot \nabla v_{i}\right] +\nabla
p_{i}=Zn_{i}e\left[ E+\frac{v_{i}}{c}\times B\right] , \\
& \partial _{t}B+c\nabla \times E=0, \\
& \partial _{t}E-c\nabla \times B=4\pi e\left[ n_{e}v_{e}-Zn_{i}v_{i}\right],
\end{split}
\label{EuMa1}
\end{equation}
together with the elliptic equations 
\begin{equation}
\hbox{div}(B)=0,\quad \hbox{div}(E)=4\pi e(Zn_{i}-n_{e})  \label{Ell}
\end{equation}
and two equations of state expressing $p_e$ and $p_i$ in terms of $n_e$ and $n_i$. These equations describe a plasma composed of electrons and one species of ions. The electrons have charge $-e$, density $n_{e}$, mass $m_{e}$, velocity $v_{e}$, and pressure $p_{e}$, and the ions have charge $Ze$, density $n_{i}$, mass $M_{i}$, velocity $v_{i}$, and pressure $p_{i}$. In addition, $c$ denotes the speed of light and $E$ and $B$ denote the electric and magnetic field. The two equations \eqref{Ell} are propagated by the dynamic flow, provided that we assume that they are satisfied at the initial time.

The full Euler-Maxwell system \eqref{EuMa1} with constraint \eqref{Ell} forms the foundation of the ``two-fluid'' model in the plasma theory, which captures the complex dynamics of a plasma due to electromagnetic interactions present in the model. Even at the linear level, there are new ion-acoustic waves, Langmuir waves, as well as light waves etc. At the nonlinear level, the Euler-Maxwell system is the origin of many well-known dispersive PDE, such as KdV \cite{GuPu}, KP \cite{LaLiSa,Pu}, Zakharov \cite{Te}, Zakharov-Kuznetsov \cite{LaLiSa,Pu} and NLS, which can be derived from \eqref{EuMa1} and \eqref{Ell} via different scaling and asymptotic expansions. We also refer to \cite{CorGre,DeDeSa,GeHKRo} for derivation of the cold-ion and quasi-neutral equations.

From a PDE viewpoint, the full Euler-Maxwell system \eqref{EuMa1} with constraint \eqref{Ell} can be classified as a system of nonlinear hyperbolic conservation laws with \textit{no dissipation and no relaxation effects}\footnote{When dissipation or relaxation is present, one expects stronger decay, even at the level of the $L^2$-norm, see e.g. \cite{Peng} and the references therein. In our case however, the evolution is time-reversible and we need a different mechanism of decay based on dispersion.}. Despite major advances in the mathematical study of hyperbolic conservation laws in one space dimension over the years, no general mathematical theory exists for the construction of global solutions in higher space dimension. One of the reasons is that, for these equations, shock waves (i.e., discontinuities) will generically  develop even from smooth initial data (see e.g. John \cite{Jo}). Even worse, a classical result of Sideris \cite{Sid} demonstrates that, for the compressible Euler equation for a neutral gas, shock waves will develop even for smooth irrotational initial data with small amplitude. This shock formation was recently further described in \cite{Ch2,ChMi} (see also \cite{Al}). 

In this paper we consider perturbations of the flat neutral equilibrium, namely $(n_{e}^{0},v_{e}^{0},n_{i}^{0},v_{i}^{0},E^{0},B^{0})=(Zn_{0},0,n_{0},0,0,0)$, for constant $n_{0}>0$ to the Euler-Maxwell system \eqref{EuMa1} and \eqref{Ell}. In order to state our main result, we normalize the Euler-Maxwell system in the following way. Assume the pressures are given by the formulas\footnote{In fact, our approach allows to treat any sufficiently smooth \textit{barotropic} pressure law, in particular the typical power law $p_{e}\sim n_{e}^{\gamma _{e}}$ for some $\gamma _{e}>0$ and similarly for $p_{i}$. We refer to Appendix \ref{AEP} for more precise statements. We use the particular quadratic laws for the pressure here only for the sake of concreteness and since it minimizes the nonlinear terms we have to consider.}: 
\begin{equation}
p_{e}=P_{e}\frac{n_{e}^{2}}{2},\quad p_{i}=P_{i}Z^{2}\frac{n_{i}^{2}}{2}.
\label{pressure}
\end{equation}
with constants $P_{e}$ and $P_{i}.$ The physical parameters are then the effective ion and electron temperatures 
\begin{equation*}
k_{B}T_{e}=n_{0}P_{e},\quad k_{B}T_{i}=n_{0}ZP_{i},
\end{equation*}
where $k_{B}$ denotes the Boltzmann constant, with corresponding electron
and ion thermal speeds\footnote{These correspond to the speed of inertial (linearized) waves if one neglects the electromagnetic field.} 
\begin{equation*}
V_{e}=\sqrt{\frac{n_{0}P_{e}}{m_{e}}}=\sqrt{\frac{k_{B}T_{e}}{m_{e}}},\quad
V_{i}=\sqrt{\frac{n_{0}P_{i}Z}{M_{i}}}=\sqrt{\frac{k_{B}T_{i}}{M_{i}}}.
\end{equation*}
We also have the Debye length 
\begin{equation*}
\frac{1}{\lambda _{D}^{2}}=4\pi e^{2}\left[ \frac{n_{0}}{k_{B}T_{e}}+\frac{Zn_{0}}{k_{B}T_{i}}\right] =4\pi e^{2}\left[ \frac{1}{P_{e}}+\frac{1}{P_{i}}\right] .
\end{equation*}

The Euler--Maxwell system can be adimensionalized to
depend only on three parameters: the ratio of the electron to ion masses
(per charge) 
\begin{equation}
\varepsilon :=Zm_{e}/M_{i},  \label{DefEpsIntro}
\end{equation}
the ratio of the temperatures 
\begin{equation}
T:=P_{e}/P_{i}=ZT_{e}/T_{i},  \label{DefTIntro}
\end{equation}
and the (normalized) ratio of the speed of light to the ion velocity 
\begin{equation}
C_{b}:=\varepsilon\frac{c^2}{V_i^2}=\frac{c^{2}}{V_{e}V_{i}}\sqrt{T\varepsilon }=\frac{c^{2}m_{e}}{%
n_{0}P_{i}}.  \label{DefCIntro}
\end{equation}
More precisely, let 
\begin{equation*}
\lambda :=\sqrt{\frac{4\pi e^{2}}{P_{i}}},\qquad \beta :=\sqrt{\frac{4\pi
n_{0}Ze^{2}}{M_{i}}},
\end{equation*}
and 
\begin{equation}\label{AddedResc}
\begin{split}
& n_{e}(x,t)=n_{0}\big[n(\lambda x,\beta t)+1\big],\qquad
n_{i}(x,t)=(n_{0}/Z)\big[\rho (\lambda x,\beta t)+1\big], \\
& v_{e}(x,t)=(\beta /\lambda )v(\lambda x,\beta t),\qquad \quad
\,\,v_{i}(x,t)=(\beta /\lambda )u(\lambda x,\beta t), \\
& E(x,t)=(4\pi en_{0}/\lambda )\tilde{E}(\lambda x,\beta t),\quad
\,\,B(x,t)=(cM_{i}\beta /(Ze))\tilde{B}(\lambda x,\beta t).
\end{split}
\end{equation}
The parameter $\beta $ is the \textit{ion plasma frequency} and $\beta /\lambda =V_{i}$ is the ion thermal velocity. In terms of $n,v,\rho ,u,\tilde{E},\tilde{B}$ the system \eqref{EuMa1}--\eqref{Ell} becomes 
\begin{equation}
\begin{split}
& \partial _{t}n+\hbox{div}((n+1)v)=0, \\
& \varepsilon \left( \partial _{t}v+v\cdot \nabla v\right) +T\nabla n+\tilde{E}+v\times \tilde{B}=0, \\
& \partial _{t}\rho +\hbox{div}((\rho +1)u)=0, \\
& \left( \partial _{t}u+u\cdot \nabla u\right) +\nabla \rho -\tilde{E}-u\times \tilde{B}=0, \\
& \partial _{t}\tilde{B}+\nabla \times \tilde{E}=0, \\
& \partial _{t}\tilde{E}-\frac{C_{b}}{\varepsilon }\nabla \times \tilde{B}=\left[ (n+1)v-(\rho +1)u\right] , \\
& \hbox{div}(\widetilde{B})=0,\quad \hbox{div}(\widetilde{E})=\rho -n,
\end{split}
\label{NewSys}
\end{equation}
where $\varepsilon $, $T$ and $C_{b}$ have been defined above. We will
assume throughout the paper that 
\begin{equation}
\varepsilon \leq 10^{-3},\qquad T\in \lbrack 1,100],\qquad C_{b}\geq 6T.  \label{condTeps}
\end{equation}

We will make two additional simplifications. Using the system \eqref{NewSys} it is easy to see that 
\begin{equation*}
\begin{split}
\partial _{t}\big[\tilde{B}-\varepsilon \nabla \times v\big]& =\nabla \times \big[v\times (\tilde{B}-\varepsilon \nabla \times v)\big], \\
\partial _{t}\big[\tilde{B}+\nabla \times u\big]& =\nabla \times \big[u\times (\tilde{B}+\nabla \times u)\big],
\end{split}
\end{equation*}
Therefore ``\textit{generalized irrotational flows}'' with the property that 
\begin{equation}
\tilde{B}=\varepsilon \nabla \times v=-\nabla \times u  \label{girrotational}
\end{equation}
are naturally preserved for all time, see Proposition \ref{Localexistence}
(iii) below for precise details.

Our main theorem is as follows:

\begin{theorem}
\label{MainThm} \label{Main1} Assume (\ref{condTeps}).  Let $N_{0}=10^{4}$
and assume that 
\begin{equation}
\begin{split}
& \Vert (n^{0},v^{0},\rho ^{0},u^{0},\tilde{E}^{0},\tilde{B}^{0})\Vert_{H^{N_{0}}}+\Vert (n^{0},v^{0},\rho^{0},u^{0},\tilde{E}^{0},\tilde{B}^{0})\Vert _{Z}=\delta _{0}\leq \overline{\delta}, \\
& \hbox{div}(\tilde{E}^{0})+n^{0}-\rho ^{0}=0,\qquad \tilde{B}^{0}=\varepsilon \nabla \times v^{0}=-\nabla \times u^{0},
\end{split}
\label{maincond2}
\end{equation}
where $\overline{\delta }=\overline{\delta }(C_{b},T,\varepsilon )>0$ is sufficiently small, and the $Z$ norm is defined in Definition \ref{MainDef}. Then there exists a unique global solution $(n,v,\rho ,u,\tilde{E},\tilde{B})\in C([0,\infty ):{H^{N_{0}}})$ of the system (\ref{NewSys}) with initial data $(n(0),v(0),\rho (0),u(0),\tilde{E}(0),\tilde{B}(0))=(n^{0},v^{0},\rho ^{0},u^{0},\tilde{E}^{0},\tilde{B}^{0})$. Moreover, for any $t\in \lbrack 0,\infty )$, 
\begin{equation}
\hbox{div}(\tilde{E})(t)+n(t)-\rho (t)=0,\qquad \tilde{B}(t)=\varepsilon
\nabla \times v(t)=-\nabla \times u(t),\text{ \ (generalized
irrotationality) }  \label{mainconcl2}
\end{equation}
and, with $\beta :=1/100$, 
\begin{equation}
\begin{split}
&\Vert (n(t),v(t),\rho (t),u(t),\tilde{E}(t), \tilde{B}(t))\Vert
_{H^{N_{0}}}\\
&+\sup_{|\alpha |\leq 4}(1+t)^{1+\beta /2}\Vert
(D_{x}^{\alpha }n(t), D_{x}^{\alpha }v(t), D_{x}^{\alpha }\rho (t), D_{x}^{\alpha}u(t), D_{x}^{\alpha }\tilde{E}(t), D_{x}^{\alpha }\tilde{B}(t)\Vert _{L^{\infty }}\lesssim \delta
_{0}.
\end{split}
\label{mainconcl2.1}
\end{equation}
\end{theorem}

Our main result demonstrates that even though the Euler-Maxwell system (\ref{EuMa1}) and (\ref{Ell}) is much more complicated than the pure Euler system for a neutral gas, it is in fact \textit{more stable} in the sense that global smooth solutions can persist globally without any shock formations. This is a stark and surprising contrast to Sideris's result for the pure Euler equations \cite{Sid}.

\begin{remark}
We make a few remarks about the assumptions in Theorem \ref{MainThm}.

\begin{itemize}
\item Condition (\ref{condTeps}) is needed for our careful analysis of the
dispersion relations that appear in the study of the linearized system (see
Lemma \ref{tech99} in Appendix \ref{lin}). It is consistent with the relevant physical ranges of the parameters.

\item Our hypothesis imply in particular that the perturbation is \textit{electrically neutral}, i.e. 
\begin{equation*}
\int_{\mathbb{R}^3}\left[ Zn_0(1+\rho^0(x))-n_0(1+n^0(x))\right]dx=0.
\end{equation*}
This is however forced by Maxwell's relation \eqref{Ell} if we assume that
the electric perturbation is integrable.

\item The smallness assumption is needed: large deviations from an
equilibrium do create shocks \cite{GuoTah}.
\end{itemize}
\end{remark}

\subsection{Important simplified models}

The result of blow-up of Sideris for the pure compressible Euler equations \cite{Sid} can be understood from the fact that small and irrotational perturbations of a constant background for the pure compressible Euler equations obey a \textit{quasilinear wave equation without null-structure} of the form 
\begin{equation}
\left( \partial _{tt}-\Delta \right) \alpha =\mathcal{Q}(\alpha ,\nabla
\alpha ,\nabla ^{2}\alpha )  \label{WaveIntro}
\end{equation}
where $\alpha $ is related to the unknown and the right-hand side denotes a quadratic nonlinearity in up to two derivatives of $\alpha $. This type of equation has slow decay of linear waves (decay like $1/t$) and strong resonances and therefore blow-up or formation of shocks is expected.

The Euler-Maxwell system \eqref{NewSys} contains a nonlinearity $\mathcal{Q}$ essentially similar to the pure compressible Euler case. However, due to self-consistent electromagnetic interaction,  the linearized Euler-Maxwell system exhibits much more complex and subtle linear and bilinear dispersive effects than that from the wave equation. The main task in the present work is to systematically track down and exploit such dispersive effects mathematically to preserve smoothness globally in time and prevent shock formation.

In order to put our result in the right context as well as to understand the wealth of dynamics involved in small perturbations of \eqref{EuMa1}-\eqref{Ell}, we need to introduce some intermediate models. The Euler-Maxwell system (\ref{EuMa1}) and (\ref{Ell}) is such a ``master equations'' describing very rich and complex plasma dynamics, that it contains several well-known simplified models in plasma physics. For instance, in all physical situations\footnote{Indeed, the ratio $m_e/M_i$ is no bigger than the ratio of the electron mass to the proton mass which equals $1/1836$.}, $m_{e}\ll M_{i}$. It is then natural to formally set $\varepsilon =0$ in \eqref{NewSys}, which leads to simplified {\it one fluid} models for either ions ($M_{i}=1$, $m_{e}=0$) or electrons ($M_{i}=\infty $, $m_{e}=1$). Moreover, if all the velocities are much smaller than the speed of light, then $C_{b}\gg 1$. Formally setting\footnote{This is called the \textit{electrostatic approximation}.} $C_{b}=\infty $ and $B\equiv 0$ replaces the Maxwell equations by the much simpler \textit{Poisson equation}. We refer to \cite{CorGre,DeDeSa} for other examples.

In the following, we will consider the simplified models in a form which is
consistent with the reformulation \eqref{NewSys} given appropriate
approximations. This might look somewhat different from the classical form
of these models. However, after an appropriate rescaling the equations
should be the same up to cubic and higher-order terms which can be ignored
in our situation (see Appendix \ref{AEP}).

\subsubsection{Single-fluid models}

The simplest model we can derive is the \textit{Euler-Poisson model for the
electrons}

\begin{equation}
\begin{split}
\partial _{t}n+\hbox{div}((1+n)v)& =0, \\
\partial _{t}v+v\cdot \nabla v+\nabla n& =\nabla \phi , \\
\Delta \phi & =n.
\end{split}
\label{EP/e}
\end{equation}
Here the magnetic field vanishes $B\equiv 0$, and the ions are treated as motionless
with a constant density and only form a fixed charged background. Such a
simplified system is used for describing Langmuir waves in the two-fluid
theory. After suitable change of unknown, \eqref{EP/e} can be reformulated
as 
\begin{equation}  \label{KGIntro}
\left(\partial_{tt}-\Delta+1\right)\alpha=\mathcal{Q}(\alpha,\nabla\alpha,\nabla^2\alpha).
\end{equation}
The linearized Euler-Poisson system for irrotational flows is no longer the acoustic (wave) equation as in the pure Euler system \eqref{WaveIntro}, but the Klein-Gordon system with \textquotedblleft mass term\textquotedblright\ created by the plasma frequency due to to the electrostatic interaction. Taking advantage of the much better properties of Klein-Gordon equations (faster time decay of linear waves like $t^{-3/2}$, absence of quadratic resonances), global smooth irrotational flows were constructed in \cite{Guo} via the normal form method of Shatah \cite{Sh}:

\begin{theorem}[Stability of a neutral equilibrium \protect\cite{Guo}]
\label{GuoThm} Solutions of \eqref{EP/e} with initial data $(n^0,v^0)$
small, smooth, neutral and irrotational in the sense that 
\begin{equation*}
\int_{\mathbb{R}^3}n^0(x)dx=0,\qquad \nabla\times v^0\equiv0
\end{equation*}
remain globally smooth and decay to $0$ in $L^\infty$ as $t\to+\infty$.
\end{theorem}

The neutral assumption was later removed in \cite{GeMaPa} and this result was extended to two spatial dimensions independently in \cite{IoPa1,LiWu} (see also \cite{Ja,JaLiZh}). Theorem \ref{GuoThm} was the first positive result indicating that the dispersive effect alone in the two-fluid theory may prevent shock formation\footnote{Another way to prevent shock formation is to introduce exponential damping of the perturbation via dissipation or relaxation (see e.g. \cite{Peng}). We will not discuss this at all in this paper.} and it started an investigation to understand to which extent the introduction of electromagnetic forces could stabilize the full Euler-Maxwell system.

\medskip

Recently, further progress was made in this direction in the study of another simplified model: the \textit{Euler-Poisson equation for the ions}\footnote{In many works (including \cite{GuPa} and \cite{GeHKRo,GuPu,LaLiSa}), the the Poisson relation in \eqref{EP/i} is replaced by
\begin{equation*}
-\Delta \phi =1+\rho -e^\phi,
\end{equation*}
but, for small perturbations, this agrees with \eqref{EP/i} up to nonlinear corrections which can be easily handled.}: 
\begin{equation}
\begin{split}
\partial _{t}\rho +\hbox{div}((1+\rho )u)& =0, \\
\partial _{t}u+u\cdot \nabla u+\nabla \rho & =-\nabla \phi , \\
-\Delta \phi  &=\rho -\phi.
\end{split}
\label{EP/i}
\end{equation}
Here the electron dynamics with constant temperature is decoupled from the
ion dynamics via the Boltzmann relation. The model equation then becomes 
\begin{equation}
\left( \partial _{tt}-\Delta +(-\Delta )(1-\Delta )^{-1}\right) \alpha
=|\nabla |\mathcal{Q}(\alpha ,\nabla \alpha )  \label{IonDispIntro}
\end{equation}
This system has intermediate behavior between \eqref{WaveIntro} and \eqref{KGIntro}. The linearized solutions decay slowly (like $t^{-4/3}$) and create many strong degeneracies near the zero frequency, where the dispersion relation is similar to the wave dispersion up to third order (see $\lambda _{i}$ in Lemma \ref{tech99}). Nevertheless, the first and third authors were able to obtain an analogue of Theorem \ref{GuoThm} for perturbations of a neutral equilibrium by using a variation on the normal form method, controlling bilinear multipliers with rough coefficients using arguments inspired by \cite{GuNaTs}. Here, a crucial property is the fact that the nonlinearity is an exact derivative, which helps compensate for the degeneracy at the $0$ frequency.

\subsubsection{Two-fluid models with different speeds}

Both systems \eqref{EP/e} and \eqref{EP/i} can be reduced (under the irrotational assumption) to a (complex) scalar quasilinear equation with one speed. This is no longer the case for more complicated two-fluid models which yield quasilinear system with different speeds. Bilinear interactions in quasilinear systems generically create resonant sets of 2D spheres in the phase space, which are very challenging to control analytically. This was first studied in \cite{Ge} for the case of semilinear systems of Klein-Gordon equations with different speeds (see also \cite{DeFaXu} for a study of a system with different masses) and led in \cite{GeMa} to the first construction of global smooth solutions for the Euler-Maxwell equation for electrons, 
\begin{equation}
\begin{split}
\partial _{t}n+\hbox{div}((1+n)v)& =0, \\
\partial _{t}v+v\cdot \nabla v+\nabla n& =-\left[ E+v\times B\right] , \\
\partial _{t}B+\nabla \times E& =0, \\
\partial _{t}E-C\nabla \times B& =(1+n)v
\end{split}
\label{EM/e}
\end{equation}
with constraints $\hbox{div}(B)=0$ and $\hbox{div}(E)=n$ :

\begin{theorem}[Stability in the Euler-Maxwell system for electrons \cite{GeMa,IoPa2}]
\label{GMThm} A solution of \eqref{EM/e} with
initial data $(n^0,v^0,E^0,B^0)$ small, smooth, compactly supported, neutral
and irrotational in the sense that 
\begin{equation*}
\begin{split}
\int_{\mathbb{R}^3}n^0(x)dx=0,\quad\int_{\mathbb{R}^3}B^0(x)dx=0,\quad
\nabla\times v^0+CB^0\equiv 0
\end{split}
\end{equation*}
remains global and smooth and decays to $0$ in $L^\infty$.
\end{theorem}

This was first shown in \cite{GeMa} under additional generic conditions on the parameters. Later in \cite{IoPa2}, the generic condition was removed and a stronger (integrable) decay was obtained, providing a robust approach even in the quasilinear case.
The model system is\begin{equation}
\begin{split}
\left( \partial _{tt}-\Delta +1\right) \alpha & =\mathcal{Q}_{1}(\alpha
,\beta ,\nabla \alpha ,\nabla \beta ,\nabla ^{2}\alpha ,\nabla ^{2}\beta ) \\
\left( \partial _{tt}-C\Delta +1\right) \beta & =\mathcal{Q}_{2}(\alpha
,\beta ,\nabla \alpha ,\nabla \beta ,\nabla ^{2}\alpha ,\nabla ^{2}\beta ).
\end{split}
\label{KGSys}
\end{equation}
It is important to note that the speed of the electron fluids is
different from the speed of the magnetic field, so that new analytical tools
are needed to estimate the 2D resonant sphere in the phase space.
The main result of \cite{IoPa2} is the natural analogue of Theorem \ref{GuoThm} and it is the foundation of the approach we use in this work. Note that in this case, we also need to introduce a decay condition on the initial data in order to be able to perform a more refined analysis of the solutions.

\subsection{Description of the Method}

We use a combination of dispersive analysis and energy estimate, relying heavily on the Fourier transform (see  \cite{Ch,GeMaSh,GeMaSh2,GuNaTs,Kl,Kl2,Sh} for previous seminal works). To overcome the quasilinear nature of the nonlinearity and ensure global existence, we use classical high order energy estimates to make up for the loss of derivatives in the nonlinearity. Global existence follows if a certain norm of lower regularity remains bounded and decays faster than 
$1/t.$

Such a crucial decay property is established by a semilinear analysis of systems of dispersive equations. Expecting some form of scattering, we express the solution as a free evolution\footnote{I.e. a solution to the linearized equation.} from a profile which varies more slowly in time. After suitable algebraic manipulations, and appropriate use of the Fourier transform, we need to study bilinear operators $T$ of
the form 
\begin{equation}
\widehat{T[f,g]}(\xi )=\int_{\mathbb{R}}\int_{\mathbb{R}^{3}}e^{it\Phi (\xi
,\eta )}m(\xi ,\eta )\widehat{f}(\xi -\eta ,t)\widehat{g}(\eta ,t)d\eta dt.
\label{OPT}
\end{equation}
with a phase $\Phi$ which is specific to each interaction and which is of the form
\begin{equation}\label{ModelPhase}
\Phi(\xi,\eta)=\Lambda_0(\xi)\pm\Lambda_{1}(\xi-\eta)\pm\Lambda_{2}(\eta),\quad\Lambda_j\in\{\Lambda_i,\Lambda_e,\Lambda_b\},
\end{equation}
where the functions $\Lambda_j$ are defined from the linearized system and given in \eqref{operators1}. As a first approximation, one may think of $f$, $g$ as being smooth bump functions and $m$ being essentially a smooth cut-off, and the main challenge is to estimate efficiently the infinite time integral. It then becomes clear that a key role is played by the properties of the function $\Phi $ and in particular by the points where it is stationary, 
\begin{equation}
\nabla _{(t,\eta )}[t\Phi (\xi ,\eta )]=0,\quad \hbox{i.e.}\quad \Phi (\xi
,\eta )=0\hbox{ and }\nabla _{\eta }\Phi (\xi ,\eta )=0.  \label{STRes}
\end{equation}
The collection of such points form the \textit{space-time resonant set}. This was already highlighted in \cite{GeMaSh} and forms the basis of the space-time resonance method. In some situations, one has no or few fully stationary points and the task is mainly to propagate enough smoothness of $\hat{f}$, $\hat{g}$ to exploit (non)-stationary phase arguments.

However, as in the case of the Euler-Maxwell equation for electrons \eqref{EM/e}, the space-time resonance set can be a 2D sphere. In addition, the linearized system for the full two fluid model is now coupled, which makes the derivation of the dispersive system in Section \ref{Algebraic} more involved and requires a careful study of the dispersion relations $\Lambda_j$ appearing in \eqref{ModelPhase} (see Lemma \ref{tech99}). Moreover, the appearance of an ``ion-like'' dispersion relation $\Lambda _{i}$, similar to that in \eqref{EP/i},  leads to additional mathematical difficulties of slower time decay of linear solutions, rough bilinear multipliers and rough phase around the zero frequency.

To overcome these new difficulties, we employ and extend the method developed in \cite{IoPa1,IoPa2}. We seek an appropriate space $B$ satisfying two requirements: first, the bilinear operator $T$ in \eqref{OPT} needs to be bounded 
\begin{equation}
T:B\times B\rightarrow B,  \label{TBdd}
\end{equation}
second, the free flow of the linearized Euler-Mawell system \eqref{NewSys} with initial data bounded in the space $B$ should belong to a space like $L_{t}^{1}L_{x}^{\infty }$, which has sufficiently strong time decay to close the energy estimate.

This strategy, initiated in the previous works \cite{IoPa1,IoPa2} shares similarities with {\it the space-time resonance method} as developed in \cite{Ge, GeMa,GeMaSh,GeMaSh2,GeMaSh3,GeMaSh4} but with new types of function space localized in both space and frequency, which are naturally compatible with the introduction of fractional powers of the weights (like $x^{-1-\varepsilon}L^2$ for the $B^1$-norm below), and with new bilinear estimates. We find this approach more precise and flexible analytically, which is crucial to analyze the complicated phase function \eqref{ModelPhase} arising in the Euler-Maxwell system. Together with orthogonality arguments and localized decay estimates, this allows to overcome the central difficulty of controlling delicate space-time resonant points in \eqref{STRes}.

We also mention the works in \cite{DeFa,Sch} which consider global existence in dispersive equations or systems with nonlinearity with small power without assuming any weights on the initial data (see also \cite{SGuo} and the references therein). It appears, however that such approach would be very difficult to carry on in the context of a system like \eqref{EuMa1} due to the loss of derivative in the nonlinearity. In addition, even in a purely semilinear setting, the analysis of the most delicate space-time resonances seems presently out of reach when the functions have too rough Fourier transforms.

\subsubsection{Choice of the norms} In order to define such a space $B$, we measure localization both in space and in frequency. We quantify all these ``coordinates'' all the way to the uncertainty principle and decompose an arbitrary function as a sum of ``atoms'':
\begin{equation*}
f=\sum_{X\cdot N\geq 1}Q_{X}P_{N}f,\quad (Q_{X}f)(x)\simeq \mathfrak{1}_{X\leq |x|\leq 2X}(x)f(x),\quad (\widehat{P_{N}f})(\xi )\simeq \mathfrak{1}_{N\leq |\xi |\leq 2N}(\xi )\hat{f}(\xi ).
\end{equation*}
We can then define the norms for the space $B$ on each atom. The simplest norm giving the appropriate decay would be a weighted space $x^{-1-\varepsilon}L^{2}$ and this is the main motivation for our ``strong'' norm $B^{1}$. Unfortunately, some interactions seem to produce outputs which are not bounded in this norm around a 2D resonant sphere. To account for this, we also introduce another kind of atoms, the ``weak'' atoms, bounded only in $B^{2}$ which barely fail to be in $x^{-1}L^{2}$, but are essentially concentrated on the $2D$-resonant spheres. Finally, each atom is allowed to be a combination of the two above types: 
\begin{equation*}
\Vert f\Vert =\sup_{X\cdot N\geq 1}\Vert Q_{X}P_{N}f\Vert _{B_{X,N}},\qquad
\Vert g\Vert _{B_{X,N}}=\Vert g\Vert_{B_{X,N}^{1}+B_{X,N}^{2}}=\inf_{g=g_{1}+g_{2}}\{\Vert g_{1}\Vert_{B_{X,N}^{1}}+\Vert g_{2}\Vert _{B_{X,N}^{2}}\}.
\end{equation*}
We refer to Definition \ref{MainDef} for the precise definition of the $Z$ norm that we use and to Lemma \ref{tech1.5} in Appendix \ref{lin} for the proof that these norms yield the desired integrability upon application of the linear flow.

\subsubsection{Analysis of the bilinear operators}

Once the form of the norm has been assessed, the main difficulty is to understand the bilinear interactions in \eqref{OPT} and to fine-tune the
norms to ensure that they are appropriately bounded as in \eqref{TBdd}. After quantifying all the information, one needs to treat a huge sum of elementary interactions (even for a single atom as an output). However, appropriate use of energy-estimate, simple orthogonality arguments and finite speed of propagation quickly limit the cases to only a few possibilities, see Proposition \ref{reduced2} and Lemma \ref{BigBound3}. Then non-stationary phase analysis allows to focus on the \textit{space-time resonant sets} as in \eqref{STRes}. This is where the bulk of the work is done, following a previous work \cite{IoPa2} where such an analysis was performed on the simplified Euler-Maxwell equation for the electrons.

After a careful analysis of the interactions done in Appendix \ref{resonantsets}, we
isolate three different problematic space-time resonant sets $\mathcal{S}$.

\begin{itemize}
\item Case A: we have the case of a ``classical'' $2D$ sphere 
\begin{equation*}
\mathcal{S}_A=\{(\xi,\eta)=(R\omega,r\omega),\,\,\omega\in\mathbb{S}^2\},\quad R\ne 0,\,\, r\ne 0
\end{equation*}
which already appears in the analysis of \eqref{KGSys} and which is responsible for the introduction of the ``weak'' atoms. Fortunately, here the phase is nondegenerate and we can perform an efficient stationary phase analysis and use additional refined orthogonality arguments as in \cite{IoPa2}.

\item Case B: we have a first degenerate sphere 
\begin{equation*}
\mathcal{S}_B=\{(\xi,\eta)=(R^\prime\omega,0),\,\,\omega\in\mathbb{S}^2\},\quad R^\prime\ne 0
\end{equation*}
where in addition, the phase is not smooth in $\eta$.

In this case, we use the fact that the speed of propagation of the singular perturbation is slower than expected, the fact that the phase is weakly elliptic and a careful adaptation of the refined orthogonality analysis of Case A, keeping track of how the bound deteriorate as $\eta\to 0$.

\item Case C: the presence of an ``ion-like'' dispersion relation brings in
a strong degenerate set at $0$ 
\begin{equation*}
\mathcal{S}_D=\{(\xi,\eta)=(0,r^\prime\omega),\,\,\omega\in\mathbb{S}
^2\},\quad r^\prime\ne 0\hbox{ or }r^\prime=0.
\end{equation*}
Here the problem comes from the strong degeneracy of the phase. Similar problems already appeared for the Euler-Poisson equation for the ions \eqref{EP/i}, but for \eqref{EuMa1} we need more refined multiplier estimate and orthogonality arguments, combined with additional finite speed of propagation estimates and use of the null-form structure coming from the presence of a derivative in front of the nonlinearity as in \eqref{IonDispIntro} in order to overcome the loss coming from the roughness of the multiplier after application of a normal form transformation.
\end{itemize}

Our approach seems flexible and robust and we illustrate this by extending
the main results to other problems of interest with a similar structure in
Appendix \ref{AEP}, Most notably variants of \eqref{EuMa1} which enjoy natural (Galilean or Lorentz) symmetry.

\subsubsection{Organization of the paper}

In Section \ref{local}, we obtain a classical local well-posedness result in the energy space. In Section \ref{Algebraic}, we reduce the Euler-Maxwell system \eqref{NewSys} into a quasilinear dispersive system and identify the linearized system, together with the main structure of the nonlinearity. In Section \ref{MainSec2}, we introduce the function space $Z$ (see \ref{sec5.4}) and prove the main Theorem \ref{MainThm} assuming boundedness of the relevant bilinear integral operators as in \eqref{OPT}-\eqref{TBdd}. In Section \ref{normproof}, we study the case of nonresonant interactions for localized atoms. Sections \ref{normproof2}-\ref{normproof3}-\ref{normproof4} are then devoted to the study of the resonant
interactions.

In Section \ref{normproof2}, we study Case A resonant interactions. We first make use of an efficient parametrization $p^{\sigma ;\mu ,\nu }$ in \eqref{DefOfP0}, \eqref{DefP1}-\eqref{DefP3}, then control precisely the output of interactions of ``atoms'' by carefully designed $B_{k,j}^{2}$ norm defined in \eqref{sec5.4} as well as additional $L^{2}$ orthogonality argument in the spirit of \cite{IoPa2}.

In Section \ref{normproof3}, we study Case B resonant interactions. We make use of a precise analytic characterization of Case B (Lemma \ref{desc2}), decay estimates Lemma A.5, as well as an orthogonality argument to control the $ L^{2}$ norm to complete the analysis.

Section \ref{normproof4} is devoted to the study of Case C. We take advantage of the geometry of angles between $\eta ,\xi $ and $\xi -\eta $ to obtain extra regularity to overcome the singularity near zero frequency.

Finally, in Appendix \ref{lin}, we isolate relevant information on the structure of the dispersion relations $\lambda_i$, $\lambda_e$ and $\lambda_b$ and provide various stationary-phase estimates that are needed throughout the proof; in Appendix \ref{resonantsets}, we classify the quadratic resonances that may appear; in Appendix \ref{multiex}, we provide, for the convenience of the reader the precise form of the $61$ multipliers that appear in the quadratic interactions, and in Appendix \ref{AEP}, we extend the results to cover various other systems with a similar structure.

{\bf Acknowledgments:} The third author expresses his thanks to B. Texier and A. Cerfon for interesting discussions and helpful references.

\section{Energy estimates and the local existence theory}\label{local}

The local existence theory for \eqref{NewSys} is based on energy estimates. These in turn are obtained from the physical energy. The (local) energy identity reads
\begin{equation*}
\begin{split}
&\partial_t\mathfrak{e}+\hbox{div}\left[\mathfrak{J}_e+\mathfrak{J}_i+\mathfrak{J}_b\right]=0,\\
&\mathfrak{e}:=T\frac{n^2}{2}+\varepsilon(n+1)\frac{\vert v\vert^2}{2}+\frac{\rho^2}{2}+(\rho+1)\frac{\vert u\vert^2}{2}+\frac{\vert \tilde{E}\vert^2}{2}+\frac{C_b}{\varepsilon}\frac{\vert \tilde{B}\vert^2}{2},\\
&\mathfrak{J}_e:=\left\{Tn+\varepsilon\frac{\vert v\vert^2}{2}\right\}(n+1)v,\quad \mathfrak{J}_i:=\left\{\rho+\frac{\vert u\vert^2}{2}\right\}(\rho+1)u,\quad \mathfrak{J}_b:=\frac{C_b}{\varepsilon}\tilde{E}\times \tilde{B}.
\end{split}
\end{equation*}
From this, we obtain our higher order energies.
For any $(n,v,\rho,u,\tilde{E},\tilde{B})\in\widetilde{H}^N$ we define
\begin{equation}\label{Addpla0}
\mathcal{E}_N:=\sum_{|\gamma|\leq N}\int_{\mathbb{R}^3}
\big[T|D_x^\gamma n|^2+\varepsilon(1+n)|D_x^\gamma v|^2
+\vert D_x^\gamma\rho\vert^2+(\rho+1)\vert D_x^\gamma u\vert^2+|D_x^\gamma \tilde{E}|^2+\frac{C_b}{\varepsilon} |D_x^\gamma \tilde{B}|^2\big]\,dx.
\end{equation}
The following proposition is our local regularity result:

\begin{proposition}\label{Localexistence}
(i) There is $\delta_1\in(0,1]$ such that if 
\begin{equation}\label{Addpla1}
\|(n^0,v^0,\rho^0,u^0,\tilde{E}^0,\tilde{B}^0)\|_{\widetilde{H}^4}\leq\delta_1
\end{equation}
then there is a unique solution $(n,v,\rho,u,\tilde{E},\tilde{B})\in C([0,1]:\widetilde{H}^4)$ of the system
\eqref{NewSys}
with
\begin{equation*}
 (n(0),v(0),\rho(0),u(0),\tilde{E}(0),\tilde{B}(0))=(n^0,v^0,\rho^0,u^0,\tilde{E}^0,\tilde{B}^0).
\end{equation*}
Moreover,
\begin{equation*}
\sup_{t\in[0,1]}\|(n(t),v(t),\rho(t),u(t),\tilde{E}(t),\tilde{B}(t))\|_{\widetilde{H}^4}\lesssim \|(n^0,v^0,\rho^0,u^0,\tilde{E}^0,\tilde{B}^0)\|_{\widetilde{H}^4}.
\end{equation*}

(ii) If $N\geq 4$ and $(n^0,v^0,\rho^0,u^0,\tilde{E}^0,\tilde{B}^0)\in\widetilde{H}^N$ satisfies \eqref{Addpla1} then $(n,v,\rho,u,\tilde{E},\tilde{B})\in C([0,1]:\widetilde{H}^N)$, and
\begin{equation}\label{Addpla2}
\mathcal{E}_N(t')-\mathcal{E}_N(t)\lesssim \int_{t}^{t'}A(s)\mathcal{E}_N(s)\,ds.
\end{equation}
for any $t\leq t'\in[0,1]$, where
\begin{equation}\label{Addpla0.5}
\begin{split}
A(s):&=\|\nabla n(s)\|_{L^\infty}+\|v(s)\|_{L^\infty}+\|\nabla v(s)\|_{L^\infty}+\Vert \nabla\rho(s)\Vert_{L^\infty}+\Vert u(s)\Vert_{L^\infty}+\Vert\nabla u(s)\Vert_{L^\infty}\\
&+\|\nabla \tilde{E}(s)\|_{L^\infty}+\|\tilde{B}(s)\|_{L^\infty}+\|\nabla \tilde{B}(s)\|_{L^\infty}.\\
\end{split}
\end{equation}

(iii) If $(n^0,v^0,\rho^0,u^0,\tilde{E}^0,\tilde{B}^0)\in\widetilde{H}^4$ satisfies \eqref{Addpla1}, and, in addition,
\begin{equation*}
\hbox{div}(\tilde{E}^0)+n^0-\rho^0=0,\qquad \tilde{B}^0=\varepsilon \nabla\times v^0=-\nabla \times u^0,
\end{equation*}
then, for any $t\in[0,1]$,
\begin{equation}\label{Addpla5}
\hbox{div}(\tilde{E})(t)+n(t)-\rho(t)=0,\qquad \tilde{B}(t)=\varepsilon \nabla\times v(t)=-\nabla \times u(t).
\end{equation}
\end{proposition}

\begin{proof}[Proof of Proposition \ref{Localexistence}] We multiply each equation by a suitable factor and rewrite the system \eqref{NewSys} as a symmetric hyperbolic system,
\begin{equation*}
\begin{split}
&T\partial_tn+T\sum_{k=1}^3v_k\partial_kn+T(1+n)\sum_{k=1}^3\partial_kv_k=0,\\
&\varepsilon(1+n)\partial_tv_j+T(1+n)\partial_jn+\varepsilon(1+n)\sum_{k=1}^3v_k\partial_kv_j=-(1+n)\tilde{E}_j-(1+n)\sum_{k,m=1}^3\in_{jmk}v_m\tilde{B}_k,\\
&\partial_t\rho+\sum_{k=1}^3u_k\partial_k\rho+(1+\rho)\sum_{k=1}^3\partial_ku_k=0,\\
&(1+\rho)\partial_tu_j+(1+\rho)\partial_j\rho+(1+\rho)\sum_{k=1}^3u_k\partial_ku_j=(1+\rho)\tilde{E}_j+(1+\rho)\sum_{k,m=1}^3\in_{jmk}u_m\tilde{B}_k,\\
&\frac{C_b}{\varepsilon}\partial_t\tilde{B}_j+\frac{C_b}{\varepsilon}\sum_{k,m=1}^3\in_{jmk}\partial_m\tilde{E}_k=0,\\
&\partial_t\tilde{E}_j-\frac{C_b}{\varepsilon}\sum_{k,m=1}^3\in_{jmk}\partial_m\tilde{B}_k=(1+n)v_j-(1+\rho)u_j.
\end{split}
\end{equation*}
Then we apply Theorem II and Theorem III in \cite{Ka} to prove the local existence claim in part (i) and the propagation of regularity claim in part (ii). 

To verify the energy inequality \eqref{Addpla2} we let, for $P=D^\gamma_x$, $|\gamma|\leq N$,
\begin{equation*}
\mathcal{E}^\prime_P:=\int_{\mathbb{R}^3}\big[T|P n|^2+\varepsilon(1+n)|P v|^2+\vert P\rho\vert^2+(1+\rho)\vert Pu\vert^2+|P \tilde{E}|^2+\frac{C_b}{\varepsilon}|P \tilde{B}|^2\big]\,dx,
\end{equation*}
Then we calculate
\begin{equation*}
\begin{split}
&\frac{d}{dt}\mathcal{E}'_P=I_P+II_P+III_P+I_P^\prime+II_P^\prime+III_P^\prime+IV_P,\\
\end{split}
\end{equation*}
where
\begin{equation*}
\begin{split}
&I_P:=\int_{\mathbb{R}^3}2TP n\cdot P\partial_tn\,dx,\\
&II_P:=\sum_{j=1}^3\varepsilon\int_{\mathbb{R}^3}\partial_tn\cdot Pv_j\cdot Pv_j\,dx,\\
&III_P:=\sum_{j=1}^3\varepsilon\int_{\mathbb{R}^3}2(1+n)\cdot Pv_j\cdot P\partial_tv_j\,dx,\\
&I_P^\prime:=\int_{\mathbb{R}^3}2P \rho\cdot P\partial_t\rho\,dx,\\
&II_P^\prime:=\sum_{j=1}^3\int_{\mathbb{R}^3}\partial_t\rho\cdot Pu_j\cdot Pu_j\,dx,\\
&III_P^\prime:=\sum_{j=1}^3\int_{\mathbb{R}^3}2(1+\rho)\cdot Pu_j\cdot P\partial_tu_j\,dx,\\
&IV_P:=\sum_{j=1}^3\int_{\mathbb{R}^3}2P \tilde{E}_j\cdot P\partial_t\tilde{E}_j\,dx+\sum_{j=1}^3\int_{\mathbb{R}^3}2\frac{C_b}{\varepsilon}P\tilde{B}_j\cdot P\partial_t\tilde{B}_j\,dx.
\end{split}
\end{equation*}

We use the general bound 
\begin{equation}\label{sobo}
\|D^\rho_xf\cdot D^{\rho'}_xg\|_{L^2}\lesssim \|\nabla_x f\|_{L^\infty}\|g\|_{H^M}+\|\nabla_x g\|_{L^\infty}\|f\|_{H^M},
\end{equation}
provided that $|\rho|+|\rho'|\leq M+1$, $M\geq 1$, and $|\rho|,|\rho'|\geq 1$. Using also the equations we estimate
\begin{equation*}
\begin{split}
&\Big|I_P+\sum_{k=1}^3\int_{\mathbb{R}^3}2TPn\cdot (1+n)\cdot P\partial_kv_k\,dx\Big|\lesssim A(t)\|(n,v,\rho,u,\tilde{E},\tilde{B})\|_{\widetilde{H}^N}^2,\\
&\Big|II_P\Big|\lesssim A(t)\|(n,v,\rho,u,\tilde{E},\tilde{B})\|_{\widetilde{H}^N}^2,\\
&\Big|III_P+\sum_{j=1}^3\int_{\mathbb{R}^3}\big[2TP\partial_jn\cdot (1+n)\cdot Pv_j+2P\tilde{E}_j\cdot Pv_j\cdot (1+n)\big]\,dx\Big|\lesssim A(t)\|(n,v,\rho,u,\tilde{E},\tilde{B})\|_{\widetilde{H}^N}^2,
\end{split}
\end{equation*}
and similarly,
\begin{equation*}
\begin{split}
&\Big|I_P^\prime+\sum_{k=1}^3\int_{\mathbb{R}^3}2P\rho\cdot (1+\rho)\cdot P\partial_ku_k\,dx\Big|\lesssim A(t)\|(n,v,\rho,u,\tilde{E},\tilde{B})\|_{\widetilde{H}^N}^2,\\
&\Big|II_P^\prime\Big|\lesssim A(t)\|(n,v,\rho,u,\tilde{E},\tilde{B})\|_{\widetilde{H}^N}^2,\\
&\Big|III_P^\prime+\sum_{j=1}^3\int_{\mathbb{R}^3}\big[2P\partial_j\rho\cdot (1+\rho)\cdot Pu_j-2P\tilde{E}_j\cdot Pu_j\cdot (1+\rho)\big]\,dx\Big|
\lesssim A(t)\|(n,v,\rho,u,\tilde{E},\tilde{B})\|_{\widetilde{H}^N}^2.\\
\end{split}
\end{equation*}
In addition,
\begin{equation*}
\Big|IV_P-\sum_{j=1}^3\int_{\mathbb{R}^3}2P\tilde{E}_j\cdot \left[Pv_j\cdot (1+n)-Pu_j\cdot (1+\rho)\right]\,dx\Big|\lesssim A(t)\|(n,v,\rho,u,\tilde{E},\tilde{B})\|_{\widetilde{H}^N}^2.
\end{equation*}
Therefore
\begin{equation*}
\Big|\frac{d}{dt}\mathcal{E}'_P\Big|\lesssim A(t)\|(n,v,\rho,u,\tilde{E},\tilde{B})\|_{\widetilde{H}^N}^2,
\end{equation*}
and the bound \eqref{Addpla2} follows since $\mathcal{E}_N=\sum_{P=D^\gamma_x,\,|\gamma|\leq N}\mathcal{E}'_P\approx \|(n,v,\rho,u,\tilde{E},\tilde{B})\|_{\widetilde{H}^N}^2$.

Finally, to verify that the identities \eqref{Addpla5} are propagated by the flow, we let
\begin{equation*}
X:=n-\rho+\hbox{div}(\tilde{E}),\qquad Y:=\tilde{B}-\varepsilon\nabla\times v,\qquad Z:=\tilde{B}+\nabla\times u.
\end{equation*}
Using the equations in \eqref{NewSys} we calculate
\begin{equation*}
\partial_tX=\partial_tn-\partial_t\rho+\sum_{j=1}^3\partial_j\partial_t\tilde{E}_j=-\sum_{j=1}^3\partial_j[(1+n)v_j-(1+\rho)u_j]+\sum_{j=1}^3\partial_j[(1+n)v_j-(1+\rho)u_j]=0,
\end{equation*}
therefore $X\equiv 0$. Moreover
\begin{equation*}
\partial_t\big(\sum_{k=1}^3\partial_k\tilde{B}_k\big)=0,
\end{equation*}
therefore
\begin{equation*}
\sum_{k=1}^3\partial_k\tilde{B}_k\equiv 0,\qquad \sum_{k=1}^3\partial_kY_k\equiv 0,\qquad \sum_{k=1}^3\partial_kZ_k\equiv 0.
\end{equation*}
Finally we notice that
\begin{equation*}
\partial_tY=\nabla\times(v\times Y),\qquad \partial_tZ=\nabla\times(u\times Z).
\end{equation*}
Using energy estimates it follows easily that $Y\equiv 0, Z\equiv 0$, as desired.
\end{proof}

\section{Derivation of the main dispersive system}\label{Algebraic}

The main part of this paper is devoted to obtain global time integrability of the function $A$ defined in \eqref{Addpla0.5}, so as to be able to propagate energy control 
using \eqref{Addpla2}. In order to do this, one needs to turn the system \eqref{NewSys}--\eqref{girrotational} into a quasilinear system of dispersive equations. This is the purpose of this section. 
The main results are summarized in Proposition \ref{algebra}.

For $\xi\in\mathbb{R}^3$ and $\al=1,2,3$ we define
\begin{equation}\label{pla1}
\begin{split}
&|\nabla|(\xi):=|\xi|,\qquad R_\alpha(\xi):=i\xi_\al/|\xi|,\qquad Q_{\alpha\beta}(\xi):=i\in_{\alpha\gamma\beta}\xi_\gamma/\vert\xi\vert,\\
&H_1(\xi):=\sqrt{1+|\xi|^2},\qquad H_\varepsilon(\xi):=\varepsilon^{-1/2}\sqrt{1+T|\xi|^2},\qquad \Lambda_b(\xi):=\varepsilon^{-1/2}\sqrt{1+\varepsilon+C_b\vert\xi\vert^2}.
\end{split}
\end{equation}
By a slight abuse of notation we also let $|\nabla|,R_\alpha,Q,H_1,H_\varepsilon,\Lambda_b$ denote the operators on $\mathbb{R}^3$ defined by the 
corresponding Fourier multipliers. Notice that
\begin{equation*}
Q^3=Q\qquad\text{ and }\qquad Q A=|\nabla|^{-1}(\nabla\times A)\text{ for any vector-field }A.
\end{equation*}

Closer inspection of the system \eqref{NewSys}--\eqref{girrotational} shows a decoupling of the magnetic unknowns $\hbox{curl}(E)$, $B$
and the electrostatic (Euler-Poisson) unknowns $n$, $\rho$, $\hbox{div}(v)$ and $\hbox{div}(u)$. More precisely, we may define
\begin{equation*}
  2U_b:=\Lambda_b\vert\nabla\vert^{-1}Q\tilde{B}-iQ^2\tilde{E},\qquad h:=-\vert\nabla\vert^{-1}\hbox{div}(v),\qquad g:=-\vert\nabla\vert^{-1}\hbox{div}(u).
\end{equation*}
Recalling that $\tilde{B}=\varepsilon\nabla\times v=-\nabla\times u$ and $\hbox{div}(\tilde{E})=\rho-n$, the functions $U_b,h,g$ together with $n$, $\rho$ allow us to recover all the physical unknowns, i.e.
\begin{equation}\label{Recover}
 \begin{split}
\tilde{B}&=2\Lambda_b^{-1}\vert\nabla\vert Q\hbox{Re}(U_b),\\
  v&=\nabla\vert\nabla\vert^{-1}h+\frac{2}{\varepsilon}\Lambda_b^{-1}\hbox{Re}(U_b),\\
  u&=\nabla\vert\nabla\vert^{-1}g-2\Lambda_b^{-1}\hbox{Re}(U_b),\\
  \tilde{E}&=-\nabla\vert\nabla\vert^{-2}\left[\rho-n\right]-2\hbox{Im}(U_b).
 \end{split}
\end{equation}

Let
\begin{equation*}
 A_\alpha=2\Lambda_b^{-1}\hbox{Re}(U_{b,\alpha}).
\end{equation*}
In terms of $n,h,\rho,g,U_b$ the system \eqref{NewSys}-\eqref{girrotational} becomes
\begin{equation}\label{pla0}
\begin{split}
\partial_t n-\vert\nabla\vert h&=-\partial_\alpha\left[n R_\alpha h\right]-(1/\varepsilon)\partial_\alpha\left[nA_\alpha\right],\\
\partial_t\rho-\vert\nabla\vert g&=-\partial_\alpha\left[\rho R_\alpha g\right]+\partial_\alpha\left[\rho A_\alpha\right],\\
\partial_th+\vert\nabla\vert^{-1}H_\varepsilon^2n-\varepsilon^{-1}\vert\nabla\vert^{-1}\rho&=-(1/2)\vert\nabla\vert\left[R_\alpha h R_\alpha h\right]-\varepsilon^{-1}\vert\nabla\vert \left[R_\alpha h A_\alpha\right]-(\varepsilon^{-2}/2)\vert\nabla\vert\left[A_\alpha A_\alpha\right],\\
\partial_t g-\vert\nabla\vert^{-1}n+\vert\nabla\vert^{-1}H_1^2\rho&=-(1/2)\vert\nabla\vert\left[R_\alpha g R_\alpha g\right]+\vert\nabla\vert\left[R_\alpha g A_\alpha\right]-(1/2)\vert\nabla\vert\left[A_\alpha A_\alpha\right],\\
\partial_tU_{b,\alpha}+i\Lambda_b U_{b,\alpha}&=-(i/2)Q^2_{\al\be}[nR_\be h-\rho R_\be g+\varepsilon^{-1}n A_\be+\rho A_\be],
\end{split}
\end{equation}
where the left-hand sides of the equations above are linear in the variables $n,h,\rho,g,U_b$, and the right-hand sides are quadratic.

We make linear changes of variables to diagonalize this system. Let
\begin{equation}\label{pla2}
\begin{split}
&\Lambda_e:=\varepsilon^{-1/2}\sqrt{\frac{(1+\varepsilon)-(T+\varepsilon)\Delta+\sqrt{\left((1-\varepsilon)-(T-\varepsilon)\Delta\right)^2+4\varepsilon}}{2}},\\
&\Lambda_i:=\varepsilon^{-1/2}\sqrt{\frac{(1+\varepsilon)-(T+\varepsilon)\Delta-\sqrt{\left((1-\varepsilon)-(T-\varepsilon)\Delta\right)^2+4\varepsilon}}{2}},
\end{split}
\end{equation}
such that
\begin{equation}\label{pla3}
(\Lambda_e^2-H_\varepsilon^2)(H_{\varepsilon}^2-\Lambda_i^2)=\varepsilon^{-1},\qquad \Lambda_e^2-H_1^2=H_{\varepsilon}^2-\Lambda_i^2.
\end{equation}
Let
\begin{equation}\label{pla4}
 R:=\sqrt{\frac{\Lambda_{e}^2-H_\varepsilon^2}{H_\varepsilon^2-\Lambda_{i}^2}},
\end{equation}
and notice that
\begin{equation}\label{pla5}
\Lambda_e^2-H_\varepsilon^2=\varepsilon^{-1/2}R,\qquad H_{\varepsilon}^2-\Lambda_i^2=\varepsilon^{-1/2}R^{-1}.
\end{equation}

Let
\begin{equation}\label{pla7}
\begin{split}
&U_e:=\frac{1}{2\sqrt{1+R^2}}\big[-\varepsilon^{1/2}|\nabla|^{-1}\Lambda_e n+R|\nabla|^{-1}\Lambda_e\rho-i\varepsilon^{1/2}h+iRg\big],\\
&U_i:=\frac{1}{2\sqrt{1+R^2}}\big[\varepsilon^{1/2}R|\nabla|^{-1}\Lambda_i n+|\nabla|^{-1}\Lambda_i\rho+i\varepsilon^{1/2}Rh+ig\big].
\end{split}
\end{equation}
Using the system \eqref{pla0} it is easy to check that the complex variables $U_e$, $U_i$ and $U_b$ satisfy the identities
\begin{equation}\label{pla8}
\begin{split}
&(\partial_t+i\Lambda_e)U_e=\mathcal{N}_e,\\
&(\partial_t+i\Lambda_i)U_i=\mathcal{N}_i,\\
&(\partial_t+i\Lambda_b)U_{b,\alpha}=\mathcal{N}_{b,\alpha},
\end{split}
\end{equation}
where
\begin{equation}\label{pla9}
\begin{split}
&\Re(\mathcal{N}_e)=\frac{\Lambda_eR_\al}{2\sqrt{1+R^2}}\big[\varepsilon^{1/2}(nR_\al h)-R(\rho R_\al g)+\varepsilon^{-1/2}(nA_\alpha)+R(\rho A_\alpha)\big],\\
&\Im(\mathcal{N}_e)=\frac{|\nabla|}{4\sqrt{1+R^2}}\big[\varepsilon^{-3/2}(\varepsilon R_\al h+A_\al)(\varepsilon R_\al h+A_\al)-R[(R_\al g-A_\al)(R_\al g-A_\al)]\big],\\
&\Re(\mathcal{N}_i)=\frac{-\Lambda_iR_\al}{2\sqrt{1+R^2}}\big[\varepsilon^{1/2}R(nR_\al h)+(\rho R_\al g)+\varepsilon^{-1/2}R(nA_\alpha)-(\rho A_\alpha)\big],\\
&\Im(\mathcal{N}_i)=\frac{-|\nabla|}{4\sqrt{1+R^2}}\big[\varepsilon^{-3/2}R[(\varepsilon R_\al h+A_\al)(\varepsilon R_\al h+A_\al)]+(R_\al g-A_\al)(R_\al g-A_\al)\big],\\
&\Re(\mathcal{N}_{b,\alpha})=0,\\
&\Im(\mathcal{N}_{b,\alpha})=-(1/2)Q^2_{\alpha\beta}\left[nR_\beta h-\rho R_\beta g+\varepsilon^{-1}n A_\beta+\rho A_\beta\right].
\end{split}
\end{equation}

The system \eqref{pla8} is our main dispersive system, which is diagonalized at the linear level. To analyze it we have to express the nonlinearities $\mathcal{N}_e$, $\mathcal{N}_i$, and $\mathcal{N}_{b,\alpha}$ in terms of the complex variables $U_e$, $U_i$, and $U_b$. Indeed, it follows from \eqref{pla7} that
\begin{equation}\label{pla20}
\begin{split}
&n=\frac{-|\nabla|\varepsilon^{-1/2}}{\sqrt{1+R^2}\Lambda_e}(U_e+\overline{U_e})+\frac{|\nabla|\varepsilon^{-1/2}R}{\sqrt{1+R^2}\Lambda_i}(U_i+\overline{U_i}),\\
&\rho=\frac{|\nabla|R}{\sqrt{1+R^2}\Lambda_e}(U_e+\overline{U_e})+\frac{|\nabla|}{\sqrt{1+R^2}\Lambda_i}(U_i+\overline{U_i}),\\
&h=\frac{i\varepsilon^{-1/2}}{\sqrt{1+R^2}}(U_e-\overline{U_e})+\frac{-i\varepsilon^{-1/2}R}{\sqrt{1+R^2}}(U_i-\overline{U_i}),\\
&g=\frac{-iR}{\sqrt{1+R^2}}(U_e-\overline{U_e})+\frac{-i}{\sqrt{1+R^2}}(U_i-\overline{U_i})\\
&A_\alpha=\Lambda_b^{-1}(U_{b,\alpha}+\overline{U_{b,\alpha}}).
\end{split}
\end{equation}

We summarize now the main results we proved in this section. Recall first the definitions of the main multipliers
\begin{equation}\label{operators1}
\begin{split}
&\Lambda_e(\xi):=\varepsilon^{-1/2}\sqrt{\frac{(1+\varepsilon)+(T+\varepsilon)|\xi|^2+\sqrt{\left((1-\varepsilon)+(T-\varepsilon)|\xi|^2\right)^2+4\varepsilon}}{2}},\\
&\Lambda_i(\xi):=\varepsilon^{-1/2}\sqrt{\frac{(1+\varepsilon)+(T+\varepsilon)|\xi|^2-\sqrt{\left((1-\varepsilon)+(T-\varepsilon)|\xi|^2\right)^2+4\varepsilon}}{2}},\\
&\Lambda_b(\xi):=\varepsilon^{-1/2}\sqrt{1+\varepsilon+C_b\vert\xi\vert^2},\\
\end{split}
\end{equation}
and
\begin{equation}\label{operators2}
\begin{split}
&|\nabla|(\xi):=|\xi|,\qquad R_\alpha(\xi):=i\xi_\al/|\xi|,\qquad Q_{\alpha\beta}(\xi):=i\in_{\alpha\gamma\beta}\xi_\gamma/\vert\xi\vert,\qquad H_1(\xi):=\sqrt{1+|\xi|^2},\\
&H_\varepsilon(\xi):=\varepsilon^{-1/2}\sqrt{1+T|\xi|^2},\qquad R(\xi):=[\Lambda_e(\xi)^2-H_{\varepsilon}(\xi)^2]^{1/2}[H_{\varepsilon}(\xi)^2-\Lambda_i(\xi)]^{-1/2}.
\end{split}
\end{equation}

The lemma below describes symbol-type properties of some of these multipliers.

\begin{lemma}\label{tech1}
In $\mathbb{R}^3$ we have
\begin{equation}\label{mk1}
\Lambda_e^2\geq H_\varepsilon^2\geq H_1^2\geq \Lambda_i^2\geq |\nabla|^2,\qquad \Lambda_i^2\lesssim |\nabla|^2,
\end{equation}
and
\begin{equation}\label{mk2}
\begin{split}
&\Lambda_e^2-H_\varepsilon^2=\varepsilon^{-1/2}R,\qquad H_{\varepsilon}^2-\Lambda_i^2=\varepsilon^{-1/2}R^{-1},\\
&\Lambda_e(\xi)^2-H_\varepsilon(\xi)^2=H_1(\xi)^2-\Lambda_i(\xi)^2 =\frac{2}{(1-\varepsilon)+(T-\varepsilon)|\xi|^2+\sqrt{\big((1-\varepsilon)+(T-\varepsilon)|\xi|^2\big)^2+4\varepsilon}}.
\end{split}
\end{equation}
In addition, for $\alpha=(\alpha_1,\alpha_2,\alpha_3)$, we have the symbol-type estimates
\begin{equation}\label{mk3}
\begin{split}
&|D^\al_\xi \Lambda_e(\xi)|+|D^\al_\xi H_\varepsilon(\xi)|+|D^\al_\xi H_1(\xi)|\lesssim_{|\alpha|} (1+|\xi|)^{1-|\alpha|},\\
&|D^\al_\xi \Lambda_i(\xi)|+|D^\al_\xi |\nabla|(\xi)|\lesssim_{|\alpha|} |\xi|^{1-|\alpha|},\\
&|D^\al_\xi R(\xi)|\lesssim_{|\alpha|} (1+|\xi|)^{-2-|\alpha|}.
\end{split}
\end{equation}
\end{lemma}

\begin{proof}[Proof of Lemma \ref{tech1}] The inequalities in \eqref{mk1} and the identities in \eqref{mk2} follow directly from definitions. The symbol-type estimates in \eqref{mk3} also follow from definitions and the additional formula
\begin{equation*}
R(\xi)=\frac{2\varepsilon^{1/2}}{(1-\varepsilon)+(T-\varepsilon)|\xi|^2+\sqrt{\big((1-\varepsilon)+(T-\varepsilon)|\xi|^2\big)^2+4\varepsilon}}.
\end{equation*}
\end{proof}

The following proposition is the main result in this section. 

\begin{proposition}\label{algebra}
With $N_0=10^4$ as in Theorem \ref{Main1}, assume that $(n,v,\rho,u,\tilde{E},\tilde{B})\in C(I:\widetilde{H}^{N_0})$ is a solution of the system \eqref{NewSys}-\eqref{girrotational}, where $I\subseteq\mathbb{R}$ is an interval. Let $\Lambda_e, \Lambda_i, \Lambda_b, |\nabla|, R_\alpha, Q, H_1, H_\varepsilon, R$ denote the operators defined by the corresponding multipliers in \eqref{operators1}-\eqref{operators2}. Let
\begin{equation}\label{pla50}
\begin{split}
&h:=-\vert\nabla\vert^{-1}\hbox{div}(v),\qquad g:=-\vert\nabla\vert^{-1}\hbox{div}(u),\qquad A_\alpha:=|\nabla|^{-1}Q_{\al\be}\tilde{B}_\be,\\
&U_e:=\frac{1}{2\sqrt{1+R^2}}\big[-\varepsilon^{1/2}|\nabla|^{-1}\Lambda_e n+R|\nabla|^{-1}\Lambda_e\rho-i\varepsilon^{1/2}h+iRg\big],\\
&U_i:=\frac{1}{2\sqrt{1+R^2}}\big[\varepsilon^{1/2}R|\nabla|^{-1}\Lambda_i n+|\nabla|^{-1}\Lambda_i\rho+i\varepsilon^{1/2}Rh+ig\big],\\
&U_b:=[\Lambda_b\vert\nabla\vert^{-1}Q\tilde{B}-iQ^2\tilde{E}]/2,
\end{split}
\end{equation}
and, for $\alpha\in\{1,2,3\}$,
\begin{equation*}
U_{e+}:=U_e,\quad U_{e-}:=\overline{U_e},\quad U_{i+}:=U_i,\quad U_{i-}:=\overline{U_i},\quad U_{b+\alpha}:=U_{b,\alpha},\quad U_{b-\alpha}:=\overline{U_{b,\alpha}}.
\end{equation*}

(i) Then $U_e,U_i,U_b\in C(I:H^{N_0})$ and, for any $t\in I$,
\begin{equation}\label{pla50.1}
\|U_e(t)\|_{H^{N_0}}+\|U_i(t)\|_{H^{N_0}}+\|U_b(t)\|_{H^{N_0}}\lesssim \|(n(t),v(t),\rho(t),u(t),\tilde{E}(t),\tilde{B}(t))\|_{\widetilde{H}^{N_0}}.
\end{equation}
Moreover, the functions $U_e:\mathbb{R}^3\times I\to\mathbb{C}$, $U_i:\mathbb{R}^3\times I\to\mathbb{C}$, $U_b:\mathbb{R}^3\times I\to\mathbb{C}^3$ satisfy the dispersive system
\begin{equation}\label{pla51}
\begin{split}
&(\partial_t+i\Lambda_e)U_e=\mathcal{N}_e,\qquad (\partial_t+i\Lambda_i)U_i=\mathcal{N}_i,\qquad (\partial_t+i\Lambda_b)U_b=\mathcal{N}_{b},
\end{split}
\end{equation}
where the quadratic nonlinearities $\mathcal{N}_e,\mathcal{N}_i,\mathcal{N}_b$ are given by
\begin{equation}\label{pla52}
\begin{split}
&\mathcal{F}(\mathcal{N}_e)(\xi,t)=c\sum_{\mu,\nu\in\mathcal{I}_0}\int_{\mathbb{R}^3}m_{e;\mu,\nu}(\xi,\eta)\widehat{U_\mu}(\xi-\eta,t)\widehat{U_\nu}(\eta,t)\,d\eta,\\
&\mathcal{F}(\mathcal{N}_i)(\xi,t)=c\sum_{\mu,\nu\in\mathcal{I}_0}\int_{\mathbb{R}^3}m_{i;\mu,\nu}(\xi,\eta)\widehat{U_\mu}(\xi-\eta,t)\widehat{U_\nu}(\eta,t)\,d\eta,\\
&\mathcal{F}(\mathcal{N}_{b})(\xi,t)=c\sum_{\mu,\nu\in\mathcal{I}_0}\int_{\mathbb{R}^3}m_{b;\mu,\nu}(\xi,\eta)\widehat{U_\mu}(\xi-\eta,t)\widehat{U_\nu}(\eta,t)\,d\eta.
\end{split}
\end{equation}
The set $\mathcal{I}_0$ is given by 
\begin{equation}\label{Iset}
\mathcal{I}_0:=\{e+,e-,i+,i-,b+1,b+2,b+3,b-1,b-2,b-3\},
\end{equation}
and the multipliers $m_{e;\mu,\nu}:\mathbb{R}^3\times\mathbb{R}^3\to\mathbb{C}$, $m_{i;\mu,\nu}:\mathbb{R}^3\times\mathbb{R}^3\to\mathbb{C}$, $m_{b;\mu,\nu}:\mathbb{R}^3\times\mathbb{R}^3\to\mathbb{C}^3$ are calculated explicitly in \eqref{mult0}--\eqref{mult3.5}.

(ii) The physical variables $(n,\rho,v,u,\tilde{E},\tilde{B})$ can be expressed in terms of the complex variables $U_e,U_i,U_b$ according to the formulas
\begin{equation}\label{pla50.2}
\begin{split}
&n=\frac{-|\nabla|\varepsilon^{-1/2}}{\sqrt{1+R^2}\Lambda_e}(U_e+\overline{U_e})+\frac{|\nabla|\varepsilon^{-1/2}R}{\sqrt{1+R^2}\Lambda_i}(U_i+\overline{U_i}),\\
&\rho=\frac{|\nabla|R}{\sqrt{1+R^2}\Lambda_e}(U_e+\overline{U_e})+\frac{|\nabla|}{\sqrt{1+R^2}\Lambda_i}(U_i+\overline{U_i}),\\
&v=\nabla\vert\nabla\vert^{-1}h+\frac{2}{\varepsilon}\Lambda_b^{-1}\hbox{Re}(U_b),\qquad h=\frac{i\varepsilon^{-1/2}}{\sqrt{1+R^2}}(U_e-\overline{U_e})+\frac{-i\varepsilon^{-1/2}R}{\sqrt{1+R^2}}(U_i-\overline{U_i}),\\
&u=\nabla\vert\nabla\vert^{-1}g-2\Lambda_b^{-1}\hbox{Re}(U_b),\qquad g=\frac{-iR}{\sqrt{1+R^2}}(U_e-\overline{U_e})+\frac{-i}{\sqrt{1+R^2}}(U_i-\overline{U_i}),\\
&\tilde{E}=-\nabla\vert\nabla\vert^{-2}\left[\rho-n\right]-2\hbox{Im}(U_b),\\
&\tilde{B}=2\Lambda_b^{-1}\vert\nabla\vert Q\hbox{Re}(U_b).
\end{split}
\end{equation}
\end{proposition}

\begin{proof}[Proof of Proposition \ref{algebra}] The claim \eqref{pla50.1} is a consequence of \eqref{mk3} and the observation that $R(0)=\varepsilon^{1/2}$. The diagonalized dispersive system \eqref{pla51} and the identities \eqref{pla50.2} were derived earlier, see \eqref{pla8}-\eqref{pla9}, \eqref{Recover}, and \eqref{pla20}. It remains only to prove the formulas \eqref{pla52}, showing that the nonlinearities $\mathcal{N}_e,\mathcal{N}_i,\mathcal{N}_b$ can be expressed as bilinear forms in terms of the complex variables $U_e,U_i,U_b$. This is easy to see by inspecting the formulas \eqref{pla9} and \eqref{pla20}. The precise, somewhat long calculations are presented in section \ref{multiex}.
\end{proof}

The precise formulas of the multipliers $m_{e;\mu,\nu}$, $m_{i;\mu,\nu}$, and $m_{b;\mu,\nu}$, derived in section \ref{multiex} below, are complicated. However, we do not use these formulas in the rest of the paper. We will only use the simple observation that these multipliers can be expressed as suitable products of multipliers satisfying inequalities of the H\"{o}rmander--Michlin type. More precisely, for any integer $n\geq 1$ let
\begin{equation}\label{symb1}
\mathcal{S}^n:=\{q:\mathbb{R}^3\to\mathbb{C}:\,\|q\|_{\mathcal{S}^n}:=\sup_{\xi\in\mathbb{R}^3\setminus\{0\}}\sup_{|\alpha|\leq n}|\xi|^{|\alpha|}|D^\alpha_\xi q(\xi)|<\infty\},
\end{equation}
and
\begin{equation}\label{symb2}
\mathcal{M}:=\{m:\mathbb{R}^3\times\mathbb{R}^3\to\mathbb{C}:\,m(\xi,\eta)=q_1(\xi)\cdot q_2(\xi-\eta)\cdot q_3(\eta),\,\sup_{n\in\{1,2,3\}}\|q_n\|_{\mathcal{S}^{100}}\leq 1\}.
\end{equation}

\begin{lemma}\label{tech1.3}
The multipliers $m_{e;\mu,\nu}(\xi,\eta)$ and $m_{b,\alpha;\mu,\nu}(\xi,\eta)$, $\al\in\{1,2,3\}$, can be written as finite sums of functions of the form
\begin{equation}\label{mk4}
(1+|\xi|^2)^{1/2}\cdot m(\xi,\eta),\qquad m\in\mathcal{M}.
\end{equation}
Similarly, the multipliers $m_{i;\mu,\nu}(\xi,\eta)$ can be written as finite sums of functions of the form
\begin{equation}\label{mk5}
|\xi|\cdot m(\xi,\eta),\qquad m\in\mathcal{M}.
\end{equation}
\end{lemma}

\begin{remark}\label{tech1.33} We notice that the multipliers $m_{i;\mu,\nu}$ satisfy better estimates at $\xi=0$ than the multipliers $m_{e;\mu,\nu}$ and $m_{b,\alpha;\mu,\nu}$; in particular these multipliers vanish at the origin. This is an indication of a certain {\it{null structure}} of the system, and is important in the analysis in sections \ref{normproof} and \ref{normproof2}.
\end{remark}

\begin{proof}[Proof of Lemma \ref{tech1.3}] The formulas \eqref{mk4} and \eqref{mk5} follow from the identities \eqref{pla9}-\eqref{pla20} and Lemma \ref{tech1}. Indeed, using \eqref{pla20} and Lemma \ref{tech1}, we notice first that the functions $n, \rho, h, g, A_\alpha$ can all be written as finite sums of Calderon--Zygmund operators applied to the complex variables $U_{e\pm},U_{i\pm},U_{b\pm\alpha}$, i. e. finite sums of expressions of the form
\begin{equation*}
TU_{e\pm},\qquad TU_{i\pm},\qquad TU_{b\pm\alpha},\qquad \qquad\text{ where }\quad \widehat{Tf}(\xi)=q(\xi)\widehat{f}(\xi)\quad\text{ for some }\quad q\in \mathcal{S}^{100}.
\end{equation*} 
Then we use again Lemma \ref{tech1} and the identities in \eqref{pla9} to complete the proof of the lemma.  
\end{proof}

\section{Main definitions and propositions}\label{MainSec2}

In this section we define our main function spaces, and state two key propositions that concern properties of solutions of the dispersive system \eqref{pla51}. Then we show how to use these propositions, together with the local regularity theory in section \ref{local} and linear dispersive bounds, to complete the proof of the main theorem.

We fix $\varphi:\mathbb{R}\to[0,1]$ an even smooth function supported in $[-8/5,8/5]$ and equal to $1$ in $[-5/4,5/4]$. For simplicity of notation, we also let $\varphi:\mathbb{R}^d\to[0,1]$ denote the corresponding radial function on $\mathbb{R}^d$, $d=2,3$. For $d\in\{1,2,3\}$ let
\begin{equation*}
\begin{split}
&\varphi_k(x)=\varphi_{k,(d)}(x):=\varphi(|x|/2^k)-\varphi(|x|/2^{k-1})\qquad\text{ for any }k\in\mathbb{Z},\,x\in\mathbb{R}^d,\\
&\varphi_I:=\sum_{m\in I\cap\mathbb{Z}}\varphi_m\text{ for any }I\subseteq\mathbb{R}.
\end{split}
\end{equation*}
Let
\begin{equation*}
\mathcal{J}:=\{(k,j)\in\mathbb{Z}\times\mathbb{Z}_+:\,k+j\geq 0\}.
\end{equation*}
For any $(k,j)\in\mathcal{J}$ let
\begin{equation*}
\phii^{(k)}_j(x):=
\begin{cases}
\varphi_{(-\infty,-k]}(x)\quad&\text{ if }k+j=0\text{ and }k\leq 0,\\
\varphi_{(-\infty,0]}(x)\quad&\text{ if }j=0\text{ and }k\geq 0,\\
\varphi_j(x)\quad&\text{ if }k+j\geq 1\text{ and }j\geq 1.
\end{cases}
\end{equation*}
and notice that, for any $k\in\mathbb{Z}$ fixed,
\begin{equation*}
\sum_{j\geq-\min(k,0)}\phii^{(k)}_j=1.
\end{equation*}
For any interval $I\subseteq\mathbb{R}$ let
\begin{equation*}
\phii^{(k)}_I(x):=\sum_{j\in I,\,(k,j)\in\mathcal{J}}\phii^{(k)}_j(x).
\end{equation*}

Let $P_k$, $k\in\mathbb{Z}$, denote the operator on $\mathbb{R}^3$ defined by the Fourier multiplier $\xi\to \varphi_k(\xi)$. Similarly, for any $I\subseteq \mathbb{R}$ let $P_I$ denote the operator on $\mathbb{R}^3$
 defined by the Fourier multiplier $\xi\to \varphi_I(\xi)$. 

\begin{definition}\label{MainDef}
Let
\begin{equation}\label{sec5.6}
\beta:=1/100,\qquad \al:=\beta/2,\qquad\gamma:=3/2-4\beta.
\end{equation}
We define
\begin{equation}\label{sec5}
Z:=\{f\in L^2(\mathbb{R}^3):\,\|f\|_{Z}:=\sup_{(k,j)\in\mathcal{J}}\|\phii^{(k)}_j(x)\cdot P_kf(x)\|_{B_{k,j}}<\infty\},
\end{equation}
where, with $\widetilde{k}:=\min(k,0)$ and $k_+:=\max(k,0)$,
\begin{equation}\label{sec5.2}
\|g\|_{B_{k,j}}:=\inf_{g=g_1+g_2}\big[\|g_1\|_{B^{1}_{k,j}}+\|g_2\|_{B^{2}_{k,j}}\big],
\end{equation}
\begin{equation}\label{sec5.3}
\|h\|_{B^{1}_{k,j}}:=(2^{\alpha k}+2^{10k})\big[2^{(1+\beta)j}\|h\|_{L^2}+2^{(1/2-\beta)\widetilde{k}}\|\widehat{h}\|_{L^\infty}\big],
\end{equation}
and
\begin{equation}\label{sec5.4}
\|h\|_{B^{2}_{k,j}}:=2^{10|k|}(2^{\alpha k}+2^{10k})\big[2^{(1-\beta)j}\|h\|_{L^2}+\|\widehat{h}\|_{L^\infty}+2^{\gamma j}\sup_{R\in[2^{-j},2^k],\,\xi_0\in\mathbb{R}^3}R^{-2}\|\widehat{h}\|_{L^1(B(\xi_0,R))}\big].
\end{equation}
\end{definition}

The definition above shows that if $\|f\|_{Z}\leq 1$ then, for any $(k,j)\in\mathcal{J}$ one can decompose
\begin{equation}\label{sec5.8}
\widetilde{\varphi}^{(k)}_{j}\cdot P_kf=(2^{\alpha k}+2^{10k})^{-1}(g+h),
\end{equation}
where{\footnote{The support condition \eqref{sec5.81} can easily be achieved by starting with a decomposition 
$\widetilde{\varphi}^{(k)}_{j}\cdot P_kf=(2^{\alpha k}+2^{10k})^{-1}(g'+h')$ that minimizes the $B_{k,j}$ norm up to a constant, and then 
redefining $g:=g'\cdot \widetilde{\varphi}^{(k)}_{[j-1,j+1]}$ and $h:=h'\cdot \widetilde{\varphi}^{(k)}_{[j-1,j+1]}$.}}
\begin{equation}\label{sec5.81}
g=g\cdot \widetilde{\varphi}^{(k)}_{[j-2,j+2]},\qquad h=h\cdot \widetilde{\varphi}^{(k)}_{[j-2,j+2]},
\end{equation}
and
\begin{equation}\label{sec5.815}
\begin{split}
&2^{(1+\be)j}\|g\|_{L^2}+2^{(1/2-\beta)\widetilde{k}}\|\widehat{g}\|_{L^\infty}\lesssim 1,\\
&2^{(1-\be)j}\|h\|_{L^2}+\|\widehat{h}\|_{L^\infty}+
2^{\gamma j}\sup_{R\in[2^{-j},2^k],\,\xi_0\in\mathbb{R}^3}R^{-2}\|\widehat{h}\|_{L^1(B(\xi_0,R))}
\lesssim 2^{-10|k|}.
\end{split}
\end{equation}
In some of the easier estimates we will often use the weaker bound, obtained by setting $R=2^k$, 
\begin{equation}\label{sec5.82}
\begin{split}
&2^{(1+\be)j}\|g\|_{L^2}+2^{(1/2-\beta)\widetilde{k}}\|\widehat{g}\|_{L^\infty}\lesssim 1,\\
&2^{(1-\be)j}\|h\|_{L^2}+\|\widehat{h}\|_{L^\infty}
+2^{\gamma j}\|\widehat{h}\|_{L^1}\lesssim 2^{-8|k|}.
\end{split}
\end{equation}

We are now ready to state our main propositions which concern solutions $U=(U_e,U_i,U_b)$ of the system \eqref{pla51}-\eqref{pla52} derived in Proposition \ref{algebra}. We claim first that smooth solutions that start with data in the space $Z$ remain in the space $Z$, in a continuous way. More precisely:

\begin{proposition}\label{Norm0}
Assume $N_0=10^4$, $T_0\geq 1$, and $U=(U_e,U_i,U_b)\in C([0,T_0]:H^{N_0})$ is a solution of the system of equations \eqref{pla51}-\eqref{pla52}. Assume that, for some $t_0\in[0,T_0]$,
\begin{equation}\label{data1}
e^{it_0\Lambda_\sigma}U_\sigma(t_0)\in Z,\qquad\text{ for }\sigma\in\{e,i,b\}.
\end{equation}
Then there is 
\begin{equation*}
\tau=\tau\Big(T_0,\sup_{\sigma\in\{e,i,b\}}\|e^{it_0\Lambda_\sigma}U_\sigma(t_0)\|_{Z},\sup_{\sigma\in\{e,i,b\}}\sup_{t\in[0,T_0]}\|U_\sigma(t)\|_{H^{N_0}}\Big)>0
\end{equation*}
such that
\begin{equation}\label{data2}
\sup_{t\in [0,T_0]\cap[t_0,t_0+\tau]}\sup_{\sigma\in\{e,i,b\}}\|e^{it\Lambda_\sigma}U_\sigma(t)\|_{Z}\leq 2\sup_{\sigma\in\{e,i,b\}}\|e^{it_0\Lambda_\sigma}U_\sigma(t_0)\|_{Z},
\end{equation}
and the mapping $t\to e^{it\Lambda_\sigma}U_\sigma(t)$ is continuous from $[0,T_0]\cap[t_0,t_0+\tau]$ to $Z$, for any $\sigma\in\{e,i,b\}$.
\end{proposition}

The proof of Proposition \ref{Norm0} is very similar to the proof of Proposition 2.4 in \cite{IoPa2}. For any integer $J\geq 0$ and $f\in H^{N_0}$ we define
\begin{equation*}
\|f\|_{Z_{J}}:=\sup_{(k,j)\in\mathcal{J}}2^{\min(0,2J-2j)}\|\phii^{(k)}_j(x)\cdot P_kf(x)\|_{B_{k,j}},
\end{equation*}
compare with Definition \ref{MainDef}, and notice that
\begin{equation*}
 \|f\|_{Z_{J}}\leq\|f\|_{Z},\qquad \|f\|_{Z_{J}}\lesssim_J\|f\|_{H^{N_0}}.
\end{equation*}
The main point is show that if $t\leq t'\in [0,T_0]\cap [t_0,t_0+1]$ and $J\in\mathbb{Z}_+$ then
\begin{equation*}
 \sup_{\sigma\in\{e,i,b\}}\|e^{it'\Lambda_\sigma}U_\sigma(t')-e^{it\Lambda_\sigma}U_\sigma(t)\|_{Z_{J}}\leq \widetilde{C} |t'-t|(1+\sup_{s\in[t,t']}\sup_{\sigma\in\{e,i,b\}}\|e^{is\Lambda_\sigma}U_\sigma(s)\|_{Z_{J}})^2,
\end{equation*}
with a suitable constant $\widetilde{C}$ that may depend only on 
\begin{equation*}
T_0,\qquad\sup_{\sigma\in\{e,i,b\}}\sup_{t\in[0,T_0]}\|U_\sigma(t)\|_{H^{N_0}},\qquad\sup_{\sigma\in\{e,i,b\}}\|e^{it_0\Lambda_\sigma}U_\sigma(t_0)\|_Z.
\end{equation*}
This is very similar to the proof of the corresponding estimate (3.2) in \cite{IoPa2}, and we refer the reader there for the details.

The key proposition in the paper is the following bootstrap estimate:

\begin{proposition}\label{Norm}
Assume $N_0=10^4$, $T_0\geq 0$, and $U=(U_e,U_i,U_b)\in C([0,T_0]:H^{N_0})$ is a solution of the system of equations \eqref{pla51}-\eqref{pla52}.
Assume that
\begin{equation}\label{sec2}
\sup_{t\in [0,T_0]}\sup_{\sigma\in\{e,i,b\}}\|e^{it\Lambda_\sigma}U_\sigma(t)\|_{H^{N_0}\cap Z}\leq\delta_1\leq 1.
\end{equation}
Then
\begin{equation}\label{sec3}
\sup_{t\in [0,T_0]}\sup_{\sigma\in\{e,i,b\}}\|e^{it\Lambda_\sigma}U_\sigma(t)-U_\sigma(0)\|_{Z}\lesssim \delta_1^2,
\end{equation}
where the implicit constant in \eqref{sec3} may depend only on the constants $T,\varepsilon,C$.
\end{proposition}

We prove Proposition \ref{Norm} in sections \ref{normproof} and \ref{normproof2}. In the rest of this section we show how to use these propositions and the local theory to complete the proof of Theorem \ref{Main1}.

\subsection{Proof of Theorem \ref{Main1}}\label{MainSec5} Theorem \ref{Main1} is a consequence of Proposition \ref{Localexistence}, Proposition \ref{algebra}, Proposition \ref{Norm0}, Proposition \ref{Norm}, and a linear dispersive estimate. Indeed, assume that we start with data $(n^0,v^0,\rho^0, u^0, \tilde{E}^0,\tilde{B}_0)$ as in \eqref{maincond2}, where $\overline{\delta}$ is taken sufficiently small. Using first Proposition \ref{Localexistence} there is $T_1\geq 1$ and a unique solution $(n,v,\rho, u, \tilde{E},\tilde{B})\in C([0,T_1]:\widetilde{H}^{N_0})$ of the system \eqref{NewSys}, such that  $(n(0),v(0),\rho(0),u(0),\tilde{E}(0),\tilde{B}(0))=(n^0,v^0,\rho^0, u^0, \tilde{E}^0,\tilde{B}_0)$,
\begin{equation}\label{ki0}
\hbox{div}(E)(t)+n(t)-\rho(t)=0,\qquad \tilde{B}(t)=\varepsilon\nabla\times v(t)=-\nabla\times u(t),\qquad t\in[0,T_1],
\end{equation}
and
\begin{equation}\label{ki1}
\sup_{t\in[0,T_1]}\|(n(t),v(t),\rho(t),u(t),\tilde{E}(t),\tilde{B}(t))\|_{\widetilde{H}^{N_0}}\leq\delta_0^{3/4}.
\end{equation}

We can now apply Proposition \ref{algebra} and construct the complex variables $U_e,U_i,U_b\in C([0,T_1]:H^{N_0})$ as in \eqref{pla50}, which satisfy the dispersive system \eqref{pla51}-\eqref{pla52}, and the uniform bound
\begin{equation}\label{Uinfo}
\sup_{t\in[0,T_1]}\big(\|U_e(t)\|_{H^{N_0}}+\|U_i(t)\|_{H^{N_0}}+\|U_b(t)\|_{H^{N_0}}\big)\lesssim\delta_0^{3/4}.
\end{equation}
Moreover, using the definition \eqref{pla50}, the assumption \eqref{maincond2}, Lemma \ref{tech1}, and Lemma \ref{CZop}, we have
\begin{equation}\label{ki6}
\|U_e(0)\|_{Z}+\|U_i(0)\|_{Z}+\|U_b(0)\|_{Z}\lesssim\delta_0.
\end{equation}

We are now ready to apply Proposition \ref{Norm0}. Let $T_2$ denote the largest number in $(0,T_1]$ with the property that
\begin{equation*}
\sup_{t\in[0,T_2)}\big[\|e^{it\Lambda_e}U_e(t)\|_{Z}+\|e^{it\Lambda_i}U_i(t)\|_{Z}+\|e^{it\Lambda_b}U_b(t)\|_{Z}\big]\leq \delta_0^{3/4}.
\end{equation*}
Such a $T_2\in(0,T_1]$ exists, in view of \eqref{ki6} and Proposition \ref{Norm0}. We apply now Proposition \ref{Norm} on the intervals $[0,T_2(1-1/n)]$, $n=2,3,\ldots$, with $\delta_1\approx \delta_0^{3/4}$. It follows that 
\begin{equation*}
\sup_{t\in[0,T_2)}\big[\|e^{it\Lambda_e}U_e(t)\|_{Z}+\|e^{it\Lambda_i}U_i(t)\|_{Z}+\|e^{it\Lambda_b}U_b(t)\|_{Z}\big]\lesssim \delta_0.
\end{equation*}
Using again Proposition \ref{Norm0} it follows that $T_2=T_1$ and 
\begin{equation}\label{ki7}
\sup_{t\in[0,T_1]}\big[\|e^{it\Lambda_e}U_e(t)\|_{Z}+\|e^{it\Lambda_i}U_i(t)\|_{Z}+\|e^{it\Lambda_b}U_b(t)\|_{Z}\big]\lesssim \delta_0.
\end{equation}

We can now return to the physical variables $(n,v,\rho,u,\tilde{E},\tilde{B})$. Using the formulas in \eqref{pla50.2}, the bounds \eqref{ki7}, and the dispersive bounds \eqref{ok10} it follows that, for any $t\in[0,T_1]$ and $|\al|\leq 4$,
\begin{equation}\label{ki8}
\begin{split}
(1+t)^{1+\beta/2}\big[&\|D^\al_x n(t)\|_{L^\infty}+\|D^\al_x \rho(t)\|_{L^\infty}+\|D^\al_x v(t)\|_{L^\infty}\\
&+\|D^\al_x u(t)\|_{L^\infty}+\|D^\al_x \tilde{E}(t)\|_{L^\infty}+\|D^\al_x \tilde{B}(t)\|_{L^\infty}\big]\lesssim \delta_0.
\end{split}
\end{equation}
Recalling the definition \eqref{Addpla0.5} and the energy estimate \eqref{Addpla2}, it follows that
\begin{equation*}
\sup_{t\in[0,T_1]}\mathcal{E}_{N_0}(t)\lesssim \delta_0.
\end{equation*}
As a consequence, if the solution $(n,v,\rho,u,\tilde{E},\tilde{B})$ satisfies the bound \eqref{ki1} on some interval $[0,T_1]$, then it has to satisfy the stronger bound
\begin{equation*}
\sup_{t\in[0,T_1]}\|(n(t),v(t),\rho(t),u(t),\tilde{E}(t),\tilde{B}(t))\|_{\widetilde{H}^{N_0}}\lesssim\delta_0.
\end{equation*}
Therefore the solution can be extended globally, and the desired bound \eqref{mainconcl2.1} follows using also \eqref{ki8}. This completes the proof of Theorem \ref{Main1}.

\section{Proof of Proposition \ref{Norm}, I: nonresonant interactions}\label{normproof}

In this section we start the proof of Proposition \ref{Norm}. We derive first several new formulas describing the solutions $U_\sigma$. 

\subsection{Renormalizations} The equations \eqref{pla51}-\eqref{pla52} give
\begin{equation}\label{norm2}
[\partial_t+i\Lambda_\sigma(\xi)]\widehat{U_{\sigma}}(\xi,t)=c\sum_{\mu,\nu\in\mathcal{I}_0}\int_{\mathbb{R}^3}m_{\sigma;\mu,\nu}(\xi,\eta)\widehat{U_\mu}(\xi-\eta,t)\widehat{U_\nu}(\eta,t)\,d\eta,
\end{equation}
for $\sigma\in\{i,e,b\}$. For any $\mu\in\mathcal{I}_0$ let $\iota_\mu\in\{+,-\}$ denote its sign and let $\sigma_\mu\in\{i,e,b\}$ denote its component, i.e.
\begin{equation}\label{trivi}
\begin{split}
&\iota_{i+}=\iota_{e+}=\iota_{b+1}=\iota_{b+2}=\iota_{b+3}:=+,\qquad \iota_{i-}=\iota_{e-}=\iota_{b-1}=\iota_{b-2}=\iota_{b-3}:=-,\\
&\sigma_{i+}=\sigma_{i-}:=i,\qquad\sigma_{e+}=\sigma_{e-}:=e,\qquad \sigma_{b+1}=\sigma_{b+2}=\sigma_{b+3}=\sigma_{b-1}=\sigma_{b-2}=\sigma_{b-3}:=b.
\end{split}
\end{equation}
Let
\begin{equation*}
\begin{split}
&V_{\sigma}(t):=e^{it\Lambda_\sigma}U_{\sigma}(t),\qquad\sigma\in\{i,e,b\},\\
&\widetilde{\Lambda}_{\mu}:=\iota_\mu\Lambda_{\sigma_\mu},\qquad V_{\mu}(t):=e^{it\widetilde{\Lambda}_\mu}U_{\mu}(t),\qquad \mu\in\mathcal{I}_0.
\end{split}
\end{equation*}
The equations \eqref{norm2} are equivalent to
\begin{equation}\label{norm3}
\begin{split}
\frac{d}{dt}[\widehat{V_{\sigma}}(\xi,t)]&=c\sum_{\mu,\nu\in\mathcal{I}_0}\int_{\mathbb{R}^3}e^{it[\Lambda_\sigma(\xi)-\widetilde{\Lambda}_{\mu}(\xi-\eta)-\widetilde{\Lambda}_{\nu}(\eta)]}m_{\sigma;\mu,\nu}(\xi,\eta)\widehat{V_\mu}(\xi-\eta,t)\widehat{V_\nu}(\eta,t)\,d\eta\\
&=c\sum_{\mu,\nu\in\mathcal{I}_0}\mathcal{F}[Q^{\sigma;\mu,\nu}_t(V_\mu(t),V_\nu(t))](\xi),
\end{split}
\end{equation}
where, by definition,
\begin{equation}\label{norm4.0}
\mathcal{F}[Q^{\sigma;\mu,\nu}_s(f,g)](\xi):=\int_{\mathbb{R}^3}e^{is[\Lambda_\sigma(\xi)-\widetilde{\Lambda}_{\mu}(\xi-\eta)-\widetilde{\Lambda}_{\nu}(\eta)]}m_{\sigma;\mu,\nu}(\xi,\eta)\widehat{f}(\xi-\eta)\widehat{g}(\eta)\,d\eta.
\end{equation}
Therefore, for any $t\in[0,T_0]$ and $\sigma\in\{i,e,b\}$,
\begin{equation}\label{norm4}
\widehat{V_{\sigma}}(\xi,t)-\widehat{V_{\sigma}}(\xi,0)=c\sum_{\mu,\nu\in\mathcal{I}_0}\int_0^t\int_{\mathbb{R}^3}
e^{is[\Lambda_\sigma(\xi)-\widetilde{\Lambda}_{\mu}(\xi-\eta)-\widetilde{\Lambda}_{\nu}(\eta)]}m_{\sigma;\mu,\nu}(\xi,\eta)
\widehat{V_\mu}(\xi-\eta,s)\widehat{V_\nu}(\eta,s)\,d\eta ds
\end{equation}

The desired bound \eqref{sec3} is equivalent to proving that
\begin{equation}\label{nh1}
\|V_{\sigma}(t)-V_{\sigma}(0)\|_{Z}\lesssim \delta_1^2,
\end{equation}
for any $t\in[0,T_0]$ and any $\sigma\in\{i,e,b\}$. Given $t\in[0,T_0]$, we fix a suitable decomposition of the function $\mathbf{1}_{[0,t]}$, i.e. we fix functions $q_0,\ldots,q_{L+1}:\mathbb{R}\to[0,1]$, $|L-\log_2(2+t)|\leq 2$, with the properties
\begin{equation}\label{nh2}
\begin{split}
&\sum_{m=0}^{L+1}q_l(s)=\mathbf{1}_{[0,t]}(s),\qquad \mathrm{supp}\,q_0\subseteq [0,2], \qquad \mathrm{supp}\,q_{L+1}\subseteq [t-2,t],\qquad\mathrm{supp}\,q_m\subseteq [2^{m-1},2^{m+1}],\\
&q_m\in C^1(\mathbb{R})\qquad\text{ and }\qquad \int_0^t|q'_m(s)|\,ds\lesssim 1\qquad\text{ for }m=1,\ldots,L.
\end{split}
\end{equation}

Recall the conclusions of Lemma \ref{tech1.3}. Using also Lemma \ref{CZop} and the formula \eqref{norm4}, for \eqref{nh1} it suffices to prove the following proposition.

\begin{proposition}\label{reduced1} Assume $t\in[0,T_0]$ is fixed and define the functions $q_m$ as in \eqref{nh2}. For any $\sigma\in\{i,e,b\}$, $\mu,\nu\in\mathcal{I}_0$ we define the bilinear operators $T_{m}^{\sigma;\mu,\nu}$ by 
\begin{equation}\label{nh5}
\mathcal{F}\big[T_{m}^{\sigma;\mu,\nu}(f,g)\big](\xi):=\int_{\mathbb{R}}\int_{\mathbb{R}^3}e^{is[\Lambda_\sigma(\xi)-\widetilde{\Lambda}_{\mu}(\xi-\eta)-\widetilde{\Lambda}_{\nu}(\eta)]}q_m(s)\cdot \widehat{f}(\xi-\eta,s)\widehat{g}(\eta,s)\,d\eta ds.
\end{equation}
For any $\mu\in\mathcal{I}_0$ we define functions $f_\mu:\mathbb{R}^3\times[0,T_0]\to\mathbb{C}$,
\begin{equation}\label{nh4.5}
f_\mu:=\delta_1^{-1}Q_\mu V_\mu,
\end{equation}
where $Q_\mu f:=\mathcal{F}^{-1}(q_\mu\cdot \widehat{f})$ for some $q_\mu\in\mathcal{S}^{100}$ with $\|q_\mu\|_{\mathcal{S}^{100}}\leq 1$. We decompose
\begin{equation}\label{ok20}
f_\mu=\sum_{k'\in\mathbb{Z}}\sum_{j'\geq \max(-k',0)}P_{[k'-2,k'+2]}(\phii_{j'}^{(k')}\cdot P_{k'}f_\mu)=\sum_{(k',j')\in\mathcal{J}}f_{k',j'}^\mu.
\end{equation}
For any $k\in\mathbb{Z}$ let
\begin{equation*}
k_i:=\min(k,0),\qquad k_e=k_b:=0.
\end{equation*}
Then
\begin{equation}\label{nh3}
\sum_{(k_1,j_1),(k_2,j_2)\in\mathcal{J}}(1+2^k)2^{k_\sigma}\big\|\widetilde{\varphi}^{(k)}_j\cdot P_kT_{m}^{\sigma;\mu,\nu}(f^\mu_{k_1,j_1},f^\nu_{k_2,j_2})\big\|_{B_{k,j}}\lesssim 2^{-\beta^4 m}
\end{equation}
for any fixed
\begin{equation}\label{nh4}
\sigma\in\{i,e,b\},\quad\mu,\nu\in\mathcal{I}_0,\quad (k,j)\in \mathcal{J},\quad m\in\{0,\ldots,L+1\}.
\end{equation}
\end{proposition}

It follows from the definition that
\begin{equation}\label{nh9.2}
\begin{split}
&T^{\sigma;\mu,\nu}_m(f,g)=\int_\mathbb{R}q_m(s)\widetilde{T}^{\sigma;\mu,\nu}_s(f(s),g(s))\,ds,\\
&\mathcal{F}\big[\widetilde{T}^{\sigma;\mu,\nu}_s(f',g')\big](\xi):=\int_{\mathbb{R}^3}e^{is[\Lambda_\sigma(\xi)-\widetilde{\Lambda}_{\mu}(\xi-\eta)-\widetilde{\Lambda}_{\nu}(\eta)]}\cdot \widehat{f'}(\xi-\eta)\widehat{g'}(\eta)\,d\eta.
\end{split}
\end{equation}
For $\sigma\in\{i,e,b\}$ and $\mu,\nu\in\mathcal{I}_0$ we define also the smooth functions $\Phi^{\sigma;\mu,\nu}:\mathbb{R}^3\times\mathbb{R}^3\to\mathbb{R}$ and $\Xi^{\mu,\nu}:\mathbb{R}^3\times\mathbb{R}^3\to\mathbb{R}^3$,
\begin{equation}\label{PsiXi}
\begin{split}
\Phi^{\sigma;\mu,\nu}(\xi,\eta)&:=\Lambda_\sigma(\xi)-\widetilde{\Lambda}_{\mu}(\xi-\eta)-\widetilde{\Lambda}_{\nu}(\eta)=\Lambda_\sigma(\xi)-\iota_\mu\Lambda_{\sigma_\mu}(\xi-\eta)-\iota_\nu\Lambda_{\sigma_\nu}(\eta),\\
\Xi^{\mu,\nu}(\xi,\eta)&:=(\nabla_\eta\Phi^{\sigma;\mu,\nu})(\xi,\eta)=-\iota_\mu(\nabla\Lambda_{\sigma_\mu})(\eta-\xi)-\iota_\nu(\nabla\Lambda_{\sigma_\nu})(\eta).
\end{split}
\end{equation}

In view of Lemma \ref{CZop} and the main hypothesis \eqref{sec2}, we have
\begin{equation}\label{nh5.5}
\sup_{t\in[0,T_0]}\|f_\mu(t)\|_{H^{N_0}\cap Z}\lesssim 1.
\end{equation}
for functions $f_\mu$ defined as in \eqref{nh4.5}. Letting
\begin{equation}\label{nh5.6}
Ef^\mu_{k',j'}(s):=e^{-is\widetilde{\Lambda}_\mu}f^\mu_{k',j'}(s),
\end{equation}
it follows from Lemma \ref{tech1.5} that for any $\mu\in\mathcal{I}_0$ and $s\in[0,T_0]$,
\begin{equation}\label{nh9}
\begin{split}
&\sum_{j'\geq \max(-k',0)}(\|Ef^\mu_{k',j'}(s)\|_{L^2}+\|f^\mu_{k',j'}(s)\|_{L^2})\lesssim \min(2^{-(N_0-1)k'},2^{(1+\beta-\alpha)k'}),\\
&\sum_{j'\geq \max(-k',0)}\|Ef^\mu_{k',j'}(s)\|_{L^\infty}\lesssim \min(2^{-6k'},2^{(1/2-\beta-\alpha)k'})(1+s)^{-1-\beta},\\
&\sup_{\xi\in\mathbb{R}^3}\big|D^\rho_\xi\widehat{f^\mu_{k',j'}}(\xi,s)\big|\lesssim_{|\rho|} (2^{\alpha k'}+2^{10k'})^{-1}\cdot 2^{-(1/2-\beta)\widetilde{k'}}2^{|\rho|j'}.
\end{split}
\end{equation}
Sometimes, we will also need the more precise bounds
\begin{equation}\label{nh9.1}
\|Ef^\mu_{k',j'}(s)\|_{L^2}+\|f^\mu_{k',j'}(s)\|_{L^2}\lesssim (2^{\alpha k'}+2^{10k'})^{-1}2^{2\beta\widetilde{k'}}2^{-(1-\be)j'},
\end{equation}
and
\begin{equation}\label{equl}
\|Ef_{k',j'}^\mu(s)\|_{L^\infty}\lesssim \min(2^{\beta k'},2^{-6k'})(1+s)^{-(5/4-10\beta)}2^{(1/4-11\beta)j'},
\end{equation}
for any $(k',j')\in\mathcal{J}$. The last bound follows using \eqref{mk15.6}--\eqref{mk15.67}, and recalling that $\alpha\in[0,\beta]$.

To integrate by parts in time, i.e. apply the method of normal forms, we need suitable information on the derivatives $\partial_sf^\mu_{k',j'}$. It follows from \eqref{norm3} and Lemma \ref{ders} that, for any $(k',j')\in\mathcal{J}$, $\mu\in\mathcal{I}_0$, and $s\in[0,T_0]$,
\begin{equation}\label{ok50repeat}
\|(\partial_sf^\mu_{k',j'})(s)\|_{L^2}\lesssim 2^{k'_\sigma}\min[(1+s)^{-1-\beta},2^{3k'/2}]\cdot\min [1,2^{-(N_0-5)k'}].
\end{equation}
Moreover
\begin{equation}\label{derv1repeat}
\text{ if }\quad 2^{k'}\in[2^{-D},2^D]\text{ and }\sigma\in\{e,b\}\quad\text{ or }\quad 2^{k'}\in(0,2^D]\text{ and }\sigma=i,
\end{equation}
then
\begin{equation}\label{derv2repeat}
\|(\partial_s\widehat{f^\mu_{k',j'}})(s)\|_{L^\infty}\lesssim (1+s)^{-1+\beta/10}2^{-k'}.
\end{equation}

\subsection{Proof of Proposition \ref{reduced1}} We will prove the key bound \eqref{nh3} in several steps. 
The main ingredients in the proof are the estimates \eqref{nh5.5}--\eqref{nh9.1} above, together with \eqref{ok50}. 
In this subsection we start by considering some of the easier cases, and reduce matters to proving Proposition \ref{reduced2} below. 
In all the cases analyzed in this subsection we can in fact control the stronger norm $B^{\sigma,1}_{k,j}$, see Definition \ref{MainDef}, instead of the required $B^\sigma_{k,j}$ norm.

\begin{lemma}\label{BigBound1}
With $D=D(\varepsilon,T,C_b)$ sufficiently large, the estimate 
\begin{equation}\label{nh7.5}
\sum_{(k_1,j_1),(k_2,j_2)\in\mathcal{J}}(1+2^k)2^{k_\sigma}\big\|\widetilde{\varphi}^{(k)}_j\cdot P_kT_{m}^{\sigma;\mu,\nu}(f^\mu_{k_1,j_1},f^\nu_{k_2,j_2})\big\|_{B^{1}_{k,j}}\lesssim 2^{-\beta^4 m}
\end{equation}
holds if
\begin{equation}\label{nh7}
j\leq \beta m/2+N'_0k_++D^2,\qquad\text{ where }\qquad N'_0:=2N_0/3-10.
\end{equation}
\end{lemma}

\begin{proof}[Proof of Lemma \ref{BigBound1}] We observe that, in view of Definition \ref{MainDef},
\begin{equation}\label{ok1}
\|\phii^{(k)}_j\cdot P_kh\|_{B_{k,j}^1}\lesssim (2^{\alpha k}+2^{10 k})\cdot 2^{3j/2}2^{(1/2-\beta)\widetilde{k}}\|\phii^{(k)}_j\cdot P_kh\|_{L^2}.
\end{equation}
Therefore it suffices to prove that
\begin{equation}\label{nh8}
\sum_{(k_1,j_1),(k_2,j_2)\in\mathcal{J}}(1+2^{k})(2^{\alpha k}+2^{10 k})2^{3j/2}2^{(1/2-\beta)\widetilde{k}}\big\|P_kT_{m}^{\sigma;\mu,\nu}(f_{k_1,j_1}^\mu,f_{k_2,j_2}^\nu)\big\|_{L^2}\lesssim 2^{-\beta^4 m}.
\end{equation}

Recalling the definition \eqref{nh5.6}, it is easy to see that
\begin{equation*}
\mathcal{F}\big[P_kT_{m}^{\sigma;\mu,\nu}(f_{k_1,j_1}^\mu,f_{k_2,j_2}^\nu)\big](\xi)=\int_{\mathbb{R}}\int_{\mathbb{R}^3}\varphi_k(\xi)e^{is\Lambda_\sigma(\xi)}q_m(s)\widehat{Ef_{k_1,j_1}^\mu}(\xi-\eta,s)\widehat{Ef^\nu_{k_2,j_2}}(\eta,s)\,d\eta ds.
\end{equation*}
Therefore, using Plancherel theorem,
\begin{equation}\label{nh9.5}
\begin{split}
&\big\|P_kT_{m}^{\sigma;\mu,\nu}(f_{k_1,j_1}^\mu,f_{k_2,j_2}^\nu)\big\|_{L^2}\\
&\lesssim \min\Big(\int_{\mathbb{R}}q_m(s)\|Ef_{k_1,j_1}^\mu(s)\|_{L^2}\|Ef_{k_2,j_2}^\nu(s)\|_{L^\infty}\,ds,\int_{\mathbb{R}}q_m(s)\|Ef_{k_1,j_1}^\mu(s)\|_{L^\infty}\|Ef_{k_2,j_2}^\nu(s)\|_{L^2}\,ds\Big).
\end{split}
\end{equation}
Using now \eqref{nh9} and recalling the properties of the functions $q_m$ (see \eqref{nh2}), 
\begin{equation}\label{nh10}
\sum_{(k_1,j_1),(k_2,j_2)\in\mathcal{J}}(1+2^{k})\big\|P_kT_{m}^{\sigma;\mu,\nu}(f_{k_1,j_1}^\mu,f_{k_2,j_2}^\nu)\big\|_{L^2}\lesssim 2^{-(N_0-4)k_+}2^{-\beta m}.
\end{equation}

It follows that the left-hand side of \eqref{nh8} is dominated by
\begin{equation*}
2^{-\beta m}2^{(1/2-\beta+\alpha)k}2^{3j/2}
\end{equation*}
when $k\leq 0$, and by
\begin{equation*}
2^{-(N_0-15)k}2^{-\beta m}2^{3j/2}
\end{equation*}
when $k\geq 0$. The bound \eqref{nh8} follows if $j\leq \beta m/2+(2N_0/3-10)k_++D^2$, as desired.
\end{proof}

\begin{lemma}\label{BigBound2}
Assume that
\begin{equation}\label{nh30}
j\geq \beta m/2+N'_0k_++D^2.
\end{equation}
Then, with the same notation as before,
\begin{equation}\label{nh31}
\sum_{(k_1,j_1),(k_2,j_2)\in\mathcal{J},\,\,\max(k_1,k_2)\geq j/N'_0}(1+2^k)2^{k_\sigma}\big\|\widetilde{\varphi}^{(k)}_j\cdot P_kT_{m}^{\sigma;\mu,\nu}(f_{k_1,j_1}^\mu,f_{k_2,j_2}^\nu)\big\|_{B^{1}_{k,j}}\lesssim 2^{-\beta^4 m},
\end{equation}
\begin{equation}\label{nh30.7}
\sum_{(k_1,j_1),(k_2,j_2)\in\mathcal{J},\,\min(k_1,k_2)\leq -10j}(1+2^k)2^{k_\sigma}\big\|\widetilde{\varphi}^{(k)}_j\cdot P_kT_{m}^{\sigma;\mu,\nu}(f_{k_1,j_1}^\mu,f_{k_2,j_2}^\nu)\big\|_{B^{1}_{k,j}}\lesssim 2^{-\beta^4 m},
\end{equation}
and
\begin{equation}\label{nh30.8}
\sum_{(k_1,j_1),(k_2,j_2)\in\mathcal{J},\,\max(j_1,j_2)\geq 10j}(1+2^k)2^{k_\sigma}\big\|\widetilde{\varphi}^{(k)}_j\cdot P_kT_{m}^{\sigma;\mu,\nu}(f_{k_1,j_1}^\mu,f_{k_2,j_2}^\nu)\big\|_{B^{1}_{k,j}}\lesssim 2^{-\beta^4 m}.
\end{equation}
\end{lemma}

\begin{proof}[Proof of Lemma \ref{BigBound2}] Using  \eqref{nh9}, \eqref{ok1}, and \eqref{nh9.5}, the left-hand side of \eqref{nh31} is dominated by
\begin{equation*}
\begin{split}
\sum_{(k_1,j_1),(k_2,j_2)\in\mathcal{J},\,\,\max(k_1,k_2)\geq j/N'_0}&(1+2^{k})(2^{\alpha k}+2^{10 k})2^{3j/2}2^{(1/2-\beta)\widetilde{k}}\big\|P_kT_{m}^{\sigma;\mu,\nu}(f_{k_1,j_1}^\mu,f_{k_2,j_2}^\nu)\big\|_{L^2}\\
&\lesssim 2^{-\beta m}2^{-(N_0-6)j/N_0^\prime}2^{3j/2}2^{(1/2-\beta)\widetilde{k}},
\end{split}
\end{equation*}
which clearly suffices, in view of \eqref{nh30}. Similarly, the left-hand side of \eqref{nh30.7} is dominated by
\begin{equation*}
\begin{split}
\sum_{(k_1,j_1),(k_2,j_2)\in\mathcal{J},\,\,\min(k_1,k_2)\leq -10j}(1+2^{k})&(2^{\alpha k}+2^{10 k})\cdot 2^{3j/2}2^{(1/2-\beta)\widetilde{k}}\big\|P_kT_{m}^{\sigma;\mu,\nu}(f_{k_1,j_1}^\mu,f_{k_2,j_2}^\nu)\big\|_{L^2}\\
&\lesssim 2^{-\beta m}2^{-3j}\cdot(2^{\alpha k}+2^{10 k})2^{3j/2}2^{(1/2-\beta)\widetilde{k}},
\end{split}
\end{equation*}
which clearly suffices. Finally, using the more precise bound \eqref{nh9.1}, the left-hand side of \eqref{nh30.8} is dominated by
\begin{equation*}
\begin{split}
\sum_{(k_1,j_1),(k_2,j_2)\in\mathcal{J},\,\,\max(j_1,j_2)\geq 10j}(1+2^{k})&(2^{\alpha k}+2^{10 k})\cdot 2^{3j/2}2^{(1/2-\beta)\widetilde{k}}\big\|P_kT_{m}^{\sigma;\mu,\nu}(f_{k_1,j_1}^\mu,f_{k_2,j_2}^\nu)\big\|_{L^2}\\
&\lesssim 2^{-\beta m}2^{-3j}\cdot(2^{\alpha k}+2^{10 k})2^{3j/2}2^{(1/2-\beta)\widetilde{k}},
\end{split}
\end{equation*}
which clearly suffices.
\end{proof}

We examine the conclusions of Lemma \ref{BigBound1} and Lemma \ref{BigBound2}, and notice that Proposition \ref{reduced1} follows from Proposition \ref{reduced2} below.

\begin{proposition}\label{reduced2}
With the same notation as in Proposition \ref{reduced1}, we have
\begin{equation}\label{ok60}
(1+2^k)2^{k_\sigma}\big\|\widetilde{\varphi}^{(k)}_j\cdot P_kT_{m}^{\sigma;\mu,\nu}(f_{k_1,j_1}^\mu,f_{k_2,j_2}^\nu)\big\|_{B_{k,j}}\lesssim 2^{-\beta^4 (m+j)},
\end{equation}
for any fixed $\mu,\nu\in\mathcal{I}_0$, $(k,j),(k_1,j_1),(k_2,j_2)\in\mathcal{J}$, and $m\in[0,L+1]\cap\mathbb{Z}$, satisfying
\begin{equation}\label{ok61}
j\geq \beta m/2+N'_0k_++D^2,\qquad -10j\leq k_1,\,k_2\leq j/N'_0,\qquad \max(j_1,j_2)\leq 10j.
\end{equation}
\end{proposition}

\subsection{Proof of Proposition \ref{reduced2}} In this subsection we will show that proving Proposition \ref{reduced2} can be further reduced 
to proving Proposition \ref{reduced3} below. The arguments are more complicated than before, and we need to examine our bilinear operators 
more carefully; however, in all cases discussed in this subsection we can still control the stronger $B^{1}_{k,j}$ norms. 

We notice that we are looking to prove the bound \eqref{ok60} for {\it{fixed}} $k,j,k_1,j_1,k_2,j_2,m$. 
We will consider several cases, depending on the relative sizes of these parameters. 

\begin{lemma}\label{BigBound3}
The bound \eqref{ok60} holds provided that \eqref{ok61} holds and, in addition,
\begin{equation}\label{nh32.5}
j\geq \max(m+D,-k(1+\beta^2)+D).
\end{equation}
\end{lemma}

\begin{proof}[Proof of Lemma \ref{BigBound3}] Using definition \eqref{sec5.3} it suffices to prove that
\begin{equation}\label{ok30}
\begin{split}
&(1+2^{k})(2^{\alpha k}+2^{10k})\cdot 2^{(1+\beta)j}\big\|\widetilde{\varphi}^{(k)}_j\cdot P_kT_{m}^{\sigma;\mu,\nu}(f_{k_1,j_1}^\mu,f_{k_2,j_2}^\nu)\big\|_{L^2}\\
&+(1+2^{k})(2^{\alpha k}+2^{10k})\cdot 2^{(1/2-\beta)\widetilde{k}}\big\|\mathcal{F}[\widetilde{\varphi}^{(k)}_j\cdot P_kT_{m}^{\sigma;\mu,\nu}(f_{k_1,j_1}^\mu,f_{k_2,j_2}^\nu)]\big\|_{L^\infty}\lesssim 2^{-\beta^4 (m+j)}.
\end{split}
\end{equation}

Assume first that
\begin{equation}\label{ok31}
\min(j_1,j_2)\leq (1-\beta^2)j.
\end{equation}
By symmetry, we may assume that $j_1\leq (1-\beta^2)j$ and write
\begin{equation*}
\begin{split}
&\widetilde{\varphi}^{(k)}_j(x)\cdot P_kT_{m}^{\sigma;\mu,\nu}(f_{k_1,j_1}^\mu,f_{k_2,j_2}^\nu)(x)\\
&=c\widetilde{\varphi}^{(k)}_j(x)\int_{\mathbb{R}^3}\int_{\mathbb{R}}\int_{\mathbb{R}^3}\varphi_k(\xi)e^{ix\cdot\xi}e^{is[\Lambda_\sigma(\xi)-\widetilde{\Lambda}_{\mu}(\xi-\eta)-\widetilde{\Lambda}_{\nu}(\eta)]}q_m(s)\cdot \widehat{f_{k_1,j_1}^\mu}(\xi-\eta,s)\widehat{f_{k_2,j_2}^\nu}(\eta,s)\,d\eta ds d\xi.
\end{split}
\end{equation*}
We examine the integral in $\xi$ in the formula above. We recall the assumptions \eqref{ok61}, \eqref{nh32.5}, and \eqref{ok31}, and the last bound in \eqref{nh9}. Notice that, using only the assumption \eqref{nh32.5} and the definition \eqref{operators1} (see also Lemma \ref{tech99}),
\begin{equation*}
\Big|\nabla_\xi\big[x\cdot\xi+s[\Lambda_\sigma(\xi)-\widetilde{\Lambda}_{\mu}(\xi-\eta)-\widetilde{\Lambda}_{\nu}(\eta)]\big]\Big|\geq |x|-s\big|\nabla_\xi[\Lambda_\sigma(\xi)-\widetilde{\Lambda}_{\mu}(\xi-\eta)\big]\big|\geq 2^{j-10}.
\end{equation*}
We apply Lemma \ref{tech5} (with $K\approx 2^j$, $\eps\approx \min(2^{-j_1},2^k)$) to conclude that
\begin{equation*}
\big|\widetilde{\varphi}^{(k)}_j(x)\cdot P_kT_{m}^{\sigma;\mu,\nu}(f_{k_1,j_1}^\mu,f_{k_2,j_2}^\nu)(x)\big|\lesssim 2^{-10j}|\widetilde{\varphi}^{(k)}_j(x)|,
\end{equation*}
and the desired bounds \eqref{ok30} follow easily.

Assume now that 
\begin{equation}\label{ok33}
\min(j_1,j_2)\geq (1-\beta^2)j.
\end{equation}
By symmetry, we may assume that $k_1\leq k_2$. We prove first the bound on the second term in the left-hand side of \eqref{ok30}: using \eqref{nh9.1} we estimate
\begin{equation*}
\begin{split}
&(1+2^{k})(2^{\alpha k}+2^{10k})\cdot 2^{(1/2-\beta)\widetilde{k}}\|\mathcal{F}[\widetilde{\varphi}^{(k)}_j\cdot P_kT_{m}^{\sigma;\mu,\nu}(f_{k_1,j_1}^\mu,f_{k_2,j_2}^\nu)]\big\|_{L^\infty}\\
&\lesssim (1+2^{k})(2^{\alpha k}+2^{10k})2^{(1/2-\beta)\widetilde{k}}\cdot 2^m\sup_{s\in[2^{m-1},2^{m+1}]}\|f_{k_1,j_1}^\mu(s)\|_{L^2}\|f_{k_2,j_2}^\nu(s)\|_{L^2}\\
&\lesssim (1+2^{k})(2^{\alpha k}+2^{10k})2^{(1/2-\beta)\widetilde{k}}2^{j}\cdot(2^{\alpha k_1}+2^{10k_1})^{-1}2^{2\beta\widetilde{k_1}}2^{-(1-\be)j_1}\cdot(2^{\alpha k_2}+2^{10k_2})^{-1}2^{2\beta\widetilde{k_2}}2^{-(1-\be)j_2}\\
&\lesssim (1+2^{k})2^j\cdot 2^{-\alpha k_1}\min(2^{(1+\beta)k_1},2^{-(1-\beta-\beta^2)j})\cdot 2^{-(1-\be-\beta^2)j}.
\end{split}
\end{equation*}
This suffices to prove the desired bound in \eqref{ok30}, as it can be easily seen by considering the cases $k_1\leq -j$ and $k_1\geq -j$.

Some more care is needed to prove the bound on the first term in the left-hand side of \eqref{ok30}. We recall that
\begin{equation*}
f_{k_1,j_1}^\mu=P_{[k_1-2,k_1+2]}(\phii_{j_1}^{(k_1)}\cdot P_{k_1}f_\mu)\quad\text{ and }\quad f_{k_2,j_2}^\nu=P_{[k_2-2,k_2+2]}(\phii_{j_2}^{(k_2)}\cdot P_{k_2}f_\nu).
\end{equation*}
Since $\|\phii_{j_1}^{(k_1)}\cdot P_{k_1}f_\mu(s)\|_{B_{k_1,j_1}}+\|\phii_{j_2}^{(k_2)}\cdot P_{k_2}f_\nu(s)\|_{B_{k_2,j_2}}\lesssim 1$, see \eqref{nh5.5}, we use \eqref{sec5.8}--\eqref{sec5.82} to decompose
\begin{equation}\label{ok36}
\begin{split}
&\phii^{(k_1)}_{j_1}\cdot P_{k_1}f_\mu(s)=(2^{\alpha k_1}+2^{10k_1})^{-1}[g_{k_1,j_1}^\mu(s)+h_{k_1,j_1}^\mu(s)],\\
&g_{k_1,j_1}^\mu(s)=g_{k_1,j_1}^\mu(s)\cdot \widetilde{\varphi}^{(k_1)}_{[j_1-2,j_1+2]},\qquad h_{k_1,j_1}^\mu(s)=h_{k_1,j_1}^\mu(s)\cdot \widetilde{\varphi}^{(k_1)}_{[j_1-2,j_1+2]},\\
&2^{(1+\be)j_1}\|g_{k_1,j_1}^\mu(s)\|_{L^2}+2^{(1/2-\beta)\widetilde{k_1}}\|\widehat{g_{k_1,j_1}^\mu}(s)\|_{L^\infty}\lesssim 1,\\
&2^{(1-\be)j_1}\|h_{k_1,j_1}^\mu(s)\|_{L^2}+\|\widehat{h_{k_1,j_1}^\mu}(s)\|_{L^\infty}+2^{\gamma j_1}\|\widehat{h_{k_1,j_1}^\mu}(s)\|_{L^1}\lesssim 2^{-8|k_1|},
\end{split}
\end{equation}
and
\begin{equation}\label{ok37}
\begin{split}
&\phii^{(k_2)}_{j_2}\cdot P_{k_2}f_\nu(s)=(2^{\alpha k_2}+2^{10k_2})^{-1}[g_{k_2,j_2}^\nu(s)+h_{k_2,j_2}^\nu(s)],\\
&g_{k_2,j_2}^\nu(s)=g_{k_2,j_2}^\nu(s)\cdot \widetilde{\varphi}^{(k_2)}_{[j_2-2,j_2+2]},\qquad h_{k_2,j_2}^\nu(s)=h_{k_2,j_2}^\nu(s)\cdot \widetilde{\varphi}^{(k_2)}_{[j_2-2,j_2+2]},\\
&2^{(1+\be)j_2}\|g_{k_2,j_2}^\nu(s)\|_{L^2}+2^{(1/2-\beta)\widetilde{k_2}}\|\widehat{g_{k_2,j_2}^\nu}(s)\|_{L^\infty}\lesssim 1,\\
&2^{(1-\be)j_2}\|h_{k_2,j_2}^\nu(s)\|_{L^2}+\|\widehat{h_{k_2,j_2}^\nu}(s)\|_{L^\infty}+2^{\gamma j_2}\|\widehat{h_{k_2,j_2}^\nu}(s)\|_{L^1}\lesssim 2^{-8|k_2|}.
\end{split}
\end{equation}
Using these decompositions and recalling the definition \eqref{nh9.2}, to prove the desired bound on the first term in the left-hand side of \eqref{ok30}, it suffices to prove that for any $s\in[2^{m-1},2^{m+1}]$
\begin{equation}\label{ok38}
\begin{split}
(1+2^{k})&(2^{\alpha k}+2^{10k})2^{(1+\beta)j}\cdot(2^{\alpha k_1}+2^{10k_1})^{-1}(2^{\alpha k_2}+2^{10k_2})^{-1}2^m\\
\Big[&\big\|\widetilde{\varphi}^{(k)}_j\cdot P_k\widetilde{T}_{s}^{\sigma;\mu,\nu}(P_{[k_1-2,k_1+2]}g_{k_1,j_1}^\mu(s),P_{[k_2-2,k_2+2]}g_{k_2,j_2}^\nu(s))\big\|_{L^2}\\
+&\big\|\widetilde{\varphi}^{(k)}_j\cdot P_k\widetilde{T}_{s}^{\sigma;\mu,\nu}(P_{[k_1-2,k_1+2]}g_{k_1,j_1}^\mu(s),P_{[k_2-2,k_2+2]}h_{k_2,j_2}^\nu(s))\big\|_{L^2}\\
+&\big\|\widetilde{\varphi}^{(k)}_j\cdot P_k\widetilde{T}_{s}^{\sigma;\mu,\nu}(P_{[k_1-2,k_1+2]}h_{k_1,j_1}^\mu(s),P_{[k_2-2,k_2+2]}g_{k_2,j_2}^\nu(s))\big\|_{L^2}\\
+&\big\|\widetilde{\varphi}^{(k)}_j\cdot P_k\widetilde{T}_{s}^{\sigma;\mu,\nu}(P_{[k_1-2,k_1+2]}h_{k_1,j_1}^\mu(s),P_{[k_2-2,k_2+2]}h_{k_2,j_2}^\nu(s))\big\|_{L^2}\Big]\lesssim 2^{-\beta^4(m+j)}.
\end{split}
\end{equation}

Recall that we assumed $k_1\leq k_2$; therefore we may also assume that $k\leq k_2+4$. Using \eqref{ok36}--\eqref{ok37} and recalling \eqref{ok33}, we estimate
\begin{equation*}
\begin{split}
\big\|P_k\widetilde{T}_{s}^{\sigma;\mu,\nu}(P_{[k_1-2,k_1+2]}g_{k_1,j_1}^\mu(s),P_{[k_2-2,k_2+2]}g_{k_2,j_2}^\nu(s))\big\|_{L^2}&\lesssim \|\mathcal{F}(P_{[k_1-2,k_1+2]}g_{k_1,j_1}^\mu)(s)\|_{L^1}\|g_{k_2,j_2}^\nu(s)\|_{L^2}\\
&\lesssim 2^{3k_1/2}2^{-(1+\beta)j_1}2^{-(1+\beta)j_2}\\
&\lesssim 2^{3k_1/2}2^{-(2+2\beta)(1-\beta^2)j},
\end{split}
\end{equation*}
\begin{equation*}
\begin{split}
\big\|P_k\widetilde{T}_{s}^{\sigma;\mu,\nu}(P_{[k_1-2,k_1+2]}h_{k_1,j_1}^\mu(s),P_{[k_2-2,k_2+2]}h_{k_2,j_2}^\nu(s))\big\|_{L^2}&\lesssim \|\widehat{h_{k_1,j_1}^\mu}(s)\|_{L^1}\|\widehat{h_{k_2,j_2}^\nu}(s)\|_{L^2}\\
&\lesssim 2^{-\gamma j_1}2^{-8|k_1|}2^{-(1-\beta)j_2}2^{-8|k_2|}\\
&\lesssim 2^{-8|k_1|}2^{-(2+2\beta)(1-\beta^2)j},
\end{split}
\end{equation*}
\begin{equation*}
\begin{split}
\big\|P_k\widetilde{T}_{s}^{\sigma;\mu,\nu}(P_{[k_1-2,k_1+2]}h_{k_1,j_1}^\mu(s),P_{[k_2-2,k_2+2]}g_{k_2,j_2}^\nu(s))\big\|_{L^2}&\lesssim \|\widehat{h_{k_1,j_1}^\mu}(s)\|_{L^1}\|\widehat{g_{k_2,j_2}^\nu}(s)\|_{L^2}\\
&\lesssim 2^{-\gamma j_1}2^{-8|k_1|}2^{-(1+\beta)j_2}\\
&\lesssim 2^{-8|k_1|}2^{-(2+2\beta)(1-\beta^2)j},
\end{split}
\end{equation*}
and
\begin{equation*}
\begin{split}
\big\|P_k\widetilde{T}_{s}^{\sigma;\mu,\nu}&(P_{[k_1-2,k_1+2]}g_{k_1,j_1}^\mu(s),P_{[k_2-2,k_2+2]}h_{k_2,j_2}^\nu(s))\big\|_{L^2}\\
&\lesssim \min\big(2^{3k_1/2}\|\widehat{g_{k_1,j_1}^\mu}(s)\|_{L^2}\|\widehat{h_{k_2,j_2}^\nu}(s)\|_{L^2},\|\widehat{g_{k_1,j_1}^\mu}(s)\|_{L^2}\|\widehat{h_{k_2,j_2}^\nu}(s)\|_{L^1}\big)\\
&\lesssim 2^{-(1+\beta)j_1}2^{-8|k_2|}\min\big(2^{-(1-\beta)j_2}2^{3k_1/2},2^{-\gamma j_2}\big)\\
&\lesssim 2^{-(2+2\beta)(1-\beta^2)j}2^{3k_1/4}2^{-8|k_2|}.
\end{split}
\end{equation*}
Since $2^m\lesssim 2^{j}$ and $(2^{\alpha k}+2^{10k})(2^{\alpha k_2}+2^{10k_2})^{-1}\lesssim 1$, the left-hand side of \eqref{ok38} is dominated by
\begin{equation*}
C(1+2^{k})2^{(1+\beta)j}\cdot(2^{\alpha k_1}+2^{10k_1})^{-1}2^{j}\cdot 2^{-(2+2\beta)(1-\beta^2)j}(2^{3k_1/2}+2^{3k_1/4}2^{-8|k_2|})\lesssim 2^{-2\beta j/3}(1+2^{k}),
\end{equation*}
which suffices since $2^{k}\lesssim 2^{j/N'_0}$. This completes the proof of the lemma.
\end{proof}

\begin{lemma}\label{BigBound4}
The bound \eqref{ok60} holds provided that \eqref{ok61} holds and, in addition,
\begin{equation}\label{ok40}
\max(m+D, j)\leq -k(1+\beta^2)+D.
\end{equation}
\end{lemma}

\begin{proof}[Proof of Lemma \ref{BigBound4}]
In view of the restrictions \eqref{ok40} and \eqref{ok61}, we may assume that $k\leq -D^2/2$. Using the definition, it is easy to see that
\begin{equation}\label{ok40.55}
\|\widetilde{\varphi}^{(k)}_j\cdot P_kh\|_{B^1_{k,j}}\lesssim (2^{\alpha k}+2^{10 k})2^{(1+\beta)j}2^{3k/2}\|\widehat{P_kh}\|_{L^\infty}.
\end{equation}
Therefore, it suffices to prove that
\begin{equation}\label{ok43}
2^{k_\sigma}2^{\alpha k}2^{(1+\beta)j}2^{3k/2}\big\|\mathcal{F}P_kT_{m}^{\sigma;\mu,\nu}(f_{k_1,j_1}^\mu,f_{k_2,j_2}^\nu)\big\|_{L^\infty}\lesssim 2^{-\beta^4 (m+j)}.
\end{equation}

Using \eqref{nh9.1} and recalling $\alpha\leq 2\beta$ we estimate
\begin{equation*}
\begin{split}
\big\|\mathcal{F}P_kT_{m}^{\sigma;\mu,\nu}(f_{k_1,j_1}^\mu,&f_{k_2,j_2}^\nu)\big\|_{L^\infty}\lesssim \int_{\mathbb{R}}q_m(s)\|f_{k_1,j_1}^\mu(s)\|_{L^2}\|f_{k_2,j_2}^\nu(s)\|_{L^2}\,ds\\
&\lesssim \|q_m\|_{L^1}\cdot(2^{\alpha k_1}+2^{10k_1})^{-1}2^{2\beta\widetilde{k_1}}2^{-(1-\beta)j_1}\cdot (2^{\alpha k_2}+2^{10k_2})^{-1}2^{2\beta\widetilde{k_2}}2^{-(1-\beta)j_2}\\
&\lesssim\|q_m\|_{L^1}\min(1,2^{-5k_1})2^{-(1-\beta)j_1}\cdot \min(1,2^{-5k_2})2^{-(1-\beta)j_2}.
\end{split}
\end{equation*}
Recalling the definitions \eqref{sec5.6} and the assumptions, the desired bound \eqref{ok43} follows if
\begin{equation*}
\sigma=i\,\,\,\text{ or }\,\,\,m=L+1\,\,\,\text{ or }\,\,\,m\leq (1-\beta)(j_1+j_2)-(1/2-\beta)k. 
\end{equation*}

It remains to prove the bound \eqref{ok43} in the case
\begin{equation}\label{ok45}
\sigma\in\{e,b\}\,\,\,\text{ and }\,\,\,m\in[1,L]\cap\mathbb{Z}\,\,\,\text{ and }\,\,\,m\geq -(1/2-\beta)k+(1-\beta)(j_1+j_2).
\end{equation}
Since $j_1+k_1\geq 0$, $j_2+k_2\geq 0$, and $k\leq -D^2/2$, the conditions \eqref{ok40} and \eqref{ok45} show that $k_1,k_2\geq k/4$ and $|k_1-k_2|\leq 10$. Using also \eqref{ok40}, for \eqref{ok43} it suffices to prove that, assuming \eqref{ok45},
\begin{equation}\label{ok46}
\big\|\mathcal{F}P_kT_{m}^{\sigma;\mu,\nu}(f_{k_1,j_1}^\mu,f_{k_2,j_2}^\nu)\big\|_{L^\infty}\lesssim 2^{-k(1/2+\alpha-\beta-2\beta^2)}.
\end{equation}

To prove \eqref{ok46} we would like to integrate by parts in $\eta$ and $s$ in the formula \eqref{ok40.5}. Recall the definitions
\begin{equation*}\label{ok40.5}
\mathcal{F}P_kT_{m}^{\sigma;\mu,\nu}(f_{k_1,j_1}^\mu,f_{k_2,j_2}^\nu)(\xi)=\varphi_k(\xi)\int_{\mathbb{R}}\int_{\mathbb{R}^3}e^{is\Phi^{\sigma;\mu,\nu}(\xi,\eta)}q_m(s)\widehat{f_{k_1,j_1}^\mu}(\xi-\eta,s)\widehat{f_{k_2,j_2}^\nu}(\eta,s)\,d\eta ds,
\end{equation*}
where
\begin{equation*}\label{ok40.6}
\Phi^{\sigma;\mu,\nu}(\xi,\eta)=\Lambda_\sigma(\xi)-\widetilde{\Lambda}_{\mu}(\xi-\eta)-\widetilde{\Lambda}_{\nu}(\eta).
\end{equation*}
We decompose
\begin{equation*}
\begin{split}
&\mathcal{F}P_kT_{m}^{\sigma;\mu,\nu}(f_{k_1,j_1}^\mu,f_{k_2,j_2}^\nu)(\xi)=G(\xi)+H(\xi),\\
&G(\xi):=\varphi_k(\xi)\int_{\mathbb{R}}\int_{\mathbb{R}^3}e^{is\Phi^{\sigma;\mu,\nu}(\xi,\eta)}\varphi(2^{20D}(1+2^{k_2})\Phi^{\sigma;\mu,\nu}(\xi,\eta))q_m(s)\widehat{f_{k_1,j_1}^\mu}(\xi-\eta,s)\widehat{f_{k_2,j_2}^\nu}(\eta,s)\,d\eta ds,\\
&H(\xi):=\varphi_k(\xi)\int_{\mathbb{R}}\int_{\mathbb{R}^3}e^{is\Phi^{\sigma;\mu,\nu}(\xi,\eta)}[1-\varphi(2^{20D}(1+2^{k_2})\Phi^{\sigma;\mu,\nu}(\xi,\eta))]q_m(s)\widehat{f_{k_1,j_1}^\mu}(\xi-\eta,s)\widehat{f_{k_2,j_2}^\nu}(\eta,s)\,d\eta ds.
\end{split}
\end{equation*}
The function $H$ can be estimated using integration by parts in $s$, \eqref{ok50repeat}, the assumptions \eqref{nh2}, and the bounds \eqref{nh9}. Indeed,
\begin{equation*}
\begin{split}
|H(\xi)|&\lesssim (1+2^{k_2})\sup_{s\in [2^{m-1},2^{m+1}]}\big[\big\|\widehat{f_{k_1,j_1}^\mu}(s)\big\|_{L^2}\big\|\widehat{f_{k_2,j_2}^\nu}(s)\big\|_{L^2}\\
&+2^m\big\|(\partial_s\widehat{f_{k_1,j_1}^\mu})(s)\big\|_{L^2}\big\|\widehat{f_{k_2,j_2}^\nu}(s)\big\|_{L^2}+2^m\big\|\widehat{f_{k_1,j_1}^\mu}(s)\big\|_{L^2}\big\|(\partial_s\widehat{f_{k_2,j_2}^\nu})(s)\big\|_{L^2}\big]\\
&\lesssim \min(1,2^{-(N_0-10)k_2}).
\end{split}
\end{equation*}
Therefore, for \eqref{ok46} it suffices to prove that
\begin{equation}\label{ok48}
\big\|G\big\|_{L^\infty}\lesssim 2^{-k(1/2+\alpha-\beta-2\beta^2)}.
\end{equation}

Recall the definitions \eqref{PsiXi},
\begin{equation}\label{ok47}
\Xi^{\mu,\nu}(\xi,\eta)=(\nabla_\eta\Phi^{\sigma;\mu,\nu})(\xi,\eta)=-\iota_1\nabla\Lambda_{\sigma_1}(\eta-\xi)-\iota_2\nabla\Lambda_{\sigma_2}(\eta),
\end{equation}
where
\begin{equation*}
\mu=(\sigma_1\iota_1),\qquad \nu=(\sigma_2\iota_2),\qquad \sigma_1,\sigma_2\in\{i,e,b\},\qquad \iota_1,\iota_2\in\{+,-\}.
\end{equation*}
For $l\in\mathbb{Z}$ let
\begin{equation}\label{ba1}
\begin{split}
G_{\leq l}(\xi):=\varphi_k(\xi)&\int_{\mathbb{R}}\int_{\mathbb{R}^3}\varphi_{(-\infty,l]}(\Xi^{\mu,\nu}(\xi,\eta))\cdot  e^{is\Phi^{\sigma;\mu,\nu}(\xi,\eta)}\\
&\varphi(2^{20D}(1+2^{k_2})\Phi^{\sigma;\mu,\nu}(\xi,\eta))q_m(s)\widehat{f_{k_1,j_1}^\mu}(\xi-\eta,s)\widehat{f_{k_2,j_2}^\nu}(\eta,s)\,d\eta ds.
\end{split}
\end{equation}
Let $G_l:=G_{\leq l}-G_{\leq l-1}$. In proving \eqref{ok48} we may assume that $j_1\leq j_2$. If $l\geq l_0=-20D-4\max(k_2,0)$ then we integrate by parts in $\eta$, using Lemma \ref{tech5} with $K\approx 2^{m+l}$ and $\eps^{-1}\approx 2^{j_2}+2^{-\min(l,0)-\min(k_2,0)}+2^{k_2}$. Using also the last bound in \eqref{nh9}, \eqref{ok40}, and \eqref{ok45} to ensure $\eps K\ge 2^{\beta^2m}$, it follows that
\begin{equation}\label{ba2}
\sum_{l\geq l_0+1}\|G_l\|_{L^\infty}\lesssim (1+2^{5k_2})^{-1}.
\end{equation}
It remains to estimate $\|G_{\leq l_0}\|_{L^\infty}$. Since $\sigma\ne i$, it follows from Proposition \ref{PropABC} that $G_{\leq l_0}\equiv 0$. This completes the proof of the lemma.
\end{proof}

\begin{lemma}\label{BigBound5}
The bound \eqref{ok60} holds provided that \eqref{ok61} holds and, in addition,
\begin{equation}\label{ok110}
j\leq m+D\quad\text{ and }\quad \max(j_1,j_2)\geq (1-\beta/10)m+k_\sigma,
\end{equation}
or
\begin{equation}\label{ok110.5}
j\leq m+D\quad\text{ and }\quad \min(k_1,k_2)\leq -9m/10.
\end{equation}
\end{lemma}

\begin{proof}[Proof of Lemma \ref{BigBound5}] Assume first that \eqref{ok110.5} holds. We estimate, assuming $k_1\leq k_2$ and using \eqref{nh9},
\begin{equation*}
\begin{split}
(1+2^k)2^{k_\sigma}&(2^{\alpha k}+2^{10k})2^{(1+\beta)j}2^{3k/2}\big\|\mathcal{F}P_kT_{m}^{\sigma;\mu,\nu}(f_{k_1,j_1}^\mu,f_{k_2,j_2}^\nu)\big\|_{L^\infty}\\
&\lesssim 2^{(1+\beta)j}2^m(1+2^{11k})2^{3k_2/2}\sup_{s\in[2^{m-1},2^{m+1}]}\|\widehat{f_{k_1,j_1}^\mu}(s)\|_{L^1}\|\widehat{f_{k_2,j_2}^\nu}(s)\|_{L^\infty}\\
&\lesssim 2^{(2+\beta)m}(1+2^{11k})2^{3k_2/2}\cdot 2^{5k_1/2}2^{-k_2/2}.
\end{split}
\end{equation*}
The desired bound \eqref{ok60} follows using also \eqref{ok40.55}.

Assume now that \eqref{ok110} holds. Using definition \eqref{sec5.3}, it suffices to prove that
\begin{equation}\label{ok112}
\begin{split}
&(1+2^{k})2^{k_\sigma}(2^{\alpha k}+2^{10k})\cdot 2^{(1+\beta)j}\big\|\widetilde{\varphi}^{(k)}_j\cdot P_kT_{m}^{\sigma;\mu,\nu}(f_{k_1,j_1}^\mu,f_{k_2,j_2}^\nu)\big\|_{L^2}\\
&+(1+2^{k})2^{k_\sigma}(2^{\alpha k}+2^{10k})\cdot 2^{(1/2-\beta)\widetilde{k}}\big\|\mathcal{F}[\widetilde{\varphi}^{(k)}_j\cdot P_kT_{m}^{\sigma;\mu,\nu}(f_{k_1,j_1}^\mu,f_{k_2,j_2}^\nu)]\big\|_{L^\infty}\lesssim 2^{-\beta^4 (m+j)}.
\end{split}
\end{equation}
By symmetry, we may assume $k_1\leq k_2$. We prove first the bounds \eqref{ok112} in the case
\begin{equation}\label{ba6}
k_1\leq-5m/6.
\end{equation}
Using \eqref{nh9}, for any $s\in[0,t]$,
\begin{equation*}
\|\widehat{f_{k_1,j_1}^\mu}(s)\|_{L^1}\lesssim 2^{3k_1}\|\widehat{f_{k_1,j_1}^\mu}(s)\|_{L^\infty}\lesssim 2^{(5/2-\alpha+\beta)k_1}.
\end{equation*}
Therefore, using \eqref{nh9} again, it follows that
\begin{equation*}
\begin{split}
\big\|\mathcal{F}[T_{m}^{\sigma;\mu,\nu}(f_{k_1,j_1}^\mu,f_{k_2,j_2}^\nu)]\big\|_{L^2}&\lesssim 2^m\sup_{s\in[2^{m-1},2^{m+1}]}\|\widehat{f_{k_1,j_1}^\mu}(s)\|_{L^1}\|\widehat{f_{k_2,j_2}^\nu}(s)\|_{L^2}\\
&\lesssim 2^m2^{(5/2-\alpha+\beta)k_1}\min(2^{-(N_0-1)k_2},2^{(1+\beta-\alpha)k_2})
\end{split}
\end{equation*}
and
\begin{equation}\label{ba6.5}
\begin{split}
\big\|\mathcal{F}[T_{m}^{\sigma;\mu,\nu}(f_{k_1,j_1}^\mu,f_{k_2,j_2}^\nu)\big\|_{L^\infty}&\lesssim 2^m\sup_{s\in[2^{m-1},2^{m+1}]}\|\widehat{f_{k_1,j_1}^\mu}(s)\|_{L^1}\|\widehat{f_{k_2,j_2}^\nu}(s)\|_{L^\infty}\\
&\lesssim 2^m2^{(5/2-\alpha+\beta)k_1}\cdot(2^{\alpha k_2}+2^{10k_2})^{-1}2^{-(1/2-\beta)\widetilde{k_2}}.
\end{split}
\end{equation}
Therefore, recalling \eqref{ba6}, if $k\leq 0$ then the left-hand side of \eqref{ok112} is dominated by
\begin{equation*}
C2^{(2+\beta)m}2^{(5/2-\alpha+\beta)k_1}\lesssim 2^{(-1/12+5\al/6+\be/6)m},
\end{equation*}
which suffices. Similarly, if $k\geq 0$ then the left-hand side of \eqref{ok112} is dominated by
\begin{equation*}
C2^{(2+\beta)m}2^{(5/2-\alpha+\beta)k_1}2^{-(N_0-15)k}+C2^{k_2}2^m2^{(5/2-\alpha+\beta)k_1}\lesssim 2^{k_2}2^{(-1/12+5\al/6+\be/6)m},
\end{equation*}
which also suffices.

To prove the bound \eqref{ok112} when $-5m/6\leq k_1\leq k_2$ we decompose, as in \eqref{ok36}--\eqref{ok37}, for any $s\in[2^{m-1},2^{m+1}]$,
\begin{equation}\label{ok113}
\begin{split}
&\phii^{(k_1)}_{j_1}\cdot P_{k_1}f_\mu(s)=(2^{\alpha k_1}+2^{10k_1})^{-1}[g_{k_1,j_1}^\mu(s)+h_{k_1,j_1}^\mu(s)],\\
&g_{k_1,j_1}^\mu(s)=g_{k_1,j_1}^\mu(s)\cdot \widetilde{\varphi}^{(k_1)}_{[j_1-2,j_1+2]},\qquad h_{k_1,j_1}^\mu(s)=h_{k_1,j_1}^\mu(s)\cdot \widetilde{\varphi}^{(k_1)}_{[j_1-2,j_1+2]},\\
&2^{(1+\be)j_1}\|g_{k_1,j_1}^\mu(s)\|_{L^2}+2^{(1/2-\beta)\widetilde{k_1}}\|\widehat{g_{k_1,j_1}^\mu}(s)\|_{L^\infty}\lesssim 1,\\
&2^{(1-\be)j_1}\|h_{k_1,j_1}^\mu(s)\|_{L^2}+\|\widehat{h_{k_1,j_1}^\mu}(s)\|_{L^\infty}+2^{\gamma j_1}\|\widehat{h_{k_1,j_1}^\mu}(s)\|_{L^1}\lesssim 2^{-8|k_1|},
\end{split}
\end{equation}
and
\begin{equation}\label{ok114}
\begin{split}
&\phii^{(k_2)}_{j_2}\cdot P_{k_2}f_\nu(s)=(2^{\alpha k_2}+2^{10k_2})^{-1}[g_{k_2,j_2}^\nu(s)+h_{k_2,j_2}^\nu(s)],\\
&g_{k_2,j_2}^\nu(s)=g_{k_2,j_2}^\nu(s)\cdot \widetilde{\varphi}^{(k_2)}_{[j_2-2,j_2+2]},\qquad h_{k_2,j_2}^\nu(s)=h_{k_2,j_2}^\nu(s)\cdot \widetilde{\varphi}^{(k_2)}_{[j_2-2,j_2+2]},\\
&2^{(1+\be)j_2}\|g_{k_2,j_2}^\nu(s)\|_{L^2}+2^{(1/2-\beta)\widetilde{k_2}}\|\widehat{g_{k_2,j_2}^\nu}(s)\|_{L^\infty}\lesssim 1,\\
&2^{(1-\be)j_2}\|h_{k_2,j_2}^\nu(s)\|_{L^2}+\|\widehat{h_{k_2,j_2}^\nu}(s)\|_{L^\infty}+2^{\gamma j_2}\|\widehat{h_{k_2,j_2}^\nu}(s)\|_{L^1}\lesssim 2^{-8|k_2|}.
\end{split}
\end{equation}

We will prove now the $L^2$ bound
\begin{equation}\label{ba8}
(1+2^{k})2^{k_\sigma}(2^{\alpha k}+2^{10k})\cdot 2^{(2+\beta)m}\big\|P_k\widetilde{T}_{s}^{\sigma;\mu,\nu}(f_{k_1,j_1}^\mu(s),f_{k_2,j_2}^\nu(s))\big\|_{L^2}\lesssim 2^{-2\beta^4 m},
\end{equation}
for any $s\in[2^{m-1},2^{m+1}]$, see \eqref{nh9.2} for the definition of the bilinear operators $\widetilde{T}_{s}^{\sigma;\mu,\nu}$. In view of the assumption \eqref{ok110} this would clearly imply the desired $L^2$ bound in \eqref{ok112}. 

Assume first that $\min(j_1,j_2)\leq (1-9\beta)m$, i.e.
\begin{equation}\label{ba9}
\min(j_1,j_2)\leq (1-15\beta)m,\qquad \max(j_1,j_2)\geq (1-\beta/10)m+k_\sigma,\qquad k_2\geq k_1\geq-5m/6.
\end{equation}
Using \eqref{nh9.1} and \eqref{equl},
\begin{equation*}
\begin{split}
\|P_k\widetilde{T}_{s}^{\sigma;\mu,\nu}&(f_{k_1,j_1}^\mu(s),f_{k_2,j_2}^\nu(s))\big\|_{L^2}\lesssim \min(\|Ef_{k_1,j_1}^\mu(s)\|_{L^\infty}\|Ef_{k_2,j_2}^\nu(s)\|_{L^2}, \|Ef_{k_1,j_1}^\mu(s)\|_{L^2}\|Ef_{k_2,j_2}^\nu(s)\|_{L^\infty})\\
&\lesssim \min(2^{\beta k_1},2^{-6k_1})\min(2^{\beta k_2},2^{-6k_2})\cdot 2^{-m(5/4-10\beta)}2^{(1/4-11\beta)\min(j_1,j_2)}2^{-(1-\beta)\max(j_1,j_2)}\\
&\lesssim 2^{-k_\sigma}(1+2^{k_2})^{-6}2^{-(2+3\beta/2)m},
\end{split}
\end{equation*}
which suffices to prove \eqref{ba8}.

Assume now that $\min(j_1,j_2)\geq (1-15\beta)m$, i.e.
\begin{equation}\label{ba10}
\min(j_1,j_2)\geq (1-15\beta)m,\qquad \max(j_1,j_2)\geq (1-\beta/10)m+k_\sigma,\qquad k_2\geq k_1\geq-5m/6.
\end{equation}
We recall that 
\begin{equation}\label{ba11}
\begin{split}
&f_{k_1,j_1}^\mu=P_{[k_1-2,k_1+2]}(\phii^{(k_1)}_{j_1}\cdot P_{k_1}f_\mu)=(2^{\alpha k_1}+2^{10k_1})^{-1}[P_{[k_1-2,k_1+2]}g_{k_1,j_1}^\mu+P_{[k_1-2,k_1+2]}h_{k_1,j_1}^\mu],\\
&f_{k_2,j_2}^\nu=P_{[k_2-2,k_2+2]}(\phii^{(k_2)}_{j_2}\cdot P_{k_2}f_\nu)=(2^{\alpha k_2}+2^{10k_2})^{-1}[P_{[k_2-2,k_2+2]}g_{k_2,j_2}^\nu+P_{[k_2-2,k_2+2]}h_{k_2,j_2}^\nu],
\end{split}
\end{equation}
and use the bounds in \eqref{ok113}--\eqref{ok114}. Then we estimate, using also \eqref{ba10},
\begin{equation*}
\begin{split}
\big\|P_k\widetilde{T}_{s}^{\sigma;\mu,\nu}(P_{[k_1-2,k_1+2]}h_{k_1,j_1}^\mu(s),&P_{[k_2-2,k_2+2]}h_{k_2,j_2}^\nu(s))\big\|_{L^2}\lesssim \|\widehat{h_{k_1,j_1}^\mu}(s)\|_{L^1}\|\widehat{h_{k_2,j_2}^\nu}(s)\|_{L^2}\\
&\lesssim 2^{-\gamma j_1}2^{-(1-\beta)j_2}2^{-8|k_1|}2^{-8|k_2|}\\
&\lesssim 2^{-(\gamma+1-25\beta)m}2^{-6|k_1|}2^{-8|k_2|}
\end{split}
\end{equation*}
\begin{equation*}
\begin{split}
\big\|P_k\widetilde{T}_{s}^{\sigma;\mu,\nu}(P_{[k_1-2,k_1+2]}h_{k_1,j_1}^\mu(s),&P_{[k_2-2,k_2+2]}g_{k_2,j_2}^\nu(s))\big\|_{L^2}\lesssim \|\widehat{h_{k_1,j_1}^\mu}(s)\|_{L^1}\|\widehat{g_{k_2,j_2}^\nu}(s)\|_{L^2}\\
&\lesssim 2^{-\gamma j_1}2^{-(1+\beta)j_2}2^{-8|k_1|}\\
&\lesssim 2^{-m(\gamma+1-25\beta)}2^{-6|k_1|},
\end{split}
\end{equation*}
\begin{equation*}
\begin{split}
\big\|P_k\widetilde{T}_{s}^{\sigma;\mu,\nu}(P_{[k_1-2,k_1+2]}g_{k_1,j_1}^\mu(s),&P_{[k_2-2,k_2+2]}h_{k_2,j_2}^\nu(s))\big\|_{L^2}\lesssim \|\widehat{g_{k_1,j_1}^\mu}(s)\|_{L^2}\|\widehat{h_{k_2,j_2}^\nu}(s)\|_{L^1}\\
&\lesssim 2^{-(1+\beta)j_1}2^{-\gamma j_2}2^{-8|k_2|}\\
&\lesssim 2^{-m(\gamma+1-25\beta)}2^{-8|k_2|}2^{-k_\sigma},
\end{split}
\end{equation*}
and, using also \eqref{mk15.6}--\eqref{mk15.67} (compare with the bounds \eqref{equl}),
\begin{equation*}
\begin{split}
&\big\|P_k\widetilde{T}_{s}^{\sigma;\mu,\nu}(P_{[k_1-2,k_1+2]}g_{k_1,j_1}^\mu(s),P_{[k_2-2,k_2+2]}g_{k_2,j_2}^\nu(s))\big\|_{L^2}\\
&\lesssim \min\big(\|e^{-is\widetilde{\Lambda}_\mu}P_{[k_1-2,k_1+2]}(g_{k_1,j_1}^\mu(s))\|_{L^\infty}\|g_{k_2,j_2}^\nu(s)\|_{L^2},\|g_{k_1,j_1}^\mu(s)\|_{L^2}\|e^{-is\widetilde{\Lambda}_\nu}P_{[k_2-2,k_2+2]}(g_{k_2,j_2}^\nu(s))\|_{L^\infty}\big)\\
&\lesssim 2^{-(1+\beta)\max(j_1,j_2)}\cdot 2^{-m(5/4-10\beta)}2^{(1/4-11\beta)\min(j_1,j_2)}(1+2^{4k_2})\\
&\lesssim 2^{-k_\sigma}(1+2^{4k_2})2^{-(2+19\beta/10)m}.
\end{split}
\end{equation*}
Therefore, using also $\al\in[0,\be/2]$ and $k_1\geq -5m/6$, the left-hand side of \eqref{ba8} is dominated by 
\begin{equation*}
C(1+2^{5k_2})2^{-\alpha k_1}2^{-9\beta m/10}\lesssim (1+2^{5k_2})2^{-29m\beta/60}. 
\end{equation*}
This completes the proof of \eqref{ba8}.

To complete the proof of \eqref{ok112} it remains to prove the $L^\infty$ bound.
This would follow from the estimate
\begin{equation}\label{ba14}
(1+2^{k})2^{k_\sigma}(2^{\alpha k}+2^{10k})\cdot 2^{(1/2-\beta)\widetilde{k}}2^m\big\|\mathcal{F}P_k\widetilde{T}_s^{\sigma;\mu,\nu}(f_{k_1,j_1}^\mu(s),f_{k_2,j_2}^\nu(s))\big\|_{L^\infty}\lesssim 2^{-2\beta^4m}
\end{equation}
for all $s\in [2^{m-2},2^{m+2}]$.
If $k_1\leq -2m/5$ then, as in \eqref{ba6.5},
\begin{equation*}
\begin{split}
\big\|\mathcal{F}P_k\widetilde{T}_s^{\sigma;\mu,\nu}(f_{k_1,j_1}^\mu(s),f_{k_2,j_2}^\nu(s))\big\|_{L^\infty}&\lesssim \|\widehat{f_{k_1,j_1}^\mu}(s)\|_{L^1}\|\widehat{f_{k_2,j_2}^\nu}(s)\|_{L^\infty}\\
&\lesssim 2^{(5/2-\alpha+\beta)k_1}(2^{\alpha k_2}+2^{10k_2})^{-1}2^{-(1/2-\beta)\widetilde{k}_2},
\end{split}
\end{equation*}
and therefore the left-hand side of \eqref{ba14} is dominated by
\begin{equation*}
C(1+2^{k})2^{k_\sigma}\cdot (2^{\alpha k}+2^{10k})(2^{\alpha k_2}+2^{10k_2})^{-1}\cdot 2^{(1/2-\beta)(\widetilde{k}-\widetilde{k}_2)}\cdot 2^m2^{(5/2-\alpha+\beta)k_1},
\end{equation*}
which is sufficient.

We now assume that $-2m/5\le k_1\le k_2$ and we decompose $f_{k_1,j_1}^\mu,f_{k_2,j_2}^\nu$ as in \eqref{ok113}, \eqref{ok114}, \eqref{ba11}. If $j_1\leq j_2$, we estimate
\begin{equation*}
\begin{split}
\big\|\mathcal{F}P_k\widetilde{T}_s^{\sigma;\mu,\nu}(P_{[k_1-2,k_1+2]}(g_{k_1,j_1}^\mu(s)+h_{k_1,j_1}^\mu(s)),&P_{[k_2-2,k_2+2]}g_{k_2,j_2}^\nu(s))\big\|_{L^\infty}\\
&\lesssim \big(\|\widehat{g_{k_1,j_1}^\mu}(s)\|_{L^2}+\|\widehat{h_{k_1,j_1}^\mu}(s)\|_{L^2}\big)\|\widehat{g_{k_2,j_2}^\nu}(s)\|_{L^2}\\
&\lesssim 2^{(1+\beta)\widetilde{k}_1}2^{-(1+\beta)j_2},
\end{split}
\end{equation*}
and
\begin{equation*}
\begin{split}
\big\|\mathcal{F}P_k\widetilde{T}_s^{\sigma;\mu,\nu}(P_{[k_1-2,k_1+2]}(g_{k_1,j_1}^\mu(s)+h_{k_1,j_1}^\mu(s)),&P_{[k_2-2,k_2+2]}h_{k_2,j_2}^\nu(s))\big\|_{L^\infty}\\
&\lesssim \big(\|\widehat{g_{k_1,j_1}^\mu}(s)\|_{L^\infty}+\|\widehat{h_{k_1,j_1}^\mu}(s)\|_{L^\infty}\big)\|\widehat{h_{k_2,j_2}^\nu}(s)\|_{L^1}\\
&\lesssim 2^{-(1/2-\beta)\widetilde{k_1}}\cdot 2^{-8|k_2|}2^{-\gamma j_2}.
\end{split}
\end{equation*}
Since $-\widetilde{k_1}\leq 2m/5$, $\alpha\le \beta$ and $2^{j_2}\gtrsim 2^{m(1-\beta/10)}2^{k_\sigma}$ it follows that if $j_1\leq j_2$ then
\begin{equation}\label{ba15}
\big\|\mathcal{F}P_k\widetilde{T}_s^{\sigma;\mu,\nu}(f_{k_1,j_1}^\mu(s),f_{k_2,j_2}^\nu(s))\big\|_{L^\infty}\lesssim 2^{-(1+\beta)k_\sigma}2^{-(1+\beta)(1-\beta/10)m}\cdot (2^{\alpha k_1}+2^{10k_1})^{-1}(2^{\alpha k_2}+2^{10k_2})^{-1}.
\end{equation}

Similarly, if $j_1\geq j_2$ we estimate
\begin{equation*}
\begin{split}
\big\|\mathcal{F}P_k\widetilde{T}_s^{\sigma;\mu,\nu}(P_{[k_1-2,k_1+2]}g_{k_1,j_1}^\mu(s),&P_{[k_2-2,k_2+2]}(g_{k_2,j_2}^\nu(s)+h_{k_2,j_2}^\nu(s)))\big\|_{L^\infty}\\
&\lesssim \|\widehat{g_{k_1,j_1}^\mu}(s)\|_{L^2}\big(\|\widehat{g_{k_2,j_2}^\nu}(s)\|_{L^2}+\|\widehat{h_{k_2,j_2}^\nu}(s)\|_{L^2}\big)\\
&\lesssim 2^{-(1+\beta)j_1}2^{(1+\beta)\widetilde{k}_2},
\end{split}
\end{equation*}
and
\begin{equation*}
\begin{split}
\big\|\mathcal{F}P_k\widetilde{T}_s^{\sigma;\mu,\nu}(P_{[k_1-2,k_1+2]}h_{k_1,j_1}^\mu(s),&P_{[k_2-2,k_2+2]}(g_{k_2,j_2}^\nu(s)+h_{k_2,j_2}^\nu(s)))\big\|_{L^\infty}\\
&\lesssim \|\widehat{h_{k_1,j_1}^\mu}(s)\|_{L^1}\big(\|\widehat{g_{k_2,j_2}^\nu}(s)\|_{L^\infty}+\|\widehat{h_{k_2,j_2}^\nu}(s)\|_{L^\infty}\big)\\
&\lesssim 2^{-\gamma j_1}2^{-6|k_1|}.
\end{split}
\end{equation*}
Since $2^{j_1}\gtrsim 2^{m(1-\beta/10)}2^{k_\sigma}$ it follows that if $j_1\geq j_2$ then
\begin{equation}\label{ba16}
\big\|\mathcal{F}P_k\widetilde{T}_s^{\sigma;\mu,\nu}(f_{k_1,j_1}^\mu(s),f_{k_2,j_2}^\nu(s))\big\|_{L^\infty}\lesssim 2^{-(1+\beta)k_\sigma}2^{-(1+\beta)(1-\beta/10)m}\cdot (2^{\alpha k_1}+2^{10k_1})^{-1}(2^{\alpha k_2}+2^{10k_2})^{-1}.
\end{equation}

Using \eqref{ba15} and \eqref{ba16}, the left-hand side of \eqref{ba14} is dominated by
\begin{equation*}
C(1+2^{k})2^{-\alpha k_1}2^{-4\beta m/5},
\end{equation*}
which suffices. This completes the proof of the lemma.
\end{proof}

Now that we have identified $m$ as the largest parameter, 
we may remove the non-resonant part of the nonlinearity. For any $\kappa\in(0,1]$ we define
\begin{equation}\label{DefOfN}
\begin{split}
T_{m}^{\sigma;\mu,\nu}(f,g)&=R_{m,\kappa}^{\sigma;\mu,\nu}(f,g)+N_{m}^{1;\sigma;\mu,\nu}(f,g)+N_{m,\kappa}^{2;\sigma;\mu,\nu}(f,g),\\
\mathcal{F}\big[N_{m}^{1;\sigma;\mu,\nu}(f,g)\big](\xi)&:=\int_{\mathbb{R}}\int_{\mathbb{R}^3}e^{is\Phi^{\sigma;\mu,\nu}(\xi,\eta)}\chi_{T}^{\sigma;\mu,\nu}(\xi,\eta)q_m(s)\cdot \widehat{f}(\xi-\eta,s)\widehat{g}(\eta,s)\,d\eta ds,\\
\chi_{T}^{\sigma;\mu,\nu}(\xi,\eta)&:=\varphi_{[1,\infty)}(2^{D^2+\max(0,k_1,k_2)}\Phi^{\sigma;\mu,\nu}(\xi,\eta)),\\
\mathcal{F}\big[N_{m,\kappa}^{2;\sigma;\mu,\nu}(f,g)\big](\xi)&:=\int_{\mathbb{R}}\int_{\mathbb{R}^3}e^{is\Phi^{\sigma;\mu,\nu}(\xi,\eta)}\chi_{S}^{\sigma;\mu,\nu}(\xi,\eta)q_m(s)\cdot \widehat{f}(\xi-\eta,s)\widehat{g}(\eta,s)\,d\eta ds,\\
\chi_{S}^{\sigma;\mu,\nu}(\xi,\eta)&:=\varphi(2^{D^2+\max(0,k_1,k_2)}\Phi^{\sigma;\mu,\nu}(\xi,\eta))\varphi_{[1,\infty)}(\vert\Xi^{\mu,\nu}(\xi,\eta)\vert/\kappa),\\
\mathcal{F}\big[R_{m,\kappa}^{\sigma;\mu,\nu}(f,g)\big](\xi)&:=\int_{\mathbb{R}}\int_{\mathbb{R}^3}e^{is\Phi^{\sigma;\mu,\nu}(\xi,\eta)}\chi_{R}^{\sigma;\mu,\nu}(\xi,\eta)q_m(s)\cdot \widehat{f}(\xi-\eta,s)\widehat{g}(\eta,s)\,d\eta ds,\\
\chi_R^{\sigma;\mu,\nu}(\xi,\eta)&:=\varphi(2^{D^2+\max(0,k_1,k_2)}\Phi^{\sigma;\mu,\nu}(\xi,\eta))\varphi(\vert\Xi^{\mu,\nu}(\xi,\eta)\vert/\kappa).
\end{split}
\end{equation}
Our last lemma in this section shows that only the resonant part of the interaction $R_{m,\kappa}^{\sigma;\mu,\nu}$ may produce more problematic outputs not in $B^{1}_{j,k}$.

\begin{lemma}\label{NewBigBound1}
Assume that $\sigma\in\{i,e,b\}$, $\mu, \nu\in\mathcal{I}_0$, $(k, j), (k_1, j_1), (k_2, j_2) \in\mathcal{J}$ , $m \in [0, L + 1] \cap\mathbb{Z}$, and
\begin{equation}\label{Bds1N1}
\begin{split}
&-9m/10\le k_1,k_2\le j/N_0^\prime,\quad\max(j_1,j_2)\le (1-\beta/10)m+k_\sigma,\\
&\beta m/2+N_0^\prime k_++D^2\le j\le m+D,\quad m\ge -k(1+\beta^2).
\end{split}
\end{equation}
Then, assuming $m\in[0,L]\cap\mathbb{Z}$,
\begin{equation}\label{NRI}
\begin{split}
(1+2^k)\Vert \widetilde{\varphi}^{(k)}_j\cdot P_kN_{m}^{1;\sigma;\mu,\nu}(f^\mu_{k_1,j_1},f^\nu_{k_2,j_2})\Vert_{B^1_{k,j}}&\lesssim 2^{-2\beta^4m},\\
(1+2^k)\Vert \widetilde{\varphi}^{(k)}_j\cdot P_kN_{m,\kappa}^{2;\sigma;\mu,\nu}(f^\mu_{k_1,j_1},f^\nu_{k_2,j_2})\Vert_{B^1_{k,j}}&\lesssim 2^{-2\beta^4m},
\end{split}
\end{equation}
for any $\kappa\in(0,1]$ satisfying
\begin{equation}\label{kappa}
2^m\kappa\geq 2^{\beta^2m}2^{\max(j_1,j_2)},\qquad 2^m\kappa\geq 2^{\beta^2 m}\kappa^{-1}2^{-\min(k_1,k_2,0)}2^{-D}.
\end{equation}
Moreover, for $m=L+1$,
\begin{equation}\label{NRI2}
(1+2^k)\Vert \widetilde{\varphi}^{(k)}_j\cdot P_kT_{L+1}^{\sigma;\mu,\nu}(f^\mu_{k_1,j_1},f^\nu_{k_2,j_2})\Vert_{B^1_{k,j}}\lesssim 2^{-2\beta^4L}.
\end{equation}
\end{lemma}

\begin{proof}[Proof of Lemma \ref{NewBigBound1}] To prove the second inequality in \eqref{NRI} we use Lemma \ref{tech5} and the assumptions \eqref{kappa} to show that
\begin{equation}\label{szn25.8}
\vert \mathcal{F}\big[N_{m,\kappa}^{2;\sigma;\mu,\nu}(f^\mu_{k_1,j_1},f^\nu_{k_2,j_2})\big](\xi)\vert\lesssim 2^{-10m}.
\end{equation}
The second inequality in \eqref{NRI} follows easily using \eqref{Bds1N1}.

To prove the first inequality in \eqref{NRI} when $m\leq L$, we first integrate by parts in $s$ and obtain that
\begin{equation}\label{IBTNRI2}
\begin{split}
\mathcal{F}\big[N_{m}^{1;\sigma;\mu,\nu}(f,g)\big](\xi)&=-\int_{\mathbb{R}}\int_{\mathbb{R}^3}e^{is\Phi^{\sigma;\mu,\nu}(\xi,\eta)}\frac{\chi_{T}^{\sigma;\mu,\nu}(\xi,\eta)}{i\Phi^{\sigma;\mu,\nu}(\xi,\eta)}\cdot\partial_s\left[q_m(s) \widehat{f}(\xi-\eta,s)\widehat{g}(\eta,s)\right]\,d\eta ds.\\
\end{split}
\end{equation}
Therefore
\begin{equation*}
\begin{split}
\mathcal{F}\big[N_{m}^{1;\sigma;\mu,\nu}(f^\mu_{k_1,j_1},f^\nu_{k_2,j_2})\big]&=i\left[N_{11}+N_{12}+N_{13}\right],\\
N_{11}(\xi)&:=\int_{\mathbb{R}}\int_{\mathbb{R}^3}e^{is\Phi^{\sigma;\mu,\nu}(\xi,\eta)}\frac{\chi_{T}^{\sigma;\mu,\nu}(\xi,\eta)}{\Phi^{\sigma;\mu,\nu}(\xi,\eta)}q^\prime_m(s)\cdot \widehat{f^\mu_{k_1,j_1}}(\xi-\eta,s)\widehat{f^\nu_{k_2,j_2}}(\eta,s)\,d\eta ds,\\
N_{12}(\xi)&:=\int_{\mathbb{R}}\int_{\mathbb{R}^3}e^{is\Phi^{\sigma;\mu,\nu}(\xi,\eta)}\frac{\chi_{T}^{\sigma;\mu,\nu}(\xi,\eta)}{\Phi^{\sigma;\mu,\nu}(\xi,\eta)}q_m(s)\cdot (\partial_s\widehat{f^\mu_{k_1,j_1}})(\xi-\eta,s)\widehat{f^\nu_{k_2,j_2}}(\eta,s)\,d\eta ds,\\
N_{13}(\xi)&:=\int_{\mathbb{R}}\int_{\mathbb{R}^3}e^{is\Phi^{\sigma;\mu,\nu}(\xi,\eta)}\frac{\chi_{T}^{\sigma;\mu,\nu}(\xi,\eta)}{\Phi^{\sigma;\mu,\nu}(\xi,\eta)}q_m(s)\cdot \widehat{f^\mu_{k_1,j_1}}(\xi-\eta,s)(\partial_s\widehat{f^\nu_{k_2,j_2}})(\eta,s)\,d\eta ds.
\end{split}
\end{equation*}

We show first that
\begin{equation}\label{SuffNRI1}
(1+2^k)(2^{\alpha k}+2^{10k})2^{(1+\beta)m}\Vert P_kN_{m}^{1;\sigma;\mu,\nu}(f^\mu_{k_1,j_1},f^\nu_{k_2,j_2})\Vert_{L^2}\lesssim 2^{-2\beta^4m}.
\end{equation}
We may assume $k_1\leq k_2$. Using symbol type estimates, it is easy to see that
\begin{equation}\label{ClaimChiT}
\Big\|\mathcal{F}^{-1}\left[\frac{\chi_{T}^{\sigma;\mu,\nu}(\xi,\eta)}{\Phi^{\sigma;\mu,\nu}(\xi,\eta)}\varphi_{[k-4,k+4]}(\xi)\varphi_{[k_1-4,k_1+4]}(\xi-\eta)\varphi_{[k_2-4,k_2+4]}(\eta)\right]\Big\|_{L^1(\mathbb{R}^3\times\mathbb{R}^3)}\lesssim 2^{20\max(0,k_2)}.
\end{equation}
Using the decomposition \eqref{IBTNRI2}, Lemma \ref{tech2}, and \eqref{ClaimChiT}, we see that
\begin{equation}\label{IBTNRI1}
\begin{split}
&\Vert P_kN_{m}^{1;\sigma;\mu,\nu}(f^\mu_{k_1,j_1},f^\nu_{k_2,j_2})\Vert_{L^2}\\
&\lesssim 2^{20\max(0,k_2)}\sup_{s\in [2^{m-2},2^{m+2}]}\Big[\min\left\{\Vert Ef^\mu_{k_1,j_1}(s)\Vert_{L^\infty}\Vert f^\nu_{k_2,j_2}(s)\Vert_{L^2},\Vert f^\mu_{k_1,j_1}(s)\Vert_{L^2}\Vert Ef^\nu_{k_2,j_2}(s)\Vert_{L^\infty}\right\}\\
&+2^m\Vert Ef^\mu_{k_1,j_1}(s)\Vert_{L^\infty}\Vert (\partial_sf^\nu_{k_2,j_2})(s)\Vert_{L^2}+2^m\Vert (\partial_sf^\mu_{k_1,j_1})(s)\Vert_{L^2}\Vert Ef^\nu_{k_2,j_2}(s)\Vert_{L^\infty}\Big].
\end{split}
\end{equation}
It follows from \eqref{nh9} and \eqref{ok50repeat} that
\begin{equation*}
2^m\Vert Ef^\mu_{k_1,j_1}(s)\Vert_{L^\infty}\Vert (\partial_sf^\nu_{k_2,j_2})(s)\Vert_{L^2}+2^m\Vert (\partial_sf^\mu_{k_1,j_1})(s)\Vert_{L^2}\Vert Ef^\nu_{k_2,j_2}(s)\Vert_{L^\infty}\lesssim 2^{-6\max(k_2,0)}2^{-(1+2\beta)m}.
\end{equation*}
Moreover, using \eqref{nh9}--\eqref{nh9.1}, 
\begin{equation*}
\begin{split}
\min\left\{\Vert Ef^\mu_{k_1,j_1}(s)\Vert_{L^\infty}\Vert f^\nu_{k_2,j_2}(s)\Vert_{L^2},\Vert f^\mu_{k_1,j_1}(s)\Vert_{L^2}\Vert Ef^\nu_{k_2,j_2}(s)\Vert_{L^\infty}\right\}\\
\lesssim 2^{-6\max(k_2,0)}2^{-(1+\beta)m}2^{-(1-\beta)\max(j_1,j_2)}.
\end{split}
\end{equation*}
Finally, if $\max(j_1,j_2)\leq 2\beta m$ then, using \eqref{nh9.1} and \eqref{equl},
\begin{equation*}
\min\left\{\Vert Ef^\mu_{k_1,j_1}(s)\Vert_{L^\infty}\Vert f^\nu_{k_2,j_2}(s)\Vert_{L^2},\Vert f^\mu_{k_1,j_1}(s)\Vert_{L^2}\Vert Ef^\nu_{k_2,j_2}(s)\Vert_{L^\infty}\right\}\lesssim 2^{-6\max(k_2,0)}2^{-(5/4-15\beta)m}.
\end{equation*}
It follows from the last three bounds and \eqref{IBTNRI1} that
\begin{equation*}
\Vert P_kN_{m}^{1;\sigma;\mu,\nu}(f^\mu_{k_1,j_1},f^\nu_{k_2,j_2})\Vert_{L^2}\lesssim 2^{15\max(k_2,0)}2^{-(1+2\beta)m},
\end{equation*}
and the desired bound \eqref{SuffNRI1} follows since $2^k\lesssim 2^{k_2}\lesssim 2^{m/N'_0}$.

We show now that
\begin{equation}\label{LINormN1}
(1+2^k)2^{(1/2-\beta)\tilde{k}}(2^{\alpha k}+2^{10k})\Vert \mathcal{F}P_kN_{m}^{1;\sigma;\mu,\nu}(f^\mu_{k_1,j_1},f^\nu_{k_2,j_2})\Vert_{L^\infty}\lesssim 2^{-2\beta^4m}.
\end{equation}
We may assume $k_1\leq k_2$, and use the Cauchy-Schwartz inequality, \eqref{nh9}, and \eqref{ok50} to see that
\begin{equation*}
\begin{split}
&\| N_{12}\|_{L^\infty}+\|N_{13}\|_{L^\infty}\\
&\lesssim 2^{\max(0,k_2)}2^m\sup_{s\in[2^{m-2},2^{m+2}]}\big[\Vert(\partial_s\widehat{f}^\mu_{k_1,j_1})(s)\Vert_{L^2}\Vert \widehat{f}^\nu_{k_2,j_2}(s)\Vert_{L^2}+\Vert\widehat{f}^\mu_{k_1,j_1}(s)\Vert_{L^2}\Vert (\partial_s\widehat{f}^\nu_{k_2,j_2})(s)\Vert_{L^2}\big]\\
&\lesssim 2^{-\beta m}(1+2^{(N_0-10)k_2})^{-1}.
\end{split}
\end{equation*}
This implies that $N_{12}$ and $N_{13}$ give acceptable contributions to \eqref{LINormN1}. 
Proceeding as above, using \eqref{nh9.1} we also get
\begin{equation*}
\| N_{11}\|_{L^\infty}\lesssim (1+2^{k_2})2^{\beta\widetilde{k}_1}2^{-(1-\beta)j_1}\min(2^{-(N_0-5)k_2},2^{-(1-\beta)j_2}).
\end{equation*}
Therefore, this gives an acceptable contribution to \eqref{LINormN1} unless
\begin{equation}\label{SmallParam}
\vert k\vert+\vert k_1\vert+\vert k_2\vert+j_1+j_2\le \beta^2 m.
\end{equation}

Assuming that \eqref{SmallParam} holds, we need to strenghten the $L^\infty$ bound on $N_{11}$ slightly. We decompose
\begin{equation*}
\begin{split}
N_{11}&=N_{11;1}+N_{11;2},\\
N_{11;1}(\xi)&:=\int_{\mathbb{R}}\int_{\mathbb{R}^3}e^{is\Phi^{\sigma;\mu,\nu}(\xi,\eta)}\frac{\chi_{T}^{\sigma;\mu,\nu}(\xi,\eta)}{\Phi^{\sigma;\mu,\nu}(\xi,\eta)}\varphi(\delta^{-1}\vert\Xi^{\mu,\nu}(\xi,\eta)\vert)q^\prime_m(s)\cdot \widehat{f^\mu_{k_1,j_1}}(\xi-\eta,s)\widehat{f^\nu_{k_2,j_2}}(\eta,s)\,d\eta ds,\\
N_{11;2}(\xi)&:=\int_{\mathbb{R}}\int_{\mathbb{R}^3}e^{is\Phi^{\sigma;\mu,\nu}(\xi,\eta)}\frac{\chi_{T}^{\sigma;\mu,\nu}(\xi,\eta)}{\Phi^{\sigma;\mu,\nu}(\xi,\eta)}[1-\varphi(\delta^{-1}\vert\Xi^{\mu,\nu}(\xi,\eta)\vert)]q^\prime_m(s)\cdot \widehat{f^\mu_{k_1,j_1}}(\xi-\eta,s)\widehat{f^\nu_{k_2,j_2}}(\eta,s)\,d\eta ds,\\
\end{split}
\end{equation*}
with $\delta:=2^{-m/3}$. Applying Lemma \ref{tech5} with $K=2^{2m/3}$, $\epsilon=2^{-m/3}$, it is easy to see that 
\begin{equation*}
\vert N_{11;2}(\xi)\vert\lesssim 2^{-10m},
\end{equation*}
provided that \eqref{SmallParam} holds, which is clearly sufficient. On the other hand, using the definition \eqref{PsiXi} and the bounds \eqref{mk3.1}, we observe that
\begin{equation*}
\begin{split}
\vert\Xi^{\mu,\nu}(\xi,\eta)\vert&\gtrsim \vert\nabla\widetilde{\Lambda}_\nu(\eta)\vert\cdot\min(\big|(\xi-\eta)/\vert\xi-\eta\vert-\eta/\vert\eta\vert\big|,\big|(\xi-\eta)/\vert\xi-\eta\vert+\eta/\vert\eta\vert\big|)\\
&\gtrsim 2^{-\beta m}\min(\big| (\xi-\eta)/\vert\xi-\eta\vert-\eta/\vert\eta\vert\big|,\big| (\xi-\eta)/\vert\xi-\eta\vert+\eta/\vert\eta\vert\big|).
\end{split}
\end{equation*}
Consequently, if $\vert \xi\vert\in [2^{k-2},2^{k+2}]$, $\vert\xi-\eta\vert\in[2^{k_1-2},2^{k_1+2}]$ and $\vert\eta\vert\in [2^{k_2-2},2^{k_2+2}]$, and $\vert\Xi^{\mu,\nu}(\xi,\eta)\vert\lesssim 2^{-m/3}$ then
\begin{equation*}
\min\big(\vert \eta/\vert\eta\vert-\xi/\vert\xi\vert\vert,\vert \eta/\vert\eta\vert+\xi/\vert\xi\vert\vert\big)\lesssim 2^{-m/4}.
\end{equation*}
Then, a simple estimate using the $L^\infty$ bounds in \eqref{nh9} gives $\vert N_{11;1}(\xi)\vert\lesssim 2^{-m/6}$, which is sufficient to finish the proof of \eqref{LINormN1}. The first bound in \eqref{NRI} follows from \eqref{SuffNRI1} and \eqref{LINormN1}.

The bound \eqref{NRI2} follows by a similar (in fact easier) argument; since $\|q_{L+1}\|_{L^1}\lesssim 1$ one does not need to integrate by parts in $s$ and one can simply estimate the appropriate $L^2$ and $L^\infty$ norms in the same way we estimated the contributions of the function $N_{11}$ in the argument above.
\end{proof}

We examine now the conclusions of Lemma \ref{BigBound3}, Lemma \ref{BigBound4}, Lemma \ref{BigBound5}, and Lemma \ref{NewBigBound1}. We notice that to complete the proof of Proposition \ref{reduced2}, it suffices to prove Proposition \ref{reduced3} below. 

\begin{proposition}\label{reduced3}
Assume that $\sigma\in\{i,e,b\}$, $\mu, \nu\in\mathcal{I}_0$, $(k, j), (k_1, j_1), (k_2, j_2) \in\mathcal{J}$ , $m \in [1, L ] \cap\mathbb{Z}$, and
\begin{equation}\label{conditions}
\begin{split}
&-9m/10\le k_1,k_2\le j/N_0^\prime,\quad\max(j_1,j_2)\le (1-\beta/10)m+k_\sigma,\\
&\beta m/2+N_0^\prime k_++D^2\le j\le m+D,\quad m\ge -k(1+\beta^2).
\end{split}
\end{equation}
Then there is $\kappa\in(0,1]$, $\kappa\geq\max\big(2^{(\beta^2m-m)/2}2^{-\min(k_1,k_2,0)/2}2^{-D/2},2^{\beta^2m-m}2^{\max(j_1,j_2)}\big)$, such that
\begin{equation}\label{ok90}
(1+2^k)2^{k_\sigma}\big\|\widetilde{\varphi}^{(k)}_j\cdot P_kR_{m,\kappa}^{\sigma;\mu,\nu}(f_{k_1,j_1}^\mu,f_{k_2,j_2}^\nu)\big\|_{B_{k,j}}\lesssim 2^{-2\beta^4m}.
\end{equation}
\end{proposition}

We prove this proposition in the next 3 sections. We consider several types of resonant interactions, which involve input and output frequencies located on spheres or at the origin, as well as the different phase functions $\Phi^{\sigma;\mu,\nu}$. We classify these interactions into 3 basic types, see Proposition \ref{PropABC}, and analyze the contributions separately in the next 3 sections. The optimal value of $\kappa$ for which we prove \eqref{ok90} depends, of course, on all the other parameters.

\section{Proof of Proposition \ref{Norm}, II: Case A resonant interactions}\label{normproof2}

In this section we consider type A interactions, see Proposition \ref{PropABC}, and prove the following proposition:

\begin{proposition}\label{reduced4}
Assume that $(k, j), (k_1, j_1), (k_2, j_2) \in\mathcal{J}$ , $m \in [1, L ] \cap\mathbb{Z}$,
\begin{equation}\label{tln1}
\begin{split}
\Phi^{\sigma;\mu,\nu}\in\mathcal{T}'_A:=\{&\Phi^{i;e+,i-}, \Phi^{i;b+,i-}, \Phi^{i;b-,e+}, \Phi^{i;b+,e-}, \Phi^{e;e+,i+}, \Phi^{e;b+,i+}, \Phi^{e;b+,i-}, \Phi^{e;b+,e-},\\
&\Phi^{b;e+,i+}, \Phi^{b;b+,i+}, \Phi^{b;e+,e+}, \Phi^{b;b+,e+},\Phi^{b;b+,e-}\},
\end{split}
\end{equation}
and
\begin{equation}\label{tln2}
-D/2\le k,k_1,k_2\le D/2,\quad\max(j_1,j_2)\le (1-\beta/10)m,\quad\beta m/2+N_0^\prime k_++D^2\le j\le m+D.
\end{equation}
Then there is $\kappa\in(0,1]$, $\kappa\geq\max\big(2^{(\beta^2m-m)/2},2^{\beta^2m-m}2^{\max(j_1,j_2)}\big)$, such that
\begin{equation}\label{tln3}
\big\|\widetilde{\varphi}^{(k)}_j\cdot P_kR_{m,\kappa}^{\sigma;\mu,\nu}(f_{k_1,j_1}^\mu,f_{k_2,j_2}^\nu)\big\|_{B_{k,j}}\lesssim 2^{-2\beta^4m}.
\end{equation}
\end{proposition}

The phases in the set $\mathcal{T}'_A$ are the same as the phases in the set $\mathcal{T}_A$, after interchanging the last two indices. Without loss of generality we may assume that $\Phi^{\sigma;\mu,\nu}\in\mathcal{T}'_A$ instead of $\Phi^{\sigma;\mu,\nu}\in\in\mathcal{T}_A$.

 The rest of the section is concerned with the proof of Proposition \ref{reduced4}. The interactions corresponding to Case A are among the most difficult to control. In particular, they produce outputs which fail to belong to the ``strong'' $B^1_{k,j}$ spaces. A key element we need is a precise description of the sizes of the various elements close to the resonant set. This is made possible by the fact that the Hessian of the phases is nondegenerate. We refer to the introduction of \cite{IoPa2} for more details.

We define first the interaction functions for the space-resonant phases in $\mathcal{T}'_A$ given in \eqref{tln1}, the functions $p^{\sigma;\mu,\nu}$ and $q^{\mu,\nu}$ defined below. They help us to characterize the vanishing set for $\Xi^{\mu,\nu}$ through the equality \eqref{XiVanish}. Only the functions $p^{\sigma;\mu,\nu}$ will play an essential role, but the functions $q^{\mu,\nu}$ appear as simpler natural intermediate functions. Our goal is to define these functions such that
\begin{equation}\label{XiVanish}
\Xi^{\mu,\nu}(q^{\mu,\nu}(\eta),\eta)=0=\Xi^{\mu,\nu}(\xi,p^{\sigma;\mu,\nu}(\xi)),
\end{equation}
where the first equality holds for all $\eta$ and the second equality holds for all $\xi$ where $p^{\sigma;\mu,\nu}(\xi)$ is well defined. 

For this we first define
\begin{equation}\label{DefOfP0}
p^{b;e+,e+}(\xi):=\xi/2,\quad q^{e+,e+}(\eta):=2\eta,\qquad t^{e,e}(r):=r.
\end{equation}
The other functions require a little more care. We first define $q^{\mu,\nu}$ and then invert the process. We define the real-valued functions $t^{ei},t^{bi},t^{be}:[0,\infty)\to[0,\infty)$ by the relation
\begin{equation*}
\lambda_e^\prime(t^{ei}(r))=\lambda_b^\prime(t^{bi}(r))=\lambda_i^\prime(r),\quad \lambda_b^\prime(t^{be}(r))=\lambda_e^\prime(r).
\end{equation*}
Since $\lambda_e^\prime$ and $\lambda_b^\prime$ are injective (see Lemma \ref{tech99}) and using also \eqref{yut9}, these functions are well defined.
We can directly see that $t^{bi}(r)\le t^{ei}(r)$, $t^{be}(r)\le r$ and since
\begin{equation*}
\lambda'_i(r)\in[\lambda_i^\prime(r_\ast),\lambda_i^\prime(0)]\subseteq\Big[\lambda_i^\prime(r_\ast),\frac{\sqrt{1+T}}{\sqrt{1+\varepsilon}}\Big],\qquad\text{ for any }r\in[0,\infty),
\end{equation*}
we get from Lemma \ref{tech99} that
\begin{equation}\label{Range}
\sqrt{\varepsilon}\lambda_i^\prime(r_\ast)/(2T)\le t^{ei}(r)\le\sqrt{3\varepsilon/T},\quad \sqrt{\varepsilon}\lambda_i^\prime(r_\ast)/C_b\le t^{bi}(r)\le \sqrt{\varepsilon/C_b},\quad 0\le t^{be}(r)\le \frac{\sqrt{T(1+\varepsilon)}}{\sqrt{C_b^2-TC_b}}.
\end{equation}
More precisely we have
\begin{equation*}
t^{bi}(r)=\frac{\sqrt{\varepsilon(1+\varepsilon)}}{\sqrt{C_b}}\frac{\lambda_i^\prime(r)}{\sqrt{C_b-\varepsilon(\lambda_i^\prime(r))^2}},\quad t^{be}(r)=\frac{\sqrt{\varepsilon(1+\varepsilon)}}{\sqrt{C_b}}\frac{\lambda_e^\prime(r)}{\sqrt{C_b-\varepsilon(\lambda_e^\prime(r))^2}},
\end{equation*}
while $t^{ei}$ has a similar behavior. Note in particular that $T(t^{ei}(r))^2\le 3\varepsilon$, $C_b(t^{bi}(r))^2\le \varepsilon$, $C_b(t^{be}(r))^2\le T(1+\varepsilon)/(C_b-T)$. Let
\begin{equation*}
(\partial t^{\sigma_1\sigma_2})(r):=\frac{dt^{\sigma_1\sigma_2}(r)}{dr},\qquad (\sigma_1,\sigma_2)\in\{(e,e),(e,i),(b,i),(b,e)\}.
\end{equation*}
Using Lemma \ref{tech99},
\begin{equation}\label{UpperBdT}
\begin{split}
&\big| (\partial t^{ei})(r)\big|=\Big|\frac{\lambda_i^{\prime\prime}(r)}{\lambda_e^{\prime\prime}(t^{ei}(r))}\Big|\le \vert\lambda_i^{\prime\prime}(r)\vert\frac{\varepsilon^{1/2}(1+Tt^{ei}(r)^2)^{3/2}}{T(1-\sqrt{\varepsilon})}\leq \frac{8\sqrt{2}\sqrt{\varepsilon}(1+3\varepsilon)^{3/2}T}{(1-\sqrt{\varepsilon})T}\le\frac{1}{2},\\
&\big| (\partial t^{bi})(r)\big|=\Big|\frac{\lambda''_i(r)}{\lambda''_b(t^{bi}(r))}\Big|\le |\lambda''_i(r)|\frac{\varepsilon^{1/2}(1+\varepsilon+C_bt^{bi}(r)^2)^{3/2}}{C_b(1+\varepsilon)}\le\frac{8\sqrt{2}\sqrt{\varepsilon}(1+\varepsilon)^{1/2}TC_b^{1/2}}{(C_b-T)^{3/2}}\leq\frac{1}{2},\\
&C_{C_b,\varepsilon}^{-1}\le (\partial t^{be})(r)=\frac{\lambda''_e(r)}{\lambda''_b(t^{be}(r))}\le \frac{(1+\sqrt{\varepsilon})T}{(1+Tr^2)^{3/2}}\frac{(1+\varepsilon+C_bt^{be}(r)^2)^{3/2}}{C_b(1+\varepsilon)}\leq \frac{(1+4\sqrt{\varepsilon})TC_b^{1/2}}{(C_b-T)^{3/2}}\leq\frac{1}{2}.
\end{split}
\end{equation}

We now define $q^{\mu,\nu}$ when $(\sigma_1,\sigma_2)\in\{(e,i),(b,i),(b,e)\}$ by the formula
\begin{equation*}
q^{\mu,\nu}(\eta):=
\eta+(\iota_1\cdot\iota_2)t^{\sigma_1\sigma_2}(\vert\eta\vert)\frac{\eta}{\vert\eta\vert}=\tilde{t}^{\mu,\nu}(\vert\eta\vert)\frac{\eta}{\vert\eta\vert},
\end{equation*}
such that $\Xi^{\mu,\nu}(q^{\mu,\nu}(\eta),\eta)=0$. Then we define the function $r^{\mu,\nu}(s)$ as the inverse function of $\widetilde{t}^{\mu,\nu}(r):=r+(\iota_1\cdot\iota_2)t^{\sigma_1\sigma_2}(r)$. Therefore
\begin{equation*}
r^{\mu,\nu}:[\iota_1\iota_2t^{\sigma_1\sigma_2}(0),\infty)\to[0,\infty)
\end{equation*} 
is a well-defined increasing function, and
\begin{equation}\label{dr}
(\partial_s r^{\mu,\nu})(s)=\frac{1}{1+\iota_1\iota_2(\partial t^{\sigma_1\sigma_2})(r^{\mu,\nu}(s))},\qquad s\in[\iota_1\iota_2t^{\sigma_1\sigma_2}(0),\infty).
\end{equation}

We can now finally define the functions $p^{\sigma;\mu,\nu}$ and $\chi_A^{\sigma;\mu,\nu}:[0,\infty)\to[0,1]$:

(a) if $\Phi^{\sigma;\mu,\nu}\in\mathcal{T}'_A\setminus \{\Phi^{e;b+,i-}\}$ then we define
\begin{equation}\label{DefP1}
I^{\sigma;\mu,\nu}:=[t^{\sigma_1\sigma_2}(0),\infty),\quad p^{\sigma;\mu,\nu}(\xi):=r^{\mu,\nu}(\vert\xi\vert)\xi/\vert\xi\vert\,\text{ for }\,|\xi|\in I^{\sigma;\mu,\nu},\quad\chi_A^{\sigma;\mu,\nu}:=\mathbf{1}_{(t^{\sigma_1\sigma_2}(0)+2^{-2D},\infty)};
\end{equation}

(b) if $\Phi^{\sigma;\mu,\nu}=\Phi^{e;b+,i-}$ then we define
\begin{equation}\label{DefP2}
I^{\sigma;\mu,\nu}:=[0,t^{bi}(0)],\quad p^{\sigma;\mu,\nu}(\xi):=-r^{\mu,\nu}(-\vert\xi\vert)\xi/\vert\xi\vert\,\text{ for }\,|\xi|\in I^{\sigma;\mu,\nu},\quad\chi_A^{\sigma;\mu,\nu}:=\mathbf{1}_{(0,t^{bi}(0)-2^{-2D})}.
\end{equation}
In both cases we also define
\begin{equation}\label{DefP3}
r^{\sigma;\mu,\nu}(|\xi|):=p^{\sigma;\mu,\nu}(\xi)\cdot\xi/|\xi|.
\end{equation}

The functions $p^{\sigma;\mu,\nu}$ are not defined (and not needed) outside the range specified above. These functions are the key to an efficient analysis of Case A through the use of the following lemma.

\begin{lemma}\label{LemP}
Assume that $\Phi^{\sigma;\mu,\nu}\in\mathcal{T}'_A$, see \eqref{tln1}, and $-D/2\le k,k_1,k_2\leq D/2$.

(i) Assume that $\delta\in[0,2^{-100D}]$ and assume that $(\xi,\eta)\in\mathbb{R}^3\times\mathbb{R}^3$ is a point such that
\begin{equation}\label{Added1}
\begin{split}
&|\xi|\in[2^{k-4},2^{k+4}],\qquad|\eta|\in[2^{k_2-4},2^{k_2+4}],\qquad|\xi-\eta|\in[2^{k_1-4},2^{k_1+4}],\\
&|\Xi^{\mu,\nu}(\xi,\eta)|\leq\delta,\qquad|\Phi^{\sigma;\mu,\nu}(\xi,\eta)|\leq 2^{-100D}.
\end{split}
\end{equation}
Then
\begin{equation}\label{kx11}
\chi_A^{\sigma;\mu,\nu}(|\xi|)=1,\quad\big|\eta-p^{\sigma;\mu,\nu}(\xi)\big|\leq 2^{40D}\delta\quad\text{ and }\quad\Xi^{\mu,\nu}(\xi,p^{\sigma;\mu,\nu}(\xi))=0,
\end{equation}
and
\begin{equation}\label{r1to1}
\begin{split}
&\min\big(|(\partial_sr^{\sigma;\mu,\nu})(|\xi|)|,|1-(\partial_sr^{\sigma;\mu,\nu})(|\xi|)|\big)\geq 2^{-4D},\\
&|(D^\rho_sr^{\sigma;\mu,\nu})(|\xi|)|\leq 2^{20D},\qquad \rho=0,1,\ldots 4.
\end{split}
\end{equation}
Moreover, if $\sigma_2=i$ then
\begin{equation}\label{CaseAAwayFromRast}
\big| |\eta|-r_\ast\big|\gtrsim_{C_b,\varepsilon} 1.
\end{equation}

(ii) Let $\Psi^{\si;\mu,\nu}:I^{\sigma;\mu,\nu}\to\mathbb{R}$ be defined by
\begin{equation}\label{kxz11}
\begin{split}
\Psi^{\si;\mu,\nu}(s)&:=\Phi^{\sigma;\mu,\nu}(se,r^{\sigma;\mu,\nu}(s)e)=\lambda_\sigma(s)-\iota_1\lambda_{\sigma_1}(|r^{\sigma;\mu,\nu}(s)-s|)
-\iota_2\lambda_{\sigma_2}(|r^{\sigma;\mu,\nu}(s)|),
\end{split}
\end{equation}
for some $e\in\mathbb{S}^2$ (the definition, of course, does not depend on the choice of $e$). Then there is some constant $\widetilde{c}=\widetilde{c}(\sigma,\mu,\nu)\in\{-1,1\}$ with the following property:
\begin{equation}\label{kx20}
\begin{split}
&\text{ the set }\widetilde{I}_k^{\sigma;\mu,\nu}:=\{s\in[2^{k-4},2^{k+4}]\cap I^{\sigma;\mu,\nu}:\,|\Psi^{\si;\mu,\nu}(s)|\leq 2^{-110D}\}\text{ is an interval};\\ &\widetilde{c}\cdot(\partial_s\Psi^{\si;\mu,\nu})(s)\geq 2^{-20D}\text{ for any }s\in\widetilde{I}_k^{\sigma;\mu,\nu}.
\end{split}
\end{equation}
\end{lemma}

\begin{proof}[Proof of Lemma \ref{LemP}] Since $\Phi^{\sigma;\mu,\nu}\in\mathcal{T}'_A$, $q^{\mu,\nu}$ is well defined. We start from the elementary formula
\begin{equation*}
\begin{split}
\vert\Xi^{\mu,\nu}(\xi,\eta)\vert&=\vert\Xi^{\mu,\nu}(\xi,\eta)-\Xi^{\mu,\nu}(q^{\mu,\nu}(\eta),\eta)\vert\\
&\approx_{C_b,\varepsilon}\vert\lambda_{\sigma_1}^\prime(\vert\xi-\eta\vert)-\lambda_{\sigma_1}^\prime(\vert q^{\mu,\nu}(\eta)-\eta\vert)\vert\vert\\
&+\max(\lambda_{\sigma_1}^\prime(\vert\xi-\eta\vert),\lambda_{\sigma_1}^\prime(\vert q^{\mu,\nu}(\eta)-\eta\vert))\Big|\frac{\xi-\eta}{\vert\xi-\eta\vert}-\frac{q^{\mu,\nu}(\eta)-\eta}{\vert q^{\mu,\nu}(\eta)-\eta\vert}\Big|.
\end{split}
\end{equation*}
Since $\lambda_{\sigma_1}^\prime(r)\ge 2^{-2D}$ and $\lambda_{\sigma_1}^{\prime\prime}(r)\ge 2^{-2D}(1+r)^{-3}$ if $r\geq 2^{-D/2-10}$, the condition $\vert\Xi^{\mu,\nu}(\xi,\eta)\vert\leq\delta$ shows that
\begin{equation*}
\big|\vert\xi-\eta\vert-\vert q^{\mu,\nu}(\eta)-\eta\vert\big|+\Big|\frac{\xi-\eta}{\vert\xi-\eta\vert}-\frac{q^{\mu,\nu}(\eta)-\eta}{\vert q^{\mu,\nu}(\eta)-\eta\vert}\Big|\le 2^{10D}\delta.
\end{equation*}
This shows that
\begin{equation}\label{XiDetermined}
\vert \xi- q^{\mu,\nu}(\eta)\vert\le 2^{20D}\delta\quad\text{ and }\quad |f^{\sigma;\mu,\nu}(|\eta|)|\leq 2^{30D}\delta.
\end{equation}
where $f^{\sigma;\mu,\nu}:[0,\infty)\to\mathbb{R}$ is defined by
\begin{equation}\label{fdef}
f^{\sigma;\mu,\nu}(r):=\Phi^{\sigma;\mu,\nu}(q^{\mu,\nu}(re),re)=\lambda_\sigma(|\widetilde{t}^{\mu,\nu}(r)|)-\iota_1\lambda_{\sigma_1}(t^{\sigma_1\sigma_2}(r))-\iota_2\lambda_{\sigma_2}(r).
\end{equation}

We turn now to the proof of the lemma. We observe first that \eqref{r1to1} follows from the formula \eqref{dr} and the bounds \eqref{CaseAAwayFromRast}, \eqref{kx11}, and \eqref{UpperBdT}. We note also that the conclusion that $\widetilde{I}_k^{\sigma;\mu,\nu}$ is a closed interval in the first line of \eqref{kx20} is a consequence of the existence of a constant $\widetilde{c}$ satisfying the inequality $\widetilde{c}(\partial_s\Psi^{\si;\mu,\nu})(s)\geq 2^{-20D}$ for any $s\in\widetilde{I}_k^{\sigma;\mu,\nu}$ in the second line of \eqref{kx20}. 

We prove the claims in the lemma by analyzing several cases.

{\bf{Case 1.}} $\Phi^{\sigma;\mu,\nu}\in\{\Phi^{b;b+,e+}, \Phi^{b;e+,e+}\}$. In this case we have
\begin{equation}\label{thy1}
\begin{split}
&t^{\sigma_1\sigma_2}(0)=0,\quad\chi^{\sigma;\mu,\nu}_A=\mathbf{1}_{(2^{-2D},\infty)},\quad\widetilde{t}^{\mu,\nu}(r)=r+t^{\sigma_1\sigma_2}(r),\quad r^{\mu,\nu}(s)\in[0,s],\\
&f^{\sigma;\mu,\nu}(r)=\lambda_b(r+t^{\sigma_1\sigma_2}(r))-\lambda_{\sigma_1}(t^{\sigma_1\sigma_2}(r))-\lambda_e(r),\\
&\Psi^{\si;\mu,\nu}(s)=\lambda_b(s)-\lambda_{\sigma_1}(s-r^{\mu,\nu}(s))-\lambda_{e}(r^{\mu,\nu}(s)),\\
&(\partial_s\Psi^{\si;\mu,\nu})(s)=\lambda'_b(s)-\lambda'_e(r^{\mu,\nu}(s)).
\end{split}
\end{equation}
The claims \eqref{kx11} and \eqref{kx20} with $\widetilde{c}=1$ follow easily (using for example \eqref{nbc2}), and the claim \eqref{CaseAAwayFromRast} is trivial.
\medskip

{\bf{Case 2.}} $\Phi^{\sigma;\mu,\nu}\in\{\Phi^{e;e+,i+}, \Phi^{e;b+,i+}, \Phi^{b;e+,i+}, \Phi^{b;b+,i+}\}$. In this case we have
\begin{equation}\label{thy2}
\begin{split}
&t^{\sigma_1\sigma_2}(0)\approx_{C_b,\varepsilon}1,\quad\chi^{\sigma;\mu,\nu}_A=\mathbf{1}_{(t^{\sigma_1\sigma_2}(0)+2^{-2D},\infty)},\quad\widetilde{t}^{\mu,\nu}(r)=r+t^{\sigma_1\sigma_2}(r),\quad r^{\mu,\nu}(s)\in[0,s],\\
&f^{\sigma;\mu,\nu}(r)=\lambda_\sigma(r+t^{\sigma_1\sigma_2}(r))-\lambda_{\sigma_1}(t^{\sigma_1\sigma_2}(r))-\lambda_i(r),\\
&\Psi^{\si;\mu,\nu}(s)=\lambda_\sigma(s)-\lambda_{\sigma_1}(s-r^{\mu,\nu}(s))-\lambda_{i}(r^{\mu,\nu}(s)),\\
&(\partial_s\Psi^{\si;\mu,\nu})(s)=\lambda'_\sigma(s)-\lambda'_i(r^{\mu,\nu}(s)).
\end{split}
\end{equation}
Notice that
\begin{equation*}
(\partial_rf^{\sigma;\mu,\nu})(r)=[1+(\partial t^{\sigma_1\sigma_2})(r)][\lambda'_\sigma(r+t^{\sigma_1\sigma_2}(r))-\lambda'_{\sigma_1}(t^{\sigma_1\sigma_2}(r))].
\end{equation*}
Therefore, using also \eqref{UpperBdT} and Lemma \eqref{tech99}, $(\partial_rf^{\sigma;\mu,\nu})(r)\geq_{C_b,\varepsilon}r(1+r^2)^{-3/2}$ and $f^{\sigma;\mu,\nu}(0)\geq 0$ if $\Phi^{\sigma;\mu,\nu}\in\{\Phi^{e;e+,i+}, \Phi^{b;e+,i+}, \Phi^{b;b+,i+}\}$. Therefore the inequality $|f^{\sigma;\mu,\nu}(|\eta|)|\leq 2^{-20D}$ in \eqref{XiDetermined} cannot be verified in these cases for any $(\xi,\eta)$ as in \eqref{Added1}, and the conclusions of the lemma are trivial.

On the other hand, if $\Phi^{\sigma;\mu,\nu}=\Phi^{e;b+,i+}$ then the claims in \eqref{kx11} follow easily, using \eqref{XiDetermined} and the hypothesis of the lemma. To prove the remaining claims we show first that
\begin{equation}\label{thy2.1}
|\eta|\leq 3T^{-1/2}/4\leq 3r_\ast/4.
\end{equation}
Indeed, starting from the inequalities $|f^{e;b+,i+}(|\eta|)|\leq 2^{-20D}$ and $t^{bi}(r)\leq \sqrt{\varepsilon/C_b}$ (see \eqref{Range}), and using also \eqref{SimpleBdLie}, it follows that
\begin{equation*}
2^{-20D}\geq \lambda_e(|\eta|)-\varepsilon^{-1/2}\sqrt{1+2\varepsilon}-\lambda_i(|\eta|)\geq \varepsilon^{-1/2}\big(\sqrt{1+T|\eta|^2}-\sqrt{1+2\varepsilon}\big)-\sqrt{(T+1)(\varepsilon+1)}|\eta|.
\end{equation*}
The desired bound \eqref{thy2.1} follows. This clearly implies the bound \eqref{CaseAAwayFromRast}.

Finally, to prove \eqref{kx20}, we calculate
\begin{equation}\label{thy2.2}
\begin{split}
&\Psi^{e;b+,i+}(t^{bi}(0))=\lambda_e(t^{bi}(0))-\lambda_{b}(t^{bi}(0))\leq -C^{-1}_{C_b,\varepsilon},\\
&(\partial_s\Psi^{e;b+,i+})(t^{bi}(0))=\lambda'_e(t^{bi}(0))-\lambda'_i(0)=\lambda'_e(t^{bi}(0))-\lambda'_b(t^{bi}(0))\leq -C^{-1}_{C_b,\varepsilon},\\
&(\partial^2_s\Psi^{e;b+,i+})(s)=\lambda''_e(s)-(\partial_s r^{b+,i+})(s)\lambda''_i(r^{b+,i+}(s)).
\end{split}
\end{equation}
Therefore $(\partial^2_s\Psi^{e;b+,i+})(s)\geq C^{-1}_{C_b,\varepsilon}$ for all $s\in[t^{bi}(0),\infty)$ for which $r^{b+,i+}(s)\leq r_\ast$. On the other hand, as in the proof of \eqref{thy2.1}, if $s\in[2^{k-4},2^{k+4}]$ has the property that $|\Psi^{e;b+,i+}(s)|\leq 2^{-20D}$ then $r^{b+,i+}(s)\leq 4r_\ast/5$. The desired conclusion \eqref{kx20} follows with $\widetilde{c}=1$ by combining the inequalities in \eqref{thy2.2}.

\medskip

{\bf{Case 3.}} $\Phi^{\sigma;\mu,\nu}\in\{\Phi^{i;b-,e+}, \Phi^{i;b+,e-},\Phi^{e;b+,e-}, \Phi^{b;b+,e-}\}$. In this case we have
\begin{equation}\label{thy3}
\begin{split}
&t^{\sigma_1\sigma_2}(0)=0,\quad\chi^{\sigma;\mu,\nu}_A=\mathbf{1}_{(2^{-2D},\infty)},\quad\widetilde{t}^{\mu,\nu}(r)=r-t^{\sigma_1\sigma_2}(r),\quad r^{\mu,\nu}(s)\in[s,\infty),\\
&f^{\sigma;\mu,\nu}(r)=\lambda_\sigma(r-t^{\sigma_1\sigma_2}(r))-\iota_1\lambda_{b}(t^{\sigma_1\sigma_2}(r))-\iota_2\lambda_e(r),\\
&\Psi^{\si;\mu,\nu}(s)=\lambda_\sigma(s)-\iota_1\lambda_{b}(r^{\mu,\nu}(s)-s)-\iota_2\lambda_{e}(r^{\mu,\nu}(s)),\\
&(\partial_s\Psi^{\si;\mu,\nu})(s)=\lambda'_\sigma(s)-\iota_2\lambda'_e(r^{\mu,\nu}(s)).
\end{split}
\end{equation}
The claims in \eqref{kx11} follow easily, using the hypothesis and \eqref{XiDetermined}. The claim \eqref{CaseAAwayFromRast} is trivial. The conclusion \eqref{kx20} also follows from the formulas above if $\iota_2=-$, with $\widetilde{c}=1$.

It remains to prove \eqref{kx20} when $\Phi^{\sigma;\mu,\nu}=\Phi^{i;b-,e+}$, in which case we set $\widetilde{c}=-1$. For $s\geq r_\ast$ we estimate
\begin{equation*}
(\partial_s\Psi^{i;b-,e+})(s)=\lambda'_i(s)-\lambda'_e(r^{\mu,\nu}(s))\leq 1-\lambda'_e(r_\ast)\leq -1,
\end{equation*}
which gives the desired conclusion \eqref{kx20} when $s\geq r_\ast$. On the other hand, we calculate
\begin{equation*}
\begin{split}
&\Psi^{i;b-,e+}(0)=\lambda_i(0)+\lambda_{b}(0)-\lambda_{e}(0)=0,\\
&(\partial_s\Psi^{i;b-,e+})(0)=\lambda'_i(0)-\lambda'_e(0)\geq C^{-1}_{C_b,\varepsilon},\\
&(\partial^2_s\Psi^{i;b-,e+})(s)=\lambda''_i(s)-(\partial_s r^{b-,e+})(s)\lambda''_e(r^{b-,e+}(s)).
\end{split}
\end{equation*}
Therefore $(\partial^2_s\Psi^{i;b-,e+})(s)\leq -C^{-1}_{C_b,\varepsilon}$ for $s\in[0,r_\ast]$, and the desired conclusion \eqref{kx20} with $\widetilde{c}=-1$ follows in this range as well.

\medskip

{\bf{Case 4.}} $\Phi^{\sigma;\mu,\nu}\in\{\Phi^{i;e+,i-}, \Phi^{i;b+,i-}\}$. In this case we have
\begin{equation}\label{thy4}
\begin{split}
&t^{\sigma_1\sigma_2}(0)\approx_{C_b,\varepsilon}1,\quad\chi^{\sigma;\mu,\nu}_A=\mathbf{1}_{(t^{\sigma_1\sigma_2}(0)+2^{-2D},\infty)},\quad\widetilde{t}^{\mu,\nu}(r)=r-t^{\sigma_1\sigma_2}(r),\quad r^{\mu,\nu}(s)\in[s,\infty),\\
&f^{\sigma;\mu,\nu}(r)=\lambda_i(|r-t^{\sigma_1\sigma_2}(r)|)-\lambda_{\sigma_1}(t^{\sigma_1\sigma_2}(r))+\lambda_i(r),\\
&\Psi^{\si;\mu,\nu}(s)=\lambda_i(s)-\lambda_{\sigma_1}(r^{\mu,\nu}(s)-s)+\lambda_{i}(r^{\mu,\nu}(s)),\\
&(\partial_s\Psi^{\si;\mu,\nu})(s)=\lambda'_i(s)+\lambda'_i(r^{\mu,\nu}(s)).
\end{split}
\end{equation}
Recalling that $t^{\sigma_1i}(r)\leq \sqrt{3\varepsilon/T}$ and $\lambda_i(r)\leq \sqrt{1+r^2}$ for any $r\in[0,\infty)$, see \eqref{Range} and \eqref{mk0.1}, we estimate
\begin{equation*}
\lambda_i(|r-t^{\sigma_1\sigma_2}(r)|)-\lambda_{\sigma_1}(t^{\sigma_1\sigma_2}(r))+\lambda_i(r)\leq -\varepsilon^{-1/2}+2\sqrt{1+r^2}
\end{equation*}
for any $r\in[0,\infty)$. The inequality $|f^{\sigma;\mu,\nu}(|\eta|)|\leq 2^{-20D}$, see \eqref{XiDetermined}, then shows that $|\eta|\geq (3\varepsilon)^{-1/2}$. Therefore $|q^{\mu,\nu}(\eta)|=|\eta|-t^{\sigma_1i}(|\eta|)\geq |\eta|/2$, and the conclusions in \eqref{kx11} follow using also \eqref{XiDetermined}. The claim \eqref{CaseAAwayFromRast} follows from $|\eta|\geq (3\varepsilon)^{-1/2}$. Finally, the conclusion \eqref{kx20} with $\widetilde{c}=1$ follows from last formula in \eqref{thy4}.

\medskip

{\bf{Case 5.}} $\Phi^{\sigma;\mu,\nu}=\Phi^{e;b+,i-}$. In this case we have
\begin{equation}\label{thy7}
\begin{split}
&t^{\sigma_1\sigma_2}(0)\approx_{C_b,\varepsilon}1,\quad\chi^{\sigma;\mu,\nu}_A=\mathbf{1}_{(0,t^{\sigma_1\sigma_2}(0)-2^{-2D})},\quad\widetilde{t}^{\mu,\nu}(r)=r-t^{bi}(r),\quad r^{\mu,\nu}(s)\in[s,\infty),\\
&f^{\sigma;\mu,\nu}(r)=\lambda_e(|r-t^{bi}(r)|)-\lambda_{b}(t^{bi}(r))+\lambda_{i}(r),\\
&\Psi^{\si;\mu,\nu}(s)=\lambda_e(s)-\lambda_b(r^{\mu,\nu}(-s)+s)+\lambda_i(r^{\mu,\nu}(-s)),\\
&(\partial_s\Psi^{\si;\mu,\nu})(s)=\lambda'_e(s)-\lambda'_b(r^{\mu,\nu}(-s)+s).
\end{split}
\end{equation}
Clearly, $-f^{\sigma;\mu,\nu}(0)\gtrsim_{C_b,\varepsilon}1$. Extending $\lambda_e$ as an even function on $\mathbb{R}$ we calculate, for $r\geq 0$,
\begin{equation*}
\begin{split}
(\partial_rf^{\sigma;\mu,\nu})(r)&=(1-(\partial t^{bi})(r))\lambda'_e(r-t^{bi}(r))-(\partial t^{bi})(r)\lambda'_{b}(t^{bi}(r))+\lambda'_i(r)\\
&=[1-(\partial t^{bi})(r)][\lambda'_{b}(t^{bi}(r))+\lambda'_e(r-t^{bi}(r))].
\end{split}
\end{equation*}
Let $r_0\in[0,\infty)$ denote the unique number with the property that $r_0=t^{bi}(r_0)$. In view of \eqref{Range}, $r_0\leq \sqrt{\varepsilon/C_b}\leq r_\ast/2$. Moreover, $r-t^{bi}(r)\geq 0$ if $r\geq r_0$ and $r-t^{bi}(r)\leq 0$ if $r\leq r_0$. Therefore $(\partial_rf^{\sigma;\mu,\nu})(r)\gtrsim_{C_b,\varepsilon}1$ if $r\geq r_0$ and $(\partial_rf^{\sigma;\mu,\nu})(r)\gtrsim_{C_b,\varepsilon}r$ if $r\in[0,r_0]$. Moreover,
\begin{equation*}
\begin{split}
f^{\sigma;\mu,\nu}(r_0)&=\lambda_e(0)-\lambda_{b}(r_0)+\lambda_{i}(r_0)=\int_{0}^{r_0}[\lambda'_i(\rho)-\lambda'_b(\rho)]\,d\rho\\
&\geq r_0\lambda'_i(r_0)-r_0\lambda'_b(r_0)+\int_{0}^{r_0}[\lambda'_b(r_0)-\lambda'_b(\rho)]\,d\rho\gtrsim_{C_b,\varepsilon}1.
\end{split}
\end{equation*}
Therefore the strictly increasing function $f^{\sigma;\mu,\nu}$ has a unique zero in the interval $(2^{-D/2},r_0-2^{-D/2})$. It follows from \eqref{XiDetermined} that if $\eta=re$ then $r\in(2^{-D},r_0-2^{-D})$ and $|\xi-(r-t^{bi}(r))e|\leq 2^{20D}\delta$. The conclusions in \eqref{kx11} follow. The conclusion \eqref{CaseAAwayFromRast} follows using also $r_0\leq r_\ast/2$. The inequality \eqref{kx20} follows using, for example, \eqref{nbc2}. 
\end{proof}

{\bf{Remark: }} The analysis in Case 2 in the proof of Lemma \ref{LemP} shows that the phases $\Phi^{e;e+,i+}$, $\Phi^{b;e+,i+}$, and $\Phi^{b;b+,i+}$ are, in fact, nonresonant, in the sense that there are no points $(\xi,\eta)\in\mathbb{R}^3\times\mathbb{R}^3$ satisfying \eqref{Added1}. Therefore in the proof of Proposition we may assume that \begin{equation}\label{thy6}
\begin{split}
\Phi^{\sigma;\mu,\nu}\in\mathcal{T}''_A:=\{&\Phi^{i;e+,i-}, \Phi^{i;b+,i-}, \Phi^{i;b-,e+}, \Phi^{i;b+,e-}, \Phi^{e;b+,i+}, \Phi^{e;b+,i-},\\
& \Phi^{e;b+,e-}, \Phi^{b;e+,e+}, \Phi^{b;b+,e+},\Phi^{b;b+,e-}\}.
\end{split}
\end{equation}

\subsection{Proof of Proposition \ref{reduced4}}

Once the functions $p^{\sigma;\mu,\nu}$ have been created, the rest of the analysis follows similar lines to the analysis of \cite[Section 4]{IoPa2}. The main ingredients we  need come from the refined $B_{k,j}$ norms and additional $L^2$ orthogonality arguments. We prove Proposition \ref{reduced4} in two steps, see Lemma \ref{BigBound7A} and Lemma \ref{BigBound8A} below, depending on the maximum in the definition of $\kappa$.

\begin{lemma}\label{BigBound7A}
The bound \eqref{tln3} holds provided that \eqref{tln2} and \eqref{thy6} hold and, in addition,
\begin{equation}\label{cle0}
\max(j_1,j_2)\leq (m-\beta^2m)/2,
\end{equation}
with
\begin{equation}\label{cle1}
\kappa:=2^{(\beta^2m-m)/2}.
\end{equation}
\end{lemma}

\begin{proof}[Proof of Lemma \ref{BigBound7A}] For simplicity of notation, let
\begin{equation}\label{CaseAG}
\begin{split}
G(\xi):=&\mathcal{F}[P_kR_{m,\kappa}^{\sigma;\mu,\nu}(f_{k_1,j_1}^\mu,f_{k_2,j_2}^\nu)](\xi),\\
=&\varphi_k(\xi)\int_{\mathbb{R}}\int_{\mathbb{R}^3}e^{is\Phi^{\sigma;\mu,\nu}(\xi,\eta)}\chi_{R}^{\sigma;\mu,\nu}(\xi,\eta)q_m(s)\cdot \widehat{f_{k_1,j_1}^\mu}(\xi-\eta,s)\widehat{f_{k_2,j_2}^\nu}(\eta,s)\,d\eta ds,
\end{split}
\end{equation}
where $\chi^{\sigma;\mu,\nu}_A$ was defined in Lemma \ref{LemP}, and, as before,
\begin{equation*}
\chi_R^{\sigma;\mu,\nu}(\xi,\eta)=\varphi(2^{D^2+\max(0,k_1,k_2)}\Phi^{\sigma;\mu,\nu}(\xi,\eta))\varphi(\vert\Xi^{\mu,\nu}(\xi,\eta)\vert/\kappa).
\end{equation*}
Using the $L^\infty$ bounds in \eqref{nh9} and \eqref{kx11} (with $\delta=4\kappa$), we see easily that
\begin{equation}\label{cle3}
\|G\|_{L^\infty}\lesssim \ka^3\cdot 2^m\lesssim 2^{-m/2}2^{3\beta^2 m/2}.
\end{equation}
This suffices to prove \eqref{tln3} if, for example, $j\leq m(1/2-4\beta)$. To cover the entire range $j\leq m+D$ we integrate by parts in $s$. 

In the argument below we may assume that $G\neq 0$; in particular this guarantees that the main assumption \eqref{Added1} is satisfied. With $\Psi^{\si;\mu,\nu}(|\xi|)=\Phi^{\si;\mu,\nu}(\xi,p^{\sigma;\mu,\nu}(\xi))$, defined as in \eqref{kxz11}, 
assume that
\begin{equation}\label{cle6}
2^m\vert\Psi^{\sigma;\mu,\nu}(\vert\xi\vert)\vert\in [2^l,2^{l+1}],\,\,l\in[\beta m,\infty)\cap\mathbb{Z}.
\end{equation}
Then, using Lemma \ref{LemP}, we see that if $|\eta-p^{\sigma;\mu,\nu}(\xi)|\leq 2^{50D}\kappa$ then
\begin{equation*}
\vert\Phi^{\sigma;\mu,\nu}(\xi,\eta)-\Psi^{\sigma;\mu,\nu}(\vert\xi\vert)\vert\le\vert\eta-p^{\sigma;\mu,\nu}(\xi)\vert\cdot\sup_{\vert \zeta-p^{\sigma;\mu,\nu}(\xi)\vert\le 2^{50D}\kappa}\vert\Xi^{\mu,\nu}(\xi,\zeta)\vert\leq 2^{60D} \kappa\vert\eta-p^{\sigma;\mu,\nu}(\xi)\vert,
\end{equation*}
since $\Xi^{\mu,\nu}(\xi,p^{\mu,\nu}(\xi))=0$. Therefore
\begin{equation*}
 2^m\vert\Phi^{\si;\mu,\nu}(\xi,\eta)\vert\in[2^{l-3},2^{l+4}]\qquad\text{ if }\qquad \chi_R^{\sigma;\mu,\nu}(\xi,\eta)\neq 0.
\end{equation*}
After integration by parts in $s$ it follows that
\begin{equation*}
 \begin{split}
 |G(\xi)|\lesssim 2^{m-l}|\varphi_k(\xi)|\int_{\mathbb{R}}\int_{\mathbb{R}^3}
&|\chi_R^{\sigma;\mu,\nu}(\xi,\eta)|\,|q'_m(s)|\,|\widehat{f_{k_1,j_1}^\mu}(\xi-\eta,s)|\,|\widehat{f_{k_2,j_2}^\nu}(\eta,s)|\\
&+|\chi_R^{\sigma;\mu,\nu}(\xi,\eta)|\,|q_m(s)|\,|(\partial_s\widehat{f_{k_1,j_1}^\mu})(\xi-\eta,s)|\,|\widehat{f_{k_2,j_2}^\nu}(\eta,s)|\\
&+|\chi_R^{\sigma;\mu,\nu}(\xi,\eta)|\,|q_m(s)|\,|\widehat{f_{k_1,j_1}^\mu}(\xi-\eta,s)|\,|(\partial_s\widehat{f_{k_2,j_2}^\nu})(\eta,s)|\,d\eta ds.
 \end{split}
\end{equation*}
We use now \eqref{nh2}, the last bound in \eqref{nh9}, \eqref{derv2repeat}, and Lemma \ref{LemP}. It follows that
\begin{equation}\label{cle7}
 |G(\xi)|\lesssim 2^{m-l}|\varphi_k(\xi)|\chi_A^{\sigma;\mu,\nu}(|\xi|)\cdot \ka^3\lesssim |\varphi_k(\xi)|\chi_A^{\sigma;\mu,\nu}(|\xi|)\cdot 
2^{-l}2^{-m/2}2^{\beta m/5}
\end{equation}
provided that \eqref{cle6} holds.

We can now prove the desired bound \eqref{tln3}. To make use of \eqref{cle6}--\eqref{cle7} we need a good description of the 
level sets of the functions $\Psi^{\si;\mu,\nu}$. Let
\begin{equation*}
 \begin{split}
&l_0:=\lfloor\beta m+2\rfloor,\quad D_{l_0}:=\{\xi\in\mathbb{R}^3:2^m|\Psi^{\si;\mu,\nu}(|\xi|)|\leq 2^{l_0}\,\text{ and }|\varphi_k(\xi)|\chi_A^{\sigma;\mu,\nu}(|\xi|)\neq 0\},\\
&D_l:=\{\xi\in\mathbb{R}^3:2^m|\Psi^{\si;\mu,\nu}(|\xi|)|\in(2^{l-1},2^l]\,\text{ and }|\varphi_k(\xi)|\chi_A^{\sigma;\mu,\nu}(|\xi|)\neq 0\},\quad l\in[l_0+1,m-100D]\cap\mathbb{Z},\\
&G=\sum_{l=l_0}^{m-100D}G_l,\qquad G_l(\xi):=G(\xi)\cdot\mathbf{1}_{D_l}(\xi).
 \end{split}
\end{equation*}
For \eqref{tln3} it remains to prove that for any $l\in[l_0,m-100D]\cap\mathbb{Z}$
\begin{equation}\label{cle8}
\big\|\widetilde{\varphi}^{(k)}_j\cdot \mathcal{F}^{-1}(G_l)\|_{B_{k,j}}\lesssim 2^{-3\beta^4m}.
\end{equation}

Using \eqref{kx20} in Lemma \ref{LemP}, it follows that there is $\theta^{\sigma;\mu,\nu}=\theta^{\sigma;\mu,\nu}(\mu,\nu,\sigma,k,k_1,k_2,l)\in [2^{-D},\infty)$ 
with the property that
\begin{equation}\label{cle9}
 D_l\subseteq\{\xi\in\mathbb{R}^3:\big||\xi|-\theta^{\sigma;\mu,\nu}\big|\lesssim 2^{l-m}\}.
\end{equation}
Therefore, using also \eqref{cle7} if $l\geq l_0+1$ and \eqref{cle3} if $l=l_0$,
\begin{equation*}
\begin{split}
 \big\|\widetilde{\varphi}^{(k)}_j\cdot \mathcal{F}^{-1}(G_l)\|_{B^1_{k,j}}
&\lesssim 2^{(1+\beta)j}\|G_l\|_{L^2}+\|G_l\|_{L^\infty}\\
&\lesssim 2^{\beta m}2^{-l}2^{-m/2}2^{\beta m/5}\cdot \big(2^{(1+\beta)j}2^{(l-m)/2}+1)\\
&\lesssim 2^{j-m}2^{-l/2}2^{11\beta m/5}+2^{-l}2^{-m/2}2^{6\beta m/5}.
\end{split}
\end{equation*}
This clearly suffices to prove \eqref{cle8} if $l\ge 6\beta m$ or $j\leq m-3\beta m$.

It remains to prove \eqref{cle8} in the remaining case
\begin{equation}\label{cle10}
l\in[l_0, 6\beta m]\cap\mathbb{Z}\qquad\text{ and }\qquad j\in[m-3\beta m,m+D]\cap\mathbb{Z}. 
\end{equation}
For this we need to use the norms $B^2_{k,j}$ defined in \eqref{sec5.4}. Assume first that $l\geq l_0+1$. As before we estimate easily
\begin{equation*}
\begin{split}
 2^{(1-\beta)j}\|G_l\|_{L^2}+\|G_l\|_{L^\infty}&\lesssim 2^{-l}2^{-m/2}2^{\beta m/5}\cdot \big(2^{(1-\beta)m}2^{(l-m)/2}+1)\\
&\lesssim 2^{-l/2}2^{-4\beta m/5}+2^{-l}2^{-m/2}2^{\beta m/5}.
\end{split}
\end{equation*}
Therefore, for \eqref{cle8} it suffices to prove that
\begin{equation}\label{cle13}
2^{\gamma j}\sup_{R\in[2^{-j},2^{k}],\,\xi_0\in\mathbb{R}^3}R^{-2}
\big\|\mathcal{F}\big[\widetilde{\varphi}^{(k)}_j\cdot \mathcal{F}^{-1}(G_l)\big]\big\|_{L^1(B(\xi_0,R))}\lesssim 2^{-3\beta^4m}.
\end{equation}
Since $\big|\mathcal{F}(\widetilde{\varphi}^{(k)}_j)(\xi)\big|\lesssim 2^{3j}(1+2^j|\xi|)^{-6}$, it follows from \eqref{cle7} that
\begin{equation*}
\begin{split}
\big|\mathcal{F}\big[\widetilde{\varphi}^{(k)}_j\cdot \mathcal{F}^{-1}(G_l)\big](\xi)\big|&\lesssim 
\int_{\mathbb{R}^3}|G_l(\xi-\eta)|\cdot 2^{3j}(1+2^j|\eta|)^{-6}\,d\eta\\
&\lesssim 2^{-l}2^{-m/2}2^{\beta m/5}\int_{\mathbb{R}^3}\mathbf{1}_{D_l}(\xi-\eta)\cdot 2^{3j}(1+2^j|\eta|)^{-6}\,d\eta.
\end{split}
\end{equation*}
Therefore, using now \eqref{cle9}, for any $R\in[2^{-j},2^k]$ and $\xi_0\in\mathbb{R}^3$,
\begin{equation*}
R^{-2}\big\|\mathcal{F}\big[\widetilde{\varphi}^{(k)}_j\cdot \mathcal{F}^{-1}(G_l)\big]\big\|_{L^1(B(\xi_0,R))}\lesssim 2^{-l}2^{-m/2}2^{\beta m/5}\cdot 2^{l-m}\lesssim 2^{-3m/2}2^{\beta m/5},
\end{equation*}
and the bound \eqref{cle13} follows.

Similarly, using \eqref{cle3} and \eqref{cle9},
\begin{equation*}
2^{(1-\beta)j}\Vert G_{l_0}\Vert_{L^2}+\Vert G_{l_0}\Vert_{L^\infty} \lesssim 2^{(1-\beta)(j-m)}2^{-\beta m+l_0/2+3\beta^2m}+2^{-m/4}\lesssim 2^{-3\beta^4m}
\end{equation*}
and
\begin{equation*}
\begin{split}
\big|\mathcal{F}\big[\widetilde{\varphi}^{(k)}_j\cdot \mathcal{F}^{-1}(G_{l_0})\big](\xi)\big|&\lesssim 
\int_{\mathbb{R}^3}|G_{l_0}(\xi-\eta)|\cdot 2^{3j}(1+2^j|\eta|)^{-6}\,d\eta\\
&\lesssim 2^{-m/2}2^{3\beta^2 m}\int_{\mathbb{R}^3}\mathbf{1}_{D_{l_0}}(\xi-\eta)\cdot 2^{3j}(1+2^j|\eta|)^{-6}\,d\eta
\end{split}
\end{equation*}
from where we conclude that, for any $R\in[2^{-j},2^k]$ and $\xi_0\in\mathbb{R}^3$,
\begin{equation*}
R^{-2}\big\|\mathcal{F}\big[\widetilde{\varphi}^{(k)}_j\cdot \mathcal{F}^{-1}(G_{l_0})\big]\big\|_{L^1(B(\xi_0,R))}\lesssim 2^{-m/2}2^{3\beta^2 m}\cdot 2^{l_0-m}\lesssim 2^{-3m/2}2^{2\beta m}.
\end{equation*}
The desired bound \eqref{cle8} follows when $l=l_0$, which completes the proof of the lemma.
\end{proof}

\begin{lemma}\label{BigBound8A}
The bound \eqref{tln3} holds provided that \eqref{tln2} and \eqref{thy6} hold and, in addition,
\begin{equation}\label{cle20}
\max(j_1,j_2)\geq (m-\beta^2m)/2,
\end{equation}
with
\begin{equation}\label{cle20.1}
\kappa:=2^{\beta^2m}2^{\max(j_1,j_2)-m}.
\end{equation}
\end{lemma}

\begin{proof}[Proof of Lemma \ref{BigBound8A}] Using definition \eqref{sec5.3}, it suffices to prove that
\begin{equation}\label{cle21}
2^{(1+\beta)j}\big\|\widetilde{\varphi}^{(k)}_j\cdot P_kR_{m,\kappa_1}^{\sigma;\mu,\nu}(f_{k_1,j_1}^\mu,f_{k_2,j_2}^\nu)\big\|_{L^2}+\big\|\mathcal{F}[\widetilde{\varphi}^{(k)}_j\cdot P_kR_{m,\kappa_1}^{\sigma;\mu,\nu}(f_{k_1,j_1}^\mu,f_{k_2,j_2}^\nu)]\big\|_{L^\infty}\lesssim 2^{-2\beta^4m}.
\end{equation}
Let $G=\mathcal{F}P_kR_{m,\kappa_1}^{\sigma;\mu,\nu}(f_{k_1,j_1}^\mu,f_{k_2,j_2}^\nu)$ be given as in \eqref{CaseAG}.
In proving \eqref{cle21} we may assume that $G\neq 0$; in particular this guarantees that the main assumptions \eqref{Added1} of Lemma \ref{LemP} are satisfied. 
We prove first the $L^\infty$ bound in \eqref{cle21}. Assume that $j_1\leq j_2$ (the case $j_1\geq j_2$ is similar). Then, see \eqref{nh9} and \eqref{sec5.8}--\eqref{sec5.815},
\begin{equation*}
 \begin{split}
&\|\widehat{f_{k_1,j_1}^\mu}(s)\|_{L^\infty}\lesssim 1,\\
&\sup_{\xi_0\in\mathbb{R}^3}\|\widehat{f_{k_2,j_2}^\nu}(s)\|_{L^1(B(\xi_0,R))}\lesssim 2^{-(1+\beta)j_2}R^{3/2},\qquad\text{ for any }R\leq 1.
 \end{split}
\end{equation*}
Using \eqref{kx11} in Lemma \ref{LemP} it follows that
\begin{equation*}
\|G\|_{L^\infty}\lesssim 2^m\cdot 2^{-(1+\beta)j_2}\kappa^{3/2}\lesssim 2^{-m/2}2^{2\beta^2 m}2^{(1/2-\beta)j_2}\lesssim 2^{-2\beta^4m},
\end{equation*}
as desired.

To get the $L^2$ bound in \eqref{cle21} it suffices to show that
\begin{equation}\label{cle24}
2^{(2+2\beta)m}\|G\|^2_{L^2}\lesssim 2^{-4\beta^4m}.
\end{equation}
To prove this we need first an orthogonality argument. Let $\chi:\mathbb{R}\to[0,1]$ denote a smooth function supported in the interval $[-2,2]$ with the 
property that
\begin{equation*}
\sum_{n\in\mathbb{Z}}\chi(x-n)=1\qquad\text{ for any }x\in\mathbb{R}.
\end{equation*}
We define the smooth function $\chi':\mathbb{R}^3\to[0,1]$, $\chi'(x,y,z):=\chi(x)\chi(y)\chi(z)$. Recall the functions $\Psi^{\sigma;\mu,\nu}$
defined in \eqref{kxz11}. We define, for any $v\in\mathbb{Z}^3$ and $n\in\mathbb{Z}$,
\begin{equation}\label{cle25}
\begin{split}
G_{v,n}(\xi):=&\chi'(\kappa^{-1}\xi-v)\varphi_k(\xi)\\
&\int_{\mathbb{R}}\int_{\mathbb{R}^3}e^{is\Phi^{\sigma;\mu,\nu}(\xi,\eta)}
\chi_R^{\sigma;\mu,\nu}(\xi,\eta)\chi(2^{-m}\kappa^{-1}s-n)q_m(s)\widehat{f_{k_1,j_1}^\mu}(\xi-\eta,s)\widehat{f_{k_2,j_2}^\nu}(\eta,s)\,d\eta ds,
\end{split}
\end{equation}
and notice that $G=\sum_{v\in\mathbb{Z}^3}\sum_{n\in\mathbb{Z}}G_{v,n}$. In view of Lemma \ref{LemP} (i) we notice also that the functions $\widetilde{G}_{v,s}$ are trivial unless
\begin{equation}\label{cle25.5}
v\in Z_\kappa^{\sigma;\mu,\nu}:=\{w\in\mathbb{Z}^3:\kappa|w|\in[2^{k-4},2^{k+4}]\cap [t^{\sigma_1\sigma_2}(0)+2^{-4D},\infty),\,|\Psi^{\sigma;\mu,\nu}(\kappa |w|)|\leq 2^{-200D}\}.
\end{equation} 

We show now that
\begin{equation}\label{cle26}
 \|G\|^2_{L^2}\lesssim \sum_{v\in Z_\kappa^{\sigma;\mu,\nu}}\sum_{n\in\mathbb{Z}}\|G_{v,n}\|_{L^2}^2+2^{-10m}.
\end{equation}
This additional orthogonality in time allows us a crucial gain of $\kappa^{1/2}$ in the time integration with respect to the trivial bound. To prove this bound we estimate
\begin{equation*}
\|G\|^2_{L^2}\lesssim \sum_{v\in Z_\kappa^{\sigma;\mu,\nu}}\Big\|\sum_{n\in\mathbb{Z}}G_{v,n}\Big\|_{L^2}^2
\lesssim \sum_{v\in Z_\kappa^{\sigma;\mu,\nu}}\sum_{n_1,n_2\in\mathbb{Z}}|\langle G_{v,n_1},G_{v,n_2}\rangle|.
\end{equation*}
Therefore, for \eqref{cle26} it suffices to prove that
\begin{equation}\label{cle27}
|\langle G_{v,n_1},G_{v,n_2}\rangle|\lesssim 2^{-20m}\qquad \text{ if }v\in Z_\kappa^{\sigma;\mu,\nu}\text{ and }|n_1-n_2|\geq 2^{100D}.
\end{equation}
Let $w_n:=n\kappa 2^m\cdot(\Psi^{\sigma;\mu,\nu})'(\kappa|v|)\cdot v/|v|$ and integrate by parts in $\xi$ using Lemma \ref{tech2} with
\begin{equation*}
K\approx |x+w_n|,\,\quad\, \epsilon^{-1}\approx 2^{-\max(j_1,j_2)}.
\end{equation*}
It follows that, for any $n\in\mathbb{Z}$,
\begin{equation*}
|\mathcal{F}^{-1}(G_{v,n})(x)|\lesssim |x+w_n|^{-200}\quad\text{ if }\quad |x+w_n|\geq 2^{50D}\kappa 2^m.
\end{equation*}
Moreover, using Lemma \ref{LemP} and \eqref{cle25.5} we conclude that $|(\Psi^{\sigma;\mu,\nu})'({\kappa}|v|)|\geq 2^{-20D}$. Therefore if $|n_1-n_2|\geq 2^{100D}$ then $|w_{n_1}-w_{n_2}|\geq 2^{70D}\kappa 2^m$ and the bound \eqref{cle27} follows. This completes the proof of \eqref{cle26}.

In view of \eqref{cle26}, for \eqref{cle24} it remains to prove that
\begin{equation}\label{cle27.1}
2^{(2+2\beta)m}\sum_{v\in Z_\kappa^{\sigma;\mu,\nu},\,n\in[2^{-10}\kappa^{-1},2^{10}\kappa^{-1}]}\|G_{v,n}\|^2_{L^2}\lesssim 2^{-4\beta^4m}.
\end{equation}
Let
\begin{equation}\label{cle27.2}
\widetilde{G}_{v,s}(\xi):=\chi'(\kappa^{-1}\xi-v)\varphi_k(\xi)\int_{\mathbb{R}^3}e^{is\Phi^{\sigma;\mu,\nu}(\xi,\eta)}
\chi_R^{\sigma;\mu,\nu}(\xi,\eta)\widehat{f_{k_1,j_1}^\mu}(\xi-\eta,s)\widehat{f_{k_2,j_2}^\nu}(\eta,s)\,d\eta,
\end{equation}
such that
\begin{equation*}
G_{v,n}(\xi)=\int_{\mathbb{R}}\widetilde{G}_{v,s}(\xi)\chi(2^{-m}\kappa^{-1}s-n)q_m(s)\,ds.
\end{equation*}
Therefore, for any $(v,n)$,
\begin{equation*}
\|G_{v,n}\|_{L^2}^2\lesssim 2^m\kappa\int_{\mathbb{R}}\|\widetilde{G}_{v,s}\|^2_{L^2}\chi(2^{-m}\kappa^{-1}s-n)q_m(s)\,ds.
\end{equation*}
Therefore, for \eqref{cle27.1} it suffices to prove that for any $s\in[2^{m-1},2^{m+1}]$
\begin{equation}\label{cle27.3}
2^{(4+2\beta)m}\kappa\sum_{v\in Z_\kappa^{\sigma;\mu,\nu}}\|\widetilde{G}_{v,s}\|^2_{L^2}\lesssim 2^{-4\beta^4m}.
\end{equation}

Assuming $v\in Z_\kappa^{\sigma;\mu,\nu}$ fixed, the variables in the definition of the function $\widetilde{G}_{v,s}$ are naturally restricted as follows:
\begin{equation*}
|\xi-\kappa v|\lesssim \kappa,\qquad |\eta-p^{\sigma;\mu,\nu}(\kappa v)|\lesssim \kappa,
\end{equation*}
where $p^{\sigma;\mu,\nu}$ is defined as in Lemma \ref{LemP}. More precisely, we define the functions $f_1^{v}$ and $f_2^{v}$ by the formulas
\begin{equation}\label{cle36}
\begin{split}
&\widehat{f_1^{v}}(\theta,s):=\varphi(2^{-50D}\kappa^{-1}(\theta-\kappa v+p^{\sigma;\mu,\nu}(\kappa v)))\cdot \widehat{f_{k_1,j_1}^\mu}(\theta,s),\\
&\widehat{f_2^{v}}(\theta,s):=\varphi(2^{-50D}\kappa^{-1}(\theta-p^{\sigma;\mu,\nu}(\kappa v))\cdot \widehat{f_{k_2,j_2}^\nu}(\theta,s).
\end{split}
\end{equation}
Since $|p^{\sigma;\mu,\nu}(\kappa v_1)-p^{\sigma;\mu,\nu}(\kappa v_2)|\geq 2^{-80D}\kappa$ and $\big|[\kappa v_1-p^{\sigma;\mu,\nu}(\kappa v_1)]-[\kappa v_2-p^{\sigma;\mu,\nu}(\kappa v_2)]\big|\geq 2^{-80D}\kappa$ whenever $|v_1-v_2|\gtrsim 1$ (these inequalities are consequences of the lower bounds in the first line of \eqref{r1to1} in Lemma \ref{LemP}), it follows by orthogonality that, for any $s\in\mathbb{R}$,
\begin{equation}\label{cle36.5}
\begin{split}
&\sum_{v\in Z_\kappa^{\sigma;\mu,\nu}}\|f_1^{v}(s)\|_{L^2}^2\lesssim\|f_{k_1,j_1}^\mu(s)\|_{L^2}^2\lesssim 2^{-2j_1+2\beta j_1},\\
&\sum_{v\in Z_\kappa^{\sigma;\mu,\nu}}\|f_2^{v}(s)\|_{L^2}^2\lesssim\|f_{k_2,j_2}^\nu(s)\|_{L^2}^2\lesssim 2^{-2j_2+2\beta j_2}.
\end{split}
\end{equation}

For any $v\in\mathbb{R}^3$ and $g_1,g_2\in L^2(\mathbb{R}^3)$ let
\begin{equation}\label{cle38}
A_{v}(g_1,g_2)(\xi):=\chi'(\kappa^{-1}\xi-v)\varphi_k(\xi)\int_{\mathbb{R}^3}
\chi_R^{\sigma;\mu,\nu}(\xi,\eta)\mathcal{F}(P_{[k_1-4,k_1+4]}g_1)(\xi-\eta)\mathcal{F}(P_{[k_2-4,k_2+4]}g_2)(\eta)\,d\eta.
\end{equation}
We observe that 
\begin{equation*}
\begin{split}
&\widetilde{G}_{v,s}(\xi)=e^{is\Lambda_\sigma(\xi)}A_{v}[Ef_1^{v}(s),Ef_2^{v}(s)](\xi),\\
&Ef_1^{v}(s)=e^{-is\widetilde{\Lambda}_\mu}f_1^{v}(s),\qquad Ef_2^{v}(s)=e^{-is\widetilde{\Lambda}_\nu}f_2^{v}(s).
\end{split}
\end{equation*}
Therefore for \eqref{cle27.3} it suffices to prove that, for any $s\in[2^{m-1},2^{m+1}]$,
\begin{equation}\label{cle39}
2^{(4+2\beta)m}\kappa\sum_{v\in Z_\kappa^{\sigma;\mu,\nu}}\|A_{v}(Ef_1^{v}(s),Ef_2^{v}(s))\|^2_{L^2}\lesssim 2^{-4\beta^4m}.
\end{equation}

We notice now that if $p,q\in[2,\infty]$, $1/p+1/q=1/2$, then
\begin{equation}\label{cle40}
\|A_{v}(g_1,g_2)\|_{L^2}\lesssim\|g_1\|_{L^p}\|g_2\|_{L^q}.
\end{equation}
Indeed, as in the proof of Lemma \ref{tech2}, we write
\begin{equation*}
\mathcal{F}^{-1}(A_{v}(g_1,g_2))(x)=c\int_{\mathbb{R}^3\times\mathbb{R}^3}g_1(y)g_2(z)K_v(x;y,z)\,dydz,
\end{equation*}
where
\begin{equation*}
\begin{split}
K_v(x;y,z):=\int_{\mathbb{R}^3\times\mathbb{R}^3}&e^{i(x-y)\cdot \xi}e^{i(y-z)\cdot \eta}\chi'(\kappa^{-1}\xi-v)\varphi(\kappa^{-1}\Xi^{\mu,\nu}(\xi,\eta))\\
&\times \varphi_k(\xi)\varphi(2^{D^2+\max(0,k_1,k_2)}\Phi^{\sigma;\mu,\nu}(\xi,\eta))\varphi_{[k_1-4,k_1+4]}(\xi-\eta)\varphi_{[k_2-4,k_2+4]}(\eta)\,d\xi d\eta.
\end{split}
\end{equation*}
We recall that $k,k_1,k_2\in[-D/2,D/2]$ and integrate by parts in $\xi$ and $\eta$. Using also Lemma \ref{LemP}, it follows that
\begin{equation*}
|K_v(x;y,z)|\lesssim\kappa^3(1+\kappa^{-1}|x-y|)^{-4}\cdot \kappa^3(1+\kappa^{-1}|y-z|)^{-4},
\end{equation*}
and the desired estimate \eqref{cle40} follows.

We can now prove the main estimate \eqref{cle39}. Assume first that
\begin{equation}\label{cle45}
\max(j_1,j_2)\leq (3/5-\beta) m.
\end{equation}
By symmetry, we may assume again that $j_1\leq j_2$ and estimate 
\begin{equation*}
\|Ef_1^{v}(s)\|_{L^\infty}\lesssim \|\widehat{f_1^{v}}(s)\|_{L^1}\lesssim \kappa^3.
\end{equation*}
Therefore, using \eqref{cle40} and \eqref{cle36.5}, the left-hand side of \eqref{cle39} is dominated by
\begin{equation*}
C2^{4m+2\beta m}\kappa\sum_{v\in Z_\kappa^{\sigma;\mu,\nu}}\kappa^6\|Ef_2^{v}(s)\|^2_{L^2}\lesssim 2^{4m+2\beta m}\kappa^7\cdot 2^{-2j_2+2\beta j_2}\lesssim 2^{-3(1-\beta)m}2^{(5+2\beta)j_2},
\end{equation*}
and the desired bound \eqref{cle39} follows provided that \eqref{cle45} holds.

Assume now that
\begin{equation}\label{cle41}
\max(j_1,j_2)\ge (3/5-\beta)m,\quad \max(j_1,j_2)-\min(j_1,j_2)\geq 8\beta m.
\end{equation}
By symmetry, we may assume that $j_1\leq j_2$ and estimate, using \eqref{CaseAAwayFromRast}, \eqref{nh9}, and either \eqref{ccc7}, \eqref{ccc25}, \eqref{ccc33} or \eqref{ccc43.5},
\begin{equation*}
\|Ef_1^{v}(s)\|_{L^\infty}\lesssim 2^{-3m/2}2^{(1/2+\beta)j_1}.
\end{equation*}
Therefore, using \eqref{cle40} and \eqref{cle36.5}, the left-hand side of \eqref{cle39} is dominated by
\begin{equation*}
\begin{split}
C2^{4m+2\beta m}\kappa\sum_{v\in Z_\kappa^{\sigma;\mu,\nu}}2^{-3m}2^{(1+2\beta)j_1}\|Ef_2^{v}(s)\|^2_{L^2}&\lesssim 2^{m+2\beta m}\kappa\cdot 2^{(1+2\beta)j_1}2^{-2j_2+2\beta j_2}\\
&\lesssim 2^{j_1-j_2}2^{3\beta m}2^{2\beta j_1}2^{2\beta j_2},
\end{split}
\end{equation*}
and the desired bound \eqref{cle39} follows provided that \eqref{cle41} holds.

Finally, assume that
\begin{equation}\label{cle47}
\max(j_1,j_2)-\min(j_1,j_2)\leq 8\beta m\qquad\text{ and }\qquad\max(j_1,j_2)\geq (3/5-\beta) m.
\end{equation}
In this case we need the more refined decomposition in \eqref{sec5.8}--\eqref{sec5.815}. More precisely, using the definitions, for $s\in[2^{m-1},2^{m+1}]$ fixed we decompose
\begin{equation*}
f^\mu_{k_1,j_1}(s)=P_{[k_1-2,k_1+2]}(g_1+h_1),\qquad f^\nu_{k_2,j_2}(s)=P_{[k_2-2,k_2+2]}(g_2+h_2),
\end{equation*}
where
\begin{equation}\label{cle48}
g_1=g_1\cdot \widetilde{\varphi}^{(k_1)}_{[j_1-2,j_1+2]},\quad g_2=g_2\cdot \widetilde{\varphi}^{(k_2)}_{[j_2-2,j_2+2]},
\end{equation}
and
\begin{equation}\label{cle49}
\begin{split}
&2^{(1+\beta)j_1}\|g_1\|_{L^2}+2^{(1-\beta)j_1}\|h_1\|_{L^2}+2^{\gamma j_1}\sup_{R\in[2^{-j_1},2^{k_1}],\theta_0\in\mathbb{R}^3}R^{-2}\|\widehat{h_1}\|_{L^1(B(\theta_0,R))}\lesssim 1,\\
&2^{(1+\beta)j_2}\|g_2\|_{L^2}+2^{(1-\beta)j_2}\|h_2\|_{L^2}+2^{\gamma j_2}\sup_{R\in[2^{-j_2},2^{k_2}],\theta_0\in\mathbb{R}^3}R^{-2}\|\widehat{h_2}\|_{L^1(B(\theta_0,R))}\lesssim 1.
\end{split}
\end{equation}
Then, we define the functions $g_1^{v}, h_1^{v}, g_2^{v}, h_2^{v}$ by the formulas (compare with \eqref{cle36}),
\begin{equation}\label{cle50}
\begin{split}
&\widehat{g_1^{v}}(\theta):=\varphi(2^{-50D}\kappa^{-1}(\theta-\kappa v+p^{\sigma;\mu,\nu}(\kappa v)))\cdot \mathcal{F}(P_{[k_1-2,k_1+2]}g_1)(\theta),\\
&\widehat{h_1^{v}}(\theta):=\varphi(2^{-50D}\kappa^{-1}(\theta-\kappa v+p^{\sigma;\mu,\nu}(\kappa v)))\cdot \mathcal{F}(P_{[k_1-2,k_1+2]}h_1)(\theta),\\
&\widehat{g_2^{v}}(\theta):=\varphi(2^{-50D}\kappa^{-1}(\theta-p^{\sigma;\mu,\nu}(\kappa v)))\cdot \mathcal{F}(P_{[k_2-2,k_2+2]}g_2)(\theta),\\
&\widehat{h_2^{v}}(\theta):=\varphi(2^{-50D}\kappa^{-1}(\theta-p^{\sigma;\mu,\nu}(\kappa v)))\cdot \mathcal{F}(P_{[k_2-2,k_2+2]}h_2)(\theta).
\end{split}
\end{equation}
As in \eqref{cle36.5}, using $L^2$ orthogonality and \eqref{cle49}, we have
\begin{equation}\label{cle50.5}
\begin{split}
&\sum_{v\in Z_\kappa^{\sigma;\mu,\nu}}\|g_1^{v}\|_{L^2}^2\lesssim 2^{-2j_1-2\beta j_1},\quad \sum_{v\in Z_\kappa^{\sigma;\mu,\nu}}\|h_1^{v}\|_{L^2}^2\lesssim 2^{-2j_1+2\beta j_1},\\
&\sum_{v\in Z_\kappa^{\sigma;\mu,\nu}}\|g_2^{v}\|_{L^2}^2\lesssim 2^{-2j_2-2\beta j_2}, \quad \sum_{v\in Z_\kappa^{\sigma;\mu,\nu}}\|h_2^{v}\|_{L^2}^2\lesssim 2^{-2j_2+2\beta j_2}.
\end{split}
\end{equation}

Let $E^\mu_sf=e^{-is\widetilde{\Lambda}_\mu}f$. Using \eqref{CaseAAwayFromRast}, and either \eqref{ccc7}, \eqref{ccc25}, \eqref{ccc33} or \eqref{ccc43.5} together with \eqref{cle48}--\eqref{cle49}, we derive the $L^\infty$ bounds
\begin{equation}\label{cle51}
\begin{split}
&\|E_s^\mu g_1^{v}\|_{L^\infty}\lesssim 2^{-3m/2}\|g_1^v\|_{L^1}\lesssim 2^{-3m/2}2^{(1/2-\beta)j_1},\\
&\|E_s^\mu h_1^{v}\|_{L^\infty}\lesssim \|\widehat{h^{v}_1}\|_{L^1}\lesssim \kappa^22^{-\gamma j_1},\\
&\|E_s^\nu g_2^{v}\|_{L^\infty}\lesssim 2^{-3m/2}\|g_2^v\|_{L^1}\lesssim 2^{-3m/2}2^{(1/2-\beta)j_2},\\
&\|E_s^\nu h_2^{v}\|_{L^\infty}\lesssim \|\widehat{h^{v}_2}\|_{L^1}\lesssim \kappa^22^{-\gamma j_2},
\end{split}
\end{equation}
for any $v\in Z_\kappa^{\sigma;\mu,\nu}$. Using \eqref{cle40} and \eqref{cle50.5}--\eqref{cle51}, we estimate, assuming $j_1\leq j_2$,
\begin{equation*}
\begin{split}
2^{4m+2\beta m}\kappa&\sum_{v\in Z_\kappa^{\sigma;\mu,\nu}}\big[\|A_{v}(E_s^\mu g_1^v,E_s^\nu g_2^v)\|^2_{L^2}+\|A_{v}(E_s^\mu h^v_1,E_s^\nu g_2^v)\|^2_{L^2}\big]\\
&\lesssim 2^{4m+2\beta m}\kappa\sum_{v\in Z_\kappa^{\sigma;\mu,\nu}}\|g_2^{v}\|^2_{L^2}(\|E_s^\nu g_1^v\|_{L^\infty}^2+\|E_s^\nu h_1^{v}\|_{L^\infty}^2)\\
&\lesssim 2^{4m+2\beta m}\kappa\cdot 2^{-2j_2-2\beta j_2}\cdot [2^{-3m}2^{(1-2\beta)j_1}+\kappa^42^{-2\gamma j_1}]\\
&\lesssim 2^{3m}2^{(2\beta +\beta^2)m}2^{-(1+2\beta) j_2}\cdot 2^{-3m}2^{(1-2\beta)j_2}\\
&\lesssim 2^{-\beta^3m}.
\end{split}
\end{equation*}
Similarly, we estimate
\begin{equation*}
\begin{split}
2^{4m+2\beta m}\kappa&\sum_{v\in Z_\kappa^{\sigma;\mu,\nu}}\big[\|A_{v}(E_s^\mu g_1^v,E_s^\nu h_2^v)\|^2_{L^2}+\|A_{v}(E_s^\mu h^v_1,E_s^\nu h_2^v)\|^2_{L^2}\big]\\
&\lesssim 2^{4m+2\beta m}\kappa\sum_{v\in Z_\kappa^{\sigma;\mu,\nu}}\|E_s^\nu h_2^{v}\|^2_{L^\infty}(\|E_s^\nu g_1^v\|_{L^2}^2+\|E_s^\nu h_1^{v}\|_{L^2}^2)\\
&\lesssim 2^{4m+2\beta m}\kappa\cdot\kappa^42^{-2\gamma j_2} \cdot 2^{-2j_1+2\beta j_1}\\
&\lesssim 2^{-m/10}.
\end{split}
\end{equation*}
The desired estimate \eqref{cle39} follows from the last two bounds and the restriction \eqref{cle47}. This completes the proof of the lemma.
\end{proof}

\section{Proof of Proposition \ref{Norm}, III: Case B resonant interactions}\label{normproof3}

In this section we consider type B interactions, see Proposition \ref{PropABC}, and prove the following proposition:

\begin{proposition}\label{reduced5}
Assume that $(k, j), (k_1, j_1), (k_2, j_2) \in\mathcal{J}$ , $m \in [1, L ] \cap\mathbb{Z}$,
\begin{equation}\label{szn1}
\Phi^{\sigma;\mu,\nu}\in\in\mathcal{T}_B=\{\Phi^{e;i+,e+},\Phi^{e;i-,e+},\Phi^{b;i+,b+},\Phi^{b;i-,b+}\},
\end{equation}
and
\begin{equation}\label{szn2}
\begin{split}
&-9m/10\le k_1,k_2\le j/N_0^\prime,\quad\max(j_1,j_2)\le (1-\beta/10)m,\quad\beta m/2+N_0^\prime k_++D^2\le j\le m+D,\\
&k_1\leq -D/3,\qquad k\geq -D/4,\qquad |k-k_2|\leq 10.
\end{split}
\end{equation}
Then there is $\kappa\in(0,1]$, $\kappa\geq\max\big(2^{(\beta^2m-m)/2}2^{-k_1/2},2^{\beta^2m-m}2^{\max(j_1,j_2)}\big)$, such that
\begin{equation}\label{szn3}
(1+2^k)\big\|\widetilde{\varphi}^{(k)}_j\cdot P_kR_{m,\kappa}^{\sigma;\mu,\nu}(f_{k_1,j_1}^\mu,f_{k_2,j_2}^\nu)\big\|_{B^1_{k,j}}\lesssim 2^{-2\beta^4m}.
\end{equation}
\end{proposition}

The rest of the section is concerned with the proof of Proposition \ref{reduced5}. We have assumed, without loss of generality, that $k_1\leq k_2$. As in Case A, the proof of the proposition relies on a careful analysis of resonant interactions. For this analysis we need to understand well the geometry of almost resonant sets. 

For $\sigma\in\{e,b\}$ let $R_\sigma$ denote the unique solutions in $(0,\infty)$ of the equations
\begin{equation}\label{szn200}
\lambda'_\sigma(R_\sigma)=\lambda'_i(0)=\sqrt{(1+T)/(1+\varepsilon)}.
\end{equation}
The numbers $R_\sigma$ are well-defined, in view of Lemma \ref{tech99}, and $R_\sigma\approx_{\varepsilon,C_b}1$. 

For $(\mu,\nu)\in\{(i+,e+),(i-,e+),(i+,b+),(i-,b+)\}$, $\mu=(i\iota_1)$, $\nu=(\sigma_2+)$, $\sigma_2\in\{e,b\}$, we define the functions $r^{\mu,\nu}:(R_{\sigma_2}-2^{-D/5},R_{\sigma_2}+2^{-D/5})\to(R_{\sigma_2}-2^{-D/10},R_{\sigma_2}+2^{-D/10})$ as the unique solutions of the equations
\begin{equation}\label{kx12}
\lambda'_{\sigma_2}(r^{\mu,\nu}(s))-\lambda'_i(s-r^{\mu,\nu}(s))=0.
\end{equation}
Notice that these functions are well defined for $s\in(R_{\sigma_2}-2^{-D/5},R_{\sigma_2}+2^{-D/5})$, since the functions $r\to\lambda'_{\sigma_2}(r)-\lambda'_i(s-r)$ are strictly increasing and vanish in the appropriate ranges, as a consequence of Lemma \ref{tech99} (i) and the observation that $\lambda''_i(0)=0$. Moreover,
\begin{equation}\label{kx12.5}
|(\partial_sr^{\mu,\nu})(s)|\approx_{C_b,\varepsilon}|s-r^{\mu,\nu}(s)|\qquad\text{ for any }s\in (R_{\sigma_2}-2^{-D/5},R_{\sigma_2}+2^{-D/5}).
\end{equation}

\begin{lemma}\label{desc2}
Assume that $\mu=(i\iota_1)$, $\iota_1\in\{+,-\}$, $\nu=(\sigma_2+)$, $\sigma_2\in\{e,b\}$, $k,k_1,k_2\in\mathbb{Z}$, $k_1\leq -D/3$, $k\geq -D/4$, $|k-k_2|\leq 10$, and $\delta\in[0,2^{-10D}]$. Assume that there is a point $(\xi,\eta)\in\mathbb{R}^3\times\mathbb{R}^3$ satisfying 
\begin{equation}\label{kx10}
|\xi|\in[2^{k-4},2^{k+4}],\quad|\eta|\in[2^{k_2-4},2^{k_2+4}],\quad|\xi-\eta|\in[2^{k_1-4},2^{k_1+4}],\quad|\Xi^{\mu,\nu}(\xi,\eta)|\leq\delta.
\end{equation}

(i) Then
\begin{equation}\label{szn201}
k,k_2\in[-D/100,D/100],\quad\text{ and }\quad\big||\xi|-R_{\sigma_2}\big|+\big||\eta|-R_{\sigma_2}\big|\lesssim_{C_b,\varepsilon} 2^{k_1}+\delta.
\end{equation}
More precisely, if $\xi=se$ for some $s>0$ and some unit vector $e\in\mathbb{S}^2$ then
\begin{equation}\label{szn202}
\begin{split}
&|s-R_{\sigma_2}|\lesssim_{C_b,\varepsilon} 2^{k_1}+\delta,\\
&\eta=re+\eta',\qquad |r-r^{\mu,\nu}(s)|\lesssim_{C_b,\varepsilon}\delta,\qquad |r-s|\approx_{C_b,\varepsilon}2^{k_1},\qquad\iota_1|s-r|=s-r,\\
&\eta'\cdot e=0,\qquad|\eta'|\lesssim_{C_b,\varepsilon}2^{k_1}\delta.
\end{split}
\end{equation}

(ii) If, in addition, $\delta\leq 2^{k_1-D/10}$ then
\begin{equation}\label{szn203}
\begin{split}
&\text{ if }\quad\iota_1=+\quad\text{ then }\quad s-R_{\sigma_2}\approx_{C_b,\varepsilon}2^{k_1}\text{ and }R_{\sigma_2}-r^{\mu,\nu}(s)\approx_{C_b,\varepsilon}2^{2k_1},\\
&\text{ if }\quad\iota_1=-\quad\text{ then }\quad R_{\sigma_2}-s\approx_{C_b,\varepsilon}2^{k_1}\text{ and }R_{\sigma_2}-r^{\mu,\nu}(s)\approx_{C_b,\varepsilon}2^{2k_1}.
\end{split}
\end{equation}
\end{lemma}

\begin{proof}[Proof of Lemma \ref{desc2}] (i) We start from the formula
\begin{equation}\label{szn205}
\Xi^{\mu,\nu}(\xi,\eta)=-\iota_1\lambda'_i(|\eta-\xi|)\frac{\eta-\xi}{|\eta-\xi|}-\lambda'_{\sigma_2}(|\eta|)\frac{\eta}{|\eta|}.
\end{equation}
Since $\big|\lambda'_i(|\eta-\xi|)-\lambda'_i(0)\big|\lesssim_{\varepsilon,C_b}2^{2k_1}$, the condition $|\Xi^{\mu,\nu}(\xi,\eta)|\leq\delta$ and the estimates in Lemma \ref{tech99} (i) show that $\big||\eta|-R_{\sigma_2}\big|\lesssim_{\varepsilon,C_b}2^{2k_1}+\delta$. The desired bounds in \eqref{szn201} follow. 

We prove now the claims in \eqref{szn202}. Letting $\xi=se$ for some $s>0,e\in\mathbb{S}^2$ and $\eta=re+\eta'$, $r\in\mathbb{R}$, $\eta'\cdot e=0$, the condition $|\Xi^{\mu,\nu}(\xi,\eta)|\leq\delta$ and the formula \eqref{szn205} show that
\begin{equation}\label{szn206}
\begin{split}
&\Big|-\iota_1\lambda'_i(\sqrt{(r-s)^2+|\eta'|^2})\frac{r-s}{\sqrt{(r-s)^2+|\eta'|^2}}-\lambda'_{\sigma_2}(\sqrt{r^2+|\eta'|^2})\frac{r}{\sqrt{r^2+|\eta'|^2}}\Big|\leq\delta,\\
&\Big|-\iota_1\lambda'_i(\sqrt{(r-s)^2+|\eta'|^2})\frac{\eta'}{\sqrt{(r-s)^2+|\eta'|^2}}-\lambda'_{\sigma_2}(\sqrt{r^2+|\eta'|^2})\frac{\eta'}{\sqrt{r^2+|\eta'|^2}}\Big|\leq\delta.
\end{split}
\end{equation}
Recall that $\lambda'_i(0)>0$. Recalling also the assumptions \eqref{kx10} and the bounds \eqref{szn201}, the second equation in \eqref{szn206} shows that $|\eta'|\lesssim_{C_b,\varepsilon}2^{k_1}\delta$ as desired. In addition, $|s-r|\approx_{C_b,\varepsilon} 2^{k_1}$, therefore
\begin{equation*}
\big|s-R_{\sigma_2}\big|+\big|r-R_{\sigma_2}\big|\lesssim_{C_b,\varepsilon} 2^{k_1}+\delta.
\end{equation*}
The first equation in \eqref{szn206} now gives
\begin{equation}\label{szn207}
\Big|\iota_1\lambda'_i(|r-s|)\frac{r-s}{|r-s|}+\lambda'_{\sigma_2}(r)\Big|\leq 2\delta.
\end{equation}
Since $\lambda'_{\sigma_2}(r)\approx_{C_b,\varepsilon}1$ it follows that $\iota_1(r-s)=-|r-s|$ and, therefore,
\begin{equation*}
\Big|-\lambda'_i(s-r)+\lambda'_{\sigma_2}(r)\Big|\leq 2\delta.
\end{equation*}
Finally, we notice that the derivative of the map $r\to -\lambda'_i(s-r)+\lambda'_{\sigma_2}(r)$ is $\approx_{C_b,\varepsilon}1$ is the appropriate ranges of $r,s$, therefore $|r-r^{\mu,\nu}(s)|\lesssim_{C_b,\varepsilon}\delta$. This completes the proof of \eqref{szn206}.

(ii) If $\delta\leq 2^{k_1-D/10}$ then, using \eqref{szn202}, $|s-r^{\mu,\nu}(s)|\approx_{C_b,\varepsilon}2^{k_1}$. Therefore, using Lemma \ref{tech99} (i),  $\lambda'_i(0)-\lambda'_i(s-r^{\mu,\nu}(s))\approx_{C_b,\varepsilon}2^{2k_1}$. Using the definition \eqref{kx12} it follows that $R_{\sigma_2}-r^{\mu,\nu}(s)\approx_{C_b,\varepsilon}2^{2k_1}$. Therefore $|r-R_{\sigma_2}|\lesssim_{C_b,\varepsilon}2^{2k_1}+\delta$. The remaining bounds in \eqref{szn203} now follow from the identity $\iota_1|s-r|=s-r$ (see \eqref{szn202}) and the assumption $\delta+2^{2k_1}\leq 2^{k_1-D/10}$.
\end{proof}  

\subsection{Proof of Proposition \ref{reduced5}} We further divide the proof into several lemmas.

\begin{lemma}\label{lemmaB1}
The bound \eqref{szn3} holds if \eqref{szn2} holds and, in addition,
\begin{equation}\label{szn4}
\Phi^{\sigma;\mu,\nu}\in\{\Phi^{e;e+,i+},\Phi^{e;e+,i-},\Phi^{b;b+,i+},\Phi^{b;b+,i-}\}\text{ or }k\notin[-D/100,D/100]\text{ or }k_2\notin[-D/100,D/100],
\end{equation}
with
\begin{equation}\label{szn5}
\kappa:=2^{-10D}.
\end{equation}
\end{lemma}

\begin{proof}[Proof of Lemma \ref{lemmaB1}] In any of these cases we have $P_kR_{m,\kappa}^{\sigma;\mu,\nu}(f_{k_1,j_1}^\mu,f_{k_2,j_2}^\nu)=0$, using either Lemma \ref{tech99} (i) (which shows that $\lambda'_i(r)\approx_{C_b,\varepsilon}1$, $r\in[0,\infty)$, and $\lambda'_e(r)\approx\lambda'_b(r)\approx_{C_b,\varepsilon}r$, $r\in[0,1]$) or Lemma \ref{desc2}.
\end{proof}

\begin{lemma}\label{lemmaB2}
The bound \eqref{szn3} holds if \eqref{szn2} holds and, in addition,
\begin{equation}\label{szn6}
\begin{split}
&\Phi^{\sigma;\mu,\nu}\in\{\Phi^{e;i+,e+},\Phi^{e;i-,e+},\Phi^{b;i+,b+},\Phi^{b;i-,b+}\},\\
&k,k_2\in[-D/100,D/100],\qquad \max(j_1,j_2)\le (m-\beta^2m)/2-k_1/2,
\end{split}
\end{equation}
with
\begin{equation}\label{szn7}
\kappa:=2^{(\beta^2m-m)/2-k_1/2}.
\end{equation}
\end{lemma}

\begin{proof}[Proof of Lemma \ref{lemmaB2}] Let
\begin{equation}\label{szn11}
\begin{split}
G(\xi):&=\varphi_k(\xi)\cdot\mathcal{F}\big[R_{m,\kappa}^{\sigma;\mu,\nu}(f_{k_1,j_1}^\mu,f_{k_2,j_2}^\nu)\big](\xi)\\
&=\varphi_k(\xi)\int_{\mathbb{R}}\int_{\mathbb{R}^3}e^{is\Phi^{\sigma;\mu,\nu}(\xi,\eta)}\chi_{R}^{\sigma;\mu,\nu}(\xi,\eta)q_m(s)\cdot \widehat{f_{k_1,j_1}^\mu}(\xi-\eta,s)\widehat{f_{k_2,j_2}^\nu}(\eta,s)\,d\eta ds.
\end{split}
\end{equation}
where, as before,
\begin{equation*}
\chi_R^{\sigma;\mu,\nu}(\xi,\eta)=\varphi(2^{D^2+\max(0,k_1,k_2)}\Phi^{\sigma;\mu,\nu}(\xi,\eta))\varphi(\vert\Xi^{\mu,\nu}(\xi,\eta)\vert/\kappa).
\end{equation*}
We have to prove that
\begin{equation}\label{szn11.2}
2^{(1+\beta)j}\|\widetilde{\varphi}^{(k)}_j\cdot\mathcal{F}^{-1}(G)\|_{L^2}+\|G\|_{L^\infty}\lesssim 2^{-2\beta^4m}.
\end{equation}

Using Lemma \ref{desc2} and the $L^\infty$ bounds in \eqref{nh9}, for any $\xi\in\mathbb{R}^3$ we have
\begin{equation}\label{szn9}
\begin{split}
|G(\xi)|&\lesssim 2^m\cdot 2^{2k_1}\kappa^2\min(2^{k_1},\kappa)\cdot 2^{-k_1/2}\cdot\mathbf{1}_{[-2^D(2^{k_1}+\kappa),2^D(2^{k_1}+\kappa)]}(|\xi|-R_{\sigma_2})\\
&\lesssim 2^{2\beta^2m}2^{k_1/2}\min(2^{k_1},\kappa)\cdot\mathbf{1}_{[-2^D(2^{k_1}+\kappa),2^D(2^{k_1}+\kappa)]}(|\xi|-R_{\sigma_2}). 
\end{split}
\end{equation}
The $L^\infty$ bound in \eqref{szn11.2} follows. 

To prove the $L^2$ bound in \eqref{szn11.2} we notice first that we may assume that $2^j\lesssim 2^{\beta^2m}2^m(\kappa+2^{k_1})$, which is stronger than the assumption $j\leq m+D$ in \eqref{szn2}. Indeed, assuming that $\xi=se$, $\eta=re+\eta'$ satisfy \eqref{kx10} with $\delta=2\kappa$, we estimate
\begin{equation}\label{szn11.4}
\big|(\nabla_\xi\Phi^{\sigma;\mu,\nu})(\xi,\eta)\big|=\Big|-(\nabla_\eta\Phi^{\sigma;\mu,\nu})(\xi,\eta)+[\nabla\Lambda_\sigma(\xi)-\nabla\Lambda_\sigma(\eta)]\Big|\lesssim 2^{k_1}+\kappa.
\end{equation}
Therefore, we make the change of variables $\eta=\xi-\theta$ in \eqref{szn11} and integrate by parts in $\xi$ using \eqref{nh9} and Lemma \ref{tech5} with $K\approx 2^{\beta^2m}2^m(\kappa+2^{k_1})$, $\epsilon^{-1}\approx 2^m(\kappa+2^{k_1})$. It follows that
\begin{equation*}
2^{(1+\beta)j}\|\widetilde{\varphi}^{(k)}_j\cdot\mathcal{F}^{-1}(G)\|_{L^2}\lesssim 2^{-m}\qquad\text{ if }\qquad 2^j\geq 2^{\beta^2m}2^m(\kappa+2^{k_1}).
\end{equation*}
Therefore, for \eqref{szn11.2} it suffices to prove that
\begin{equation}\label{szn11.6}
2^{(1+\beta)m}(\kappa+2^{k_1})^{1+\beta}\|G\|_{L^2}\lesssim 2^{-2\beta^2m}.
\end{equation}

{\bf{Case 1.}} It follows from \eqref{szn9} that the left-hand side of \eqref{szn11.6} is dominated by
\begin{equation*}
C2^{(1+\beta)m}2^{2\beta^2m}2^{k_1/2}2^{k_1}\kappa(2^{k_1}+\kappa)^{1/2}\lesssim 2^{(1/2+2\beta)m}(2^{3k_1/2}+2^{k_1}\kappa^{1/2}).
\end{equation*}
The desired bound \eqref{szn11.6} follows if $k_1\leq -m/3-4\beta m$.

{\bf{Case 2.}} Assume now that
\begin{equation}\label{szn10}
-m/3+\beta m\leq k_1\leq -D/3.
\end{equation}
In this case we need to improve on the bound \eqref{szn9}. 
We use Lemma \ref{desc2} with $\delta=2\kappa$, and notice that, as a consequence of \eqref{szn10}, $\delta\leq 2^{k_1-D/10}$. Assuming $\xi=se$ and $\eta=re+\eta'$ satisfy \eqref{kx10} we estimate, using also Lemma \ref{tech99}, Lemma \ref{desc2}, and \eqref{szn207},
\begin{equation}\label{szn10.5}
\begin{split}
\Phi^{\si;\mu,\nu}(\xi,\eta)&=\lambda_{\sigma_2}(s)-\iota_1\lambda_i(\sqrt{(r-s)^2+|\eta'|^2})-\lambda_{\sigma_2}(\sqrt{r^2+|\eta'|^2})\\
&=\lambda_{\sigma_2}(s)-\iota_1\lambda_i(|r-s|)-\lambda_{\sigma_2}(r)+O_{C_b,\varepsilon}(\kappa^2)\\
&=\lambda_{\sigma_2}(s)-\lambda_{\sigma_2}(r)-\lambda_i(s-r)+O_{C_b,\varepsilon}(\kappa^2)\\
&=\lambda_{\sigma_2}(s)-\lambda_{\sigma_2}(R_{\sigma_2})-\lambda_i(s-R_{\sigma_2})+O_{C_b,\varepsilon}(\kappa^2+2^{3k_1})\\
&=\lambda_{\sigma_2}(s)-\lambda_{\sigma_2}(R_{\sigma_2})-\lambda'_{\sigma_2}(R_{\sigma_2})\cdot (s-R_{\sigma_2})+O_{C_b,\varepsilon}(\kappa^2+2^{3k_1})\\
&\approx_{C_b,\varepsilon} 2^{2k_1}.
\end{split}
\end{equation}
We can integrate by parts in $s$ in the formula \eqref{szn11} to conclude that
\begin{equation}\label{szn10.6}
 \begin{split}
 |G(\xi)|\lesssim 2^{-2k_1}|\varphi_k(\xi)|\int_{\mathbb{R}}\int_{\mathbb{R}^3}
&|\varphi(|\Xi^{\mu,\nu}(\xi,\eta)|/\ka)|\,|q'_m(s)|\,|\widehat{f_{k_1,j_1}^\mu}(\xi-\eta,s)|\,|\widehat{f_{k_2,j_2}^\nu}(\eta,s)|\\
&+|\varphi(|\Xi^{\mu,\nu}(\xi,\eta)|/\ka)|\,|q_m(s)|\,|(\partial_s\widehat{f_{k_1,j_1}^\mu})(\xi-\eta,s)|\,|\widehat{f_{k_2,j_2}^\nu}(\eta,s)|\\
&+|\varphi(|\Xi^{\mu,\nu}(\xi,\eta)|/\ka)|\,|q_m(s)|\,|\widehat{f_{k_1,j_1}^\mu}(\xi-\eta,s)|\,|(\partial_s\widehat{f_{k_2,j_2}^\nu})(\eta,s)|\,d\eta ds.
 \end{split}
\end{equation}
We use now \eqref{nh2}, the last bound in \eqref{nh9}, and the bound \eqref{derv2repeat}. In view of Lemma \ref{desc2}, the volume of integration is $\approx (2^{k_1}\kappa)^2\kappa$ and it follows that
\begin{equation}\label{szn10.7}
\begin{split}
 |G(\xi)|&\lesssim 2^{-2k_1}\mathbf{1}_{[0,2^D]}(2^{-k_1}|s-R_{\sigma_2}|)\cdot 2^{2k_1}\ka^3\cdot 2^{\beta m/10}2^{-k_1}\\
 &\lesssim \mathbf{1}_{[0,2^D]}(2^{-k_1}|s-R_{\sigma_2}|)\cdot 2^{-3m/2}2^{\beta m/5}2^{-5k_1/2}.
 \end{split}
\end{equation}
Therefore the left-hand side of \eqref{szn11.6} is dominated by
\begin{equation*}
2^{(1+\beta)m}2^{k_1}\cdot 2^{-3m/2}2^{\beta m/5}2^{-2k_1}\lesssim 2^{-k_1}2^{-m/2}2^{2\beta m},
\end{equation*}
and the desired bound \eqref{szn11.6} follows using also \eqref{szn10}.

{\bf{Case 3.}} It remains to prove the bound \eqref{szn11.6} in the case
\begin{equation}\label{szn25.1}
-m/3-4\beta m\leq k_1\leq -m/3+\beta m\quad\text{ and }\quad 2^{-m/3-\beta m}\leq\kappa\leq 2^{-m/3+3\beta m}.
\end{equation}
We define
\begin{equation*}
G'(\xi):=\varphi_k(\xi)\int_{\mathbb{R}}\int_{\mathbb{R}^3}e^{is\Phi^{\sigma;\mu,\nu}(\xi,\eta)}\varphi(2^{D^2+\max(0,k_1,k_2)}\Phi^{\sigma;\mu,\nu}(\xi,\eta))q_m(s)\cdot \widehat{f_{k_1,j_1}^\mu}(\xi-\eta,s)\widehat{f_{k_2,j_2}^\nu}(\eta,s)\,d\eta ds
\end{equation*}
and notice that, using integration by parts in $\eta$ as in the proof of Lemma \ref{NewBigBound1},
\begin{equation*}
\|G-G'\|_{L^2}\lesssim 2^{-10m}.
\end{equation*}
Moreover, using Lemma \ref{tech2} and the $L^\infty$ bounds \eqref{szn91.5} below,
\begin{equation*}
\begin{split}
\|G'\|_{L^2}&\lesssim 2^m\sup_{s\in[2^{m-1},2^{m+1}]}\min\left\{\Vert Ef^\mu_{k_1,j_1}(s)\Vert_{L^\infty}\Vert f^\nu_{k_2,j_2}(s)\Vert_{L^2},\Vert f^\mu_{k_1,j_1}(s)\Vert_{L^2}\Vert Ef^\nu_{k_2,j_2}(s)\Vert_{L^\infty}\right\}\\
&\lesssim 2^m\cdot 2^{-3m/2}2^{-\max(j_1,j_2)(1/2-2\beta)}.
\end{split}
\end{equation*}
The desired bound \eqref{szn11.6} follows if $\max(j_1,j_2)\geq m/2$, using also \eqref{szn25.1}.

Finally, assume that
\begin{equation}\label{szn25.2}
-m/3-4\beta m\leq k_1\leq -m/3+\beta m,\qquad 2^{-m/3-\beta m}\leq\kappa\leq 2^{-m/3+3\beta m},\qquad \max(j_1,j_2)\leq m/2.
\end{equation}
In this case we need to improve slightly on the pointwise bound \eqref{szn9}. Assuming $\xi=r'e$, $r'\in(0,\infty)$, $e\in\mathbb{S}^2$ and letting $\eta=re+\eta'$, $\eta'\cdot e=0$, we define, for any $l\in\mathbb{Z}$,
\begin{equation*}
\begin{split}
G'_{\leq l}(\xi):=&\varphi_k(r'e)\int_{\mathbb{R}}\int_{\mathbb{R}^2}\int_{\mathbb{R}}e^{is\Phi^{\sigma;\mu,\nu}(r'e,re+\eta')}\varphi(2^{D^2+\max(0,k_1,k_2)}\Phi^{\sigma;\mu,\nu}(r'e,re+\eta'))q_m(s)\\
&\times \varphi((r-r^{\mu,\nu}(r'))/2^l\kappa)\varphi(\eta'/(2^{D+k_1}\kappa))\cdot\widehat{f_{k_1,j_1}^\mu}(r'e-re-\eta',s)\widehat{f_{k_2,j_2}^\nu}(re+\eta',s)\,dr d\eta' ds.
\end{split}
\end{equation*}
Clearly, $\|G-G'_{\leq D}\|_{L^2}\lesssim 2^{-10m}$, see \eqref{szn202}. Estimating as in \eqref{szn9},
\begin{equation*}
\begin{split}
|G'_{\leq l}(\xi)|&\lesssim 2^m\cdot 2^{2k_1}\kappa^22^l\kappa\cdot 2^{-k_1/2}\cdot\mathbf{1}_{[-2^D(2^{k_1}+\kappa),2^D(2^{k_1}+\kappa)]}(|\xi|-R_{\sigma_2})\\
&\lesssim 2^l2^{2\beta^2m-m/2}\cdot\mathbf{1}_{[-2^D(2^{k_1}+\kappa),2^D(2^{k_1}+\kappa)]}(|\xi|-R_{\sigma_2}). 
\end{split}
\end{equation*}
Therefore, setting $l_0:=-\lfloor 8\beta m\rfloor$ we estimate
\begin{equation}\label{szn25.3}
2^{(1+\beta)m}(\kappa+2^{k_1})^{1+\beta}\|G'_{\leq l_0}\|_{L^2}\lesssim 2^{m+\beta m}2^{-m/3+3\beta m}\cdot 2^{l_0}2^{2\beta^2m-m/2}2^{-m/6+3\beta m/2}\lesssim 2^{-\beta^2m}.
\end{equation}
Finally we notice that we can integrate by parts in $r$ and use Lemma \ref{tech5} with $K\approx 2^m2^l\kappa$ and $\epsilon^{-1}\approx 2^{\beta m}[2^{\max(j_1,j_2)}+2^{-l}(2^{k_1}+\kappa)^{-1}]$ to show that 
\begin{equation}\label{szn25.4}
|G_{\leq l+1}(\xi)-G_{\leq l}(\xi)|\lesssim 2^{-10m}
\end{equation}
if $l\in[l_0,D]$. Indeed, it follows from \eqref{szn25.2} that $K\epsilon\gtrsim 2^{\beta m}$. Moreover, for $(r,\eta')$ in the relevant support,
\begin{equation*}
\begin{split}
\Big|\frac{d}{dr}\Phi^{\sigma;\mu,\nu}(r'e,re+\eta')\Big|&=\Big|-\iota_1\lambda'_i(\sqrt{(r-r')^2+|\eta'|^2})\frac{r-r'}{\sqrt{(r-r')^2+|\eta'|^2}}-\lambda'_{\sigma_2}(\sqrt{r^2+|\eta'|^2})\frac{r}{\sqrt{r^2+|\eta'|^2}}\Big|\\
&=\Big|\frac{\iota_1(r'-r)}{|r-r'|}\lambda'_i(r'-r)-\lambda'_{\sigma_2}(r)\Big|+O_{C_b,\varepsilon}(\kappa^2+2^{2k_1})\\
&\gtrsim_{C_b,\varepsilon} 2^l\kappa.
\end{split}
\end{equation*}
The bound \eqref{szn25.4} follows from Lemma \ref{tech5}. The desired estimate \eqref{szn11.6} follows using also \eqref{szn25.3}.
\end{proof}

\begin{lemma}\label{lemmaB3}
The bound \eqref{szn3} holds if \eqref{szn2} holds and, in addition,
\begin{equation}\label{szn50}
\begin{split}
&\Phi^{\sigma;\mu,\nu}\in\{\Phi^{e;i+,e+},\Phi^{e;i-,e+},\Phi^{b;i+,b+},\Phi^{b;i-,b+}\},\\
&k,k_2\in[-D/100,D/100],\qquad \max(j_1,j_2)\geq (m-\beta^2m)/2-k_1/2,
\end{split}
\end{equation}
with
\begin{equation}\label{szn51}
\kappa:=2^{\max(j_1,j_2)+\beta^2m-m}.
\end{equation}
\end{lemma}

\begin{proof}[Proof of Lemma \ref{lemmaB3}] We define the function $G$ as in \eqref{szn11}; it suffices to prove that
\begin{equation}\label{szn52}
2^{(1+\beta)j}\|\widetilde{\varphi}^{(k)}_j\cdot\mathcal{F}^{-1}(G)\|_{L^2}+\|G\|_{L^\infty}\lesssim 2^{-2\beta^4m}.
\end{equation}
The $L^\infty$ bound in \eqref{szn52} is easy: if $j_1\leq j_2$ then we use the bounds
\begin{equation*}
\begin{split}
&\|\widehat{f_{k_1,j_1}^\mu}(s)\|_{L^\infty}\lesssim 2^{-k_1/2},\\
&|\varphi_k(\xi)|\int_{\mathbb{R}^3}\big|\widehat{f_{k_2,j_2}^\nu}(\eta,s)\big|\,\big|\varphi(\vert\Xi^{\mu,\nu}(\xi,\eta)\vert/\kappa)\big|\mathbf{1}_{[2^{k_1-4},2^{k_1+4}]}(|\xi-\eta|)\,d\eta \lesssim 2^{-(1+\beta)j_2}(\kappa^3 2^{2k_1})^{1/2},
\end{split}
\end{equation*}
which follow from Lemma \ref{desc2}, the bounds \eqref{nh5.5}, and Definition \ref{MainDef}. Therefore, in this case,
\begin{equation*}
\|G\|_{L^\infty}\lesssim 2^m\cdot 2^{-k_1/2}2^{-(1+\beta)j_2}(\kappa^3 2^{2k_1})^{1/2}\lesssim 2^{-\beta m/4}2^{(j_2-m)/8},
\end{equation*}
which suffices. Similarly, if $j_1\geq j_2$ then we use the bounds
\begin{equation*}
\begin{split}
&\|\widehat{f_{k_2,j_2}^\nu}(s)\|_{L^\infty}\lesssim 1,\\
&|\varphi_k(\xi)|\int_{\mathbb{R}^3}\big|\widehat{f_{k_1,j_1}^\mu}(\xi-\eta,s)\big|\,\big|\varphi(\vert\Xi^{\mu,\nu}(\xi,\eta)\vert/\kappa)\big|\mathbf{1}_{[2^{k_2-4},2^{k_2+4}]}(|\eta|)\,d\eta \lesssim 2^{-(1+\beta)j_1}(\kappa^3 2^{2k_1})^{1/2},
\end{split}
\end{equation*}
and the desired $L^\infty$ bound on $G$ follows as before.

The $L^2$ bound in \eqref{szn52} is more complicated. We notice first that the same argument as in the proof of Lemma \ref{lemmaB2}, using the estimate \eqref{szn11.4}, shows that
\begin{equation}\label{szn53}
2^{(1+\beta)j}\|\widetilde{\varphi}^{(k)}_j\cdot\mathcal{F}^{-1}(G)\|_{L^2}\lesssim 2^{-m}\qquad\text{ if }\qquad 2^j\geq 2^{\beta^2m}2^m(\kappa+2^{k_1}).
\end{equation}
To continue we consider three cases.
\medskip

{\bf{Case 1.}} Assume first that
\begin{equation}\label{szn54}
2^{k_1-D}\leq\kappa.
\end{equation}
In view of \eqref{szn53}, in this case it remains to prove that
\begin{equation}\label{szn55}
2^{(1+\beta)m}\kappa^{1+\beta}\|G\|_{L^2}\lesssim 2^{-2\beta^2m}.
\end{equation}
We argue as in the proof of Lemma \ref{BigBound5}. We define first
\begin{equation*}
G'(\xi):=\varphi_k(\xi)\int_{\mathbb{R}}\int_{\mathbb{R}^3}e^{is\Phi^{\sigma;\mu,\nu}(\xi,\eta)}\varphi(2^{D^2+\max(0,k_1,k_2)}\Phi^{\sigma;\mu,\nu}(\xi,\eta))q_m(s)\cdot \widehat{f_{k_1,j_1}^\mu}(\xi-\eta,s)\widehat{f_{k_2,j_2}^\nu}(\eta,s)\,d\eta ds
\end{equation*}
and notice that, using integration by parts in $\eta$ and Lemma \ref{tech5} with $K\approx 2^m\kappa$, $\epsilon^{-1}\approx 2^{\max(j_1,j_2)}$, 
\begin{equation}\label{szn56}
\|G-G'\|_{L^2}\lesssim 2^{-10m}.
\end{equation}

Since $\|\phii_{j_1}^{(k_1)}\cdot P_{k_1}f_\mu(s)\|_{B_{k_1,j_1}}+\|\phii_{j_2}^{(k_2)}\cdot P_{k_2}f_\nu(s)\|_{B_{k_2,j_2}}\lesssim 1$, see \eqref{nh5.5}, we use \eqref{sec5.8}--\eqref{sec5.82} to decompose
\begin{equation}\label{szn57}
\begin{split}
&\phii^{(k_1)}_{j_1}\cdot P_{k_1}f_\mu(s)=2^{-\alpha k_1}[g_{k_1,j_1}^\mu(s)+h_{k_1,j_1}^\mu(s)],\\
&2^{(1+\be)j_1}\|g_{k_1,j_1}^\mu(s)\|_{L^2}+2^{(1/2-\beta)\widetilde{k_1}}\|\widehat{g_{k_1,j_1}^\mu}(s)\|_{L^\infty}\lesssim 1,\\
&2^{(1-\be)j_1}\|h_{k_1,j_1}^\mu(s)\|_{L^2}+\|\widehat{h_{k_1,j_1}^\mu}(s)\|_{L^\infty}+2^{\gamma j_1}\|\widehat{h_{k_1,j_1}^\mu}(s)\|_{L^1}\lesssim 2^{-8|k_1|},
\end{split}
\end{equation}
and
\begin{equation}\label{szn58}
\begin{split}
&\phii^{(k_2)}_{j_2}\cdot P_{k_2}f_\nu(s)=[g_{k_2,j_2}^\nu(s)+h_{k_2,j_2}^\nu(s)],\\
&2^{(1+\be)j_2}\|g_{k_2,j_2}^\nu(s)\|_{L^2}+\|\widehat{g_{k_2,j_2}^\nu}(s)\|_{L^\infty}\lesssim 1,\\
&2^{(1-\be)j_2}\|h_{k_2,j_2}^\nu(s)\|_{L^2}+\|\widehat{h_{k_2,j_2}^\nu}(s)\|_{L^\infty}+2^{\gamma j_2}\|\widehat{h_{k_2,j_2}^\nu}(s)\|_{L^1}\lesssim 1.
\end{split}
\end{equation}
For $f,g\in L^2(\mathbb{R}^3)$, $\xi\in\mathbb{R}^3$, and $s\in[2^{m-1},2^{m+1}]$ let
\begin{equation}\label{szn59.5}
\widetilde{G}'_s(f,g)(\xi):=\varphi_k(\xi)\int_{\mathbb{R}^3}e^{is\Phi^{\sigma;\mu,\nu}(\xi,\eta)}\varphi(2^{D^2+\max(0,k_1,k_2)}\Phi^{\sigma;\mu,\nu}(\xi,\eta))\cdot \widehat{f}(\xi-\eta)\widehat{g}(\eta)\,d\eta.
\end{equation}
Using also \eqref{szn56}, for \eqref{szn55} it suffices to prove that, for any $s\in[2^{m-1},2^{m+1}]$,
\begin{equation}\label{szn60}
2^{2\beta^2m}2^{-\alpha k_1}2^{(2+\beta)m}\kappa^{1+\beta}\|\widetilde{G}'_s(f,g)\|_{L^2}\lesssim 1,
\end{equation}
where $f\in \{P_{k_1-2,k_1+2}g_{k_1,j_1}^\mu(s),P_{k_1-2,k_1+2}h_{k_1,j_1}^\mu(s)\}$, $g\in\{P_{k_2-2,k_2+2}g_{k_2,j_2}^\nu(s),P_{k_2-2,k_2+2}h_{k_2,j_2}^\nu(s)\}$.

Using Lemma \ref{tech2} and \eqref{mk3} it follows that
\begin{equation}\label{szn61}
\|\widetilde{G}'_s(f,g)\|_{L^2}\lesssim \min\big(\|E^\mu_sf\|_{L^2}\|E^\nu_sg\|_{L^\infty},\|E^\mu_sf\|_{L^\infty}\|E^\nu_sg\|_{L^2}\big),
\end{equation}
where $E^\mu_sf:=e^{-is\widetilde{\Lambda}_\mu} f$ and $E^\nu_sg:=e^{-is\widetilde{\Lambda}_\nu} g$. In view of Lemma \ref{tech1.5}, see also \eqref{equl},
\begin{equation}\label{szn62}
\|E^\mu_sf\|_{L^\infty}\lesssim 2^{\beta k_1}2^{-m(5/4-10\beta)}2^{j_1(1/4-11\beta)},\qquad \|E^\nu_sg\|_{L^\infty}\lesssim 2^{-m(5/4-10\beta)}2^{j_2(1/4-11\beta)},
\end{equation}
for $f\in \{P_{k_1-2,k_1+2}g_{k_1,j_1}^\mu(s),P_{k_1-2,k_1+2}h_{k_1,j_1}^\mu(s)\}$, $g\in\{P_{k_2-2,k_2+2}g_{k_2,j_2}^\nu(s),P_{k_2-2,k_2+2}h_{k_2,j_2}^\nu(s)\}$.

If $|j_1-j_2|\geq 10\beta m$ then we use \eqref{szn61}--\eqref{szn62}, together with the estimate $2^{\max(j_1,j_2)}\approx \kappa 2^m2^{-\beta^2m}$ and the $L^2$ bounds $\|f\|_{L^2}\lesssim 2^{2\beta k_1}2^{-(1-\beta)j_1}$, $\|g\|_{L^2}\lesssim 2^{-(1-\beta)j_2}$, to estimate
\begin{equation*}
\begin{split}
\|\widetilde{G}'_s(f,g)\|_{L^2}&\lesssim 2^{\beta k_1}2^{-m(5/4-10\beta)}2^{\min(j_1,j_2)(1/4-11\beta)}\cdot 2^{-(1-\beta)\max(j_1,j_2)}\\
&\lesssim 2^{\beta k_1}2^{-m(5/4-10\beta)}2^{-10\beta m(1/4-11\beta)}(\kappa 2^m2^{-\beta^2m})^{-3/4-10\beta}\\
&\lesssim 2^{\beta k_1}\kappa^{-1}2^{-2m}2^{-5\beta m/4},
\end{split}
\end{equation*}
for $f\in \{P_{k_1-2,k_1+2}g_{k_1,j_1}^\mu(s),P_{k_1-2,k_1+2}h_{k_1,j_1}^\mu(s)\}$, $g\in\{P_{k_2-2,k_2+2}g_{k_2,j_2}^\nu(s),P_{k_2-2,k_2+2}h_{k_2,j_2}^\nu(s)\}$. The desired bound \eqref{szn60} follows in this case.

On the other hand, if $|j_1-j_2|\leq 10\beta m$ then we estimate, using \eqref{szn57}--\eqref{szn58} and \eqref{szn61}--\eqref{szn62},
\begin{equation*}
\begin{split}
\|\widetilde{G}'_s(P_{k_1-2,k_1+2}g_{k_1,j_1}^\mu(s),P_{k_2-2,k_2+2}g_{k_2,j_2}^\nu(s))\|_{L^2}&\lesssim 2^{-m(5/4-10\beta)}2^{\min(j_1,j_2)(1/4-11\beta)}\cdot 2^{-(1+\beta)\max(j_1,j_2)}\\
&\lesssim 2^{-m(5/4-10\beta)}(\kappa 2^m2^{-\beta^2m})^{-3/4-12\beta}\\
&\lesssim \kappa^{-1}2^{-(2+2\beta)m}2^{\beta^2m}.
\end{split}
\end{equation*}
Moreover, for $g\in\{P_{k_2-2,k_2+2}g_{k_2,j_2}^\nu(s),P_{k_2-2,k_2+2}h_{k_2,j_2}^\nu(s)\}$ we estimate, using also the assumption \eqref{szn54},
\begin{equation*}
\begin{split}
\|\widetilde{G}'_s(P_{k_1-2,k_1+2}h_{k_1,j_1}^\mu(s),g)\|_{L^2}&\lesssim \|\widehat{h_{k_1,j_1}^\mu}(s)\|_{L^1}\|g\|_{L^2}\lesssim 2^{-\gamma j_1}2^{8k_1}2^{-(1-\beta)j_2}\\
&\lesssim 2^{8k_1}2^{16\beta m}(\kappa 2^m)^{-(1-\beta+\gamma)}\lesssim 2^{k_1}2^{-m(1-17\beta+\gamma)}.
\end{split}
\end{equation*}
Finally,
\begin{equation*}
\begin{split}
\|\widetilde{G}'_s(P_{k_1-2,k_1+2}&g_{k_1,j_1}^\mu(s),P_{k_2-2,k_2+2}h_{k_2,j_2}^\nu(s))\|_{L^2}\\
&\lesssim \min\big[\|\widehat{g_{k_1,j_1}^\mu}(s)\|_{L^2}\|\widehat{h_{k_2,j_2}^\nu}(s)\|_{L^1},2^{3k_1/2}\|\widehat{g_{k_1,j_1}^\mu}(s)\|_{L^2}\|\widehat{h_{k_2,j_2}^\nu}(s)\|_{L^2}\big]\\
&\lesssim \min\big[2^{-(1+\beta)j_1}2^{-\gamma j_2},2^{3k_1/2}2^{-(1+\beta)j_1}2^{-(1-\beta)j_2}\big]\\
&\lesssim \min\big[2^{16\beta m}(\kappa 2^m)^{-(1+\beta+\gamma)},\kappa^{3/2}2^{16\beta m}(\kappa 2^m)^{-2}\big]\\
&\lesssim 2^{16\beta m}\min\big[2^{-(1+\gamma)m}\kappa^{-(1+\gamma)},2^{-2m}\kappa^{-1/2}\big].
\end{split}
\end{equation*}
The desired bound \eqref{szn60} follows from these last three estimates, which completes the proof in Case 1.

\medskip

{\bf{Case 2.}} Assume now that
\begin{equation}\label{szn70}
\kappa\leq 2^{k_1-D}\quad\text{ and }\quad\kappa\leq 2^{-m(1/3+\beta/2)}.
\end{equation}
In view of \eqref{szn53}, in this case it remains to prove that
\begin{equation}\label{szn71}
2^{(1+\beta)m}2^{(1+\beta)k_1}\|G\|_{L^2}\lesssim 2^{-2\beta^2m}.
\end{equation}

As in Lemma \ref{lemmaB2}, see \eqref{szn10.5}--\eqref{szn10.7}, we estimate pointwise
\begin{equation*}
\begin{split}
|G(\xi)|&\lesssim 2^{-2k_1}\cdot 2^{2k_1}\kappa^3\cdot 2^{\beta m/10}2^{-k_1}\cdot\mathbf{1}_{[-2^{k_1+D},2^{k_1+D}]}(|\xi|-R_{\sigma_2})\\
&\lesssim \kappa^32^{\beta m/10}2^{-k_1}\cdot\mathbf{1}_{[-2^{k_1+D},2^{k_1+D}]}(|\xi|-R_{\sigma_2}).
\end{split}
\end{equation*}
Therefore
\begin{equation*}
2^{(1+\beta)m}2^{(1+\beta)k_1}\|G\|_{L^2}\lesssim 2^{(1+5\beta/4)m}\kappa^3,
\end{equation*}
and the desired estimate \eqref{szn71} follows since $\kappa\leq 2^{-m(1/3+\beta/2)}$.

\medskip

{\bf{Case 3.}} Finally assume that 
\begin{equation}\label{szn72}
2^{-m(1/3+\beta/2)}\leq \kappa\leq 2^{k_1-D}.
\end{equation}
In view of \eqref{szn53}, in this case it remains to prove that
\begin{equation}\label{szn74}
2^{(2+2\beta)m}2^{(2+2\beta)k_1}\|G\|^2_{L^2}\lesssim 2^{-4\beta^2m}.
\end{equation}

{\bf{Step 1.}} We need first a suitable decomposition and an orthogonality argument. Let $\chi:\mathbb{R}\to[0,1]$ denote a smooth function supported in the interval $[-2,2]$ with the 
property that
\begin{equation*}
\sum_{n\in\mathbb{Z}}\chi(x-n)=1\qquad\text{ for any }x\in\mathbb{R}.
\end{equation*}
We define the smooth function $\chi':\mathbb{R}^3\to[0,1]$, $\chi'(x,y,z):=\chi(x)\chi(y)\chi(z)$. We define, for any $v\in\mathbb{Z}^3$ and $n\in\mathbb{Z}$,
\begin{equation}\label{szn75}
\begin{split}
G_{v,n}(\xi):=&\chi'(\kappa^{-1}\xi-v)\varphi_k(\xi)\\
&\int_{\mathbb{R}}\int_{\mathbb{R}^3}e^{is\Phi^{\sigma;\mu,\nu}(\xi,\eta)}
\chi_R^{\sigma;\mu,\nu}(\xi,\eta)\chi(2^{k_1-m}\kappa^{-1}s-n)q_m(s)\cdot \widehat{f_{k_1,j_1}^\mu}(\xi-\eta,s)\widehat{f_{k_2,j_2}^\nu}(\eta,s)\,d\eta ds.
\end{split}
\end{equation}
and notice that $G=\sum_{v\in\mathbb{Z}^3}\sum_{n\in\mathbb{Z}}G_{v,n}$. 

We show now that
\begin{equation}\label{szn76}
 \|G\|^2_{L^2}\lesssim \sum_{v\in\mathbb{Z}^3}\sum_{n\in\mathbb{Z}}\|G_{v,n}\|_{L^2}^2+2^{-10m}.
\end{equation}
Indeed, we clearly have
\begin{equation*}
\|G\|^2_{L^2}\lesssim \sum_{v\in\mathbb{Z}^3}\Big\|\sum_{n\in\mathbb{Z}}G_{v,n}\Big\|_{L^2}^2
\lesssim \sum_{v\in\mathbb{Z}^3}\sum_{n_1,n_2\in\mathbb{Z}}|\langle G_{v,n_1},G_{v,n_2}\rangle|.
\end{equation*}
Therefore, for \eqref{szn76} it suffices to prove that
\begin{equation}\label{szn77}
|\langle G_{v,n_1},G_{v,n_2}\rangle|\lesssim 2^{-20m}\qquad \text{ if }v\in\mathbb{Z}^3\text{ and }|n_1-n_2|\geq 2^{100D}.
\end{equation}
To prove this we need to estimate $|\mathcal{F}^{-1}(G_{v,n})(x)|$. We would like to integrate by parts in the formula \eqref{szn75}. Using Lemma \ref{desc2} and Lemma \ref{tech99} (i), for $\xi=se$, $\eta=re+\eta'$ satisfying \eqref{kx10} with $\delta=2\kappa$ and $|\xi-\kappa v|\lesssim \kappa$, we estimate
\begin{equation*}
\begin{split}
(\nabla_\xi\Phi^{\sigma;\mu,\nu})(\xi,\eta)&=-(\nabla_\eta\Phi^{\sigma;\mu,\nu})(\xi,\eta)+[\nabla\Lambda_\sigma(\xi)-\nabla\Lambda_\sigma(\eta)]\\
&=\lambda'_{\sigma_2}(s)e-\lambda'_{\sigma_2}(r^{\mu,\nu}(s))e+O_{C_b,\varepsilon}(\kappa)\\
&=[\lambda'_{\sigma_2}(\kappa|v|)-\lambda'_{\sigma_2}(r^{\mu,\nu}(\kappa|v|))]\cdot v/|v|+O_{C_b,\varepsilon}(\kappa).
\end{split}
\end{equation*}
In particular, $|\nabla_\xi\Phi^{\sigma;\mu,\nu}(\xi,\eta)|\approx 2^{k_1}$. After repeated integration by parts in $\xi$, it follows that
\begin{equation*}
\begin{split}
&|\mathcal{F}^{-1}(G_{v,n})(x)|\lesssim |x+w_n|^{-200}\qquad\text{ if }|x+w_n|\geq 2^{50D}2^m\kappa,\\
&w_n:=n\kappa2^{m-k_1}[\lambda'_{\sigma_2}(\kappa|v|)-\lambda'_{\sigma_2}(r^{\mu,\nu}(\kappa|v|))]\cdot v/|v|,
\end{split}
\end{equation*}
for any $n\in\mathbb{Z}$. Therefore if $|n_1-n_2|\geq 2^{100D}$ then $|w_{n_1}-w_{n_2}|\geq 2^{70D}\kappa 2^m$ and the desired bound \eqref{szn77} follows. This completes the proof of \eqref{szn76}.

In view of \eqref{szn76} and Lemma \ref{desc2}, for \eqref{szn74} it remains to prove that
\begin{equation}\label{szn90}
2^{(2+2\beta)m}2^{(2+2\beta)k_1}\sum_{(v,n)\in \mathbb{Z}^3\times\mathbb{Z}}\|G_{v,n}\|^2_{L^2}\lesssim 2^{-4\beta^2m},
\end{equation}
Let 
\begin{equation*}
\begin{split}
G_{n}(\xi)&:=\sum_{v\in\mathbb{Z}^3}G_{v,n}(\xi)\\
&=\varphi_k(\xi)\int_{\mathbb{R}}\int_{\mathbb{R}^3}e^{is\Phi^{\sigma;\mu,\nu}(\xi,\eta)}
\chi_R^{\sigma;\mu,\nu}(\xi,\eta)\chi(2^{k_1-m}\kappa^{-1}s-n)q_m(s)\cdot \widehat{f_{k_1,j_1}^\mu}(\xi-\eta,s)\widehat{f_{k_2,j_2}^\nu}(\eta,s)\,d\eta ds.
\end{split}
\end{equation*}
and{\footnote{In some arguments that involve the use of Lemma \ref{tech2} it is necessary to pass to operators that contain "smooth" symbols, such as the symbol $ (\xi,\eta)\to\varphi(2^{D^2+\max(0,k_1,k_2)}\Phi^{\sigma;\mu,\nu}(\xi,\eta))$ in the operators $\widetilde{G}'_s$ below. Lemma \ref{tech2} is not directly compatible with "rough" symbols such as $(\xi,\eta)\to\chi_R^{\sigma;\mu,\nu}(\xi,\eta)$ since the $L^1$ norm of the inverse Fourier transform of such symbols is very large.}}
\begin{equation*}
\begin{split}
G'_{n}(\xi):=\varphi_k(\xi)\int_{\mathbb{R}}\int_{\mathbb{R}^3}&e^{is\Phi^{\sigma;\mu,\nu}(\xi,\eta)}\varphi(2^{D^2+\max(0,k_1,k_2)}\Phi^{\sigma;\mu,\nu}(\xi,\eta))\\
&\times\chi(2^{k_1-m}\kappa^{-1}s-n)q_m(s)\cdot \widehat{f_{k_1,j_1}^\mu}(\xi-\eta,s)\widehat{f_{k_2,j_2}^\nu}(\eta,s)\,d\eta ds.
\end{split}
\end{equation*}
Notice that 
\begin{equation}\label{szn91}
\sum_{v\in\mathbb{Z}^3}\|G_{v,n}\|^2_{L^2}\lesssim \|G_n\|_{L^2}^2,\qquad \|G_n-G'_n\|_{L^2}\lesssim 2^{-10m},
\end{equation}
for any $n\in\mathbb{Z}$. Since $G'_n\equiv 0$ unless $n\in[2^{k_1-4}\kappa^{-1},2^{k_1+4}\kappa^{-1}]$, for \eqref{szn90} it suffices to prove that
\begin{equation}\label{szn92}
\sup_{n\in[2^{k_1-4}\kappa^{-1},2^{k_1+4}\kappa^{-1}]}2^{(1+\beta)m}2^{(1+\beta)k_1}2^{k_1/2}\kappa^{-1/2}\|G'_n\|_{L^2}\lesssim 2^{-2\beta^2m}.
\end{equation}

For $f,g\in L^2(\mathbb{R}^3)$, $\xi\in\mathbb{R}^3$, and $s\in[2^{m-1},2^{m+1}]$ let, as in \eqref{szn59.5},
\begin{equation*}
\widetilde{G}'_s(f,g)(\xi)=\varphi_k(\xi)\int_{\mathbb{R}^3}e^{is\Phi^{\sigma;\mu,\nu}(\xi,\eta)}\varphi(2^{D^2+\max(0,k_1,k_2)}\Phi^{\sigma;\mu,\nu}(\xi,\eta))\cdot \widehat{f}(\xi-\eta)\widehat{g}(\eta)\,d\eta.
\end{equation*}
The left-hand side of \eqref{szn92} is dominated by
\begin{equation*}
2^{(1+\beta)m}2^{(1+\beta)k_1}2^{k_1/2}\kappa^{-1/2}\cdot 2^{m}\kappa 2^{-k_1}\sup_{s\in[2^{m-1},2^{m+1}]}\|\widetilde{G}'_s(f_{k_1,j_1}^\mu,f_{k_2,j_2}^\nu)\|_{L^2}.
\end{equation*}
Therefore, it remains to prove that
\begin{equation}\label{szn91.8}
2^{(2+\beta)m}2^{(1/2+\beta)k_1}\kappa^{1/2}\sup_{s\in[2^{m-1},2^{m+1}]}\|\widetilde{G}'_s(f_{k_1,j_1}^\mu(s),f_{k_2,j_2}^\nu(s))\|_{L^2}\lesssim 2^{-2\beta^2m}.
\end{equation}

{\bf{Step 2.}} We decompose $\phii_{j_1}^{(k_1)}\cdot P_{k_1}f_\mu(s)=2^{-\alpha k_1}[g_{k_1,j_1}^\mu(s)+h_{k_1,j_1}^\mu(s)]$ and $\phii_{j_2}^{(k_2)}\cdot P_{k_2}f_\nu(s)=[g_{k_2,j_2}^\nu(s)+h_{k_2,j_2}^\nu(s)]$ as in \eqref{szn57}--\eqref{szn58}. In this proof we will also need the stronger bounds \eqref{sec5.815} on the functions $h_{k_1,j_1}^\mu(s)$ and $h_{k_2,j_2}^\nu(s)$,
\begin{equation}\label{szn90.5}
\begin{split}
&2^{\gamma j_1}\sup_{R\in[2^{-j_1},2^{k_1}],\,\xi_0\in\mathbb{R}^3}R^{-2}\|\widehat{h_{k_1,j_1}^\mu}(s)\|_{L^1(B(\xi_0,R))}\lesssim 2^{10k_1},\\
&2^{\gamma j_2}\sup_{R\in[2^{-j_2},2^{k_2}],\,\xi_0\in\mathbb{R}^3}R^{-2}\|\widehat{h_{k_2,j_2}^\nu}(s)\|_{L^1(B(\xi_0,R))}\lesssim 1,
\end{split}
\end{equation}
and the support properties \eqref{sec5.81}. Recall the $L^2$ bounds
\begin{equation}\label{szn91.6}
\begin{split}
&\|g_{k_1,j_1}^\mu(s)\|_{L^2}\lesssim 2^{-(1+\beta)j_1},\qquad\|h_{k_1,j_1}^\mu(s)\|_{L^2}\lesssim 2^{8k_1}2^{-(1-\beta)j_1},\\
&\|g_{k_2,j_2}^\nu(s)\|_{L^2}\lesssim 2^{-(1+\beta)j_2},\qquad\|h_{k_2,j_2}^\nu(s)\|_{L^2}\lesssim 2^{-(1-\beta)j_2}.
\end{split}
\end{equation}
With $E^\mu_sf=e^{-is\widetilde{\Lambda}_\mu} f$ and $E^\nu_sg=e^{-is\widetilde{\Lambda}_\nu} g$ as in the proof in Case 1, we use the kernel bounds \eqref{ccc33}, \eqref{ccc7}, and \eqref{ccc25} (as in the proof of Lemma \ref{tech1.5}) to conclude that
\begin{equation}\label{szn91.5}
\begin{split}
&\|E^\mu_sP_{[k_1-2,k_1+2]}(g_{k_1,j_1}^\mu(s))\|_{L^\infty}\lesssim 2^{k_1/2}2^{-3m/2}2^{j_1(1/2-\beta)},\\
&\|E^\mu_sP_{[k_1-2,k_1+2]}(h_{k_1,j_1}^\mu(s))\|_{L^\infty}\lesssim 2^{8k_1}\min[2^{-3m/2}2^{j_1(1/2+\beta)},2^{-\gamma j_1}]\\
&\|E^\nu_sP_{[k_2-2,k_2+2]}(g_{k_2,j_2}^\nu(s))\|_{L^\infty}\lesssim 2^{-3m/2}2^{j_2(1/2-\beta)},\\
&\|E^\nu_sP_{[k_2-2,k_2+2]}(h_{k_2,j_2}^\nu(s))\|_{L^\infty}\lesssim \min[2^{-3m/2}2^{j_2(1/2+\beta)},2^{-\gamma j_2}],
\end{split}
\end{equation}
for any $s\in[2^{m-1},2^{m+1}]$. We combine these bounds and Lemma \ref{tech2}. It follows from \eqref{szn91.5} that $\|Ef_{k_1,j_1}^\mu(s)\|_{L^\infty}\lesssim 2^{-3m/2}2^{j_1(1/2+\beta)}$ and $\|Ef_{k_2,j_2}^\nu(s)\|_{L^\infty}\lesssim 2^{-3m/2}2^{j_2(1/2+\beta)}$. Recalling that $2^{\max(j_1,j_2)}\approx\kappa2^m2^{-\beta^2m}$ (see \eqref{szn51}), we have
\begin{equation*}
\begin{split}
\|\widetilde{G}'_s(f_{k_1,j_1}^\mu(s),f_{k_2,j_2}^\nu(s))\|_{L^2}&\lesssim 2^{-3m/2}2^{\min(j_1,j_2)(1/2+\beta)}\cdot 2^{-(1-\beta)\max(j_1,j_2)}\\
&\lesssim 2^{-|j_1-j_2|(1/2+\beta)}2^{-3m/2}2^{(-1/2+2\beta)\max(j_1,j_2)}\\
&\lesssim 2^{-|j_1-j_2|(1/2+\beta)}2^{-2m}\kappa^{-1/2}\cdot 2^{\beta^2m}2^{2\beta m},
\end{split}
\end{equation*}
for any $s\in[2^{m-1},2^{m+1}]$. The desired bound \eqref{szn91.8} follows if $2^{-|j_1-j_2|}2^{k_1}\leq 2^{-6\beta m}$. 

It remains to prove \eqref{szn91.8} in the case
\begin{equation}\label{szn94}
2^{|j_1-j_2|}2^{-k_1}\leq 2^{6\beta m}.
\end{equation}
We start by using the bounds \eqref{szn91.6}--\eqref{szn91.5} more carefully. We estimate
\begin{equation*}
\|\widetilde{G}'_s(f,g)\|_{L^2}\lesssim 2^{-3m/2}2^{\min(j_1,j_2)(1/2-\beta)}\cdot 2^{-(1+\beta)\max(j_1,j_2)}\lesssim 2^{-(2+2\beta)m}\kappa^{-1/2}\kappa^{-2\beta}2^{\beta^2m},
\end{equation*}
if $(f,g)=\big(P_{[k_1-2,k_1+2]}(g_{k_1,j_1}^\mu(s)),P_{[k_2-2,k_2+2]}(g_{k_2,j_2}^\nu(s))\big)$. This is consistent with the desired bound \eqref{szn91.8}, if we recall that $\kappa^{-1}\lesssim 2^{m/3+\beta m/2}$ (see \eqref{szn72}). Therefore it remains the prove that
\begin{equation}\label{szn95}
2^{(2+\beta)m}2^{k_1/2}\kappa^{1/2}\sup_{s\in[2^{m-1},2^{m+1}]}\|\widetilde{G}'_s(f,g)\|_{L^2}\lesssim 2^{-2\beta^2m},
\end{equation}
if \eqref{szn94} holds, and
\begin{equation}\label{szn95.5}
\begin{split}
(f,g)\in\Big\{&(P_{[k_1-2,k_1+2]}(g_{k_1,j_1}^\mu(s)),P_{[k_2-2,k_2+2]}(h_{k_2,j_2}^\nu(s))\big),\\
&\big(P_{[k_1-2,k_1+2]}(h_{k_1,j_1}^\mu(s)),P_{[k_2-2,k_2+2]}(g_{k_2,j_2}^\nu(s))\big),\\
&\big(P_{[k_1-2,k_1+2]}(h_{k_1,j_1}^\mu(s)),P_{[k_2-2,k_2+2]}(h_{k_2,j_2}^\nu(s))\big)\Big\}.
\end{split}
\end{equation}

One could try arguing as before: recalling that $\gamma=3/2-4\beta$ and using \eqref{szn94}, for $(f,g)$ as in \eqref{szn95.5} we estimate
\begin{equation*}
\begin{split}
\|\widetilde{G}'_s(f,g)\|_{L^2}&\lesssim 2^{-\gamma\min(j_1,j_2)}\cdot 2^{-(1-\beta)\max(j_1,j_2)}\\
&\lesssim 2^{\gamma|j_1-j_2|}2^{-(\gamma+1-\beta)\max(j_1,j_2)}\\
&\lesssim 2^{-5m/2+15\beta m}\kappa^{-5/2+5\beta}2^{3\beta^2m},
\end{split}
\end{equation*}
Therefore the left-hand side of \eqref{szn95} is dominated by
\begin{equation*}
C2^{k_1/2}2^{(2+\beta)m}\kappa^{1/2}\cdot 2^{3\beta^2m}2^{-5m/2+15\beta m}\kappa^{-5/2+5\beta}\lesssim 2^{3\beta^2m}2^{-m(1/2-16\beta)}\kappa^{-(2-5\beta)}.
\end{equation*}
The desired bound \eqref{szn95} follows if $\kappa^{-1}$ is sufficiently small, say $\kappa^{-1}\leq 2^{m/6}$, but not in the full range $\kappa^{-1}\leq 2^{m(1/3+\beta/2)}$ (see \eqref{szn72}). To cover the full range we need an additional argument that uses the stronger bounds \eqref{szn90.5}.

{\bf{Step 3.}} We prove now \eqref{szn95}. We reinsert first the cutoff function $\chi_R^{\sigma;\mu,\nu}$, i.e., we define
\begin{equation}\label{szn93}
\begin{split}
&\widetilde{G}''_s(f,g)(\xi):=\varphi_k(\xi)\int_{\mathbb{R}^3}e^{is\Phi^{\sigma;\mu,\nu}(\xi,\eta)}\chi_R^{\sigma;\mu,\nu}(\xi,\eta)\cdot \widehat{f}(\xi-\eta)\widehat{g}(\eta)\,d\eta.\\
&\chi_R^{\sigma;\mu,\nu}(\xi,\eta)=\varphi(2^{D^2+\max(0,k_1,k_2)}\Phi^{\sigma;\mu,\nu}(\xi,\eta))\varphi(\vert\Xi^{\mu,\nu}(\xi,\eta)\vert/\kappa),
\end{split}
\end{equation}
where $(f,g)$ are as in \eqref{szn95.5}. As before, integrating by parts in $\eta$, we notice that $\|\widetilde{G}''_s(f,g)-\widetilde{G}'_s(f,g)\|_{L^2}\lesssim 2^{-10m}$. Then we decompose
\begin{equation}\label{szn93.5}
\begin{split}
&\widetilde{G}''_s(f,g)=\sum_{v\in\mathbb{Z}^3}\widetilde{G}''_{v,s}(f,g),\\
&\widetilde{G}''_{v,s}(f,g):=\chi'(\kappa^{-1}\xi-v)\varphi_k(\xi)\int_{\mathbb{R}^3}e^{is\Phi^{\sigma;\mu,\nu}(\xi,\eta)}\chi_R^{\sigma;\mu,\nu}(\xi,\eta)\cdot \widehat{f}(\xi-\eta)\widehat{g}(\eta)\,d\eta,
\end{split}
\end{equation}
where $\chi'$ is as before. For \eqref{szn95} it remains to prove that
\begin{equation}\label{cle35}
2^{(4+2\beta)m}2^{k_1}\kappa\sum_{v\in\mathbb{Z}^3}\|\widetilde{G}''_{v,s}(f,g)\|^2_{L^2}\lesssim 2^{-4\beta^2m},
\end{equation}
for any $s\in[2^{m-1},2^{m+1}]$ and $(f,g)$ as in \eqref{szn95.5}.

In view of Lemma \ref{desc2} the variables in the definition of the function $\widetilde{G}''_{v,s}(f,g)$ are naturally restricted as follows:
\begin{equation*}
\begin{split}
&v\in\mathbb{Z}^3,\qquad \big|\kappa|v|-R_{\sigma_2}\big|\lesssim_{C_b,\varepsilon}2^{k_1}\\
&\xi=a\widehat{v}+\xi',\qquad \xi'\cdot \widehat{v}=0,\qquad |\xi'|\lesssim_{C_b,\varepsilon}\kappa,\qquad \big|a-\kappa|v|\big|\lesssim_{C_b,\varepsilon} \kappa,\\
&\xi-\eta=b\widehat{v}+\theta',\qquad \theta'\cdot \widehat{v}=0,\qquad |\theta'|\lesssim_{C_b,\varepsilon}2^{k_1}\kappa,\qquad \big|b-(\kappa |v|-r^{\mu,\nu}(\kappa|v|))\big|\lesssim_{C_b,\varepsilon} \kappa,
\end{split}
\end{equation*}
where $\widehat{v}=v/|v|$. More precisely, for any $v$ fixed we define the functions $f^v$ and $g^v$ by the formulas
\begin{equation}\label{szn80}
\begin{split}
&\widehat{f^v}(\theta):=\varphi[|\theta'|/(\kappa 2^{k_1+D})]\varphi\big[[\rho-\kappa |v|+r^{\mu,\nu}(\kappa|v|)]/(\kappa 2^D)\big]\cdot \widehat{f}(\theta),\\
&\widehat{g^v}(\theta):=\varphi(|\theta'|/(\kappa 2^D))\varphi\big[[\rho-r^{\mu,\nu}(\kappa|v|)]/(\kappa 2^D)\big]\cdot \widehat{g}(\theta),
\end{split}
\end{equation}
where $\theta=\rho\widehat{v}+\theta'$, $\rho\in\mathbb{R}$, $\theta'\cdot\widehat{v}=0$. In view of Lemma \ref{desc2} and \eqref{kx12.5}, the functions $\widehat{f^v}$ (respectively $\widehat{g^v}$) have essentially pairwise disjoint supports, i.e.,
\begin{equation}\label{szn81}
\sum_{v\in\mathbb{Z}^3}\|\widehat{f^v}\|_{L^2}^2\lesssim \|\widehat{f}\|_{L^2}^2,\qquad \sum_{v\in\mathbb{Z}^3}\|\widehat{g^v}\|_{L^2}^2\lesssim 2^{-k_1}\|\widehat{g}\|_{L^2}^2.
\end{equation}
Moreover, they suffice to determine the functions $\widetilde{G}''_{v,s}(f,g)$, i.e.,
\begin{equation*}
\widetilde{G}''_{v,s}(f,g)(\xi)=\chi'(\kappa^{-1}\xi-v)\varphi_k(\xi)\int_{\mathbb{R}^3}e^{is\Phi^{\sigma;\mu,\nu}(\xi,\eta)}\chi_R^{\sigma;\mu,\nu}(\xi,\eta)\cdot \widehat{f^v}(\xi-\eta)\widehat{g^v}(\eta)\,d\eta.
\end{equation*}

We use \eqref{szn90.5}, \eqref{szn91.6}, and \eqref{szn81}. For $(f,g)=\big(P_{[k_1-2,k_1+2]}(g_{k_1,j_1}^\mu(s)),P_{[k_2-2,k_2+2]}(h_{k_2,j_2}^\nu(s))\big)$ or  $(f,g)=\big(P_{[k_1-2,k_1+2]}(h_{k_1,j_1}^\mu(s)),P_{[k_2-2,k_2+2]}(h_{k_2,j_2}^\nu(s))\big)$, we estimate
\begin{equation*}
\sum_{v\in\mathbb{Z}^3}\|\widetilde{G}''_{v,s}(f,g)\|^2_{L^2}\lesssim \sum_{v\in\mathbb{Z}^3}\|\widehat{f^v}\|_{L^2}^2\|\widehat{g^v}\|_{L^1}^2\lesssim \|\widehat{f}\|_{L^2}^2\sup_{v\in\mathbb{Z}^3}\|\widehat{g^v}\|_{L^1}^2\lesssim 2^{-2j_1+2\beta j_1}2^{-2\gamma j_2}\kappa^4.
\end{equation*}
For $(f,g)=\big(P_{[k_1-2,k_1+2]}(h_{k_1,j_1}^\mu(s)),P_{[k_2-2,k_2+2]}(g_{k_2,j_2}^\nu(s))\big)$ we estimate
\begin{equation*}
\sum_{v\in\mathbb{Z}^3}\|\widetilde{G}''_{v,s}(f,g)\|^2_{L^2}\lesssim \sum_{v\in\mathbb{Z}^3}\|\widehat{f^v}\|_{L^1}^2\|\widehat{g^v}\|_{L^2}^2\lesssim 2^{-k_1}\|\widehat{g}\|_{L^2}^2\sup_{v\in\mathbb{Z}^3}\|\widehat{f^v}\|_{L^1}^2\lesssim 2^{-2j_2+2\beta j_2}2^{-2\gamma j_1}\kappa^4.
\end{equation*}
Therefore, using also the assumption \eqref{szn94}, the left-hand side of \eqref{cle35} is dominated by
\begin{equation*}
C2^{(4+2\beta)m}\kappa\cdot \kappa^42^{-2\gamma\min(j_1,j_2)}2^{-(2-2\beta)\max(j_1,j_2)}\lesssim 2^{(4+2\beta)m}\kappa^52^{2\gamma|j_1-j_2|}2^{-(2\gamma+2-2\beta)\max(j_1,j_2)}\lesssim 2^{-\beta m},
\end{equation*}
and the desired bound \eqref{cle35} follows. This completes the proof of the lemma.
\end{proof}

\section{Proof of Proposition \ref{Norm}, IV: Case C resonant interactions}\label{normproof4}

\begin{proposition}\label{CaseCProp}
Assume that $(k, j), (k_1, j_1), (k_2, j_2) \in\mathcal{J}$ , $m \in [1, L ] \cap\mathbb{Z}$,
\begin{equation}\label{tcv1}
\Phi^{\sigma;\mu,\nu}\in\in\mathcal{T}_C=\{\Phi^{i;i+,i+},\Phi^{i;i+,i-},\Phi^{i;i-,i-},\Phi^{i;e+,e-},\Phi^{i;e+,b-},\Phi^{i;e-,b+},\Phi^{i;b+,b-}\},
\end{equation}
and
\begin{equation}\label{Bds1ND2}
\begin{split}
&-9m/10\le k_1,k_2\le j/N_0^\prime,\qquad\max(j_1,j_2)\le (1-\beta/10)m+k,\qquad m\geq -k(1+\beta^2),\\
&\beta m/2+N_0^\prime k_++D^2\le j\le m+D,\qquad k\leq -D/4.
\end{split}
\end{equation}
Then there is $\kappa\in(0,1]$, $\kappa\geq\max\big(2^{(\beta^2m-m)/2}2^{-\min(k_1,k_2,0)/2}2^{-D/2},2^{\beta^2m-m}2^{\max(j_1,j_2)}\big)$, such that
\begin{equation}\label{CaseCOK}
2^k\big\|\widetilde{\varphi}^{(k)}_j\cdot P_kR_{m,\kappa}^{\sigma;\mu,\nu}(f_{k_1,j_1}^\mu,f_{k_2,j_2}^\nu)\big\|_{B^1_{k,j}}\lesssim 2^{-2\beta^4m}.
\end{equation}
\end{proposition}

The rest of the section is concerned with the proof of Proposition \ref{CaseCProp}. Many of the easier cases can be handled using the following lemma:

\begin{lemma}\label{CaseDIBPS}
(i) With the hypothesis in Proposition \ref{CaseCProp}, assume in addition that either
\begin{equation}\label{CaseDMultiplier1}
\Big\| \mathcal{F}^{-1}
\left\{\frac{1}{\Phi^{i;\mu,\nu}(\xi,\eta)}
\varphi_{[k-4,k+4]}(\xi)\varphi_{[k_1-4,k_1+4]}(\xi-\eta)\varphi_{[k_2-4,k_2+4]}(\eta)\right\}
\Big\|_{L^1(\mathbb{R}^3\times\mathbb{R}^3)}\lesssim 2^{-k}2^{6\max(k_1,k_2,0)},
\end{equation}
or that
\begin{equation}\label{CaseDMultiplier2}
\begin{split}
&10\beta k\leq\max(k_1,k_2)\leq -D/100,\qquad\max(j_1,j_2)\le 3\beta m,\\
&\Big\| \mathcal{F}^{-1}
\left\{\frac{1}{\Phi^{i;\mu,\nu}(\xi,\eta)}\varphi_{[k-4,k+4]}(\xi)\varphi_{[k_1-4,k_1+4]}(\xi-\eta)\varphi_{[k_2-4,k_2+4]}(\eta)\right\}
\Big\|_{L^1(\mathbb{R}^3\times\mathbb{R}^3)}\lesssim 2^{-k}2^{m/3},\\
&\Big\| \frac{1}{\Phi^{i;\mu,\nu}(\xi,\eta)}\varphi_{[k-4,k+4]}(\xi)\varphi_{[k_1-4,k_1+4]}(\xi-\eta)\varphi_{[k_2-4,k_2+4]}(\eta)
\Big\|_{L^\infty(\mathbb{R}^3\times\mathbb{R}^3)}\lesssim 2^{-k-k_1-k_2}.
\end{split}
\end{equation}
Then, for any $\kappa\in(0,1]$, $\kappa\geq\max\big(2^{(\beta^2m-m)/2}2^{-\min(k_1,k_2,0)/2}2^{-D/2},2^{\beta^2m-m}2^{\max(j_1,j_2)}\big)$,
\begin{equation}\label{CaseDOKii}
2^{k}\Vert \widetilde{\varphi}^{(k)}_j\cdot P_kR_{m,\kappa}^{\sigma;\mu,\nu}(f^\mu_{k_1,j_1},f^\nu_{k_2,j_2})\Vert_{B^1_{k,j}}\lesssim 2^{-2\beta^4 m}.
\end{equation}

(ii) The inequality \eqref{CaseDMultiplier1} holds if
\begin{equation}\label{bvc0}
|\Phi^{i;\mu,\nu}(\xi,\eta)|\gtrsim 2^{\widetilde{k}}+2^{\widetilde{k_1}}+2^{\widetilde{k_2}}
\end{equation}
for all $(\xi,\eta)\in\mathbb{R}^3\times\mathbb{R}^3$ satisfying $|\xi|\in[2^{k-6},2^{k+6}]$, $|\xi-\eta|\in[2^{k_1-6},2^{k_1+6}]$, $|\eta|\in [2^{k_2-6},2^{k_2+6}]$.
\end{lemma}

\begin{proof}[Proof of Lemma \ref{CaseDIBPS}] 
(i) The proof is similar to the bound on $N^{1;\sigma;\mu,\nu}_m$ in Lemma \ref{NewBigBound1}. In view of \eqref{NRI} it suffices to prove that
\begin{equation}\label{CaseDOKii2}
2^{k}\Vert \widetilde{\varphi}^{(k)}_j\cdot P_kT_{m}^{i;\mu,\nu}(f^\mu_{k_1,j_1},f^\nu_{k_2,j_2})\Vert_{B^1_{k,j}}\lesssim 2^{-2\beta^4 m}.
\end{equation}
After integration by parts in $s$ in \eqref{nh5}, we obtain that
\begin{equation}\label{IBTNRI2D}
\begin{split}
2^k\mathcal{F}[T_{m}^{i;\mu,\nu}(f^\mu_{k_1,j_1},f^\nu_{k_2,j_2})]&=i\left[T_{1}+T_{2}+T_{3}\right],\\
T_{1}(\xi)&:=\int_{\mathbb{R}}\int_{\mathbb{R}^3}e^{is\Phi^{i;\mu,\nu}(\xi,\eta)}\frac{2^k}{\Phi^{i;\mu,\nu}(\xi,\eta)}q^\prime_m(s)\cdot \widehat{f^\mu_{k_1,j_1}}(\xi-\eta,s)\widehat{f^\nu_{k_2,j_2}}(\eta,s)\,d\eta ds,\\
T_{2}(\xi)&:=\int_{\mathbb{R}}\int_{\mathbb{R}^3}e^{is\Phi^{i;\mu,\nu}(\xi,\eta)}\frac{2^k}{\Phi^{i;\mu,\nu}(\xi,\eta)}q_m(s)\cdot (\partial_s\widehat{f^\mu_{k_1,j_1}})(\xi-\eta,s)\widehat{f^\nu_{k_2,j_2}}(\eta,s)\,d\eta ds,\\
T_{3}(\xi)&:=\int_{\mathbb{R}}\int_{\mathbb{R}^3}e^{is\Phi^{i;\mu,\nu}(\xi,\eta)}\frac{2^k}{\Phi^{i;\mu,\nu}(\xi,\eta)}q_m(s)\cdot \widehat{f^\mu_{k_1,j_1}}(\xi-\eta,s)(\partial_s\widehat{f^\nu_{k_2,j_2}})(\eta,s)\,d\eta ds.
\end{split}
\end{equation}

Recall Definition \ref{MainDef}. We first show that
\begin{equation}\label{LINormN1D}
2^{(1/2-\beta+\alpha)k}\cdot 2^k\Vert \varphi_k\cdot \mathcal{F}T_{m}^{i;\mu,\nu}(f^\mu_{k_1,j_1},f^\nu_{k_2,j_2})\Vert_{L^\infty}\lesssim 2^{-2\beta^4m}.
\end{equation}
Indeed, using Cauchy-Schwartz inequality, \eqref{nh9}, \eqref{ok50repeat}, and either \eqref{CaseDMultiplier1} or \eqref{CaseDMultiplier2}, we see that, letting $\Lambda:=2^{-20\beta k}+2^{6\max(k_1,k_2,0)}$,
\begin{equation*}
\begin{split}
\|\varphi_k\cdot T_{2}\|_{L^\infty}&\lesssim 2^m \Lambda\sup_{s\in[2^{m-2},2^{m+2}]}\Vert(\partial_s\widehat{f^\mu_{k_1,j_1}}(s))\Vert_{L^2}\Vert \widehat{f^\nu_{k_2,j_2}}(s)\Vert_{L^2}\lesssim \Lambda 2^{-\beta m} 2^{-10\max(k_1,k_2,0)},\\
\|\varphi_k\cdot T_{3}\|_{L^\infty}&\lesssim 2^m\Lambda\sup_{s\in[2^{m-2},2^{m+2}]}\Vert \widehat{f}^\mu_{k_1,j_1}(s)\Vert_{L^2}\Vert (\partial_s\widehat{f}^\nu_{k_2,j_2}(s))\Vert_{L^2}\lesssim \Lambda 2^{-\beta m}2^{-10\max(k_1,k_2,0)},
\end{split}
\end{equation*}
and this gives acceptable contributions. Proceeding as above, using \eqref{nh9.1} we get
\begin{equation*}
\|\varphi_k\cdot T_{1}\|_{L^\infty}\lesssim \Lambda 2^{-(1-\beta)(j_1+j_2)}(1+2^{k_1})^{-10}(1+2^{k_2})^{-10}.
\end{equation*}
Therefore, this gives an acceptable contribution to \eqref{LINormN1D} unless
\begin{equation}\label{SmallParamD}
\vert k\vert+\vert k_1\vert+\vert k_2\vert+j_1+j_2\le \beta^2 m.
\end{equation}
Now, assuming that \eqref{SmallParamD} holds, we can strengthen the $L^\infty$ bound. Indeed, we decompose
\begin{equation*}
\begin{split}
T_{1}(\xi)&=T_{1;1}(\xi)+T_{1;2}(\xi)\\
T_{1;1}(\xi)&:=\int_{\mathbb{R}}\int_{\mathbb{R}^3}e^{is\Phi^{i;\mu,\nu}(\xi,\eta)}\frac{2^k}{\Phi^{i;\mu,\nu}(\xi,\eta)}\varphi(\delta^{-1}\vert\Xi^{\mu,\nu}(\xi,\eta)\vert)q^\prime_m(s)\cdot \widehat{f^\mu_{k_1,j_1}}(\xi-\eta,s)\widehat{f^\nu_{k_2,j_2}}(\eta,s)\,d\eta ds,\\
T_{1;2}(\xi)&:=\int_{\mathbb{R}}\int_{\mathbb{R}^3}e^{is\Phi^{i;\mu,\nu}(\xi,\eta)}\frac{2^k}{\Phi^{i;\mu,\nu}(\xi,\eta)}(1-\varphi(\delta^{-1}\vert\Xi^{\mu,\nu}(\xi,\eta))\vert)q^\prime_m(s)\cdot \widehat{f^\mu_{k_1,j_1}}(\xi-\eta,s)\widehat{f^\nu_{k_2,j_2}}(\eta,s)\,d\eta ds,\\
\end{split}
\end{equation*}
with $\delta=2^{-m/3}$. Applying Lemma \ref{tech5} with $K\approx 2^{2m/3}$, $\epsilon\approx2^{-m/3}$, it is easy to see that 
\begin{equation*}
\|\varphi_k\cdot T_{1;2}\|_{L^\infty}\lesssim 2^{-10m}
\end{equation*}
if \eqref{SmallParamD} holds, which is clearly sufficient. On the other hand, we observe that
\begin{equation*}
\begin{split}
\vert\Xi^{\mu,\nu}(\xi,\eta)\vert&\gtrsim \vert\nabla\widetilde{\Lambda}_\nu(\eta)\vert\cdot\min(\big| (\xi-\eta)/\vert\xi-\eta\vert-\eta/\vert\eta\vert\big|,\big|\vert (\xi-\eta)/\vert\xi-\eta\vert+\eta/\vert\eta\vert\big|)\\
&\gtrsim 2^{-\beta m}\min(\big| (\xi-\eta)/\vert\xi-\eta\vert-\eta/\vert\eta\vert\big|,\big| (\xi-\eta)/\vert\xi-\eta\vert+\eta/\vert\eta\vert\big|)
\end{split}
\end{equation*}
Consequently if $\vert \xi\vert\in [2^{k-2},2^{k+2}]$, $\vert\xi-\eta\vert\in[2^{k_1-2}2^{k_1+2}]$, $\vert\eta\vert\in [2^{k_2-2},2^{k_2+2}]$, and $\vert\Xi^{\mu,\nu}(\xi,\eta))\vert\lesssim 2^{-m/3}$ then 
\begin{equation*}
\min(\big| (\xi-\eta)/\vert\xi-\eta\vert-\eta/\vert\eta\vert\big|,\big| (\xi-\eta)/\vert\xi-\eta\vert+\eta/\vert\eta\vert\big|)\lesssim 2^{-m/4}.
\end{equation*}
A simple estimate using the $L^\infty$ bounds in \eqref{nh9} then gives $\|\varphi_k\cdot T_{1;1}\|_{L^\infty}\lesssim 2^{-m/6}$, which suffices to finish the proof of \eqref{LINormN1D}.

To finish the proof of \eqref{CaseDOKii2}, it suffices to prove that
\begin{equation}\label{SuffNRI1D}
2^{(1+\alpha)  k}2^{(1+\beta)m}\Vert P_kT_{m}^{i;\mu,\nu}(f^\mu_{k_1,j_1},f^\nu_{k_2,j_2})\Vert_{L^2}\lesssim 2^{-2\beta^4m}.
\end{equation}
Assume first that \eqref{CaseDMultiplier1} holds. Then we use the $L^2$ bounds
\begin{equation}\label{bvc1}
\begin{split}
&\| f^\mu_{k_1,j_1}(s)\|_{L^2}\lesssim (2^{\alpha k_1}+2^{10k_1})^{-1}2^{-(1-\beta)j_1}2^{2\beta k_1},\\
&\| f^\nu_{k_2,j_2}(s)\|_{L^2}\lesssim (2^{\alpha k_2}+2^{10k_2})^{-1}2^{-(1-\beta)j_2}2^{2\beta k_2},\\
&\|(\partial_s f^\mu_{k_1,j_1})(s)\|_{L^2}+\|(\partial_s f^\nu_{k_2,j_2})(s)\|_{L^2}\lesssim 2^{-m(1+\beta)},
\end{split}
\end{equation}
and the $L^\infty$ bounds
\begin{equation}\label{bvc2}
\begin{split}
&\| Ef^\mu_{k_1,j_1}(s)\|_{L^\infty}\lesssim \min(2^{-6k_1},2^{(1/2-\alpha-\beta)k_1})2^{-m(1+\beta)},\\
&\| Ef^\mu_{k_1,j_1}(s)\|_{L^\infty}\lesssim 2^{-m(5/4-10\beta)}2^{j_1(1/4-11\beta)},\\
&\| Ef^\nu_{k_2,j_2}(s)\|_{L^\infty}\lesssim \min(2^{-6k_2},2^{(1/2-\alpha-\beta)k_2})2^{-m(1+\beta)},\\
&\| Ef^\nu_{k_2,j_2}(s)\|_{L^\infty}\lesssim 2^{-m(5/4-10\beta)}2^{j_2(1/4-11\beta)},
\end{split}
\end{equation}
for any $s\in[2^{m-1},2^{m+1}]$, see \eqref{nh9}--\eqref{ok50repeat}. Using the assumption \eqref{CaseDMultiplier1} and Lemma \ref{tech2}, it follows that
\begin{equation*}
\begin{split}
&2^{\alpha k}2^{(1+\beta)m}\Vert P_kT_1\Vert_{L^2}\lesssim 2^{(1+\beta)m}\cdot 2^{-m(5/4-10\beta)}2^{\min(j_1,j_2)(1/4-11\beta)}2^{-(1-\beta)\max(j_1,j_2)}\lesssim 2^{-\beta m},\\
&2^{\alpha k}2^{(1+\beta)m}\Vert P_kT_2\Vert_{L^2}\lesssim 2^{(1+\beta)m}\cdot 2^m2^{-m(1+\beta)}2^{-m(1+\beta)}\lesssim 2^{-\beta m},\\
&2^{\alpha k}2^{(1+\beta)m}\Vert P_kT_3\Vert_{L^2}\lesssim 2^{(1+\beta)m}\cdot 2^m2^{-m(1+\beta)}2^{-m(1+\beta)}\lesssim 2^{-\beta m},
\end{split}
\end{equation*}
and the desired bound \eqref{SuffNRI1D} follows.

Assume now that \eqref{CaseDMultiplier2} holds. We estimate as before, using however the stronger $L^\infty$ bounds
\begin{equation}\label{bvc3}
\begin{split}
&\| Ef^\mu_{k_1,j_1}(s)\|_{L^\infty}\lesssim 2^{-m(3/2-10\beta)}2^{j_1(1/2-11\beta)},\\
&\| Ef^\nu_{k_2,j_2}(s)\|_{L^\infty}\lesssim 2^{-m(3/2-10\beta)}2^{j_2(1/2-11\beta)},
\end{split}
\end{equation}
see \eqref{mk15.6} and \eqref{mk15.65}. Then we estimate, using \eqref{bvc1}, \eqref{bvc3}, and Lemma \ref{tech2},
\begin{equation*}
\begin{split}
2^{\alpha k}2^{(1+\beta)m}\Vert P_kT_1\Vert_{L^2}&\lesssim 2^{(1+\beta)m}\cdot 2^{m/3}\cdot 2^{-m(3/2-10\beta)}2^{\min(j_1,j_2)(1/2-11\beta)}2^{-(1-\beta)\max(j_1,j_2)}\\
&\lesssim 2^{m(4/3-3/2+12\beta)}\lesssim 2^{-\beta^2 m}.
\end{split}
\end{equation*}
Similarly, we estimate, for $l\in\{2,3\}$
\begin{equation*}
\begin{split}
2^{\alpha k}2^{(1+\beta)m}\Vert P_kT_l\Vert_{L^2}&\lesssim 2^{(1+\beta)m}\cdot 2^m2^{m/3}2^{-m(1+\beta)}2^{-m(3/2-10\beta)}2^{\max(j_1,j_2)(1/2-11\beta)}\\
&\lesssim 2^{m(4/3-3/2+12\beta)}\lesssim 2^{-\beta^2m}.
\end{split}
\end{equation*}
The desired bound \eqref{SuffNRI1D} follows from the last two estimates.

(ii) Without loss of generality we may assume that $k_1\geq k_2$. Let
\begin{equation*}
K(x,y):=\int_{\mathbb{R}^3}\int_{\mathbb{R}^3}e^{ix\cdot\xi}e^{iy\cdot\eta}\frac{1}{\Phi^{i;\mu,\nu}(\xi,\eta)}
\varphi_{[k-4,k+4]}(\xi)\varphi_{[k_1-4,k_1+4]}(\xi-\eta)\varphi_{[k_2-4,k_2+4]}(\eta)\,d\xi d\eta.
\end{equation*}
We can integrate by parts to prove suitable estimates on the kernel $K$. We use the general formula
\begin{equation}\label{Der1/Phi}
\partial^a_\xi\partial^b_\eta\frac{1}{\Phi}=\sum_{n\leq |a|+|b|,\,|a_1|+\ldots+|a_n|=|a|,\,|b_1|+\ldots+|b_n|=|b|}c_{a_1,\ldots a_n,b_1,\ldots b_n}\frac{1}{\Phi}\frac{\partial^{a_1}_\xi\partial^{b_1}_\eta\Phi}{\Phi}\ldots\frac{\partial^{a_n}_\xi\partial^{b_n}_\eta\Phi}{\Phi},
\end{equation}
for any multi-indices $a$ and $b$, which follows easily by induction. 

It follows from \eqref{bvc0} that
\begin{equation*}
|K(x,y)|(1+2^{\widetilde{k}}|x|)^4(1+2^{\widetilde{k}_2}|y|)^4\lesssim 2^{3k}2^{3k_2}\left[2^{\widetilde{k}}+2^{\widetilde{k_1}}+2^{\widetilde{k_2}}\right]^{-1}\lesssim 2^{-\widetilde{k}_1}2^{3(k+k_2)}
\end{equation*}
for any $x,y\in\mathbb{R}^3$. Therefore $\|K\|_{L^1(\mathbb{R}^3\times\mathbb{R}^3)}\lesssim 2^{-\widetilde{k}_1}2^{6\max(0,k_1)}$ as desired.
\end{proof}

We will also need the following result whose proof is identical to the one of the first case of $(i)$ above since we can always use Lemma \ref{CZop} to pass from $f^\mu_{k_1,j_1}$ to $Q_1f^\mu_{k_1,j_1}$ and from $f^\nu_{k_2,j_2}$ to $Q_2f^\nu_{k_2,j_2}$.

\begin{lemma}\label{AddedLemmaCasD1.0001}

Let $q_0,q_1,q_2\in\mathcal{S}^{100}$ and define the operators $Q_0$, $Q_1$ and $Q_2$ by
\begin{equation*}
\mathcal{F}\left[Q_jf\right](\xi)=q_j(\xi)\widehat{f}(\xi).
\end{equation*}
With the hypothesis in Proposition \ref{CaseCProp}, assume in addition that
\begin{equation}\label{CaseDMultiplier1Q}
\Big\| \mathcal{F}^{-1}
\left\{\frac{q_0(\xi)q_1(\xi-\eta)q_2(\eta)}{\Phi^{i;\mu,\nu}(\xi,\eta)}
\varphi_{[k-4,k+4]}(\xi)\varphi_{[k_1-4,k_1+4]}(\xi-\eta)\varphi_{[k_2-4,k_2+4]}(\eta)\right\}
\Big\|_{L^1(\mathbb{R}^3\times\mathbb{R}^3)}\lesssim 2^{-k}2^{6\max(k_1,k_2,0)},
\end{equation}
then, for any $\kappa\in(0,1]$, $\kappa\geq\max\big(2^{(\beta^2m-m)/2}2^{-\min(k_1,k_2,0)/2}2^{-D/2},2^{\beta^2m-m}2^{\max(j_1,j_2)}\big)$,
\begin{equation}\label{CaseDQOKii}
2^{k}\Vert \widetilde{\varphi}^{(k)}_j\cdot P_kR_{m,\kappa}^{\sigma;\mu,\nu}(Q_1f^\mu_{k_1,j_1},Q_2f^\nu_{k_2,j_2})\Vert_{B^1_{k,j}}\lesssim 2^{-2\beta^4 m}.
\end{equation}

\end{lemma}

\subsection{Proof of Proposition \ref{CaseCProp}} 
We divide the proof in several cases. We consider first the easier phases.

\begin{lemma}\label{CaseDProp}
The bound \eqref{CaseCOK} holds if \eqref{Bds1ND2} holds and, in addition,
\begin{equation}\label{bvc80}
\Phi^{\sigma;\mu,\nu}\in\in\{\Phi^{i;i-,i-},\Phi^{i;e+,e-},\Phi^{i;e+,b-},\Phi^{i;e-,b+},\Phi^{i;b+,b-}\},
\end{equation}
with $\kappa=\max\big(2^{(\beta^2m-m)/2}2^{-\min(k_1,k_2,0)/2}2^{-D/2},2^{\beta^2m-m}2^{\max(j_1,j_2)}\big)$.
\end{lemma}

\begin{proof}[Proof of Lemma \ref{CaseDProp}] The condition \eqref{bvc0} is clearly satisfied if $\Phi^{\sigma;\mu,\nu}=\Phi^{i;i-,i-}$. This condition is also satisfied if 
\begin{equation*}
\Phi^{\sigma;\mu,\nu}\in\in\{\Phi^{i;e+,e-},\Phi^{i;e+,b-},\Phi^{i;e-,b+},\Phi^{i;b+,b-}\}\qquad\text{ and }\qquad 2\max(k_1,k_2)\le k-D/10.
\end{equation*}
Indeed, in this case
\begin{equation*}
\vert\Phi^{i;\mu,\nu}(\xi,\eta)\vert\ge \Lambda_i(\xi)-|\Lambda_{\sigma_1}(\xi-\eta)-\Lambda_{\sigma_2}(\eta)|\gtrsim \vert\xi\vert-C_{C_b,\varepsilon}(\vert\xi-\eta\vert^2+\vert\eta\vert^2)\gtrsim 2^k.
\end{equation*}
The desired bound \eqref{CaseCOK} then follows from Lemma \ref{CaseDIBPS} in these cases.

It remains to prove the proposition in the case
\begin{equation}\label{bvc20}
\Phi^{\sigma;\mu,\nu}\in\in\{\Phi^{i;e+,e-},\Phi^{i;e+,b-},\Phi^{i;e-,b+},\Phi^{i;b+,b-}\}\qquad\text{ and }\qquad k\le 2\max(k_1,k_2)+D/10.
\end{equation}
In this case, since $k\le -D/4$ from \eqref{Bds1ND2}, we remark that $k<\min(k_1,k_2)-20$, $\vert k_1-k_2\vert\le 8$ (otherwise, we would have $k\ge\max(k_1,k_2)-4$, which is incompatible with \eqref{bvc20} and $k\le -D/4$) and that
\begin{equation}\label{CaseDXi}
\begin{split}
\vert\Xi^{\mu,\nu}(\xi,\eta)\vert\gtrsim_{C_b,\varepsilon}
\begin{cases}
\vert\xi\vert(1+\vert\xi-\eta\vert+\vert\eta\vert)^{-3}&\quad\hbox{if }\sigma_1=\sigma_2,\\
(\vert\eta\vert+\vert\xi-\eta\vert)(1+\vert\xi-\eta\vert+\vert\eta\vert)^{-3}&\quad\hbox{if }\sigma_1\ne\sigma_2.
\end{cases}
\end{split}
\end{equation}
These inequalities follow from Lemma \ref{tech99}. Consequently, we see that if
\begin{equation}\label{CaseDLargek}
\begin{split}
k&\ge 2\beta m-m/2-\min(k_1,k_2,0)/2\quad\hbox{if}\quad\sigma_1=\sigma_2,\\
\max(k_1,k_2)&\ge 2\beta m-m/2-\min(k_1,k_2,0)/2\quad\hbox{if}\quad\sigma_1\ne\sigma_2,
\end{split}
\end{equation}
then $P_kR^{i;\mu,\nu}_{m,\kappa}=0$, since $\kappa=\max(2^{-(1-\beta^2)m/2}2^{-\min(k_1,k_2,0)/2}2^{-D/2},2^{\max(j_1,j_2)-(1-\beta^2)m})$. The desired bound \eqref{CaseCOK} becomes trivial in this case.

Independently, using Lemma \ref{tech2} and \eqref{ccc7}, \eqref{ccc25}, we directly see that
\begin{equation*}
\begin{split}
\Vert P_kT_{m}^{i;\mu,\nu}&(f^\mu_{k_1,j_1},f^\nu_{k_2,j_2})\Vert_{L^2}\\
&\lesssim 2^m\sup_{s\in[2^{m-4},2^{m+4}]}\min(\Vert Ef^\mu_{k_1,j_1}(s)\Vert_{L^\infty}\Vert f^\nu_{k_2,j_2}(s)\Vert_{L^2},\Vert f^\mu_{k_1,j_1}(s)\Vert_{L^2}\Vert Ef^\nu_{k_2,j_2}(s)\Vert_{L^\infty})\\
&\lesssim 2^m2^{-3m/2},\\
\end{split}
\end{equation*}
from which we deduce that
\begin{equation*}
\begin{split}
2^{(1+\alpha)k}2^{(1+\beta)j}\Vert P_kT_{m}^{i;\mu,\nu}(f^\mu_{k_1,j_1},f^\nu_{k_2,j_2})\Vert_{L^2}
\lesssim 2^{(1+\alpha)k}2^{(1/2+\beta)m}.
\end{split}
\end{equation*}
In addition
\begin{equation*}
\begin{split}
2^{(3/2-\beta+\alpha)k}\Vert \mathcal{F}P_kT_{m}^{i;\mu,\nu}(f^\mu_{k_1,j_1},f^\nu_{k_2,j_2})\Vert_{L^\infty}&\lesssim 2^{(3/2-\beta+\alpha)k}2^m\sup_{s\in [2^{m-4},2^{m+4}]}\Vert f^\mu_{k_1,j_1}(s)\Vert_{L^2}\Vert f^\nu_{k_2,j_2}(s)\Vert_{L^2}\\
&\lesssim 2^{(3/2-\beta+\alpha)k}2^m.
\end{split}
\end{equation*}
Therefore
\begin{equation}\label{bvc95}
\begin{split}
2^k\Vert \widetilde{\varphi}^{(k)}_j\cdot P_kT_{m}^{i;\mu,\nu}(f^\mu_{k_1,j_1},f^\nu_{k_2,j_2})\Vert_{B^1_{k,j}}&\lesssim 2^{(1+\alpha)k}2^{(1/2+\beta)m}
+2^{(3/2-\beta+\alpha)k}2^m.
\end{split}
\end{equation}

Using also \eqref{NRI}, this gives the desired bound \eqref{CaseCOK} if $\sigma_1\ne\sigma_2$ and \eqref{CaseDLargek} does not hold, since in this case the right-hand side of \eqref{bvc95} is dominated by $C2^{-m/10}$.

If $\sigma_1=\sigma_2$, and $k\le -3m/4$, then \eqref{CaseCOK} follows from \eqref{bvc95} and \eqref{NRI}. If $k\ge-3m/4$, since $\Lambda_{\sigma_1}$ is smooth when $\sigma_1\in\{e,b\}$, we observe that
\begin{equation*}
\vert \partial_{\eta}^{\rho}\Phi^{i;\sigma_1+,\sigma_1-}(\xi,\eta)\vert\lesssim_{\rho} 2^k, \quad\forall\,\, \vert\rho\vert\ge 2,\,\,\sigma_1\in\{e,b\},
\end{equation*}
as long as $|\xi|\leq 2^{k+4}$. Besides, from \eqref{Bds1ND2}, we have that $\max(j_1,j_2)\le (1-\beta/10)m+k$. Therefore we can use Lemma \ref{tech5} with $K=2^{(1-\beta/20)m+k}$, $\epsilon=2^{-\max(j_1,j_2)}$ to conclude that
\begin{equation*}
\vert P_kT_{m}^{i;\mu,\nu}(f^\mu_{k_1,j_1},f^\nu_{k_2,j_2})(\xi)\vert\lesssim 2^{-10m}
\end{equation*}
from which the desired inequality \eqref{CaseCOK} follows easily.
\end{proof}

We consider now the remaining two phases. A key observation is the weak ellipticity bound
\begin{equation}\label{BdPhiiii}
\lambda_i(a)+\lambda_i(b)-\lambda_i(a+b)\gtrsim_{C_b,\varepsilon}a\min(1,b)^2\qquad\text{ if }0\leq a\leq b\text{ and }a\in[0,2^{-D/20}].
\end{equation}
Indeed, using Lemma \ref{tech99}, if $b\leq r_\ast/2$ then
\begin{equation*}
\lambda_i(a)+\lambda_i(b)-\lambda_i(a+b)=\int_0^a[\lambda'(r)-\lambda'(b+r)]\,dr=\int_0^a\int_0^b -\lambda''_i(r+s)\,drds\approx_{C_b,\varepsilon}ab^2.
\end{equation*}
On the other hand, if $b\geq r_\ast/2$ then
\begin{equation*}
\lambda_i(a)+\lambda_i(b)-\lambda_i(a+b)=\int_0^a[\lambda'(r)-\lambda'(b+r)]\,dr\geq \int_0^a[q_i(r)+rq'_i(r)-q_i(b+r)]\,dr\approx_{C_b,\varepsilon}a,
\end{equation*}
and the desired lower bound \eqref{BdPhiiii} follows.

We prove first the required $L^\infty$ bounds. 

\begin{lemma}\label{CaseCLinftyLem}
Assume that \eqref{Bds1ND2} holds and
\begin{equation*}
\Phi^{\sigma;\mu,\nu}\in\in\{\Phi^{i;i+,i+},\Phi^{i;i+,i-}\}.
\end{equation*}
Then, for any $\kappa\in(0,1]$, $\kappa\geq \max\big(2^{(\beta^2m-m)/2}2^{-\min(k_1,k_2,0)/2}2^{-D/2},2^{\beta^2m-m}2^{\max(j_1,j_2)}\big)$, we have
\begin{equation}\label{CaseCLinftyNormOK}
2^{(3/2+\alpha-\beta)k}\Vert \mathcal{F}P_kR_{m,\kappa}^{i;\mu,\nu}(f^\mu_{k_1,j_1},f^\nu_{k_2,j_2})\Vert_{L^\infty}\lesssim 2^{-2\beta^4m}.
\end{equation}
\end{lemma}

\begin{proof}[Proof of Lemma \ref{CaseCLinftyLem}] Integration by parts in $s$ gives
\begin{equation*}
\begin{split}
\mathcal{F}T_m^{i;\mu,\nu}(f^\mu_{k_1,j_1},f^\nu_{k_2,j_2})(\xi)&=i\left[T_1(\xi)+T_2(\xi)+T_3(\xi)\right]\\
T_1(\xi)&:=\int_{\mathbb{R}}\int_{\mathbb{R}^3}\frac{e^{is\Phi^{i;\mu,\nu}(\xi,\eta)}}{\Phi^{i;\mu,\nu}(\xi,\eta)}q^\prime_m(s)\widehat{f^\mu_{k_1,j_1}}(\xi-\eta,s)\widehat{f^\nu_{k_2,j_2}}(\eta,s)d\eta ds\\
T_2(\xi)&:=\int_{\mathbb{R}}\int_{\mathbb{R}^3}\frac{e^{is\Phi^{i;\mu,\nu}(\xi,\eta)}}{\Phi^{i;\mu,\nu}(\xi,\eta)}q_m(s)(\partial_s\widehat{f^\mu_{k_1,j_1}})(\xi-\eta,s)\widehat{f^\nu_{k_2,j_2}}(\eta,s)d\eta ds\\
T_3(\xi)&:=\int_{\mathbb{R}}\int_{\mathbb{R}^3}\frac{e^{is\Phi^{i;\mu,\nu}(\xi,\eta)}}{\Phi^{i;\mu,\nu}(\xi,\eta)}q_m(s)\widehat{f^\mu_{k_1,j_1}}(\xi-\eta,s)(\partial_s\widehat{f^\nu_{k_2,j_2}})(\eta,s)d\eta ds.
\end{split}
\end{equation*}
First, using \eqref{BdPhiiii}, \eqref{nh9}, and \eqref{ok50repeat}, we see that
\begin{equation*}
\begin{split}
2^{(3/2+\alpha-\beta)k}\Vert \varphi_k\cdot T_2\Vert_{L^\infty}&\lesssim 2^{(3/2+\alpha-\beta)k}2^{-k-\widetilde{k_1}-\widetilde{k_2}}2^m\sup_{s\in[2^{m-4},2^{m+4}]}\Vert(\partial_s\widehat{f^\mu_{k_1,j_1}})(s)\Vert_{L^2}\Vert\widehat{f^\nu_{k_2,j_2}}(s)\Vert_{L^2}\\
&\lesssim 2^{(3/2+\alpha-\beta)k}2^{-k-\widetilde{k_1}-\widetilde{k_2}}2^m\cdot 2^{\widetilde{k_1}}2^{-(1+\beta)m}2^{(1+\beta-\alpha)\widetilde{k_2}}\\
&\lesssim 2^{-\beta^3m}.
\end{split}
\end{equation*}
Similarly,
\begin{equation*}
2^{(3/2+\alpha-\beta)k}\Vert \varphi_k\cdot T_3\Vert_{L^\infty}\lesssim 2^{-\beta^3m}.
\end{equation*}
Assuming that
\begin{equation*}
\vert k\vert+\vert k_1\vert+\vert k_2\vert+\vert j_1\vert+\vert j_2\vert\ge\beta^2m,
\end{equation*}
and using H\"older's inequality, we find that
\begin{equation*}
\begin{split}
2^{(3/2+\alpha-\beta)k}&\Vert \varphi_k\cdot T_1\Vert_{L^\infty}
\lesssim 2^{(3/2+\alpha-\beta)k}2^{-k-\widetilde{k_1}-\widetilde{k_2}}\sup_{s\in[2^{m-4},2^{m+4}]}\Vert \widehat{f^\mu_{k_1,j_1}}(s)\Vert_{L^2}\Vert\widehat{f^\nu_{k_2,j_2}}(s)\Vert_{L^2}\\
&\lesssim 2^{(3/2+\alpha-\beta)k}2^{-k-\widetilde{k_1}-\widetilde{k_2}}2^{2\beta(\widetilde{k_1}+\widetilde{k_2})}(2^{\alpha k_1}+2^{10 k_1})^{-1}(2^{\alpha k_2}+2^{10 k_2})^{-1}2^{-(1-\beta)(j_1+j_2)}\\
&\lesssim 2^{-\beta^3m}.
\end{split}
\end{equation*}
On the other hand, if
\begin{equation*}
\vert k\vert+\vert k_1\vert+\vert k_2\vert+\vert j_1\vert+\vert j_2\vert\le\beta^2m,
\end{equation*}
then we can decompose $T_1=T_{1;1}+T_{1;2}$ as in the proof of Lemma \ref{CaseDIBPS} and estimate also $2^{(3/2+\alpha-\beta)k}\Vert \varphi_k\cdot T_1\Vert_{L^\infty}\lesssim 2^{-\beta^3m}$. The desired bound \eqref{CaseCLinftyNormOK} follows using also \eqref{NRI}.
\end{proof}

We prove now the weighted $L^2$ bounds, in two steps. 

\begin{lemma}\label{CaseDProp2}
Assume that \eqref{Bds1ND2} holds and
\begin{equation*}
\Phi^{\sigma;\mu,\nu}\in\in\{\Phi^{i;i+,i+},\Phi^{i;i+,i-}\}.
\end{equation*}
Then the $L^2$ bound
\begin{equation}\label{bvc90.5}
2^{(1+\alpha)k}2^{(1+\beta)j}\Vert P_kR_{m,\kappa}^{i;\mu,\nu}(f^\mu_{k_1,j_1},f^\nu_{k_2,j_2})\Vert_{L^2}\lesssim 2^{-\beta^4m}
\end{equation}
holds for any $\kappa\in(0,1]$, $\kappa\geq \max\big(2^{(\beta^2m-m)/2}2^{-\min(k_1,k_2,0)/2}2^{-D/2},2^{\beta^2m-m}2^{\max(j_1,j_2)}\big)$, provided that either
\begin{equation}\label{bvc91}
\max(k_1,k_2)\geq -D/10,
\end{equation}
or
\begin{equation}\label{bvc92}
\max(k_1,k_2)\leq -D/10\qquad\text{ and }\qquad k+\min(k_1,k_2)/2\leq(\beta^2m-m)/2+D,
\end{equation}
or
\begin{equation}\label{bvc93}
\max(k_1,k_2)\leq -D/10\qquad\text{ and }\qquad k+\max(k_1,k_2)+(1-\beta)j\leq 2D,
\end{equation}
or
\begin{equation}\label{bvc96}
\Phi^{\sigma;\mu,\nu}=\Phi^{i;i+,i-}\qquad\text{ and }\qquad k_2\geq k_1+9,
\end{equation}
or
\begin{equation}\label{bvc97}
\Phi^{\sigma;\mu,\nu}=\Phi^{i;i-,i+}\qquad\text{ and }\qquad k_1\geq k_2+9.
\end{equation}
\end{lemma}

\begin{proof}[Proof of Lemma \ref{CaseDProp2}] We use Lemma \ref{CaseDIBPS} first: the conclusion follows if either \eqref{bvc96} or \eqref{bvc97} are satisfied, since the lower bound \eqref{bvc0} holds in these cases. 

The lower bound \eqref{bvc0} also holds if $\max(k_1,k_2)\geq -D/10$ and $\Phi^{\sigma;\mu,\nu}=\Phi^{i;i+,i+}$. On the other hand, if $\max(k_1,k_2)\geq -D/10$ and $\Phi^{\sigma;\mu,\nu}\in\in\{\Phi^{i;i+,i-}\}$ then
\begin{equation*}
\begin{split}
|(D^\rho_\xi\Phi^{\sigma;\mu,\nu})(\xi,\eta)|&\lesssim 2^{k(1-|\rho|)},\\
|(D^\rho_\eta\Phi^{\sigma;\mu,\nu})(\xi,\eta)|=|(D^\rho\Lambda_i)(\eta-\xi)-(D^\rho\Lambda_i)(\eta)|&\lesssim 2^k,\\
|(D^{\rho_1}_\xi D^{\rho_2}_\eta\Phi^{\sigma;\mu,\nu})(\xi,\eta)|&\lesssim 1,\\
|\Phi^{\sigma;\mu,\nu}(\xi,\eta)|&\gtrsim 2^k,
\end{split}
\end{equation*} 
as long as $|\xi|\in[2^{k-6},2^{k+6}]$, $|\xi-\eta|\in[2^{k_1-6},2^{k_1+6}]$, $|\eta|\in[2^{k_2-6},2^{k_2+6}]$, and $|\rho|\in[0,4]$, $|\rho_1|,|\rho_2|\in[1,4]$. The last bound follows from \eqref{BdPhiiii}. We can use the formula \eqref{Der1/Phi} and integrate by parts as in the proof of Lemma \ref{CaseDIBPS} (ii) to conclude that \eqref{CaseDMultiplier1} holds. The desired bound \eqref{bvc90.5} follows from  Lemma \ref{CaseDIBPS} (i).

Assume now that \eqref{bvc93} holds. We then estimate, using Plancherel, \eqref{nh9.1} and \eqref{ccc33},
\begin{equation*}
\begin{split}
\Vert T_m^{i;\mu,\nu}&(f^\mu_{k_1,j_1},f^\nu_{k_2,j_2})\Vert_{L^2}\\
&\lesssim 2^m\sup_{s\in[2^{m-4},2^{m+4}]}\min(\Vert Ef^\mu_{k_1,j_1}(s)\Vert_{L^\infty}\Vert f^\nu_{k_2,j_2}(s)\Vert_{L^2},\Vert f^\mu_{k_1,j_1}(s)\Vert_{L^2}\Vert Ef^\nu_{k_2,j_2}(s)\Vert_{L^\infty})\\
&\lesssim 2^m\cdot 2^{-3/2m}2^{(2\beta-\alpha)(k_1+k_2)}\min\big(2^{k_1/2}2^{(1/2+\beta)j_1}2^{-(1-\beta)j_2},2^{k_2/2}2^{(1/2+\beta)j_2}2^{-(1-\beta)j_1}\big)\\
&\lesssim 2^{-m/2}2^{(1/2+2\beta-\alpha)(k_1+k_2)}(2^{-2\beta k_1}+2^{-2\beta k_2}).
\end{split}
\end{equation*}
Therefore, assuming for example that $k_1\geq k_2$,
\begin{equation*}
\begin{split}
2^{(1+\alpha)k}2^{(1+\beta)j}\Vert T_m^{i;\mu,\nu}(f^\mu_{k_1,j_1},f^\nu_{k_2,j_2})\Vert_{L^2}
&\lesssim 2^{(1+\alpha)k}2^{-(1+3\beta)(k+k_1)}2^{-m/2}2^{(1+2\beta-2\alpha)k_1}\\
&\lesssim 2^{-3\beta (k+k_1)}2^{-m/2}\\
&\lesssim 2^{-\beta^2m}
\end{split}
\end{equation*}
by \eqref{bvc93}. The desired bound \eqref{bvc90.5} follows using also \eqref{NRI}.

Finally, assume that \eqref{bvc92} holds. Assume first that $k\le\min(k,k_1,k_2)+10$ and $\max(j_1,j_2)\ge 3\beta m$ or $\max(k_1,k_2)\le -3\beta m/2$. In this case, we may simply use Plancherel, \eqref{nh9.1} and \eqref{ccc33} to estimate
\begin{equation*}
\begin{split}
\Vert T_m^{i;\mu,\nu}&(f^\mu_{k_1,j_1},f^\nu_{k_2,j_2})\Vert_{L^2}\\
&\lesssim 2^m\sup_{s\in[2^{m-4},2^{m+4}]}\min(\Vert Ef^\mu_{k_1,j_1}(s)\Vert_{L^\infty}\Vert f^\nu_{k_2,j_2}(s)\Vert_{L^2},\Vert f^\mu_{k_1,j_1}(s)\Vert_{L^2}\Vert Ef^\nu_{k_2,j_2}(s)\Vert_{L^\infty})\\
&\lesssim 2^m\cdot2^{-3/2m}2^{(1/2+3\beta)\max(k_1,k_2)}2^{(1/2+\beta)(\min(j_1,j_2)-\max(j_1,j_2))}2^{-(1/2-2\beta)\max(j_1,j_2)},
\end{split}
\end{equation*}
and therefore, assuming for example that $k_1\geq k_2$,
\begin{equation*}
\begin{split}
2^{(1+\alpha)k}2^{(1+\beta)j}\Vert T_m^{i;\mu,\nu}(f^\mu_{k_1,j_1},f^\nu_{k_2,j_2})\Vert_{L^2}\lesssim 2^{(1+\alpha)[k+k_1/2]}2^{(1/2+\beta)m}2^{-(1/2-2\beta)\max(j_1,j_2)}2^{(11\beta/4)k_1}\lesssim 2^{-\beta^3m}.
\end{split}
\end{equation*}
The desired bound \eqref{bvc90.5} follows using also \eqref{NRI}. 

Assume now that $\min(k_1,k_2)\le k-10$. In this case $k\ge\max(k_1,k_2)-4$ and necessarily $\max(j_1,j_2)\ge m/8$ by \eqref{bvc92}. We may assume that $k_2\le k_1$. If $j_1\le j_2-6\beta m$, then, using Plancherel, \eqref{nh9.1} and \eqref{ccc33}, we get that
\begin{equation*}
\begin{split}
\Vert T_m^{i;\mu,\nu}(f^\mu_{k_1,j_1},f^\nu_{k_2,j_2})\Vert_{L^2}
&\lesssim 2^m\sup_{s\in[2^{m-4},2^{m+4}]}\Vert Ef^\mu_{k_1,j_1}(s)\Vert_{L^\infty}\Vert f^\nu_{k_2,j_2}(s)\Vert_{L^2}\\
&\lesssim 2^m\cdot 2^{-3/2m}2^{k_1/2}2^{(1/2+\beta)j_1}2^{-(1-\beta)j_2}2^{3/2\beta(k_1+k_2)}\\
\end{split}
\end{equation*}
and therefore,
\begin{equation*}
\begin{split}
2^{(1+\alpha)k}2^{(1+\beta)j}\Vert T_m^{i;\mu,\nu}(f^\mu_{k_1,j_1},f^\nu_{k_2,j_2})\Vert_{L^2}\lesssim 2^{(1+\alpha)[k+k_2/2]}2^{(1/2+\beta)m}2^{-(1/2+\beta/4)(j_1-j_2)}2^{3/4\beta j_1}\lesssim 2^{-\beta^3m}.
\end{split}
\end{equation*}
The desired bound \eqref{bvc90.5} follows using also \eqref{NRI}.
If $j_1\ge j_2-6\beta m$ and $\max(j_1,j_2)\ge m/8$, then
\begin{equation*}
(1/2+\beta)j_2-(1-\beta)j_1\le -(1/2-2\beta)j_1+2\beta m+4\beta^2m\le -4\beta m
\end{equation*}
and using Plancherel, \eqref{nh9.1} and \eqref{ccc33}, we get that
\begin{equation*}
\begin{split}
\Vert T_m^{i;\mu,\nu}(f^\mu_{k_1,j_1},f^\nu_{k_2,j_2})\Vert_{L^2}
&\lesssim 2^m\sup_{s\in[2^{m-4},2^{m+4}]}\Vert f^\mu_{k_1,j_1}(s)\Vert_{L^2}\Vert Ef^\nu_{k_2,j_2}(s)\Vert_{L^\infty}\\
&\lesssim 2^m\cdot 2^{-3/2m}2^{k_2/2}2^{(1/2+\beta)j_2}2^{-(1-\beta)j_1}2^{3/2\beta(k_1+k_2)}\\
\end{split}
\end{equation*}
and therefore,
\begin{equation*}
\begin{split}
2^{(1+\alpha)k}2^{(1+\beta)j}\Vert T_m^{i;\mu,\nu}(f^\mu_{k_1,j_1},f^\nu_{k_2,j_2})\Vert_{L^2}\lesssim 2^{(1+\alpha)[k+k_2/2]}2^{(1/2+\beta)m}2^{[(1/2+\beta)j_2-(1-\beta)j_1)}2^{3/2\beta (k_1+k_2/2)}\lesssim 2^{-\beta^3m}.
\end{split}
\end{equation*}
The desired bound \eqref{bvc90.5} follows using also \eqref{NRI}.

Finally, assume that
\begin{equation*}
k\le (2\beta m-m)/2+D,\qquad \max(j_1,j_2)\le  3\beta m,\qquad -D/10\geq k_1\geq k_2\geq -3\beta m/2-10.
\end{equation*}
If $\Phi^{\sigma;\mu,\nu}=\Phi^{i;i+,i+}$ then the lower bound \eqref{bvc0} holds, and the desired conclusion follows from Lemma \ref{CaseDIBPS}. On the other hand, if $\Phi^{\sigma;\mu,\nu}=\Phi^{i;i+,i-}$, then we claim that \eqref{CaseDMultiplier2} holds. Indeed, using also \eqref{BdPhiiii},
\begin{equation*}
\begin{split}
\vert\Phi^{i;i+,i-}(\xi,\eta)\vert&\gtrsim 2^{k+2k_1},\\
\vert D^\rho_\xi\Phi^{i;i+,i-}(\xi,\eta)\vert&\lesssim 2^{(1-\vert\rho\vert)k},\,\,\quad\vert\rho\vert\ge1,\\
\vert D^{\rho}_\eta\Phi^{i;i+,i-}(\xi,\eta)\vert&\lesssim 2^k2^{-\vert\rho\vert k_1},\quad\vert\rho\vert\ge1\\
\vert D_\xi^{\rho_1}D^{\rho_2}_\eta\Phi^{i;i+,i-}(\xi,\eta)\vert&\lesssim 2^{(1-\vert\rho_1\vert-\vert\rho_2\vert) k_1},\quad\vert\rho_1\vert\cdot\vert\rho_2\vert\ge1.
\end{split}
\end{equation*}
We use \eqref{Der1/Phi} and proceed as in the proof of Lemma \ref{CaseDIBPS} (ii). The desired bound \eqref{CaseDMultiplier2} follows and we can apply Lemma \ref{CaseDIBPS} to complete the proof.
\end{proof}

In view of Lemma \ref{CaseDProp} and Lemma \ref{CaseDProp2}, for Proposition \ref{CaseCProp} it remains to prove the following lemma.

\begin{lemma}\label{CaseCLem}
Assume that $\Phi^{\sigma;\mu,\nu}\in\{\Phi^{i;i+,i+},\Phi^{i;i+,i-}\}$ and that
\begin{equation}\label{Bds1C1}
\begin{split}
&-9m/10\le k_1,k_2\le j/N_0^\prime,\quad\max(j_1,j_2)\le (1-\beta/10)m+k,\quad k_2\le k_1+9\le -D/10+9,\\
&\beta m/2+D^2\le j\le m+D,\quad(1-\beta)j\ge-k-k_1+D,\\
& (\beta^2m-m)+2D\le 2k+k_2\le-D/2,
\end{split}
\end{equation}
and set $\kappa:=2^{k-D}$. Then
\begin{equation}\label{CaseCROK1}
2^{(1+\alpha)k}2^{(1+\beta)j}\Vert \widetilde{\varphi}^{(k)}_j\cdot P_k R^{i;\mu,\nu}_{m,\kappa}(f^\mu_{k_1,j_1},f^\nu_{k_2,j_2})\Vert_{L^2}\lesssim 2^{-2\beta^4m}.
\end{equation}
\end{lemma}

Note that, in view of \eqref{Bds1C1}, $\kappa=2^{k-D}$ satisfies the necessary assumption
\begin{equation*}
\kappa\geq \max\big(2^{(\beta^2m-m)/2}2^{-\min(k_1,k_2,0)/2}2^{-D/2},2^{\beta^2m-m}2^{\max(j_1,j_2)}\big).
\end{equation*}

We remark that Lemma \ref{CaseCLem} covers the case of $k_2\le k_1$ for $\Phi^{i;i+,i+}$. By symmetry, this also covers the case $k_1\le k_2$.

\begin{proof}[Proof of Lemma \ref{CaseCLem}] In this lemma we need to use finer decompositions and additional orthogonality arguments. We will use repeatedly the following observation from the sine law in the triangle formed by $\xi$, $\xi-\eta$, $\eta$:
\begin{equation}\label{SineLaw}
\frac{\sin(\angle(\xi,\eta))}{\vert\xi-\eta\vert}=\frac{\sin(\angle(\xi,\xi-\eta))}{\vert\eta\vert}=\frac{\sin(\angle(\xi-\eta,\eta))}{\vert\xi\vert}.
\end{equation}

We need to introduce some angular localizations. For $\delta<1/10$ to be fixed in each case below, we choose $\{\omega_i\}_i$ a maximal family of $\delta$-separated points on the sphere which is symmetrical in the sense that $\omega\in\{\omega_i\}\Rightarrow -\omega\in\{\omega_i\}$,
and we define
\begin{equation}\label{AngLocD1}
\begin{split}
\chi_{e_1}^0(\xi)&=\varphi(2^{-k}\xi_1-1)\varphi(2^{-k}\delta^{-1}\xi_2)\varphi(2^{-k}\delta^{-1}\xi_3),\quad \chi_{\omega_i}^0(\xi)=\chi_{e_1}^0(R_{\omega_i}\xi)
\end{split}
\end{equation}
where $R_{\omega_i}$ is a rotation satisfying $R_{\omega_i}(\omega_i)=e_1$.
We similarly define
\begin{equation}\label{AngLocD2}
\begin{split}
\chi_{e_1}^1(\eta)&=\varphi(2^{-k_1-10}\eta_1-1)\varphi(2^{-k_1-10}\delta^{-1}\eta_2)\varphi(2^{-k_1-10}\delta^{-1}\eta_3),\quad \chi^1_{\omega_i}(\eta)=\chi^1_{e_1}(R_{\omega_i}\eta)\\
\chi_{e_1}^2(\eta)&=\varphi(2^{-k_2-10}\eta_1-1)\varphi(2^{-k_2-10}\delta^{-1}\eta_2)\varphi(2^{-k_2-10}\delta^{-1}\eta_3),\quad \chi^2_{\omega_i}(\eta)=\chi^2_{e_1}(R_{\omega_i}\eta).
\end{split}
\end{equation}
We then define the corresponding operators $Q^j_{\omega}$ by the formulas
\begin{equation*}
\widehat{Q^j_\omega f}(\xi)=\chi^j_{\omega}(\xi)\widehat{f}(\xi).
\end{equation*}
Since $\{\omega_i\}$ is maximal, there holds that, for any $f\in L^2(\mathbb{R}^3)$,
\begin{equation}\label{QOrtho}
\Vert f\Vert_{L^2}^2\lesssim \sum_{\omega_i}\Vert Q^j_{\omega_i}f\Vert_{L^2}^2\lesssim \Vert f\Vert_{L^2}^2,\quad 0\le j\le 2,
\end{equation}
uniformly in $k$, $k_1$, $k_2$, $\delta$. In addition, $Q^j_\omega$ satisfies
\begin{equation}\label{CaseCQCZ}
\Vert Q^j_\omega f\Vert_{L^p}\lesssim \Vert f\Vert_{L^p},\quad 0\le j\le 2,\,\,1\le p\le\infty.
\end{equation}

{\bf Case 1}: assume that
\begin{equation*}
j\ge m+k_2+D.
\end{equation*}
In this case, we only need a coarse localization which separates the poles and we use the finite speed of propagation to treat parallel interactions.

We set $\delta=2^{-D/10}$ in the definition \eqref{AngLocD1}-\eqref{AngLocD2} and decompose
\begin{equation}\label{bvc120}
R^{i;\mu,\nu}_{m,\kappa}(f^\mu_{k_1,j_1},f^\nu_{k_2,j_2})=\sum_{\omega_{0},\omega_{1},\omega_2}Q^0_{\omega_0}R^{i;\mu,\nu}_{m,\kappa}(Q^1_{\omega_1}f^\mu_{k_1,j_1},Q^2_{\omega_2}f^\nu_{k_2,j_2}).
\end{equation}
Since this sum is finite, we need only show that for any choice of $\omega_0,\omega_1,\omega_2$, there holds that
\begin{equation}\label{EndCase1CaseD}
2^{(1+\alpha)k}2^{(1+\beta)j}\Vert \varphi^{(k)}_j\cdot P_kQ^0_{\omega_0}R^{i;\mu,\nu}_{m,\kappa}(Q^1_{\omega_1}f^\mu_{k_1,j_1},Q^2_{\omega_2}f^\nu_{k_2,j_2})\Vert_{L^2}\lesssim 2^{-\beta^3m}
\end{equation}
Indeed, we need only consider two subcases.

$(i)$ If $\angle(\omega_0,\omega_1)>2^{-D/20}$, then, in case $\Phi^{i;\mu,\nu}=\Phi^{i;i+,i+}$, we find that
\begin{equation}\label{EstimXiCaseC1.1}
\vert\Xi^{i+,i+}(\xi,\eta)\vert\gtrsim \angle(\omega_1,\omega_2)
\end{equation}
and therefore \eqref{EndCase1CaseD} follows trivially. In case $\Phi^{i;\mu,\nu}=\Phi^{i;i+,i-}$, we claim that
\begin{equation}\label{Mult1}
\begin{split}
&\Big\|
\mathcal{F}\Big\{\frac{\chi^0_{\omega_0}(\xi)\chi^1_{\omega_1}(\xi-\eta)\chi^2_{\omega_2}(\eta)}{\Phi^{i;i+,i-}(\xi,\eta)}\varphi_{[k-4,k+4]}(\xi)\varphi_{[k_1-4,k_1+4]}(\xi-\eta)\varphi_{[k_2-4,k_2+4]}(\eta)\Big\}\Big\|_{L^1(\mathbb{R}^3\times\mathbb{R}^3)}\lesssim 2^{-k}.
\end{split}
\end{equation}
Assuming \eqref{Mult1}, the estimate \eqref{EndCase1CaseD} follows from Lemma \ref{AddedLemmaCasD1.0001}. To prove \eqref{Mult1}, we first observe that, using \eqref{BdPhiiii} and Lemma \ref{tech99},
\begin{equation*}
\begin{split}
\Phi^{i;i+,i-}(\xi,\eta)&=\lambda_i(\vert\xi\vert)+\lambda_i(\vert\eta\vert)-\lambda_i(\vert\xi\vert+\vert\eta\vert)+\left[\lambda_i(\vert\xi\vert+\vert\eta\vert)-\lambda_i(\vert\xi-\eta\vert)\right]\\
&\ge \lambda_i(\vert\xi\vert+\vert\eta\vert)-\lambda_i(\vert\xi-\eta\vert)\\
&\gtrsim \vert\xi\vert+\vert\eta\vert-\vert\xi-\eta\vert\\
&\gtrsim \vert\xi\vert\left[1-\cos(\angle\xi,\xi-\eta)\right]+\vert\eta\vert\left[1-\cos(\angle\eta-\xi,\eta)\right]\\
&\gtrsim\vert\xi\vert
\end{split}
\end{equation*}
and the inequalities
\begin{equation*}
\begin{split}
\vert \Phi^{i;i+,i-}(\xi,\eta)\vert&\ge\lambda_i(\vert\xi\vert)\gtrsim 2^k,\\
\vert D^{\rho}_{\xi}\Phi^{i;i+,i-}(\xi,\eta)\vert&\lesssim 2^{(1-\vert\rho\vert)k},\quad\vert\rho\vert\ge 1,\\
\vert D^\rho_\eta\Phi^{i;i+,i-}(\xi,\eta)\vert&\lesssim 2^k2^{-\vert\rho\vert k_2},\quad\vert\rho\vert\ge 1,\\
\vert D_\xi^{\rho_1}D^{\rho_2}_\eta\Phi^{i;i+,i-}(\xi,\eta)\vert&\lesssim 2^{(1-\vert\rho_1\vert-|\rho_2|) k_1},\quad\vert\rho_1\vert\cdot\vert\rho_2\vert\ge 1,
\end{split}
\end{equation*}
hold in the support of the multiplier in \eqref{Mult1}. Using the formula \eqref{Der1/Phi} and integrating by parts we derive the inequality \eqref{Mult1}.

$(ii)$ if $\angle(\omega_0,\omega_1)\le2^{-D/20}$, we claim that, on the support of
\begin{equation*}
\begin{split}
&\chi^{i;\mu,\nu}_C(\xi,\eta):=\\
&\varphi(2^{D^2}\Phi^{i;\mu,\nu}(\xi,\eta))\varphi(2^{D-k}\vert\Xi^{\mu,\nu}(\xi,\eta)\vert)\varphi_k(\xi)\varphi_{[k_1-2,k_1+2]}(\xi-\eta)\varphi_{[k_2-2,k_2+2]}(\eta)\chi^0_{\omega_0}(\xi)\chi^1_{\omega_1}(\xi-\eta)\chi^2_{\omega_2}(\eta)
\end{split}
\end{equation*}
we have that
\begin{equation}\label{CaseCFSP}
\vert\nabla_\xi\Phi^{i;\mu,\nu}(\xi,\eta)\vert\lesssim 2^{k_2}.
\end{equation}
Indeed, on the support of $\chi^{i;\mu,\nu}_C$, we have that $\angle(\xi,\xi-\eta)\lesssim\sin(\angle(\xi,\xi-\eta))$ and therefore, using also \eqref{SineLaw} and the smallness of $\vert\Xi^{\mu,\nu}(\xi,\eta)\vert$, if $\chi^{i;\mu,\nu}_C(\xi,\eta)\ne0$, then
\begin{equation*}
\begin{split}
\vert\nabla_\xi\Phi^{i;\mu,\nu}(\xi,\eta)\vert
&\lesssim \vert\lambda_i^\prime(\vert\xi\vert)-\lambda_i^\prime(\vert\xi-\eta\vert)\vert+\angle(\xi,\xi-\eta)\\
&\lesssim 2^{k_1+k_2}+2^{k_2-k}\sin(\angle(\xi-\eta,\eta))\\
&\lesssim 2^{k_2}+2^{k_2-k}\vert\Xi^{\mu,\nu}(\xi,\eta)\vert\lesssim 2^{k_2}.
\end{split}
\end{equation*}

From \eqref{CaseCFSP} and the assumption $j\ge m+k_2+D$, it follows that if $\vert x\vert\ge 2^{j-4}$, $s\le 2^{m+4}$, then
\begin{equation*}
\vert\nabla_\xi\left[x\cdot\xi+s\Phi^{i;\mu,\nu}(\xi,\eta)\right]\vert\ge\vert x\vert-2^{m+k_2+D/2}\ge 2^{j-8}.
\end{equation*}
We may then use Lemma \ref{tech5} with $f=2^{-j}\left[x\cdot\xi+s\Phi^{i;\mu,\nu}(\xi,\eta)\right]$, $K=2^j$ and $\epsilon=\min(2^{-\min(j_1,j_2))},2^{2k})$
to conclude that
\begin{equation*}
\vert \widetilde{\varphi}^{(k)}_j(x)P_kQ^0_{\omega_0}R^{i;\mu,\nu}_{m,\kappa}(Q^1_{\omega_1}f^\mu_{k_1,j_1},Q^2_{\omega_2}f^\nu_{k_2,j_2})(x)\vert\lesssim 2^{-100m}
\end{equation*}
from which \eqref{EndCase1CaseD} follows easily.

\medskip

{\bf Case 2:} assume that
\begin{equation}\label{FSPCaseC}
j\le m+k_2+D.
\end{equation}
In this case, we need a finer decomposition in order to obtain an efficient bilinear estimate in \eqref{CaseCMultEst}.
We define
\begin{equation}\label{DefChiCCase2}
\chi^{\mu,\nu}_C(\xi,\eta):=\varphi(2^{D^2}\Phi^{i;\mu,\nu}(\xi,\eta))\varphi(2^{D-k}\vert\Xi^{\mu,\nu}(\xi,\eta)\vert)\varphi_k(\xi)\varphi_{[k_1-2,k_1+2]}(\xi-\eta)\varphi_{[k_2-2,k_2+2]}(\eta)
\end{equation}
and fix $\delta=2^{\max(k,k_2)}$ in the definition \eqref{AngLocD1}-\eqref{AngLocD2}. By \eqref{Bds1C1}, we see that $\delta\sim 2^{k_1}$. An important property is that
\begin{equation}\label{CaseCAngleLoc}
\begin{split}
\chi^0_{\omega}(\xi)\chi^{i+,i+}_C(\xi,\eta)=&\chi^0_{\omega}(\xi)\chi^{i+,i+}_C(\xi,\eta)\chi^1_{\omega}(\xi-\eta)\chi^2_{\omega}(\eta),\\
\chi^0_{\omega}(\xi)\chi^{i+,i-}_C(\xi,\eta)=&\chi^0_{\omega}(\xi)\chi^{i+,i-}_C(\xi,\eta)\chi^1_{-\omega}(\xi-\eta)\chi^2_{\omega}(\eta)+\chi^0_{\omega}(\xi)\chi^{i+,i-}_C(\xi,\eta)\chi^1_{\omega}(\xi-\eta)\chi^2_{-\omega}(\eta).
\end{split}
\end{equation}
Indeed, for $\chi_C^{i+,i+}$, this follows from \eqref{EstimXiCaseC1.1}. For $\chi_C^{i+,i-}$, we similarly see that
\begin{equation*}
\vert\Xi^{i+,i-}(\xi,\eta)\vert\gtrsim\angle(\eta-\xi,\eta)
\end{equation*}
and, using \eqref{SineLaw},
\begin{equation*}
\begin{split}
\sin(\angle(\xi,\eta))&=\frac{\vert\xi-\eta\vert}{\vert\xi\vert}\sin(\angle(\xi-\eta,\eta))\lesssim 2^{k_1},\quad\sin(\angle(\xi,\xi-\eta))=\frac{\vert\eta\vert}{\vert\xi\vert}\sin(\angle(\xi-\eta,\eta))\lesssim 2^{k_1}.\\
\end{split}
\end{equation*}
Now, we perform an integration by parts in $s$ and obtain that
\begin{equation*}
\begin{split}
Q_\omega R^{i;i+,i-}_{m,\kappa}(f^\mu_{k_1,j_1},f^\nu_{k_2,j_2})&=i\int_{\mathbb{R}}\left[q_m^\prime(s)\Pi_{1,\omega}(s)+q_m(s)\Pi_{2,\omega}(s)+q_m(s)\Pi_{3,\omega}(s)\right] ds;\\
\widehat{\Pi_{1,\omega}}(\xi,s)&:=\int_{\mathbb{R}^3}e^{is\Phi^{i;\mu,\nu}(\xi,\eta)}\frac{\chi_\omega(\xi)\chi_{C}^{\mu,\nu}(\xi,\eta)}{\Phi^{i;\mu,\nu}(\xi,\eta)}\widehat{Q^1_\omega f^\mu_{k_1,j_1}}(\xi-\eta,s)\widehat{Q^2_\omega f^\nu_{k_2,j_2}}(\eta,s)d\eta ds,\\
\widehat{\Pi_{2,\omega}}(\xi,s)&:=\int_{\mathbb{R}^3}e^{is\Phi^{i;\mu,\nu}(\xi,\eta)}\frac{\chi_\omega(\xi)\chi_{C}^{\mu,\nu}(\xi,\eta)}{\Phi^{i;\mu,\nu}(\xi,\eta)}\widehat{Q^1_\omega (\partial_sf^\mu_{k_1,j_1})}(\xi-\eta,s)\widehat{Q^2_{\omega}f^\nu_{k_2,j_2}}(\eta,s)d\eta ds,\\
\widehat{\Pi_{3,\omega}}(\xi,s)&:=\int_{\mathbb{R}^3}e^{is\Phi^{i;\mu,\nu}(\xi,\eta)}\frac{\chi_\omega(\xi)\chi_{C}^{\mu,\nu}(\xi,\eta)}{\Phi^{i;\mu,\nu}(\xi,\eta)}\widehat{Q^1_\omega f^\mu_{k_1,j_1}}(\xi-\eta,s)\widehat{Q^2_\omega (\partial_sf^\nu_{k_2,j_2})}(\eta,s)d\eta ds.\\
\end{split}
\end{equation*}
Using \eqref{tech2}, \eqref{ccc33} and \eqref{CaseCMultEst}, we easily see that
\begin{equation*}
\begin{split}
\Vert \Pi_{1,\omega}(s)\Vert_{L^2}&\lesssim 2^{-k-k_1-k_2}\min(\Vert EQ^1_\omega f^\mu_{k_1,j_1}(s)\Vert_{L^\infty}\Vert Q^2_\omega f^\nu_{k_2,j_2}(s)\Vert_{L^2},\Vert Q^1_\omega f^\mu_{k_1,j_1}(s)\Vert_{L^2}\Vert EQ^2_\omega f^\nu_{k_2,j_2}(s)\Vert_{L^\infty})\\
&\lesssim 2^{-k-k_1-k_2}2^{-3/2m}2^{3\min(j_1,j_2)/2}2^{k_1/2}\Vert Q^1_\omega f^\mu_{k_1,j_1}(s)\Vert_{L^2}\Vert Q^2_\omega f^\nu_{k_2,j_2}(s)\Vert_{L^2}
\end{split}
\end{equation*}
and therefore, using orthogonality property \eqref{QOrtho},
\begin{equation*}
\begin{split}
&\Vert \int_{\mathbb{R}}q_m^\prime(s)\cdot\sum_\omega\Pi_{1,\omega}(s)ds\Vert_{L^2}\\
&\lesssim \int_{\mathbb{R}}q_m^\prime(s)\Vert \sum_\omega\Pi_{1,\omega}(s)\Vert_{L^2}ds\\
&\lesssim 2^{-k-k_2}2^{-k_1/2}2^{-3/2m}2^{3\min(j_1,j_2)/2}\int_{\mathbb{R}}q_m^\prime(s)\cdot\left(\sum_\omega\Vert Q^1_\omega f^\mu_{k_1,j_1}(s)\Vert_{L^2}\Vert Q^2_\omega f^\nu_{k_2,j_2}(s)\Vert_{L^2}\right)ds\\
&\lesssim 2^{-k-k_2}2^{-k_1/2}2^{-3/2m}2^{3\min(j_1,j_2)/2}\sup_{s\in[2^{m-4},2^{m+4}]}\Vert f^\mu_{k_1,j_1}(s)\Vert_{L^2}\Vert f^\nu_{k_2,j_2}(s)\Vert_{L^2}
\end{split}
\end{equation*}
and finally, with \eqref{nh9.1}, \eqref{Bds1C1} and \eqref{FSPCaseC},
\begin{equation}\label{CaseCPiOkPi1}
\begin{split}
&2^{(1+\alpha)k}2^{(1+\beta)j}\Vert \int_{\mathbb{R}}q_m^\prime(s)\cdot\sum_\omega\Pi_{1,\omega}(s)ds\Vert_{L^2}\\
&\lesssim 2^{\alpha (k+k_1+k_2)}2^{k+k_2}\cdot 2^{-k-k_2}2^{-k_1/2}2^{-(1/2-\beta)m}2^{3\min(j_1,j_2)/2}2^{-(1-\beta)(j_1+j_2)}2^{2\beta(k_1+k_2)}\\
&\lesssim 2^{\alpha (k+k_1+k_2)}2^{-k_1/2}2^{-(1/2-\beta)m}2^{-(1/2-2\beta)\max(j_1,j_2)}2^{2\beta(k_1+k_2)}\\
&\lesssim 2^{-\beta^4m}.
\end{split}
\end{equation}

\medskip

We now turn to $\Pi_{2,\omega}$ and $\Pi_{3,\omega}$. Using \eqref{mk15.65} and \eqref{ok50}, we see that for any $s\in[2^{m-4},2^{m+4}]$ and $j_1$, $j_2$ satisfying \eqref{Bds1ND2},
\begin{equation}\label{BdPi2}
\begin{split}
\Vert Ef^\mu_{k_1,j_1}(s)\Vert_{L^\infty}+\Vert (\partial_sf^\mu_{k_1,j_1})\Vert_{L^2}&\lesssim 2^{(1-5\beta)k_1}2^{-(1+\beta)m},\\
\Vert Ef^\nu_{k_2,j_2}(s)\Vert_{L^\infty}+\Vert (\partial_sf^\nu_{k_2,j_2})\Vert_{L^2}&\lesssim 2^{(1-5\beta)k_2}2^{-(1+\beta)m}.
\end{split}
\end{equation}
Therefore, using \eqref{CaseCQCZ}, \eqref{CaseCAngleLoc}, \eqref{CaseCMultEst} and Lemma \ref{tech2}, we see that
\begin{equation*}
\begin{split}
\Vert \sum_{\omega} \Pi_{2,\omega}(s)\Vert_{L^2}\lesssim \left(\sum_{\omega}\Vert \Pi_{2,\omega}(s)\Vert_{L^2}^2\right)^\frac{1}{2}
&\lesssim 2^{-k-k_1-k_2}\left(\sum_{\omega}\Vert E f^\mu_{k_1,j_1}(s)\Vert_{L^\infty}^2\Vert Q^2_{\omega}(\partial_sf^\nu_{k_2,j_2})(s)\Vert_{L^2}^2\right)^\frac{1}{2}\\
&\lesssim 2^{-k-k_1-k_2}\Vert E f^\mu_{k_1,j_1}(s)\Vert_{L^\infty}\Vert (\partial_sf^\nu_{k_2,j_2})(s)\Vert_{L^2}\\
\end{split}
\end{equation*}
and therefore, using \eqref{BdPi2},
\begin{equation*}
\int_{\mathbb{R}}q_m(s)\Vert \sum_{\omega} \Pi_{2,\omega}(s)\Vert_{L^2}ds
\lesssim 2^{-k}2^{-5\beta(k_1+k_2)}2^{-(1+2\beta)m},
\end{equation*}
so that, using \eqref{FSPCaseC},
\begin{equation*}
2^{(1+\alpha)k}2^{(1+\beta)j}\Vert \int_{\mathbb{R}}q_m(s)\sum_{\omega} \Pi_{2,\omega}(s)ds\Vert_{L^2}\lesssim 2^{\alpha k}2^{(1-10\beta)k_2}2^{-\beta m}
\end{equation*}
which is sufficient. A similar bound holds for $\Vert \int_{\mathbb{R}}q_m(s)\sum_{\omega} \Pi_{3,\omega}(s)ds\Vert_{L^2}$. This gives \eqref{CaseCROK1} and finishes the proof.
\end{proof}

\begin{lemma}
With $\chi_C^{\mu,\nu}$ defined as in \eqref{DefChiCCase2}, there holds that
\begin{equation}\label{CaseCMultEst}
\begin{split}
\sup_{\omega}\Vert\mathcal{F}\left\{
\frac{1}{\Phi^{i;\mu,\nu}(\xi,\eta)}\chi_{\omega}(\xi)
\chi_{C}^{\mu,\nu}(\xi,\eta)\right\}\Vert_{L^1(\mathbb{R}^3\times\mathbb{R}^3)}
&\lesssim 2^{-(k+k_1+k_2)}.
\end{split}
\end{equation}

\end{lemma}

\begin{proof}
We may assume that $\omega=e_1$. We begin with the case $\Phi^{i;\mu,\nu}=\Phi^{i;i+,i+}$ and we write
\begin{equation*}
K(x,y):=\int_{\mathbb{R}^3\times\mathbb{R}^3}e^{-i(x\cdot\xi+y\cdot\eta)}
\frac{1}{\Phi^{i;\mu,\nu}(\xi,\eta)}\chi_{\omega}(\xi)
\chi_{C}^{\mu,\nu}(\xi,\eta)d\xi d\eta.
\end{equation*}
The inequality \eqref{CaseCMultEst} follows from \eqref{BdPhiiii} and the following bounds for $(\xi,\eta)$ in the support of $\chi^0_{e_1}\chi_C^{\mu,\nu}$:
\begin{equation}\label{CaseCME1}
\begin{split}
i)\quad\vert \Phi^{i;\mu,\nu}(\xi,\eta)\vert\cdot\vert\partial_{\xi_1}^\rho\partial_{\xi_i}^{\sigma}\partial_{\eta_1}^{\tau}\partial_{\eta_j}^{\upsilon}\frac{1}{\Phi^{i;\mu,\nu}(\xi,\eta)}\vert&\lesssim 2^{-\vert\rho\vert k}2^{-\vert\sigma\vert(k+k_1)}2^{-\vert\tau\vert k_2}2^{-\vert\upsilon\vert(k_1+k_2)},\quad 2\le i,j\le 3,\\
ii)\quad\vert\partial_{\xi_1}^\rho\partial_{\xi_i}^{\sigma}\partial_{\eta_1}^{\tau}\partial_{\eta_j}^{\upsilon}\varphi(2^{-k}\vert\Xi^{\mu,\nu}(\xi,\eta)\vert)\vert&\lesssim 2^{-\vert\rho\vert k}2^{-\vert\sigma\vert(k+k_1)}2^{-\vert\tau\vert k_2}2^{-\vert\upsilon\vert(k_1+k_2)},\quad 2\le i,j\le 3,\\
\end{split}
\end{equation}
Indeed, assuming \eqref{CaseCME1}, we deduce that, whenever $\vert x_1\vert\sim A$, $\vert(x_2,x_3)\vert\sim B$, $\vert y_1\vert\sim C$ and $\vert (y_2,y_3)\vert\sim D$, there holds that
\begin{equation*}
\begin{split}
A^pB^{2q}C^rD^{2s}\vert K(x,y)\vert
&\lesssim\left\vert \int_{\mathbb{R}^3\times\mathbb{R}^3}e^{-i(x\cdot\xi+y\cdot\eta)}
\partial_{\xi_1}^p\Delta_{x_2,x_3}^q\partial_{\eta_1}^r\Delta_{y_2,y_3}^s
\left\{\frac{1}{\Phi^{i;\mu,\nu}(\xi,\eta)}\chi_{\omega}(\xi)
\chi_{C}^{\mu,\nu}(\xi,\eta)\right
\}d\xi d\eta\right\vert\\
&\lesssim 2^{-(k+k_1+k_2)}2^{k}2^{-pk}\cdot 2^{2(k+k_1)}2^{-2q(k+k_1)}\cdot 2^{k_2}2^{-rk_2}\cdot 2^{2(k_1+k_2)}2^{-2s(k_1+k_2)}
\end{split}
\end{equation*}
for $p,q,r,s\in\{0,2\}$ and therefore,
\begin{equation*}
\begin{split}
\Vert K\Vert_{L^1(\mathbb{R}^3\times\mathbb{R}^3)}&\lesssim \sum_{A,B,C,D}\int_{\vert x_1\vert\sim A,\,\vert (x_2,x_3)\vert\sim B,\,\vert y_1\vert\sim C,\,\vert(y_2,y_3)\vert\sim D}\vert K(x,y)\vert dx dy\\
&\lesssim 2^{-(k+k_1+k_2)}\sum_{A,B,C,D}\left[A2^k\right]^{1-p}\left[B2^{k+k_1}\right]^{2(1-q)}\left[C2^{k_2}\right]^{1-r}\left[D2^{k_1+k_2}\right]^{2(1-s)}
\end{split}
\end{equation*}
where the summation is over $A$, $B$, $C$, $D$ dyadic numbers. Since the sum converges, this gives \eqref{CaseCMultEst}.

\medskip

It remains to prove \eqref{CaseCME1}. We start with $i)$. Using \eqref{Der1/Phi}, it suffices to obtain that
\begin{equation}\label{CaseCBdPhi}
\vert\partial_{\xi_1}^\rho\partial_{\xi_i}^{\sigma}\partial_{\eta_1}^{\tau}\partial_{\eta_j}^{\upsilon}\Phi^{i;\mu,\nu}(\xi,\eta)\vert\lesssim 2^{k+k_1+k_2}2^{-\vert\rho\vert k}2^{-\vert\sigma\vert(k+k_1)}2^{-\vert\tau\vert k_2}2^{-\vert\upsilon\vert(k_1+k_2)},\quad 2\le i,j\le 3,\\
\end{equation}
Assume first that $\min(\rho+\vert\sigma\vert,\tau+\vert\upsilon\vert)\ge1$, then
\begin{equation*}
\vert\partial_{\xi_1}^\rho\partial_{\xi_i}^{\sigma}\partial_{\eta_1}^{\tau}\partial_{\eta_j}^{\upsilon}\Phi^{i;\mu,\nu}(\xi,\eta)\vert=\vert\partial_{\xi_1}^\rho\partial_{\xi_i}^{\sigma}\partial_{\eta_1}^{\tau}\partial_{\eta_j}^{\upsilon}\left[\Lambda_i(\xi)-\Lambda_i(\xi-\eta)-\iota_2\Lambda_i(\eta)\right]\vert=\vert\partial_{\xi_1}^{\rho+\tau}\partial_{\xi_i}^{\sigma+\upsilon}\Lambda_i(\xi-\eta)\vert.
\end{equation*}
If $\vert\sigma\vert+\vert\upsilon\vert\ge 2$ we obtain the claim directly since $\partial^{\theta}_\xi\Lambda_i(\vert\xi-\eta\vert)\lesssim 2^{(1-\vert\theta\vert)k_1}$. If $\rho+\tau=1=\vert\sigma\vert+\vert\upsilon\vert$, then
\begin{equation*}
\vert\partial_{\xi_1}\partial_{\xi_i}\Lambda_i(\xi-\eta)\vert\le \vert\lambda_i^{(2)}(\vert\xi-\eta\vert)\frac{\xi_1-\eta_1}{\vert\xi-\eta\vert}\frac{\xi_i-\eta_i}{\vert\xi-\eta\vert}\vert+\lambda_i^\prime(\vert\xi-\eta\vert)\frac{\xi_1-\eta_1}{\vert\xi-\eta\vert}\frac{\xi_i-\eta_i}{\vert\xi-\eta\vert^2}\vert\lesssim 1
\end{equation*}
and the bound follows. If $\rho+\tau\ge 2$, we compute that
\begin{equation*}
\partial_{\xi_1}^2[\Lambda_i(\xi-\eta)]=\lambda_i^{\prime\prime}(\vert\xi-\eta\vert)\frac{(\xi_1-\eta_1)^2}{\vert\xi-\eta\vert^2}+\frac{\lambda^\prime_i(\vert\xi-\eta\vert)}{\vert\xi-\eta\vert}\frac{\vert\xi_\perp-\eta_\perp\vert^2}{\vert\xi-\eta\vert^2},\quad\xi=\xi_1e_2+\xi_\perp,\,\,\eta=\eta_1e_1+\eta_\perp.
\end{equation*}
This can be bounded by $2^{k_1}$ and since any further derivative amounts to multiply by another factor of $2^{-k_1}$, we obtain the claim in this case as well.

We now claim that, on the support of $\chi^0_{e_1}\chi_C^{i+,i+}$,
\begin{equation}\label{ControlOfPhixi}
\vert\partial^\rho_{\xi_1}\partial^{\sigma}_{\xi_i}\Phi^{i;i+,i+}(\xi,\eta)\vert\lesssim 2^{k+k_1+k_2}2^{-\rho k}2^{-\vert\sigma\vert(k+k_1)}.
\end{equation}
Note that, on the support of $\chi_C^{i+,i+}$, there holds that $\vert\eta\vert\lesssim\min(\vert\xi\vert,\vert\xi-\eta\vert)$.
Applying the simple formula
\begin{equation*}
\begin{split}
\partial_{x_i}^\rho f(g(x))=\sum_{\substack{\vert\rho_1\vert,\dots,\vert\rho_k\vert\ge 1\\\rho_1+\dots+\rho_k=\rho}}c_{\rho_1,\dots,\rho_k}f^{(k)}(g)\cdot\partial_{x_i}^{\rho_1}g\dots\partial_{x_i}^{\rho_k}g,
\end{split}
\end{equation*}
for appropriate constants $c_{\rho_1,\dots,\rho_k}$, which follows by induction, with functions $f(r)=\lambda_i(r)$ and $g(\xi)=\vert\xi\vert$, we see that \eqref{ControlOfPhixi} will follow from the following bounds:
\begin{equation*}
\begin{split}
\vert \lambda_i^{(\rho)}(\vert\xi\vert)-\lambda_i^{(\rho)}(\vert\xi-\eta\vert)\vert\lesssim \vert\eta\vert 2^{2k_1}\cdot\vert\xi\vert^{-\rho}&\lesssim 2^{k_2}2^{(2-\rho)k_1},\quad\quad 1\le \rho\le 6\\
\vert \partial_{\xi_1}^{\rho}\partial_{\xi_i}^{\sigma}\vert\xi\vert-\partial_{\xi_1}^{\rho}\partial_{\xi_i}^{\sigma}\vert\xi-\eta\vert\vert\lesssim \vert\eta\vert 2^{(2-\vert\sigma\vert)k_1} \cdot\vert\xi\vert^{-\rho-\vert\sigma\vert}&\lesssim 2^{k_2}2^{(2-\rho-2\vert\sigma\vert)k_1},\quad 0\le \rho\le 2, 0\le \vert\sigma\vert\le 4,\,\rho+\vert\sigma\vert\ge 1\\
\vert \lambda_i^{(\vert\sigma\vert)}(\vert\xi\vert)\vert+\vert\lambda_i^{(\vert\sigma\vert)}(\vert\xi-\eta\vert)\vert+\vert \partial^\sigma_{\xi}\vert\xi\vert\vert+\vert\partial_{\xi}^{\sigma}\vert\xi-\eta\vert\vert\lesssim \vert\xi\vert^{1-\vert\sigma\vert}&\lesssim 2^{(1-\vert\sigma\vert)k_1},\quad\quad 1\le \vert\sigma\vert\le 6.\\
\end{split}
\end{equation*}
The first line follows directly from the Taylor expansion of $\lambda_i$. The second and third lines are straightforward unless $\vert\sigma\vert\le 1$. When $\sigma=0$, $1\le \rho\le 2$ or $\sigma=1$, $\rho=0$, the bound
\begin{equation*}
\vert\partial_{\xi_1}^\rho\left[\vert\xi\vert-\vert\xi-\eta\vert\right]\vert\lesssim \vert\eta\vert 2^{2k_1}\cdot\vert\xi\vert^{-\rho},\quad \vert\partial_{\xi_j}\left[\vert\xi\vert-\vert\xi-\eta\vert\right]\vert\lesssim \vert\eta\vert 2^{k_1}\cdot\vert\xi\vert^{-1}
\end{equation*}
follows from the formulas\footnote{Here for a choice of $1\le i\le 3$, we write $\xi=\xi_i+\xi_\perp$, $e_i\cdot\xi_\perp=0$.}
\begin{equation*}
\begin{split}
\partial^\rho_i\vert\xi\vert=
\begin{cases}
\vert\xi\vert^{-3}\xi_\perp^2&,\rho=2\\
\vert\xi\vert^{-5}\xi_\perp^2\xi_i&,\rho=3.
\end{cases}
\end{split}
\end{equation*}
The other cases follow from similar computations.

\medskip

We now claim that
\begin{equation}\label{ControlOfPhieta}
\vert\partial^\rho_{\eta_1}\partial^{\sigma}_{\eta_i}\Phi^{i;i+,i+}(\xi,\eta)\vert\lesssim 2^{k+k_1+k_2}2^{-\rho k_2}2^{-\vert\sigma\vert(k_1+k_2)}.
\end{equation}
Again, if $\vert\sigma\vert\ge 2$, this follows from the bound
\begin{equation*}
\vert\partial^\theta_\eta\Lambda_i(\xi-\eta)\vert+\vert\partial^\theta_\eta\Lambda_i(\eta)\vert\lesssim \vert\eta\vert^{1-\vert\theta\vert}.
\end{equation*}
If $\sigma=0$, we may simply compute that
\begin{equation*}
\vert\partial_{\eta_1}\Phi^{i;i+,i+}(\xi,\eta)\vert=\vert e_1\cdot\Xi^{i+,i+}(\xi,\eta)\vert\lesssim \vert\lambda_i^\prime(\vert \xi-\eta\vert)-\lambda_i^\prime(\vert\eta\vert)\vert+[\angle(\xi-\eta,\eta)]^2\lesssim 2^{k_1+k}
\end{equation*}
while the other bounds follow solely from the estimate
\begin{equation*}
\begin{split}
\vert\lambda_i^{(\rho)}(\vert\xi-\eta\vert)\vert+\vert\lambda_i^{(\rho)}(\vert\eta\vert)\vert+\vert\partial_{\eta_1}\vert\eta\vert\vert+\vert\partial_{\eta_1}\vert\xi-\eta\vert\vert&\lesssim 2^{2k_1}\vert\eta\vert^{(1-\rho)k_2},\quad\rho\ge2.
\end{split}
\end{equation*}
If $\sigma=1$, then
\begin{equation*}
\vert\partial_{\xi_j}\Phi^{i;i+,i+}(\xi,\eta)\vert\le \vert\Xi^{i+,i+}(\xi,\eta)\vert\lesssim 2^k
\end{equation*}
and if $\rho\ge1$, the claim follows from the bounds above and
\begin{equation*}
\vert\partial_{\xi_1}^\rho\partial_{\xi_j}\vert\eta\vert\vert+\vert\partial_{\xi_1}^\rho\partial_{\xi_j}\vert\xi-\eta\vert\vert\lesssim 2^{k_1}\vert\eta\vert^{-\rho k_2},\quad\rho\ge1.
\end{equation*}

\medskip

We now prove \eqref{CaseCME1} $ii)$. Since we can easily see that
\begin{equation*}
\begin{split}
&\partial_{\xi_1}^\rho\partial_{\xi_i}^{\sigma}\partial_{\eta_1}^{\tau}\partial_{\eta_j}^{\upsilon}\varphi(2^{-k}\vert\Xi^{\mu,\nu}(\xi,\eta)\vert)\\
&=\sum_{\vert \theta_1\vert,\dots\vert\theta_d\vert\ge1,\,\, \theta_1+\dots+\theta_d=(\rho,\sigma,\tau,\upsilon)}c_{\theta_1,\dots,\theta_d}2^{-dk}\varphi^{(d)}(2^{-k}\vert\Xi^{\mu,\nu}(\xi,\eta)\vert)\partial^{\theta_1}\Xi^{\mu,\nu}(\xi,\eta)\dots\partial^{\theta_d}\Xi^{\mu,\nu}(\xi,\eta)
\end{split}
\end{equation*}
for some appropriate coefficients $c_{\theta_1,\dots,\theta_d}$. Therefore, we see that it suffices to show that
\begin{equation*}
\vert\partial_{\xi_1}^\rho\partial_{\xi_i}^{\sigma}\partial_{\eta_1}^{\tau}\partial_{\eta_j}^{\upsilon}\Xi^{\mu,\nu}(\xi,\eta)\vert\lesssim 2^k2^{-\vert\rho\vert k}2^{-\vert\sigma\vert(k+k_1)}2^{-\vert\tau\vert k_2}2^{-\vert\upsilon\vert(k_1+k_2)}
\end{equation*}
but this follows from \eqref{CaseCBdPhi}.

\medskip

In case $\Phi^{i;\mu,\nu}=\Phi^{i;i+,i-}$, we decompose
\begin{equation*}
\begin{split}
\chi^0_{\omega}(\xi)\chi^{i+,i-}_C(\xi,\eta)&=\chi^0_{\omega}(\xi)\chi^{i+,i-}_C(\xi,\eta)\chi^1_{-\omega}(\xi-\eta)\chi^2_{\omega}(\eta)+\chi^0_{\omega}(\xi)\chi^{i+,i-}_C(\xi,\eta)\chi^1_{\omega}(\xi-\eta)\chi^2_{-\omega}(\eta)\\
&=\chi^{i+,i-;1}_{C,\omega}(\xi,\eta)+\chi^{i+,i-;2}_{C,\omega}(\xi,\eta).
\end{split}
\end{equation*}
For $\chi^{i+,i-;2}_{C,\omega}(\xi,\eta)$, changing variable $(\xi,\eta)\to (\widetilde{\xi},\widetilde{\eta})$, where $\widetilde{\xi}=\xi-\eta$ and $\widetilde{\eta}=-\eta$ if $k_2\le k$ and $\widetilde{\eta}=\xi$ if $k\le k_2$, the previous analysis for $\Phi^{i;i+,i+}$ applies. For $\chi^{i+,i-;1}_{C,\omega}(\xi,\eta)$, using that, on its support, $\vert\xi-\eta\vert\simeq \vert\eta\vert\gtrsim \vert\xi\vert$ and
\begin{equation*}
\begin{split}
\vert \Phi^{i;i+,i-}(\xi,\eta)\vert&\gtrsim 2^{k},\\
 \vert\partial_{\xi_1}^\rho\partial_{\xi_i}^{\sigma}\partial_{\eta_1}^{\tau}\partial_{\eta_j}^{\upsilon}\Xi^{i+,i-}(\xi,\eta)\vert&\lesssim 2^k2^{-\vert\rho\vert k}2^{-\vert\sigma\vert(k+k_1)}2^{-\vert\tau\vert k_2}2^{-\vert\upsilon\vert(k_1+k_2)}
\end{split}
\end{equation*}
we easily obtain \eqref{CaseCME1}. This ends the proof.
\end{proof}

\appendix

\section{General estimates and the functions $\Lambda_i,\Lambda_e,\Lambda_b$}\label{lin}

In this section we summarize the linear and the bilinear estimates we use in the paper. We also provide precise descriptions of the eigenvalues $\Lambda_i,\Lambda_e,\Lambda_b$ defined in \eqref{operators1}.

We note first that Calderon--Zygmund operators are compatible with the spaces constructed in Definition \ref{MainDef}. More precisely:

\begin{lemma}\label{CZop}
Assume $q\in\mathcal{S}^{10}$, see \eqref{symb1}, and $T_qf:=\mathcal{F}^{-1}(q\cdot\widehat{f})$. Then
\begin{equation*}
\|T_qf\|_{Z}\lesssim \|q\|_{\mathcal{S}^{10}}\|f\|_{Z},\qquad\text{ for any }f\in Z.
\end{equation*}
\end{lemma}

We omit the proof of this lemma, since it is identical to the proof of Lemma 5.1 in \cite{IoPa2}. The following general oscillatory integral estimate is used repeatedly in the proofs. 

\begin{lemma}\label{tech5} Assume that $0<\eps\leq 1/\eps\leq K$, $N\geq 1$ is an integer, and $f,g\in C^N(\mathbb{R}^n)$. Then
\begin{equation}\label{ln1}
\Big|\int_{\mathbb{R}^n}e^{iKf}g\,dx\Big|\lesssim_N (K\eps)^{-N}\big[\sum_{|\rho|\leq N}\eps^{|\rho|}\|D^\rho_xg\|_{L^1}\big],
\end{equation}
provided that $f$ is real-valued, 
\begin{equation}\label{ln2}
|\nabla_x f|\geq \mathbf{1}_{{\mathrm{supp}}\,g},\quad\text{ and }\quad\|D_x^\rho f \cdot\mathbf{1}_{{\mathrm{supp}}\,g}\|_{L^\infty}\lesssim_N\eps^{1-|\rho|},\,2\leq |\rho|\leq N.
\end{equation}
\end{lemma}

\begin{proof}[Proof of Lemma \ref{tech5}] We localize first to balls of size $\approx \eps$. Using the assumptions in \eqref{ln2} we may assume that inside each small ball, one of the directional derivatives of $f$ is bounded away from $0$, say $|\partial_1f|\gtrsim_N 1$. Then we integrate by parts $N$ times in $x_1$, and the desired bound \eqref{ln1} follows.  
\end{proof}

We will also use repeatedly the following simple bilinear estimate:

\begin{lemma}\label{tech2}
Assume $p,q\in[2,\infty]$ satisfy $1/p+1/q=1/2$, and $m\in L^\infty(\mathbb{R}^3\times\mathbb{R}^3)$. Then, for any $f,g\in L^2(\mathbb{R}^3)$,
\begin{equation}\label{mk6}
\Big\|\int_{\mathbb{R}^3}m(\xi,\eta)\cdot\widehat{f}(\xi-\eta)\widehat{g}(\eta)\,d\eta\Big\|_{L^2_\xi}\lesssim\|\mathcal{F}^{-1}m\|_{L^1(\mathbb{R}^3\times\mathbb{R}^3)}\|f\|_{L^p}\|g\|_{L^q}.
\end{equation}
\end{lemma}

\begin{proof}[Proof of Lemma \ref{tech2}] Let
\begin{equation*}
K(x,y):=(\mathcal{F}^{-1}m)(x,y)=\int_{\mathbb{R}^3\times\mathbb{R}^3}m(\xi,\eta)e^{ix\cdot\xi}e^{iy\cdot\eta}\,d\xi d\eta.
\end{equation*}
Then, for any $h\in C^\infty_0(\mathbb{R}^3)$ we estimate
\begin{equation*}
\begin{split}
\Big|\int_{\mathbb{R}^3\times\mathbb{R}^3}m(\xi,\eta)\widehat{f}(\xi-\eta)\cdot\widehat{g}(\eta)\widehat{h}(\xi)\,d\xi d\eta\Big|&=C\Big|\int_{\mathbb{R}^3\times\mathbb{R}^3\times\mathbb{R}^3}f(x)g(y)h(z)\cdot K(-x-z,x-y)\,dxdydz\Big|\\
&\lesssim \int_{\mathbb{R}^3\times\mathbb{R}^3\times\mathbb{R}^3}|f(x)g(x-v)h(-x-u)|\cdot |K(u,v)|\,dxdudv\\
&\lesssim \|K\|_{L^1(\mathbb{R}^3\times\mathbb{R}^3)}\cdot \|f\|_{L^p}\|g\|_{L^q}\|h\|_{L^2},
\end{split}
\end{equation*}
and the desired bound \eqref{mk6} follows.
\end{proof}

Recall the functions $\Lambda_i, \Lambda_e, \Lambda_b$ defined in \eqref{operators1}. Let $\lambda_i,\lambda_e,\lambda_b:[0,\infty)\to[0,\infty)$,
\begin{equation}\label{mk0.1}
\begin{split}
&\lambda_i(r):=\varepsilon^{-1/2}\sqrt{\frac{(1+\varepsilon)+(T+\varepsilon)r^2-\sqrt{\left((1-\varepsilon)+(T-\varepsilon)r^2\right)^2+4\varepsilon}}{2}},\\
&\lambda_e(r):=\varepsilon^{-1/2}\sqrt{\frac{(1+\varepsilon)+(T+\varepsilon)r^2+\sqrt{\left((1-\varepsilon)+(T-\varepsilon)r^2\right)^2+4\varepsilon}}{2}},\\
&\lambda_b(r):=\varepsilon^{-1/2}\sqrt{1+\varepsilon+C_br^2},
\end{split}
\end{equation}
such that $\Lambda_\sigma(\xi)=\lambda_\sigma(|\xi|)$, $\sigma\in\{i,e,b\}$. We also define
\begin{equation*}
c_\sigma=\lim_{r\to+\infty}\lambda_\sigma^\prime(r),\quad c_i=1,\,\,c_e=\sqrt{T/\varepsilon},\,\,c_b=\sqrt{C_b/\varepsilon}.
\end{equation*}

\begin{lemma}\label{tech99}
(i) The functions $\lambda_i,\lambda _e,\lambda_b$ are smooth on $[0,\infty)$ and satisfy
\begin{equation}\label{mk3.1}
\begin{split}
&\lambda_i(0)=0,\qquad\lambda''_i(0)=0,\qquad\lambda'''_i(0)\approx_{C_b,\varepsilon} -1,\qquad\lambda'_i(r)\approx_{C_b,\varepsilon} 1\,\,\text{ for any }\,\,r\in[0,\infty),\\
&\lambda_e(0)=\sqrt{\varepsilon^{-1}+1},\qquad\lambda'_e(0)=0,\qquad\lambda''_e(r)\approx_{C_b,\varepsilon} (1+r^2)^{-3/2}\,\,\text{ for any }\,\,r\in[0,\infty),\\
&\lambda_b(0)=\sqrt{\varepsilon^{-1}+1},\qquad\lambda'_b(0)=0,\qquad\lambda''_b(r)\approx_{C_b,\varepsilon} (1+r^2)^{-3/2}\,\,\text{ for any }\,\,r\in[0,\infty).
\end{split}
\end{equation}
In addition, there is a constant $r_\ast\in(T^{-1/2},4T^{-1/2}+4T^{-1/4})$ such that,
\begin{equation}\label{mk3.2}
\lambda_i^{\prime\prime}(r_\ast)=0,\qquad|\lambda''_i(r)|\approx_{C_b,\varepsilon}\min(r,r^{-3})\,\,\text{ if }\,\,|r-r_\ast|\geq 2^{-D},\qquad |\lambda'''_i(r)|\approx_{C_b,\varepsilon} 1\,\,\text{ if }\,\,|r-r_\ast|\leq 2^{-D/2}.
\end{equation}
Moreover,
\begin{equation}\label{SimpleBdLie}
\begin{split}
&r\le\lambda_i(r)\le \sqrt{(T+1)(\varepsilon+1)}r,\quad r\ge0,\\
&\max(\lambda_\sigma(0),c_\sigma r)\le\lambda_{\sigma}(r)\le \lambda_\sigma(0)+c_{\sigma}r,\quad\sigma\in\{e,b\},\,\, r\ge0.
\end{split}
\end{equation}

(ii) Letting $h_\varepsilon(r):=\varepsilon^{-1/2}\sqrt{1+Tr^2}$, we have
\begin{equation}\label{Le1}
\vert D^\rho_r(\lambda_e-h_\varepsilon)(r)\vert\le \sqrt{\varepsilon}\vert D_r^\rho h_\varepsilon(r)\vert,\quad 0\le\rho\le 2.
\end{equation}

(iii) We have $\lambda_i(r)=rq_i(r)$ for some $1\le q_i(r)\le \sqrt{(1+T)/(1+\varepsilon)}$, $q_i(r)\to1$ as $r\to+\infty$ such that
\begin{equation}\label{qprime}
q_i^\prime(r)\leq-\frac{1}{2}\frac{T^2r}{\left[1+T+Tr^2\right]^2}.
\end{equation}
Moreover 
\begin{equation}\label{alo5}
\vert \lambda_i^{\prime\prime}(r)\vert\le8\sqrt{2}T\qquad\text{ and }\qquad\lambda_i^{\prime\prime}(r)\le 10r^{-3}\text{ for }r\ge 4T^{-1/2}+4T^{-1/4}.
\end{equation}
\end{lemma}

\begin{proof}[Proof of Lemma \ref{tech99}] (i) Recall the assumptions \eqref{condTeps}, which are used implicitly many times in this lemma. The claims in \eqref{mk3.1} and \eqref{SimpleBdLie} are straightforward consequences of the definitions. To prove \eqref{mk3.2}, we use first the formula
\begin{equation*}
\sqrt{\varepsilon}\lambda_e(r)\lambda_i(r)=r\left[1+T+Tr^2\right]^{1/2}
\end{equation*}
to see that one can extend $\lambda_i(r)$ into a smooth odd function of $r$. Starting from the relation
\begin{equation}\label{LiP}
\lambda_i^2(r)=\frac{1}{2\varepsilon}\left[1+\varepsilon+(T+\varepsilon)r^2-\sqrt{u^2+4\varepsilon}\right],\quad u=1-\varepsilon+(T-\varepsilon)r^2,
\end{equation}
and deriving up to three times, we find that
\begin{equation}\label{LiPP}
\begin{split}
2\lambda_i(r)\lambda_i^\prime(r)&=\frac{T+\varepsilon}{\varepsilon}r-\frac{T-\varepsilon}{\varepsilon}ru(u^2+4\varepsilon)^{-1/2},\\
2(\lambda_i^\prime(r))^2+2\lambda_i(r)\lambda_i^{\prime\prime}(r)&=\frac{T+\varepsilon}{\varepsilon}-\frac{T-\varepsilon}{\varepsilon}u(u^2+4\varepsilon)^{-1/2}-8(T-\varepsilon)^2r^2(u^2+4\varepsilon)^{-3/2},\\
6\lambda_i^\prime(r)\lambda_i^{\prime\prime}(r)+2\lambda_i(r)\lambda_i^{(3)}(r)&=-\frac{24(T-\varepsilon)^2r}{\left[u^2+4\varepsilon\right]^\frac{5}{2}}\left[(1+\varepsilon)^2-(T-\varepsilon)^2r^4\right]:=A(r).
\end{split}
\end{equation}
In particular, $\lambda_i^\prime>0$.
Since $\lambda_i$ is odd, its even derivatives vanish at $0$. Dividing by $r$ and letting $r\to0$ in the first and third lines gives
\begin{equation*}
(\lambda_i^\prime(0))^2=(1+T)/(1+\varepsilon),\quad8\lambda_i^\prime(0)\lambda_i^{(3)}(0)=-24(T-\varepsilon)^2(1+\varepsilon)^{-3}.
\end{equation*}
Since $\lambda_i^\prime(0)>0$, we see that $\lambda_i^{\prime\prime}<0$ on some interval $(0,\delta)$. Let $R_A=\sqrt{(1+\varepsilon)/(T-\varepsilon)}$ be the positive root of $A(r)$. We claim that $\lambda_i^{\prime\prime}<0$ on $(0,R_A)$. Indeed, we see from \eqref{LiPP} that, on this interval, so long as $\lambda_i^{\prime\prime}(r)\ge A(r)/(12\lambda_i^\prime(r))$, $\lambda_i^{(3)}<0$ and $\lambda_i^{\prime\prime}$ is decreasing. Hence $\lambda_i^{\prime\prime}(R_A)\le 0$. If $\lambda_i^{\prime\prime}(R_A)=0$, using \eqref{LiPP}, we see that $\lambda_i^{(3)}(R_A)=0$ and $R_A$ is a single root for $\lambda_i^{(3)}$ and a double root for $\lambda_i^{\prime\prime}$. Dividing by $r-R_A$ and letting $r\to R_A$, we therefore find that
\begin{equation*}
2\lambda_i(R_A)\lambda_i^{(4)}(R_A)=\lim_{r\to R_A}\frac{A(r)}{r-R_A}>0.
\end{equation*}
But then $\lambda_i^{\prime\prime}(r)>0$ for some $r< R_A$, a contradiction.

It is clear that the argument above can be made quantitative, and prove that
\begin{equation*}
\text{ for any }\delta>0\text{ there is }\delta'=\delta'(\delta,T,\varepsilon)>0\text{ such that }\lambda_i^{\prime\prime}(r)\leq-\delta'\text{ for any }r\in[\delta,R_A].
\end{equation*}
This suffices to prove the desired claim \eqref{mk3.2} for $r\in[0,R_A]$.

We now claim that $\lambda_i^{\prime\prime}$ vanishes exactly once on $(R_A,+\infty)$. Indeed, using again \eqref{LiPP}, we see that if $\lambda_i^{\prime\prime}(r_\ast)=0$, then
\begin{equation*}
\lambda_i^{(3)}(r_\ast)=A(r_\ast)/(2\lambda_i(r_\ast))>0.
\end{equation*}
Let $r_{\ast\ast}$ be the next zero of $\lambda_i^{\prime\prime}$. Since $\lambda_i^{\prime\prime}\ge0$ on $(r_\ast,r_{\ast\ast})$, we have that $\lambda_i^{(3)}(r_{\ast\ast})\le 0$. Plugging $r=r_{\ast\ast}$ in the third line of \eqref{LiPP} gives a contradiction. Finally, we remark that there exists such $r_\ast$ since we will show below that $\lambda_i^{\prime\prime}>0$ for $r$ large enough.

Indeed, using the second equation in \eqref{LiPP},
\begin{equation}\label{alo4}
\lambda_i(r)\lambda_i^{\prime\prime}(r)=1-(\lambda_i^\prime(r))^2+\frac{T-\varepsilon}{2\varepsilon}[1-u(u^2+4\varepsilon)^{-1/2}]-4(T-\varepsilon)^2r^2(u^2+4\varepsilon)^{-3/2}.
\end{equation}
Therefore, using \eqref{LiP} and \eqref{LiPP} and letting $v:=(1-u(u^2+4\varepsilon)^{-1/2})/(2\varepsilon)$,
\begin{equation*}
\begin{split}
(\lambda_i(r))^3\lambda_i^{\prime\prime}(r)&=(\lambda_i(r))^2-(\lambda_i(r))^2(\lambda_i^\prime(r))^2+(\lambda_i(r))^2(T-\varepsilon)v-4(\lambda_i(r))^2(T-\varepsilon)^2r^2(u^2+4\varepsilon)^{-3/2}\\
&=r^2+1-v(u^2+4\varepsilon)^{1/2}-r^2(1+(T-\varepsilon)v)^2\\
&+(\lambda_i(r))^2(T-\varepsilon)v-4(\lambda_i(r))^2(T-\varepsilon)^2r^2(u^2+4\varepsilon)^{-3/2}.
\end{split}
\end{equation*}
Notice that $v\leq u^{-2}\leq (T-\varepsilon)^{-2}r^{-4}$, $(u^2+4\varepsilon)^{1/2}\leq (T-\varepsilon)r^2+1+\varepsilon$, and $(\lambda_i(r))^2\in[r^2,r^2+1]$. Therefore, if $(T-\varepsilon)r^2\geq 10(1+\sqrt{T})$ then 
\begin{equation}\label{alo3}
(\lambda_i(r))^3\lambda_i^{\prime\prime}(r)\in[1/10,1],
\end{equation}
and the desired conclusion \eqref{mk3.2} follows.

(ii) We calculate
\begin{equation*}
h_\varepsilon^\prime(r)=\varepsilon^{-1/2}Tr(1+Tr^2)^{-1/2},\qquad h_\varepsilon^{\prime\prime}(r)=\varepsilon^{-1/2}T(1+Tr^2)^{-3/2}.
\end{equation*}
Starting from the formula
\begin{equation*}
\lambda_e^2(r)=\frac{1}{\varepsilon}\left[1+Tr^2+\frac{\sqrt{u^2+4\varepsilon}-u}{2}\right],
\end{equation*}
where, as before, $u=1-\varepsilon+(T-\varepsilon)r^2$. Therefore $\lambda_e(r)\geq h_\varepsilon(r)$ and
\begin{equation}\label{alo1}
\lambda_e(r)-h_\varepsilon(r)=\frac{\lambda^2_e(r)-h^2_\varepsilon(r)}{\lambda_e(r)+h_\varepsilon(r)}\leq \frac{\sqrt{u^2+4\varepsilon}-u}{4\varepsilon h_\varepsilon(r)}\leq \frac{1}{2uh_\varepsilon(r)}.
\end{equation}
The desired bound \eqref{Le1} follows for $\rho=0$. 

Using the formulas above, we also calculate
\begin{equation*}
2\lambda_e(r)\lambda_e^\prime(r)=\frac{2Tr}{\varepsilon}-\frac{(T-\varepsilon)r}{\varepsilon}\frac{\sqrt{u^2+4\varepsilon}-u}{\sqrt{u^2+4\varepsilon}}.
\end{equation*}
Therefore $\lambda_e^\prime(r)\leq h_\varepsilon^\prime(r)$ and, using also \eqref{alo1},
\begin{equation}\label{alo2}
h_\varepsilon^\prime(r)-\lambda_e^\prime(r)=\frac{2h_\varepsilon(r) h_\varepsilon^\prime(r)-2\lambda_e(r)\lambda_e^\prime(r)+2h_\varepsilon^\prime(r)(\lambda_e(r)-h_\varepsilon(r))}{2\lambda_e(r)}\leq \frac{2Tr}{u^2\lambda_e}.
\end{equation}
The desired bound \eqref{Le1} follows for $\rho=1$.

Finally, we calculate
\begin{equation*}
2\lambda_e(r)\lambda_e^{\prime\prime}(r)+2(\lambda_e^\prime(r))^2=\frac{2T}{\varepsilon}+E\qquad\text{ where }\qquad E:=-\frac{(T-\varepsilon)}{\varepsilon}\frac{\sqrt{u^2+4\varepsilon}-u}{\sqrt{u^2+4\varepsilon}}+\frac{8(T-\varepsilon)^2r^2}{(u^2+4\varepsilon)^{3/2}}.
\end{equation*}
Therefore, using also \eqref{alo1} and \eqref{alo2},
\begin{equation*}
|\lambda_e^{\prime\prime}(r)-h_\varepsilon^{\prime\prime}(r)|\leq \frac{|E|+2|\lambda_e(r)-h_\varepsilon(r)|h_\varepsilon^{\prime\prime}(r)+2|(h_\varepsilon^\prime(r))^2-(\lambda_e^\prime(r))^2|}{2\lambda_e}\leq\frac{20(T+1)}{u^2\lambda_e}.
\end{equation*}
The desired bound \eqref{Le1} follows for $\rho=2$.

(iii) Starting from the formula
\begin{equation*}
\sqrt{\varepsilon}\lambda_e(r)\lambda_i(r)=r\sqrt{1+T+Tr^2},
\end{equation*}
we calculate
\begin{equation}\label{alo7.5}
q_i(r)=\frac{\sqrt{1+T+Tr^2}}{\sqrt{\varepsilon}\lambda_e}=\Big[1-\frac{T}{1+T+Tr^2}+\frac{\sqrt{u^2+4\varepsilon}-u}{2(1+T+Tr^2)}\Big]^{-1/2},
\end{equation}
and, with $v=(1-u(u^2+4\varepsilon)^{-1/2})/(2\varepsilon)$ as before,
\begin{equation*}
q_i^{\prime}(r)=-\Big[1-\frac{T}{1+T+Tr^2}+\frac{\sqrt{u^2+4\varepsilon}-u}{2(1+T+Tr^2)}\Big]^{-3/2}\Big[\frac{T^2r-Tr\varepsilon v\sqrt{u^2+4\varepsilon}-(1+T+Tr^2)(T-\varepsilon)r\varepsilon v}{(1+T+Tr^2)^2}\Big].
\end{equation*}
This suffices to prove \eqref{qprime}.

The second bound in \eqref{alo5} follows from \eqref{alo3}. To prove the first bound in \eqref{alo5}, we notice that it follows from part (i) that there are two values $r_{\min}\in (0,r_\ast)$ and $r_{\max}\in(r_\ast,\infty)$ such that
\begin{equation*}
\lambda_i^{\prime\prime}(r)\in [\lambda_i^{\prime\prime}(r_{\min}),\lambda_i^{\prime\prime}(r_{\max})]\qquad\text{ for any }r\in[0,\infty).
\end{equation*}
Using the identity in the second line of \eqref{LiPP} it follows that $\lambda_i(r)\lambda_i^{\prime\prime}(r)\leq 1$ for any $r\geq \sqrt{(1+\varepsilon)/(T-\varepsilon)}$. Since $r_{\max}\geq r_\ast\geq \sqrt{(1+\varepsilon)/(T-\varepsilon)}$, it follows that
\begin{equation}\label{alo8}
\lambda_i^{\prime\prime}(r_{\max})\leq 1/\lambda_i(r_{\max})\leq 1/r_{\max}\leq\sqrt{T}.
\end{equation}

To estimate $|\lambda_i^{\prime\prime}(r_{\min})|$ we use \eqref{alo4} and the observation $|\lambda_i^\prime(r)|\leq q_i(r)\leq \sqrt{(1+T)/(1+\varepsilon)}$ to write
\begin{equation*}
\begin{split}
\lambda_i(r)\lambda_i^{\prime\prime}(r)&\geq 1-\frac{1+T}{1+\varepsilon}+\frac{T-\varepsilon}{2\varepsilon}[1-u(u^2+4\varepsilon)^{-1/2}]-4(T-\varepsilon)^2r^2(u^2+4\varepsilon)^{-3/2}\\
&\geq \frac{T-\varepsilon}{1+\varepsilon}\Big[-1+\frac{1+\varepsilon}{2\varepsilon}(1-u(u^2+4\varepsilon)^{-1/2})\Big]-4(T-\varepsilon)^{3/2}r\\
&\geq -8(T-\varepsilon)^{3/2}r.
\end{split}
\end{equation*}
Moreover, since $\lambda_i^{(3)}(r_{\min})=0$ and $\lambda_i^{\prime\prime}(r_{\min})\leq 0$, it follows from the identity in the last line of \eqref{LiPP} that $r_{\min}\leq R_A=\sqrt{(1+\varepsilon)/(T-\varepsilon)}$. Therefore, using also the fact that $q_i$ is decreasing on $[0,\infty)$, see \eqref{qprime}, and the identity \eqref{alo7.5}, it follows that
\begin{equation*}
-\lambda_i^{\prime\prime}(r_{\min})\leq \frac{8(T-\varepsilon)^{3/2}r_{\min}}{\lambda_i(r_{\min})}\leq \frac{8(T-\varepsilon)^{3/2}}{q_i(R_A)}\leq 8\sqrt{2}(T-\varepsilon).
\end{equation*}
The desired estimate in \eqref{alo5} follows using also \eqref{alo8}.
\end{proof}

\begin{lemma}\label{tech1.5} Assume $\|f\|_{Z}\leq 1$, $t\in\mathbb{R}$, $(k,j)\in\mathcal{J}$, and let $\widetilde{k}=\min(k,0)$ and
\begin{equation*}
f_{k,j}:=P_{[k-2,k+2]}[\widetilde{\varphi}^{(k)}_j\cdot P_kf].
\end{equation*}

(i) Then
\begin{equation}\label{mk15.5}
\|f_{k,j}\|_{L^2}\lesssim (2^{\alpha k}+2^{10k})^{-1}\cdot 2^{2\beta\widetilde{k}}2^{-(1-\be)j}
\end{equation}
and
\begin{equation}\label{mk15.55}
\sup_{\xi\in\mathbb{R}^3}\big|D^\rho_\xi\widehat{f_{k,j}}(\xi)\big|\lesssim_{|\rho|} (2^{\alpha k}+2^{10k})^{-1}\cdot 2^{-(1/2-\beta)\widetilde{k}}2^{|\rho|j}.
\end{equation}
Moreover, if $k\leq 0$ and $\sigma\in\{e,b\}$ then
\begin{equation}\label{mk15.6}
\big\|e^{it\Lambda_\sigma}f_{k,j}\big\|_{L^\infty}\lesssim 2^{(3/2-\alpha)k}2^{-(1+\beta)j}(1+2^{-j}|t|2^k)^{-3/2+10\beta}.
\end{equation}
If $k\geq 0$ and $\sigma\in\{e,b\}$ then
\begin{equation}\label{mk15.61}
\big\|e^{it\Lambda_\sigma}f_{k,j}\big\|_{L^\infty}\lesssim 2^{-6k}2^{-(1+\beta)j}(1+2^{-j}|t|)^{-3/2+10\beta}.
\end{equation}

With $r_\ast$ defined as in \eqref{mk3.2}, let $k_\ast:=\log_2r_\ast$. If $k\leq k_\ast-3$ and $\sigma=i$ then
\begin{equation}\label{mk15.65}
\big\|e^{it\Lambda_\sigma}f_{k,j}\big\|_{L^\infty}\lesssim 2^{(3/2-\alpha)k}2^{-(1+\beta)j}(1+2^{-j}|t|2^{2k/3})^{-3/2+10\beta}.
\end{equation}
If $k\in[k_\ast-3,k_\ast+3]$ and $\sigma=i$ then
\begin{equation}\label{mk15.66}
\big\|e^{it\Lambda_\sigma}f_{k,j}\big\|_{L^\infty}\lesssim 2^{-(1+\beta)j}(1+2^{-j}|t|)^{-5/4+10\beta}.
\end{equation}
If $k\geq k_\ast+3$ and $\sigma=i$ then
\begin{equation}\label{mk15.67}
\big\|e^{it\Lambda_\sigma}f_{k,j}\big\|_{L^\infty}\lesssim 2^{-6k}2^{-(1+\beta)j}(1+2^{-j}|t|)^{-3/2+10\beta}.
\end{equation}

(ii) As a consequence
\begin{equation}\label{ok9}
\sum_{j\geq \max(-k,0)}\|f_{k,j}\|_{L^2}\lesssim \min(2^{(1+\beta-\alpha)k},2^{-10k})
\end{equation}
and\footnote{In many places we will be able to use the simpler bound \eqref{ok10}, instead of the more precise bounds \eqref{mk15.6}--\eqref{mk15.67}.}, for any $\sigma\in\{i,e,b\}$,
\begin{equation}\label{ok10}
\sum_{j\geq \max(-k,0)}\big\|e^{it\Lambda_\sigma}f_{k,j}\big\|_{L^\infty}\lesssim \min(2^{(1/2-\beta-\alpha)k},2^{-6k})(1+\vert t\vert)^{-1-\beta}.
\end{equation}
\end{lemma}

\begin{proof}[Proof of Lemma \ref{tech1.5}] We start by decomposing, as in \eqref{sec5.8}--\eqref{sec5.815},
\begin{equation}\label{compat70}
\widetilde{\varphi}^{(k)}_{j}\cdot P_kf=(2^{\alpha k}+2^{10k})^{-1}(g_{1,j}+g_{2,j}),\qquad g_{1,j}=g_{1,j}\cdot \widetilde{\varphi}^{(k)}_{[j-2,j+2]},\qquad g_{2,j}=g_{2,j}\cdot \widetilde{\varphi}^{(k)}_{[j-2,j+2]},
\end{equation}
such that
\begin{equation}\label{compat70.1}
2^{(1+\be)j}\|g_{1,j}\|_{L^2}+2^{(1/2-\beta)\widetilde{k}}\|\widehat{g_{1,j}}\|_{L^\infty}\lesssim 1,
\end{equation}
and
\begin{equation}\label{compat70.2}
2^{(1-\be)j}\|g_{2,j}\|_{L^2}+\|\widehat{g_{2,j}}\|_{L^\infty}+2^{\gamma j}\sup_{R\in[2^{-j},2^k],\,\xi_0\in\mathbb{R}^3}R^{-2}\|\widehat{g_{2,j}}\|_{L^1(B(\xi_0,R))}\lesssim 2^{-10|k|}.
\end{equation}
The bound \eqref{mk15.5} follows easily.

To prove \eqref{mk15.55} we use the formulas in \eqref{compat70} to write, for $\mu=1,2$,
\begin{equation*}
\widehat{g_{\mu,j}}(\xi)=c\int_{\mathbb{R}^3}\widehat{g_{\mu,j}}(\eta)\mathcal{F}(\widetilde{\varphi}^{(k)}_{[j-2,j+2]})(\xi-\eta)\,d\eta.
\end{equation*}
Therefore
\begin{equation}\label{ccc99}
D^\rho_\xi\widehat{g_{\mu,j}}(\xi)=c\int_{\mathbb{R}^3}\widehat{g_{\mu,j}}(\eta)\mathcal{F}(x^\rho\cdot\widetilde{\varphi}^{(k)}_{[j-2,j+2]})(\xi-\eta)\,d\eta.
\end{equation}
The desired bounds \eqref{mk15.55} follow using the bounds $\|\widehat{g_{\mu,j}}\|_{L^\infty}\lesssim2^{-(1/2-\beta)\widetilde{k}}$, see \eqref{compat70.1}-\eqref{compat70.2}.

We consider now the $L^\infty$ bounds \eqref{mk15.6}-\eqref{mk15.67}. Using \eqref{compat70}-\eqref{compat70.2}, we have
\begin{equation*}
\|\widehat{f_{k,j}}\|_{L^1(B(\xi_0,R))}\lesssim (2^{\alpha k}+2^{10k})^{-1}\min(R^32^{-(1/2-\beta)\widetilde{k}},R^{3/2}2^{-(1+\beta)j}),
\end{equation*}
for any $\xi_0\in\mathbb{R}^3$ and $R\leq 2^{k}$. Therefore, for any $k\in\mathbb{Z}$ and $\sigma\in\{i,e,b\}$,
\begin{equation}\label{ccc2}
\big\|e^{it\Lambda_\sigma}f_{k,j}\big\|_{L^\infty}\lesssim (2^{\alpha k}+2^{10k})^{-1}\cdot 2^{3k/2}2^{-(1+\beta)j}.
\end{equation}

{\bf{Step 1.}} We consider first the simplest case
\begin{equation}\label{ccc3}
\sigma\in\{e,b\},\qquad k\leq 0,\qquad |t|\geq 2^{j-k+D},
\end{equation}
and prove that, for any $x\in\mathbb{R}^3$,
\begin{equation}\label{ccc4}
\Big|\int_{\mathbb{R}^3}e^{ix\cdot\xi}e^{it\Lambda_\sigma(\xi)}\widehat{f_{k,j}}(\xi)\,d\xi\Big|\lesssim 2^{(3/2-\alpha) k}2^{-(1+\beta)j}\rho_1^{3/2-10\beta},\qquad \rho_1:=2^{j-k}|t|^{-1}.
\end{equation}
The bound \eqref{mk15.6} would clearly follow from \eqref{ccc2} and \eqref{ccc4}. 
Using the decomposition \eqref{compat70}, it suffices to prove that, for $\mu\in\{1,2\}$,
\begin{equation}\label{ccc5}
\begin{split}
&\big\|g_{\mu,j}\ast K^\sigma_{k,t}\big\|_{L^\infty}\lesssim 2^{3k/2}2^{-(1+\beta)j}\rho_1^{3/2-10\beta},\\
&K^\sigma_{k,t}(x):=\int_{\mathbb{R}^3}e^{ix\cdot\xi}e^{it\Lambda_\sigma(\xi)}\varphi_{[k-2,k+2]}(\xi)\,d\xi.
\end{split}
\end{equation}

Recall that the kernel of the operator on $\mathbb{R}^3$ defined by the radial multiplier $\xi\to p(|\xi|)$ is
\begin{equation}\label{ccc6}
K(x)=c\int_0^\infty p(s)s^2\frac{e^{is|x|}-e^{-is|x|}}{s|x|}\,ds.
\end{equation}
We show that
\begin{equation}\label{ccc7}
\|K^\sigma_{k,t}\|_{L^\infty}\lesssim |t|^{-3/2}.
\end{equation}
In view of \eqref{ccc6} it suffices to prove that
\begin{equation}\label{mk20}
\Big|\int_0^\infty s^2\varphi^1_{[k-2,k+2]}(s)e^{it\lambda_\sigma(s)}\frac{e^{isr}-e^{-isr}}{sr}\,ds\Big|\lesssim |t|^{-3/2},
\end{equation}
for any $r\in(0,\infty)$. Recall the assumption \eqref{ccc3}, it particular $|t|\geq 2^{D-2k}$. Since $\lambda'_\sigma(s)\approx \min(s,1)$ (see \eqref{mk3.1}), the bound \eqref{mk20} follows by integration by parts unless $r\approx |t|2^k$. On the other hand, if $r\approx |t|2^k$ then the bound \eqref{mk20} follows by stationary phase, using $\lambda''_\sigma(s)\approx (1+s^2)^{-3/2}$, see \eqref{mk3.1}.

In view of \eqref{ccc7} and the assumptions \eqref{compat70}--\eqref{compat70.2}, it follows that
\begin{equation*}
\big\|g_{1,j}\ast K^\sigma_{k,t}\big\|_{L^\infty}\lesssim \|g_{1,j}\|_{L^1}\|K^\sigma_{k,t}\|_{L^\infty}\lesssim 2^{3j/2}2^{-(1+\beta)j}|t|^{-3/2}\lesssim 2^{3 k/2}2^{-(1+\beta)j}\rho_1^{3/2},
\end{equation*}
and
\begin{equation*}
\big\|g_{2,j}\ast K^\sigma_{k,t}\big\|_{L^\infty}\lesssim \|g_{2,j}\|_{L^1}\|K^\sigma_{k,t}\|_{L^\infty}\lesssim 2^{3j/2}2^{-(1-\beta)j}2^{2\beta k}|t|^{-3/2}\lesssim 2^{3 k/2}2^{-(1+\beta)j}\rho_1^{3/2}2^{2\beta(j+k)}.
\end{equation*}
The bounds \eqref{ccc5} follow if $\mu=1$ or if $\mu=2$ and $2^{j+k}\leq\rho_1^{-5}$. On the other hand, if $\mu=2$ and $2^{j+k}\geq\rho_1^{-5}$ then, using the $L^1$ bounds on $\widehat{g_{2,j}}$ in \eqref{compat70.2},
\begin{equation*}
\big\|g_{2,j}\ast K^\sigma_{k,t}\big\|_{L^\infty}\lesssim \big\|\widehat{g_{2,j}}\big\|_{L^1}\lesssim 2^{3 k/2}2^{-(1+\beta)j}2^{-(\gamma-\beta-1)(j+k)}\lesssim 2^{3 k/2}2^{-(1+\beta)j}\rho_1^{5(\gamma-\beta-1)},
\end{equation*}
which suffices to prove \eqref{ccc5} in this case as well.

{\bf{Step 2.}} We consider now the case
\begin{equation}\label{ccc20}
\sigma\in\{e,b\},\qquad k\geq 0,\qquad |t|\geq 2^{j+k+D},
\end{equation}
and prove that, for any $x\in\mathbb{R}^3$,
\begin{equation}\label{ccc21}
\Big|\int_{\mathbb{R}^3}e^{ix\cdot\xi}e^{it\Lambda_\sigma(\xi)}\widehat{f_{k,j}}(\xi)\,d\xi\Big|\lesssim 2^{-6 k}2^{-(1+\beta)j}\rho_2^{3/2-10\beta},\qquad \rho_2:=2^{j}|t|^{-1}.
\end{equation}
The bound \eqref{mk15.61} would clearly follow from \eqref{ccc2} and \eqref{ccc21}. 
Using the decomposition \eqref{compat70}, it suffices to prove that, for $\mu\in\{1,2\}$,
\begin{equation}\label{ccc24}
\big\|g_{\mu,j}\ast K^\sigma_{k,t}\big\|_{L^\infty}\lesssim 2^{4k}2^{-(1+\beta)j}\rho_2^{3/2-10\beta},
\end{equation}
where $K^\sigma_{k,t}$ is defined as in \eqref{ccc5}. 

As before, we show that
\begin{equation}\label{ccc25}
\|K^\sigma_{k,t}\|_{L^\infty}\lesssim |t|^{-3/2}2^{3k}.
\end{equation}
In view of \eqref{ccc6} it suffices to prove that
\begin{equation}\label{ccc26}
\Big|\int_0^\infty s^2\varphi^1_{[k-2,k+2]}(s)e^{it\lambda_\sigma(s)}\frac{e^{isr}-e^{-isr}}{sr}\,ds\Big|\lesssim |t|^{-3/2}2^{3k},
\end{equation}
for any $r\in(0,\infty)$. Recall the assumption \eqref{ccc20}, in particular $|t|\geq 2^{D+k}$. Since $\lambda'_\sigma(s)\approx \min(s,1)$ (see \eqref{mk3.1}), the bound \eqref{ccc26} follows by integration by parts unless $r\approx |t|$. On the other hand, if $r\approx |t|$ then the bound \eqref{ccc26} follows by stationary phase, using $\lambda^{\prime\prime}_\sigma(s)\approx (1+s^2)^{-3/2}$, see \eqref{mk3.1}.

In view of \eqref{ccc25} and the assumptions \eqref{compat70}--\eqref{compat70.2}, it follows that
\begin{equation*}
\big\|g_{1,j}\ast K^\sigma_{k,t}\big\|_{L^\infty}\lesssim \|g_{1,j}\|_{L^1}\|K^\sigma_{k,t}\|_{L^\infty}\lesssim 2^{3j/2}2^{-(1+\beta)j}2^{3k}|t|^{-3/2}\lesssim 2^{3 k}2^{-(1+\beta)j}\rho_2^{3/2},
\end{equation*}
and
\begin{equation*}
\big\|g_{2,j}\ast K^\sigma_{k,t}\big\|_{L^\infty}\lesssim \|g_{2,j}\|_{L^1}\|K^\sigma_{k,t}\|_{L^\infty}\lesssim 2^{3j/2}2^{-(1-\beta)j}2^{3k}|t|^{-3/2}\lesssim 2^{3 k}2^{-(1+\beta)j}\rho_2^{3/2}2^{2\beta j}.
\end{equation*}
The bounds \eqref{ccc24} follow if $\mu=1$ or if $\mu=2$ and $2^{j}\leq\rho_2^{-5}$. On the other hand, if $\mu=2$ and $2^{j}\geq\rho_2^{-5}$ then, using the $L^1$ bounds on $\widehat{g_{2,j}}$ in \eqref{compat70.2},
\begin{equation*}
\big\|g_{2,j}\ast K^\sigma_{k,t}\big\|_{L^\infty}\lesssim \big\|\widehat{g_{2,j}}\big\|_{L^1}\lesssim 2^{-(1+\beta)j}2^{-(\gamma-\beta-1)j}\lesssim 2^{-(1+\beta)j}\rho_2^{5(\gamma-\beta-1)},
\end{equation*}
which suffices to prove \eqref{ccc24} in this case as well.

An identical argument shows that
\begin{equation}\label{ccc28}
\text{ if }\quad k\geq k_\ast+3,\qquad |t|\geq 2^{j+k+D},
\end{equation}
then, for any $x\in\mathbb{R}^3$,
\begin{equation}\label{ccc29}
\Big|\int_{\mathbb{R}^3}e^{ix\cdot\xi}e^{it\Lambda_i(\xi)}\widehat{f_{k,j}}(\xi)\,d\xi\Big|\lesssim 2^{-6 k}2^{-(1+\beta)j}(2^j/|t|)^{3/2-10\beta}.
\end{equation}
The bound \eqref{mk15.67} clearly follows from \eqref{ccc2} and \eqref{ccc29}.

{\bf{Step 3.}} We consider now the case
\begin{equation}\label{ccc30}
k\leq k_\ast-3,\qquad |t|\geq 2^{j-2k/3+D},
\end{equation}
and prove that, for any $x\in\mathbb{R}^3$,
\begin{equation}\label{ccc31}
\Big|\int_{\mathbb{R}^3}e^{ix\cdot\xi}e^{it\Lambda_i(\xi)}\widehat{f_{k,j}}(\xi)\,d\xi\Big|\lesssim 2^{(3/2-\alpha) k}2^{-(1+\beta)j}\rho_3^{3/2-10\beta},\qquad \rho_3:=2^{j-2k/3}|t|^{-1}.
\end{equation}
The bound \eqref{mk15.65} would clearly follow from \eqref{ccc2} and \eqref{ccc31}. 
Using the decomposition \eqref{compat70}, it suffices to prove that, for $\mu\in\{1,2\}$,
\begin{equation}\label{ccc32}
\begin{split}
&\big\|g_{\mu,j}\ast K^i_{k,t}\big\|_{L^\infty}\lesssim 2^{3k/2}2^{-(1+\beta)j}\rho_3^{3/2-10\beta},\\
&K^i_{k,t}(x):=\int_{\mathbb{R}^3}e^{ix\cdot\xi}e^{it\Lambda_i(\xi)}\varphi_{[k-2,k+2]}(\xi)\,d\xi.
\end{split}
\end{equation}

As before, we show that
\begin{equation}\label{ccc33}
\|K^i_{k,t}\|_{L^\infty}\lesssim 2^{k/2}|t|^{-3/2}.
\end{equation}
In view of \eqref{ccc6} it suffices to prove that
\begin{equation}\label{ccc34}
\Big|\int_0^\infty s^2\varphi^1_{[k-2,k+2]}(s)e^{it\lambda_i(s)}\frac{e^{isr}-e^{-isr}}{sr}\,ds\Big|\lesssim 2^{k/2}|t|^{-3/2},
\end{equation}
for any $r\in(0,\infty)$. Recall the assumption \eqref{ccc3}, it particular $|t|\geq 2^{D-5k/3}$. Since $\lambda'_i(s)\approx 1$ (see \eqref{mk3.1}), the bound \eqref{ccc34} follows by integration by parts unless $r\approx |t|$. On the other hand, if $r\approx |t|$ then the bound \eqref{ccc34} follows by stationary phase, using $\lambda''_\sigma(s)\approx s$, see \eqref{mk3.2}.

As before, we can now prove \eqref{ccc32}. Using \eqref{ccc7} and \eqref{compat70}--\eqref{compat70.2}, it follows that
\begin{equation*}
\big\|g_{1,j}\ast K^i_{k,t}\big\|_{L^\infty}\lesssim \|g_{1,j}\|_{L^1}\|K^i_{k,t}\|_{L^\infty}\lesssim 2^{3j/2}2^{-(1+\beta)j}2^{k/2}|t|^{-3/2}\lesssim 2^{3 k/2}2^{-(1+\beta)j}\rho_3^{3/2},
\end{equation*}
and
\begin{equation*}
\big\|g_{2,j}\ast K^i_{k,t}\big\|_{L^\infty}\lesssim \|g_{2,j}\|_{L^1}\|K^i_{k,t}\|_{L^\infty}\lesssim 2^{3j/2}2^{-(1-\beta)j}2^{2\beta k}2^{k/2}|t|^{-3/2}\lesssim 2^{3 k/2}2^{-(1+\beta)j}\rho_3^{3/2}2^{2\beta(j+k)}.
\end{equation*}
The bounds \eqref{ccc32} follow if $\mu=1$ or if $\mu=2$ and $2^{j+k}\leq\rho_3^{-5}$. On the other hand, if $\mu=2$ and $2^{j+k}\geq\rho_3^{-5}$ then, using the $L^1$ bounds on $\widehat{g_{2,j}}$ in \eqref{compat70.2},
\begin{equation*}
\big\|g_{2,j}\ast K^i_{k,t}\big\|_{L^\infty}\lesssim \big\|\widehat{g_{2,j}}\big\|_{L^1}\lesssim 2^{3 k/2}2^{-(1+\beta)j}2^{-(\gamma-\beta-1)(j+k)}\lesssim 2^{3 k/2}2^{-(1+\beta)j}\rho_3^{5(\gamma-\beta-1)},
\end{equation*}
which suffices to prove \eqref{ccc32} in this case as well.

{\bf{Step 4.}} Finally, we consider the case
\begin{equation}\label{ccc40}
k\in[k_\ast-3,k_\ast+3],\qquad |t|\geq 2^{j+4D},
\end{equation}
and prove that, for any $x\in\mathbb{R}^3$,
\begin{equation}\label{ccc41}
\Big|\int_{\mathbb{R}^3}e^{ix\cdot\xi}e^{it\Lambda_i(\xi)}\widehat{f_{k,j}}(\xi)\,d\xi\Big|\lesssim 2^{-(1+\beta)j}\rho_2^{5/4-10\beta},\qquad \rho_2=2^{j}|t|^{-1}.
\end{equation}
The bound \eqref{mk15.66} would clearly follow from \eqref{ccc2} and \eqref{ccc41}. 

Using the decomposition \eqref{compat70}--\eqref{compat70.2}, it suffices to prove that, for $\mu\in\{1,2\}$ and $x\in\mathbb{R}^3$,
\begin{equation}\label{ccc42}
\Big|\int_{\mathbb{R}^3}e^{ix\cdot\xi}e^{it\Lambda_i(\xi)}\varphi_{[k-2,k+2]}(\xi)\big[1-\varphi[(|\xi|-r_\ast)/\rho_2^{1/2}]\big]\widehat{g_{\mu,j}}(\xi)\,d\xi\Big|\lesssim 2^{-(1+\beta)j}\rho_2^{5/4-10\beta},
\end{equation}
and
\begin{equation}\label{ccc43}
\Big|\int_{\mathbb{R}^3}e^{ix\cdot\xi}e^{it\Lambda_i(\xi)}\varphi[(|\xi|-r_\ast)/\rho_2^{1/2}]\widehat{g_{\mu,j}}(\xi)\,d\xi\Big|\lesssim 2^{-(1+\beta)j}\rho_2^{5/4-10\beta}.
\end{equation}

Letting
\begin{equation*}
K_{k,t;\delta}(x):=\int_{\mathbb{R}^3}e^{ix\cdot\xi}e^{it\Lambda_i(\xi)}\varphi_{[k-2,k+2]}(\xi)\big[1-\varphi[(|\xi|-r_\ast)/\delta]\big]\,d\xi
\end{equation*}
and arguing as in the proof of \eqref{ccc33}, it is easy to see that
\begin{equation}\label{ccc43.5}
\|K_{k,t;\delta}\|_{L^\infty}\lesssim |t|^{-3/2}\delta^{-1/2}
\end{equation}
provided that $\delta\in[|t|^{-1/2},2^{-D}]$. As before, this suffices to prove the bounds \eqref{ccc42}.

To prove \eqref{ccc43}, we may assume, without loss of generality, that $x/t=(-z_1,0,0)$ for some $z_1\in[0,\infty)$. The formula \eqref{ccc99}, together with the bounds in \eqref{compat70.1} and \eqref{compat70.2} show that
\begin{equation}\label{ccc44}
\begin{split}
&2^{(1+\beta)j}\|D^\rho_\xi\widehat{g_{1,j}}(\xi)\|_{L^2}+\|D^\rho_\xi\widehat{g_{1,j}}(\xi)\|_{L^\infty}\lesssim 2^{|\rho|j},\\
&2^{(1-\beta)j}\|D^\rho_\xi\widehat{g_{2,j}}(\xi)\|_{L^2}+\|D^\rho_\xi\widehat{g_{2,j}}(\xi)\|_{L^\infty}+2^{\gamma j}\sup_{R\in[2^{-j},1],\,\xi_0\in\mathbb{R}^3}R^{-2}\|D^\rho_\xi\widehat{g_{2,j}}(\xi)\|_{L^1(B(\xi_0,R))}\lesssim 2^{|\rho|j}.
\end{split}
\end{equation}
With $\xi=(\xi_1,\xi_2,\xi_3)=(\xi_1,\xi')$ and $l\in\mathbb{Z}\cap(-\infty,D/10]$ we define
\begin{equation*}
I^\mu_{\leq l}:=\int_{\mathbb{R}^3}\varphi(|\xi'|/2^l)e^{it(\Lambda_i(\xi)-z_1\xi_1)}\varphi[(|\xi|-r_\ast)/\rho_2^{1/2}]\widehat{g_{\mu,j}}(\xi)\,d\xi,\qquad I^\mu_l:=I^\mu_{\leq l}-I^\mu_{\leq l-1}.
\end{equation*}
We fix $l_0\in\mathbb{Z}$ such that
\begin{equation*}
2^{l_0}\leq \rho_2+|t|^{-1/2}\leq 2^{l_0+1}.
\end{equation*}
We use \eqref{ccc44} with $|\rho|=0$ to estimate, if $2^j\geq |t|^{1/2}$,
\begin{equation*}
\begin{split}
&|I^1_{\leq l_0}|\lesssim 2^{l_0}\rho_2^{1/4}\|\widehat{g_{1,j}}\|_{L^2}\lesssim 2^{-(1+\beta)j}2^{l_0}\rho_2^{1/4},\\
&|I^2_{\leq l_0}|\lesssim \frac{\rho_2^{1/2}}{2^{l_0}}\cdot 2^{2l_0}2^{-\gamma_j}\lesssim 2^{-(1+\beta)j}2^{l_0}\rho_2^{1/2}.
\end{split}
\end{equation*}
On the other hand, if $2^j\leq |t|^{1/2}$ then
\begin{equation*}
|I^1_{\leq l_0}|+|I^2_{\leq l_0}|\lesssim 2^{2l_0}\rho_2^{1/2}(\|\widehat{g_{1,j}}\|_{L^\infty}+\|\widehat{g_{2,j}}\|_{L^\infty})\lesssim 2^{2l_0}\rho_2^{1/2}.
\end{equation*}
Therefore, in both cases,
\begin{equation}\label{ccc45}
|I^1_{\leq l_0}|+|I^2_{\leq l_0}|\lesssim 2^{-(1+\beta)j}\rho_2^{5/4}.
\end{equation}

To estimate $|I_l^\mu|$ for $l\geq l_0+1$ we integrate by parts in $\xi'$, using Lemma \ref{tech5} with $K\approx |t|2^l$ and $\eps^{-1}\approx 2^j+2^{-l}+2^l\rho_2^{-1/2}$. Arguing as before, we estimate, if $2^j=\max(2^j,2^{-l},2^l\rho_2^{-1/2})$,
\begin{equation*}
\begin{split}
&|I^1_l|\lesssim 2^{2j}/(|t|^22^{2l})\cdot 2^{l}\rho_2^{1/4}2^{-(1+\beta)j}\lesssim \rho_2^{5/4}2^{-(1+\beta)j}\cdot \rho_22^{-l},\\
&|I^2_l|\lesssim 2^{2j}/(|t|^22^{2l})\cdot 2^l(2^l+\rho_2^{1/2})2^{-\gamma j}\lesssim 2^{-\gamma j}\rho_2^2+\rho_2^{3/2}2^{-(1+\beta)j}\cdot \rho_22^{-l}.
\end{split}
\end{equation*}
On the other hand, if $2^{-l}=\max(2^j,2^{-l},2^l\rho_2^{-1/2})$ then
\begin{equation*}
|I^1_l|+|I^2_l|\lesssim 2^{-2l}/(|t|^22^{2l})\cdot 2^{2l}\rho_2^{1/2}\lesssim 2^{-2l}|t|^{-2}\rho_2^{1/2}.
\end{equation*}
Finally, if $2^l\rho_2^{-1/2}=\max(2^j,2^{-l},2^l\rho_2^{-1/2})$ then
\begin{equation*}
|I^1_l|+|I^2_l|\lesssim (2^{2l}\rho_2^{-1})/(|t|^22^{2l})\cdot 2^{2l}\rho_2^{1/2}\lesssim 2^{2l}|t|^{-2}\rho_2^{-1/2}.
\end{equation*}
Therefore
\begin{equation}\label{ccc46}
\sum_{l\geq l_0+1}|I^1_{l}|+|I^2_{l}|\lesssim 2^{-(1+\beta)j}\rho_2^{5/4}.
\end{equation}
The desired bound \eqref{ccc43} follows from \eqref{ccc45} and \eqref{ccc46}.
\end{proof}

Our last lemma in this section is a bilinear estimate. Recall the operators $Q_s^{\sigma;\mu,\nu}$ defined in \eqref{norm4.0},
\begin{equation*}
\mathcal{F}[Q_s^{\sigma;\mu,\nu}(f,g)](\xi)=\int_{\mathbb{R}^3} e^{is[\Lambda_\sigma(\xi)-\widetilde{\Lambda}_{\mu}(\xi-\eta)-\widetilde{\Lambda}_{\nu}(\eta)]}m_{\sigma;\mu,\nu}(\xi,\eta)
\widehat{f}(\xi-\eta)\widehat{g}(\eta)\,d\eta.
\end{equation*}

\begin{lemma}\label{ders}
Assume $s\in\mathbb{R}$, $\sigma\in\{i,e,b\}$, $\mu,\nu\in\mathcal{I}_0$, and 
\begin{equation}\label{ders2}
\|f\|_{Z\cap H^{N_0}}\leq 1,\qquad \|g\|_{Z\cap H^{N_0}}\leq 1.
\end{equation}
Then, for any $k'\in\mathbb{Z}$,
\begin{equation}\label{ok50}
\|P_{k'}Q_s^{\sigma;\mu,\nu}(f,g)\|_{L^2}\lesssim 2^{k'_\sigma}\min[(1+s)^{-1-\beta},2^{3k'/2}]\cdot\min [1,2^{-(N_0-5)k'}].
\end{equation}
Moreover
\begin{equation}\label{derv1}
\text{ if }\quad 2^{k'}\in[2^{-D},2^D]\text{ and }\sigma\in\{e,b\}\quad\text{ or }\quad 2^{k'}\in(0,2^D]\text{ and }\sigma=i,
\end{equation}
then
\begin{equation}\label{derv2}
\|\mathcal{F}P_{k'}Q_s^{\sigma;\mu,\nu}(f,g)\|_{L^\infty}\lesssim (1+s)^{-1+\beta/10}2^{-k'}.
\end{equation}
\end{lemma}

\begin{proof}[Proof of Lemma \ref{ders}] Clearly, the left-hand side of \eqref{ok50} is dominated by
\begin{equation}\label{ok51}
\begin{split}
C\sum_{k_1,k_2\in\mathbb{Z}}\Big\|\varphi_{k'}(\xi)\int_{\mathbb{R}^3}e^{-is[\widetilde{\Lambda}_{\mu}(\xi-\eta)+\widetilde{\Lambda}_{\nu}(\eta)]}m_{\sigma;\mu,\nu}(\xi,\eta)\widehat{P_{k_1}f}(\xi-\eta,s)\widehat{P_{k_2}g}(\eta,s)\,d\eta\Big\|_{L^2_\xi}.
\end{split}
\end{equation}
Using \eqref{ok9}--\eqref{ok10} and the assumption \eqref{ders2},
\begin{equation}\label{ok52}
\begin{split}
&\|P_{k''}f\|_{L^2}+\|P_{k''}g\|_{L^2}\lesssim \min(2^{(1+\beta-\alpha)k''},2^{-N_0k''}),\\
&\|e^{-is\widetilde{\Lambda}_\mu}P_{k''}f\|_{L^\infty}+\|e^{-is\widetilde{\Lambda}_\nu}P_{k''}g\|_{L^\infty}\lesssim\min(2^{(1/2-\beta-\alpha)k''},2^{-6k''}) (1+s)^{-1-\beta},
\end{split}
\end{equation}
for any $k''\in\mathbb{Z}$. Using \eqref{ok52} and the description of the symbols $m_{\sigma;\mu,\nu}$ in Lemma \ref{tech1.3}, the expression in \eqref{ok51} is dominated by
\begin{equation*}
\begin{split}
C2^{k'_\sigma}\sum_{k_1,k_2\in\mathbb{Z},\,k_1\leq k_2,\,k'\leq k_2+4} \min(2^{(1+\beta-\alpha)k_2},2^{-(N_0-2)k_2})\cdot \min(2^{(1/2-\beta-\alpha)k_1},2^{-6k_1})(1+s)^{-1-\beta}\\
\lesssim 2^{k'_\sigma}(1+s)^{-1-\beta}\min(1,2^{-(N_0-5)k'}).
\end{split}
\end{equation*}
Moreover, if $k'\leq 0$, we use again \eqref{ok52} and Lemma \ref{tech1.3} to estimate the expression in \eqref{ok51} by
\begin{equation*}
\begin{split}
&C\sum_{k_1,k_2\in\mathbb{Z}}2^{3k'/2}\Big\|\int_{\mathbb{R}^3}e^{-is[\widetilde{\Lambda}_{\mu}(\xi-\eta)+\widetilde{\Lambda}_{\nu}(\eta)]}m_{\sigma;\mu,\nu}(\xi,\eta)\widehat{P_{k_1}f}(\xi-\eta,s)\widehat{P_{k_2}g}(\eta,s)\,d\eta\Big\|_{L^\infty_\xi}\\
&\lesssim 2^{3k'/2}2^{k'_\sigma}\sum_{k_1,k_2\in\mathbb{Z}}\min(2^{(1+\beta-\alpha)k_1},2^{-(N_0-2)k_1})\cdot \min(2^{(1+\beta-\alpha)k_2},2^{-(N_0-2)k_2})\\
&\lesssim 2^{3k'/2}2^{k'_\sigma}.
\end{split}
\end{equation*}
The desired bound \eqref{ok50} follows.

To prove \eqref{derv2} we use first Lemma \ref{tech1.3} and Lemma \ref{CZop} and decompose the functions $f,g$ in suitable atoms. It suffices to prove that if 
\begin{equation}\label{nxc2}
\|h_1\|_{Z\cap H^{N_0}}+\|h_2\|_{Z\cap H^{N_0}}\leq 1,
\end{equation}
and we decompose
\begin{equation*}
h_i=\sum_{(k_i,j_i)\in\mathcal{J}}h^i_{k_i,j_i},\qquad h^i_{k_i,j_i}:=P_{[k_i-2,k_i+2]}(\phii^{(k_i)}_{j_i}\cdot P_{k_i}h_i),\qquad i=1,2,
\end{equation*}
then
\begin{equation}\label{nxc3}
\sum_{(k_1,j_1),(k_2,j_2)\in\mathcal{J}}2^{k'}\Big|\varphi_{k'}(\xi)\int_{\mathbb{R}^3}e^{is[\Lambda_\sigma(\xi)-\widetilde{\Lambda}_{\mu}(\xi-\eta)-\widetilde{\Lambda}_{\nu}(\eta)]}\widehat{h^1_{k_1,j_1}}(\xi-\eta)\widehat{h^2_{k_2,j_2}}(\eta)\,d\eta\Big|\lesssim (1+s)^{-1+\beta/10}2^{-k'},
\end{equation}
for any $\xi\in\mathbb{R}^3$, $\mu,\nu\in\mathcal{I}_0$, $s\in\mathbb{R}$, and $k',\sigma$ as in \eqref{derv1}.

We use first only the $L^2$ bounds
\begin{equation}\label{nxc3.5}
\|h^1_{k_1,j_1}\|_{L^2}\lesssim \min(2^{-N_0 k_1},2^{(2\beta-\alpha)\widetilde{k_1}}2^{-(1-\beta)j_1}), \quad \|h^2_{k_2,j_2}\|_{L^2}\lesssim \min(2^{-N_0 k_2},2^{(2\beta-\alpha)\widetilde{k_2}}2^{-(1-\beta)j_2}),
\end{equation}
see \eqref{nxc2} and \eqref{mk15.5}. The full bound \eqref{nxc3} follows easily if $s^{1-\beta/10}\leq 2^{D^2}2^{-2k'}$. Assuming $s^{1-\beta/10}\geq 2^{D^2}2^{-2k'}$ we estimate easily
\begin{equation*}
\sum_{((k_1,j_1),(k_2,j_2))\in J_1}2^{k'}\Big|\varphi_{k'}(\xi)\int_{\mathbb{R}^3}e^{is[\Lambda_\sigma(\xi)-\widetilde{\Lambda}_{\mu}(\xi-\eta)-\widetilde{\Lambda}_{\nu}(\eta)]}\widehat{h^1_{k_1,j_1}}(\xi-\eta)\widehat{h^2_{k_2,j_2}}(\eta)\,d\eta\Big|\lesssim s^{-1}2^{-k'},
\end{equation*}
where
\begin{equation*}
J_1:=\{((k_1,j_1),(k_2,j_2))\in\mathcal{J}\times\mathcal{J}: 2^{\max(k_1,k_2)}\geq s^{2/N_0}\text{ or }2^{\max(j_1,j_2)}\geq s^{1+4\beta}2^{k'}\}.
\end{equation*}

Let
\begin{equation*}
J_2:=\{((k_1,j_1),(k_2,j_2))\in\mathcal{J}\times\mathcal{J}:\,2^{\max(k_1,k_2)}\leq s^{2/N_0}\text{ and }2^{\max(j_1,j_2)}\leq s^{1+4\beta}2^{k'}\},
\end{equation*}
and notice that $J_2$ has at most $C(\ln s)^4$ elements. Therefore, for \eqref{nxc3} it suffices to prove that
\begin{equation}\label{nxc4}
\Big|\varphi_{k'}(\xi)\int_{\mathbb{R}^3}e^{-is[\widetilde{\Lambda}_{\mu}(\xi-\eta)+\widetilde{\Lambda}_{\nu}(\eta)]}\widehat{h^1_{k_1,j_1}}(\xi-\eta)\widehat{h^2_{k_2,j_2}}(\eta)\,d\eta\Big|\lesssim 2^{-2k'}s^{-1+\beta/11},
\end{equation}
provided that $\xi\in\mathbb{R}^3$, $\mu,\nu\in\mathcal{I}_0$, $k'\in\mathbb{Z}\cap (-\infty,D]$, $s^{1-\beta/10}\geq 2^{D^2}2^{-2k'}$, and $((k_1,j_1),(k_2,j_2))\in J_2$.

Assume first that
\begin{equation}\label{nxc5}
2^{\max(j_1,j_2)}\geq 2^{-D^2}s^{1-\beta/11}2^{k'}.
\end{equation}
Without loss of generality, in proving \eqref{nxc4} we may assume that $j_1\leq j_2$. Then, using \eqref{sec5.8}, \eqref{sec5.82} and the assumption \eqref{nxc2}, we have
\begin{equation}\label{nxc5.5}
\|\widehat{h^2_{k_2,j_2}}\|_{L^1}\lesssim 2^{-(1+\beta)j_2}2^{3k_2/2}(2^{\alpha k_2}+2^{10k_2})^{-1}.
\end{equation}
Using \eqref{mk15.55}, $\|\widehat{h^1_{k_1,j_1}}\|_{L^\infty}\lesssim 2^{-\widetilde{k_1}/2}$. Using also \eqref{nxc3.5} we estimate the left-hand side of \eqref{nxc4} by
\begin{equation*}
C\min(\|\widehat{h^1_{k_1,j_1}}\|_{L^\infty}\|\widehat{h^2_{k_2,j_2}}\|_{L^1},\|\widehat{h^1_{k_1,j_1}}\|_{L^2}\|\widehat{h^2_{k_2,j_2}}\|_{L^2})\lesssim \min(2^{-\widetilde{k_1}/2}2^{-(1+\beta)j_2},2^{\widetilde{k_1}(1+\beta-\alpha)}2^{-(1-\beta)j_2})\lesssim 2^{-j_2}.
\end{equation*}
The desired bound \eqref{nxc4} follows if we assume \eqref{nxc5}.

Assume now that
\begin{equation}\label{nxc5.6}
2^{2\min(k_1,k_2)}\leq 2^{D^2}2^{-2k'}s^{-1+\beta/11}.
\end{equation}
Without loss of generality, in proving \eqref{nxc4} we may assume that $k_2\leq k_1$. Then, using \eqref{nxc5.5} we estimate the left-hand side of \eqref{nxc4} by
\begin{equation*}
C\|\widehat{h^1_{k_1,j_1}}\|_{L^\infty}\|\widehat{h^2_{k_2,j_2}}\|_{L^1}\lesssim 2^{-k_1/2}2^{5k_2/2}\lesssim 2^{2k_2},
\end{equation*}
as desired

Finally it remains to prove \eqref{nxc4} assuming that
\begin{equation}\label{nxc6}
2^{\max(j_1,j_2)}\leq 2^{-D^2}s^{1-\beta/11}2^{k'}\qquad\text{and}\qquad 2^{2\min(k_1,k_2)}\geq 2^{D^2}2^{-2k'}s^{-1+\beta/11}.
\end{equation}
In this case we would like to integrate by parts in $\eta$ to estimate the integral in \eqref{nxc4}.
Using the bounds \eqref{ln1} and \eqref{mk15.55},
\begin{equation}\label{nxc7}
\Big|\int_{\mathbb{R}^3}[1-\varphi_{\leq 0}(\delta^{-1}\Xi^{\mu,\nu}(\xi,\eta))]e^{-is[\widetilde{\Lambda}_{\mu}(\xi-\eta)+\widetilde{\Lambda}_{\nu}(\eta)]}\widehat{h^1_{k_1,j_1}}(\xi-\eta)\widehat{h^2_{k_2,j_2}}(\eta)\,d\eta\Big|\lesssim s^{-2}
\end{equation}
as long as
\begin{equation}\label{nxc7.5}
\delta\in(0,1],\quad s\delta\geq s^{\beta^2}2^{\max(j_1,j_2)},\quad s\delta\geq s^{\beta^2}\delta^{-1}2^{-\min(k_1,k_2,0)}.
\end{equation}
Therefore, letting
\begin{equation}\label{nxc9}
D(\xi,\delta):=\{\eta\in\mathbb{R}^3:\,|\eta|\in[2^{k_2-4},2^{k_2+4}],\,|\xi-\eta|\in [2^{k_1-4},2^{k_1+4}],\,|\Xi^{\mu,\nu}(\xi,\eta)|\leq 2\delta\},
\end{equation}
for \eqref{nxc4} it remains to prove that, for some $\delta$ satisfying \eqref{nxc7.5},
\begin{equation}\label{nxc10}
\Big|\varphi_{k'}(\xi)\int_{\mathbb{R}^3}\mathbf{1}_{D(\xi,\delta)}(\eta)\big|\widehat{h^1_{k_1,j_1}}(\xi-\eta)\big|\,\big|\widehat{h^2_{k_2,j_2}}(\eta)\big|\,d\eta\Big|\lesssim 2^{-2k'}s^{-1+\beta/11},
\end{equation}
provided that $\xi\in\mathbb{R}^3$, $\mu,\nu\in\mathcal{I}_0$, $k'\in\mathbb{Z}\cap (-\infty,D]$, and $((k_1,j_1),(k_2,j_2))\in J_2$ satisfies \eqref{nxc6}. Without loss of generality, we may assume that $k_2\leq k_1$.

We examine now the sets $D(\xi,\delta)$ defined in \eqref{nxc9}. Assume that $\mu=(\sigma_1\iota_1)$, $\nu=(\sigma_2\iota_2)$, $\sigma_1,\sigma_2\in\{i,e,b\}$, $\iota_1,\iota_2\in\{+,-\}$. Notice that
\begin{equation*}
\Xi^{\mu,\nu}(\xi,\eta)=-\iota_1\frac{\lambda'_{\sigma_1}(|\eta-\xi|)}{|\eta-\xi|}(\eta-\xi)-\iota_2\frac{\lambda'_{\sigma_2}(|\eta|)}{|\eta|}\eta=A(\eta-\xi)+B\eta,
\end{equation*}
where
\begin{equation*}
A:=-\iota_1\frac{\lambda'_{\sigma_1}(|\eta-\xi|)}{|\eta-\xi|},\qquad B:=-\iota_2\frac{\lambda'_{\sigma_2}(|\eta|)}{|\eta|}.
\end{equation*}
In view of Lemma \ref{tech99} (i), we have $\min(|A|,|B|)\gtrsim_{C_b,\varepsilon}2^{-\max(k_1,0)}$. Letting $\xi=se$, $e\in\mathbb{S}^2$, $s\in[2^{k'-2},2^{k'+2}]$, and $\eta=re+\eta'$, $r\in\mathbb{R}$, $\eta'\cdot e=0$, and assuming $\eta\in D(\xi,\delta)$, it follows that
\begin{equation*}
|(A+B)\eta'|\leq 2\delta,\qquad |(A+B)r-As|\leq 2\delta.
\end{equation*}
We let $\delta:=\max(s^{\beta^2-1}2^{\max(j_1,j_2)},s^{(\beta^2-1)/2}2^{-\min(k_2,0)/2})$, such that \eqref{nxc7.5} is satisfied. In view of \eqref{nxc6}, it follows that $|(A+B)r|\gtrsim_{C_b,\varepsilon}2^{-\max(k_1,0)}2^{k'}$, therefore $|A+B|\gtrsim_{C_b,\varepsilon}2^{-\max(k_1,0)}2^{k'}2^{-k_2}$. This shows that $|\eta'|\lesssim 2^{\max(k_1,0)}2^{-k'}2^{k_2}\delta$. In other words, we proved that if $\xi=se$, $e\in\mathbb{S}^2$, $s\in[2^{k'-2},2^{k'+2}]$, then
\begin{equation}\label{nxc11}
D(\xi,\delta)\subseteq\{\eta=re+\eta'\in\mathbb{R}^3:|r|+|\eta'|^2\leq 2^{k_2+4},\,\eta'\cdot e=0,\,|\eta'|\lesssim 2^{\max(k_1,0)}2^{-k'}2^{k_2}\delta\}.
\end{equation}

Using \eqref{nxc11} and the $L^\infty$ bounds $\|\widehat{h^1_{k_1,j_1}}\|_{L^\infty}\lesssim 2^{-k_1/2}$, $\|\widehat{h^2_{k_2,j_2}}\|_{L^\infty}\lesssim 2^{-k_2/2}$, we can bound the left-hand side of \eqref{nxc10} by
\begin{equation*}
C2^{-k_1/2}2^{-k_2/2}\cdot (2^{\max(k_1,0)}2^{-k'}2^{k_2}\delta)^22^{k_2}\lesssim 2^{2\min(k_2,0)}\delta^22^{-2k'}s^{8/N_0}.
\end{equation*}
This suffices to prove \eqref{nxc10} if $2^{\max(j_1,j_2)}2^{\min(k_2,0)}\leq s^{1/2}$. On the other hand, if $j_1\leq j_2$ and $2^{j_2}2^{\min(k_2,0)}\geq s^{1/2}$ then we estimate the left-hand side of \eqref{nxc10} by
\begin{equation*}
C2^{-k_1/2}\|\mathbf{1}_{D(\xi,\delta)}(\eta)\cdot \widehat{h^2_{k_2,j_2}}(\eta)\|_{L^1_\eta}\lesssim 2^{-k_1/2}\cdot 2^{k_2/2}2^{\max(k_1,0)}2^{-k'}2^{k_2}\delta\cdot 2^{-j_2}\lesssim 2^{-k'}s^{-1+\beta/11},
\end{equation*}
which also suffices to prove \eqref{nxc10}. Finally, if $j_1\geq j_2$ and $2^{j_1}2^{\min(k_2,0)}\geq s^{1/2}$ then we estimate the left-hand side of \eqref{nxc10} by
\begin{equation*}
C2^{-k_2/2}\|\mathbf{1}_{D(\xi,\delta)}(\eta)\cdot \widehat{h^1_{k_1,j_1}}(\xi-\eta)\|_{L^1_\eta}\lesssim 2^{-k_2/2}\cdot 2^{k_2/2}2^{\max(k_1,0)}2^{-k'}2^{k_2}\delta\cdot 2^{-j_1}\lesssim 2^{-k'}s^{-1+\beta/11},
\end{equation*}
which also suffices to prove \eqref{nxc10}. This completes the proof of the lemma.
\end{proof}

\section{Classification of resonances}\label{resonantsets}

We define the order $i<e<b$. Recall that we introduced a large number $D\gg(\varepsilon^{-1}+C_b)^{10}$, depending only on $\varepsilon$ and $C_b$. For $\sigma\in\{i,e,b\}$ and $\mu,\nu\in\mathcal{I}_0$, see definition \eqref{Iset},
\begin{equation}\label{jb1}
\mu=(\sigma_1\iota_1),\qquad \nu=(\sigma_2\iota_2),\qquad \sigma_1,\sigma_2\in\{i,e,b\},\qquad \iota_1,\iota_2\in\{+,-\},
\end{equation}
recall the definitions of the smooth functions $\Lambda_\sigma:\mathbb{R}^3\to(0,\infty)$, $\Phi^{\sigma;\mu,\nu}:\mathbb{R}^3\times\mathbb{R}^3\to\mathbb{R}$ and $\Xi^{\mu,\nu}:\mathbb{R}^3\times\mathbb{R}^3\to\mathbb{R}^3$,
\begin{equation}\label{jb2}
\begin{split}
&\Phi^{\sigma;\mu,\nu}(\xi,\eta)=\Lambda_\sigma(\xi)-\iota_1\Lambda_{\sigma_1}(\xi-\eta)-\iota_2\Lambda_{\sigma_2}(\eta),\\
&\Xi^{\mu,\nu}(\xi,\eta)=(\nabla_\eta\Phi^{\sigma;\mu,\nu})(\xi,\eta)=-\iota_1\nabla\Lambda_{\sigma_1}(\eta-\xi)-\iota_2\nabla\Lambda_{\sigma_2}(\eta).
\end{split}
\end{equation}
In this subsection we prove several lemmas describing the structure of almost resonant sets, which are the sets where both $|\Phi^{\sigma;\mu,\nu}(\xi,\eta)|$ and $|\Xi^{\mu,\nu}(\xi,\eta)|$ are small. Recall the sets
\begin{equation}\label{jb3}
\begin{split}
\mathcal{L}^{\sigma;\mu,\nu}_{k,k_1,k_2;\delta_1,\delta_2}=\{(\xi,\eta)\in\mathbb{R}^3\times\mathbb{R}^3:&\,|\xi|\in[2^{k-4},2^{k+4}],\,|\xi-\eta|\in[2^{k_1-4},2^{k_1+4}],\,|\eta|\in[2^{k_2-4},2^{k_2+4}],\\
&\,|\Xi^{\mu,\nu}(\xi,\eta)|\leq\delta_1,\,|\Phi^{\sigma;\mu,\nu}(\xi,\eta)|\leq\delta_2\}.
\end{split}
\end{equation}
defined for $\sigma\in\{i,e,b\}$, $\mu,\nu\in\mathcal{I}_0$, $k,k_1,k_2\in\mathbb{Z}$, $\delta_1,\delta_2\in(0,\infty)$. We define also
\begin{equation*}
\mathcal{L}_{k,k_1,k_2}:=\{(\xi,\eta)\in\mathbb{R}^3\times\mathbb{R}^3:\,|\xi|\in[2^{k-4},2^{k+4}],\,|\xi-\eta|\in[2^{k_1-4},2^{k_1+4}],\,|\eta|\in[2^{k_2-4},2^{k_2+4}]\}.
\end{equation*}

Given a phase $\Phi^{\sigma;\mu,\nu}$ and a set of phases $\mathcal{T}$, we denote $\Phi^{\sigma;\mu,\nu}\in\in \mathcal{T}$ if either $\Phi^{\sigma;\mu,\nu}\in \mathcal{T}$ or $\Phi^{\sigma;\nu,\mu}\in \mathcal{T}$, and $\Phi^{\sigma\mu,\nu}\notin\notin \mathcal{T}$ if neither possibility holds.

We show first that some phases do not contribute in the analysis of resonant interactions. We define the $39$ strongly elliptic phases,
\begin{equation*}
\begin{split}
\mathcal{T}_{Sell}:=\{&\Phi^{i;i+,e+}, \Phi^{i,i+,e-}, \Phi^{i;i+,b+}, \Phi^{i;i+,b-}, \Phi^{i,i-,e-}, \Phi^{i;i-,b-}, \Phi^{i,e+,e+}, \Phi^{i;e+,b+}, \Phi^{i,e-,e-}, \Phi^{i;e-,b-},\\
&\Phi^{i;b+,b+}, \Phi^{i;b-,b-}, \Phi^{e;i+,i-}, \Phi^{e;i+,e-}, \Phi^{e;i+,b-}, \Phi^{e;i-,i-}, \Phi^{e;i-,e-}, \Phi^{e;i-,b-}, \Phi^{e;e+,e+}, \Phi^{e;e+,e-},\\
& \Phi^{e;e+,b+}, \Phi^{e;e+,b-}, \Phi^{e;e-,e-}, \Phi^{e;e-,b-}, \Phi^{e;b+,b+}, \Phi^{e;b-,b-}, \Phi^{b;i+,i-}, \Phi^{b;i+,e-}, \Phi^{b;i+,b-}, \Phi^{b;i-,i-},\\
 &\Phi^{b;i-,e-}, \Phi^{b;i-,b-}, \Phi^{b;e+,e-}, \Phi^{b;e+,b-}, \Phi^{b;e-,e-}, \Phi^{b;e-,b-}, \Phi^{b;b+,b+}, \Phi^{b;b+,b-},\Phi^{b;b-,b-}\}.
\end{split}
\end{equation*}
We also define $4$ additional nonresonant phases
\begin{equation*}
\mathcal{T}_{NR}=\{\Phi^{e;i+,i+},\Phi^{e;b+,b-},\Phi^{b;i+,i+},\Phi^{b;i-,e+}\}.
\end{equation*}

\begin{lemma}\label{StronglyEll}
Assume that $\Phi^{\sigma;\mu,\nu}\in\in\mathcal{T}_{Sell}\cup\mathcal{T}_{NR}$. If
\begin{equation*}
\delta_1\le 2^{-D}2^{-4\max(k_1,k_2,0)},\,\,\delta_2\leq 2^{-D}2^{-\max(k_1,k_2,0)}\quad\hbox{then}\quad
\mathcal{L}^{\sigma;\mu,\nu}_{k,k_1,k_2;\delta_1,\delta_2}=\emptyset.
\end{equation*}
\end{lemma}
\begin{proof}[Proof of Lemma \ref{StronglyEll}] We claim that if $\Phi^{\sigma;\mu,\nu}\in\in\mathcal{T}_{Sell}$, then we have
\begin{equation}\label{nbc1}
\vert\Phi^{\sigma;\mu,\nu}(\xi,\eta)\vert\gtrsim_{\varepsilon,C_b} 2^{-\max(k_1,k_2,0)}.
\end{equation}
This would clearly suffice to prove the claim for the strongly elliptic phases.

If $(\iota_1,\iota_2)=(-,-)$, the proof of \eqref{nbc1} is a direct consequence of the fact that $\lambda_b\ge\lambda_e\ge 1$. To deal with the remaining $22$ phases in $\mathcal{T}_{Sell}$ we observe that, as a consequence of Lemma \ref{tech99}, we have, for any $r\in[0,\infty)$,
\begin{equation}\label{nbc2}
r\leq\lambda_i(r)\leq\lambda_e(r)\leq\lambda_b(r),\qquad \lambda_b(r)-\lambda_e(r)\gtrsim_{\varepsilon,C_b}r,\qquad \lambda_e(r)-\lambda_i(r)\gtrsim_{\varepsilon,C_b}1+r.
\end{equation}
In addition, for any $r_1,r_2\in[0,\infty)$,
\begin{equation}\label{nbc3}
\begin{split}
&\lambda_i(r_1)+\lambda_i(r_2)-\lambda_i(r_1+r_2)\geq 0,\\
&\lambda_e(r_1)+\lambda_e(r_2)-\lambda_e(r_1+r_2)\gtrsim_{\varepsilon,C_b}(1+\min(r_1,r_2))^{-1},\\
&\lambda_b(r_1)+\lambda_b(r_2)-\lambda_b(r_1+r_2)\gtrsim_{\varepsilon,C_b}(1+\min(r_1,r_2))^{-1}.
\end{split}
\end{equation}
Indeed, the first bound in \eqref{nbc3} follows from the formula $\lambda_i(r)=rq_i(r)$ in Lemma \ref{tech99}, and the fact that $q_i$ is decreasing. For the second bound in \eqref{nbc3} we use the fact that the function $r\to \lambda_e(r)-h_\varepsilon(r)$ is nonnegative and decreasing on $[0,\infty)$ (see \eqref {alo1}), therefore
\begin{equation*}
\lambda_e(r_1)+\lambda_e(r_2)-\lambda_e(r_1+r_2)\geq h_\varepsilon(r_1)+h_\varepsilon(r_2)-h_\varepsilon(r_1+r_2)\gtrsim_{\varepsilon,C_b}(1+\min(r_1,r_2))^{-1}.
\end{equation*}
The third bound follows directly from the definition.

Using \eqref{nbc2}, \eqref{nbc3}, and the monotonicity of the functions $\lambda_i,\lambda_e,\lambda_b$ on $[0,\infty)$, we can now prove lower bounds for the absolute values of the 22 phases in $\mathcal{T}_{Sell}$, which correspond to $(\iota_1,\iota_2)=(+,+)$,
\begin{equation*}
\begin{split}
-\Phi^{i;i+,e+}(\xi,\eta)&=[-\Lambda_i(\xi)+\Lambda_i(\xi-\eta)+\Lambda_i(\eta)]+[\Lambda_e(\eta)-\Lambda_i(\eta)]\gtrsim_{\varepsilon,C_b} 1,\\
-\Phi^{i;i+,b+}(\xi,\eta)&=[-\Lambda_i(\xi)+\Lambda_i(\xi-\eta)+\Lambda_i(\eta)]+[\Lambda_b(\eta)-\Lambda_i(\eta)]\gtrsim_{\varepsilon,C_b} 1,\\
-\Phi^{i;e+,e+}(\xi,\eta)&=[-\Lambda_i(\xi)+\Lambda_i(\xi-\eta)+\Lambda_i(\eta)]+[\Lambda_e(\xi-\eta)-\Lambda_i(\xi-\eta)]+[\Lambda_e(\eta)-\Lambda_i(\eta)]\gtrsim_{\varepsilon,C_b}1,\\
-\Phi^{i;e+,b+}(\xi,\eta)&=[-\Lambda_i(\xi)+\Lambda_i(\xi-\eta)+\Lambda_i(\eta)]+[\Lambda_e(\xi-\eta)-\Lambda_i(\xi-\eta)]+[\Lambda_b(\eta)-\Lambda_i(\eta)]\gtrsim_{\varepsilon,C_b}1,\\
-\Phi^{i;b+,b+}(\xi,\eta)&=[-\Lambda_i(\xi)+\Lambda_i(\xi-\eta)+\Lambda_i(\eta)]+[\Lambda_b(\xi-\eta)-\Lambda_i(\xi-\eta)]+[\Lambda_b(\eta)-\Lambda_i(\eta)]\gtrsim_{\varepsilon,C_b}1,\\
-\Phi^{e;e+,e+}(\xi,\eta)&=[-\Lambda_e(\xi)+\Lambda_e(\xi-\eta)+\Lambda_e(\eta)]\gtrsim_{\varepsilon,C_b}2^{-\max(k_1,k_2,0)},\\
-\Phi^{e;e+,b+}(\xi,\eta)&=[-\Lambda_e(\xi)+\Lambda_e(\xi-\eta)+\Lambda_e(\eta)]+[\Lambda_b(\eta)-\Lambda_e(\eta)]\gtrsim_{\varepsilon,C_b}2^{-\max(k_1,k_2,0)},\\
-\Phi^{e;b+,b+}(\xi,\eta)&=[-\Lambda_b(\xi)+\Lambda_b(\xi-\eta)+\Lambda_b(\eta)]+[\Lambda_b(\xi)-\Lambda_e(\xi)]\gtrsim_{\varepsilon,C_b}2^{-\max(k_1,k_2,0)},\\
-\Phi^{b;b+,b+}(\xi,\eta)&=[-\Lambda_b(\xi)+\Lambda_b(\xi-\eta)+\Lambda_b(\eta)]\gtrsim_{\varepsilon,C_b}2^{-\max(k_1,k_2,0)},
\end{split}
\end{equation*}
or $(\iota_1,\iota_2)=(+,-)$,
\begin{equation*}
\begin{split}
\Phi^{i;i+,e-}(\xi,\eta)&=[\Lambda_i(\xi)-\Lambda_i(\xi-\eta)+\Lambda_i(\eta)]+[\Lambda_e(\eta)-\Lambda_i(\eta)]\gtrsim_{\varepsilon,C_b} 1,\\
\Phi^{i;i+,b-}(\xi,\eta)&=[\Lambda_i(\xi)-\Lambda_i(\xi-\eta)+\Lambda_i(\eta)]+[\Lambda_b(\eta)-\Lambda_i(\eta)]\gtrsim_{\varepsilon,C_b} 1,\\
\Phi^{e;i+,i-}(\xi,\eta)&=[\Lambda_i(\xi)-\Lambda_i(\xi-\eta)+\Lambda_i(\eta)]+[\Lambda_e(\xi)-\Lambda_i(\xi)]\gtrsim_{\varepsilon,C_b} 1,\\
\Phi^{e;i+,e-}(\xi,\eta)&=[\Lambda_i(\xi)-\Lambda_i(\xi-\eta)+\Lambda_i(\eta)]+[\Lambda_e(\xi)-\Lambda_i(\xi)]+[\Lambda_e(\eta)-\Lambda_i(\eta)]\gtrsim_{\varepsilon,C_b} 1,\\
\Phi^{e;i+,b-}(\xi,\eta)&=[\Lambda_i(\xi)-\Lambda_i(\xi-\eta)+\Lambda_i(\eta)]+[\Lambda_e(\xi)-\Lambda_i(\xi)]+[\Lambda_b(\eta)-\Lambda_i(\eta)]\gtrsim_{\varepsilon,C_b} 1,\\
\Phi^{e;e+,e-}(\xi,\eta)&=[\Lambda_e(\xi)-\Lambda_e(\xi-\eta)+\Lambda_e(\eta)]\gtrsim_{\varepsilon,C_b}2^{-\max(k_1,k_2,0)},\\
\Phi^{e;e+,b-}(\xi,\eta)&=[\Lambda_e(\xi)-\Lambda_e(\xi-\eta)+\Lambda_e(\eta)]+[\Lambda_b(\eta)-\Lambda_e(\eta)]\gtrsim_{\varepsilon,C_b}2^{-\max(k_1,k_2,0)},\\
\Phi^{b;i+,i-}(\xi,\eta)&=[\Lambda_i(\xi)-\Lambda_i(\xi-\eta)+\Lambda_i(\eta)]+[\Lambda_b(\xi)-\Lambda_i(\xi)]\gtrsim_{\varepsilon,C_b} 1,\\
\Phi^{b;i+,e-}(\xi,\eta)&=[\Lambda_i(\xi)-\Lambda_i(\xi-\eta)+\Lambda_i(\eta)]+[\Lambda_b(\xi)-\Lambda_i(\xi)]+[\Lambda_e(\eta)-\Lambda_i(\eta)]\gtrsim_{\varepsilon,C_b} 1,\\
\Phi^{b;i+,b-}(\xi,\eta)&=[\Lambda_i(\xi)-\Lambda_i(\xi-\eta)+\Lambda_i(\eta)]+[\Lambda_b(\xi)-\Lambda_i(\xi)]+[\Lambda_b(\eta)-\Lambda_i(\eta)]\gtrsim_{\varepsilon,C_b} 1,\\
\Phi^{b;e+,e-}(\xi,\eta)&=[\Lambda_e(\xi)-\Lambda_e(\xi-\eta)+\Lambda_e(\eta)]+[\Lambda_b(\xi)-\Lambda_e(\xi)]\gtrsim_{\varepsilon,C_b}2^{-\max(k_1,k_2,0)},\\
\Phi^{b;e+,b-}(\xi,\eta)&=[\Lambda_b(\xi)-\Lambda_b(\xi-\eta)+\Lambda_b(\eta)]+[\Lambda_b(\xi-\eta)-\Lambda_e(\xi-\eta)]\gtrsim_{\varepsilon,C_b}2^{-\max(k_1,k_2,0)},\\
\Phi^{b;b+,b-}(\xi,\eta)&=[\Lambda_b(\xi)-\Lambda_b(\xi-\eta)+\Lambda_b(\eta)]\gtrsim_{\varepsilon,C_b}2^{-\max(k_1,k_2,0)}.
\end{split}
\end{equation*}
The desired lower bound \eqref{nbc1} follows for all phases $\Phi^{\sigma;\mu,\nu}\in\mathcal{T}_{Sell}$.

We now consider the phases $\Phi^{\sigma;\mu,\nu}\in\in\mathcal{T}_{NR}$. Assume first $\sigma_1=\sigma_2=i$. Then, since $\lambda_i^\prime\gtrsim_{C_b,\varepsilon} 1$, we see that smallness of $\vert\Xi^{\mu,\nu}(\xi,\eta)\vert$ implies that $(\xi-\eta)\cdot\eta\geq 0$. But then $\vert\xi\vert\ge\max(\vert\xi-\eta\vert,\vert\eta\vert)$ and, using \eqref{SimpleBdLie},
\begin{equation*}
\min[\Phi^{e;i+,i+}(\xi,\eta),\Phi^{b;i+,i+}(\xi,\eta)]\ge\lambda_e(\vert\xi\vert)-2\lambda_i(\vert\xi\vert)\ge h_\varepsilon(\vert\xi\vert)-2\sqrt{(T+1)(\varepsilon+1)}|\xi|\gtrsim_{\varepsilon,C_b} 1,
\end{equation*}
and the desired conclusion $\mathcal{L}^{\sigma;i+,i+}_{k,k_1,k_2;\delta_1,\delta_2}=\emptyset$, $\sigma\in\{e,b\}$, follows.

Assume now that $\Phi^{\sigma;\mu,\nu}\in\in\{\Phi^{e;b+,b-}\}$. If $\max(k_1,k_2)\le -D/10$, then, using \eqref{mk3.1}, $\Phi^{e;b+,b-}(\xi,\eta)\ge 1$. On the other hand, if $\max(k_1,k_2)\ge -D/10$, we see from smallness of $\vert\Xi^{b+,b-}(\xi,\eta)\vert$ that $\vert \xi\vert\le 2^{-D}$. But in this case,
\begin{equation*}
\vert\lambda_b(\vert\xi-\eta\vert)-\lambda_b(\vert\eta\vert)\vert\lesssim_{\varepsilon,C_b} \vert\xi\vert\leq\lambda_e(0)/2,
\end{equation*}
hence $\Phi^{e;b+,b-}(\xi,\eta)\geq 1$. 

Finally, assume that $\Phi^{\sigma;\mu,\nu}\in\in\{\Phi^{b;i-,e+}\}$. By symmetry we may assume that
\begin{equation*}
\Phi^{\sigma;\mu,\nu}(\xi,\eta)=\Phi^{b;i-,e+}(\xi,\eta)=\lambda_b(|\xi|)+\lambda_i(|\xi-\eta|)-\lambda_e(|\eta|).
\end{equation*}
The condition $|\Xi^{\mu,\nu}(\xi,\eta)|\leq 2^{-D}$ and Lemma \ref{tech99} (i) show that $2^{k_2}\gtrsim_{C_b,\varepsilon}1$. The condition $|\Phi^{\sigma;\mu,\nu}(\xi,\eta)|\leq 2^{-D}$ then shows that $|\eta|\geq |\xi|$. Since
\begin{equation*}
\lambda'_i(r)\leq \frac{\sqrt{1+T}}{\sqrt{1+\varepsilon}}\qquad\text{ and }\qquad\lambda'_e(r)\geq \frac{(1-\sqrt\varepsilon)Tr}{\sqrt{\varepsilon}\sqrt{1+Tr^2}},
\end{equation*}
for any $r\in[0,\infty)$, see Lemma \ref{tech99} (ii) and (iii), the restriction $|\Xi^{\mu,\nu}(\xi,\eta)|\leq 2^{-D}$ shows that $|\eta|\leq\sqrt{\varepsilon/T}$. Therefore $|\xi|\leq \sqrt{\varepsilon/T}$ and $|\xi-\eta|\leq 2\sqrt{\varepsilon/T}$. Since $r_\ast\geq T^{-1/2}$, see Lemma \ref{tech99} (i), it follows that $|\xi-\eta|\leq r_\ast/2$. Therefore $\lambda'_i$ is decreasing on the interval $[0,|\xi-\eta|]$ and we estimate, recalling that $2^{-D}\geq \big|\Xi^{\mu,\nu}(\xi,\eta)\big|\geq\big|\lambda'_i(|\xi-\eta|)-\lambda'_e(|\eta|)\big|$,
\begin{equation*}
\begin{split}
\Phi^{\sigma;\mu,\nu}(\xi,\eta)&=\int_0^{|\xi-\eta|}\lambda'_i(s)\,ds+\int_0^{|\xi|}\lambda'_b(s)\,ds-\int_0^{|\eta|}\lambda'_e(s)\,ds\\
&\geq C_{C_b,\varepsilon}^{-1}+|\xi-\eta|\lambda'_i(|\xi-\eta|)-(|\eta|-|\xi|)\lambda'_e(|\eta|)\\
&\gtrsim_{C_b,\varepsilon}1.
\end{split}
\end{equation*}
This provides the contradiction.
\end{proof} 

We consider now the remaining 20 phases, and define three sets of phases
\begin{equation}\label{ResPhases}
\begin{split}
\mathcal{T}_{A}:=&\{\Phi^{i;i-,e+}, \Phi^{i;i-,b+}, \Phi^{i;e+,b-}, \Phi^{i;e-,b+}, \Phi^{e; i+,e+}, \Phi^{e;i+,b+}, \Phi^{e; i-,b+}, \Phi^{e;e-,b+},\\
&\,\,\Phi^{b;i+,e+}, \Phi^{b;i+,b+}, \Phi^{b;e+,e+}, \Phi^{b;e+,b+},\Phi^{b;e-,b+}\},\\
\mathcal{T}_{B}:=&\{\Phi^{e;i+,e+},\Phi^{e;i-,e+},\Phi^{b;i+,b+},\Phi^{b;i-,b+}\},\\
\mathcal{T}_{C}:=&\{\Phi^{i;i+,i+},\Phi^{i;i+,i-},\Phi^{i;i-,i-},\Phi^{i;e+,e-},\Phi^{i;e+,b-},\Phi^{i;e-,b+},\Phi^{i;b+,b-}\}.
\end{split}
\end{equation}
Notice that some phases, such as $\Phi^{e;i+,e+}$, belong to more than one set. The set $\mathcal{T}_A$ corresponds to phases having nondegenerate stationary points on spheres, while the sets $\mathcal{T}_B,\mathcal{T}_C$ consist of phases with degenerate behavior around $0$ (in $\eta$, $\xi-\eta$, or $\xi$). More precisely:

\begin{proposition}\label{PropABC}
Assume $k,k_1,k_2\in\mathbb{Z}$ and that there is a point $(\xi,\eta)\in\mathcal{L}_{k,k_1,k_2}$ satisfying 
\begin{equation*}
|\Xi^{\mu,\nu}(\xi,\eta)|\leq \delta_1=2^{-10D}2^{-4\max(0,k_1,k_2)},\quad\vert\Phi^{\sigma;\mu,\nu}(\xi,\eta)\vert\le \delta_2=2^{-10D}2^{-\max(0,k_1,k_2)}.
\end{equation*}
Then one of the three following possibilities holds:

{\bf Case A}: $-D/2\le k,k_1,k_2\le D/2$ and $\Phi^{\sigma;\mu,\nu}\in\in\mathcal{T}_A$.

{\bf Case B}: $\min(k_1,k_2)\le -D/3$, $k\ge -D/4$ and $\Phi^{\sigma;\mu,\nu}\in\in\mathcal{T}_B$.

{\bf Case C}: $k\le -D/4$ and $\Phi^{\sigma;\mu,\nu}\in\in\mathcal{T}_C$.
\end{proposition}

\begin{proof}[Proof of Proposition \ref{PropABC}] We divide the proof in several steps:

{\bf{Step 1.}} Assume that
\begin{equation}\label{yut1}
\Phi^{\sigma;\mu,\nu}\in\in\{\Phi^{i;i+,i+},\Phi^{i;i+,i-},\Phi^{i;i-,i-}\}.
\end{equation}
These phases are only in the set $\mathcal{T}_C$, and we have to prove that if $\mathcal{L}^{\sigma;\mu,\nu}_{k,k_1,k_2;\delta_1,\delta_2}\neq\emptyset$ then
\begin{equation}\label{yut2}
k\leq -D/2.
\end{equation}

Assume for contradiction that $k\geq -D/2$. Using Lemma \ref{tech99} (iii) it is easy to see that if $\Phi^{\sigma;\mu,\nu}\in\in\{\Phi^{i;i-,i-}\}$ then $|\Phi^{\sigma;\mu,\nu}(\xi,\eta)|\gtrsim_{\varepsilon,C_b}\max(|\xi|,|\eta|,|\xi-\eta|)\gtrsim_{\varepsilon,C_b}2^{-D/2}$, which is not possible. On the other hand, using Lemma \ref{tech99} (iii), for any $r,s\in[0,\infty)$
\begin{equation*}
\lambda_i(r)+\lambda_i(s)-\lambda_i(r+s)=r(q_i(r)-q_i(r+s))+s(q_i(s)-q_i(r+s))\gtrsim_{\varepsilon,C_b}\frac{\min(r,s)(r+s)^2}{(1+(r+s)^2)(1+\min(r,s)^2)}.
\end{equation*}
Therefore if $\Phi^{\sigma;\mu,\nu}\in\in\{\Phi^{i;i+,i+},\Phi^{i;i+,i-}\}$, $k\geq -D/2$, and $\mathcal{L}^{\sigma;\mu,\nu}_{k,k_1,k_2;\delta_1,\delta_2}\neq\emptyset$, then $\min(k_1,k_2)\leq -5D$. On the other hand, if $\min(k_1,k_2)\leq -5D$ and $\max(k_1,k_2)\geq -D/2-10$, then we use the bound
\begin{equation*}
\lambda'_i(r)-\lambda'_i(s)\geq q_i(0)-C_{C_b,\varepsilon}r-q_i(s)\geq 2^{-2D},
\end{equation*}
whenever $0\leq r\leq 2^{-4D}$ and $s\geq 2^{-D}$, which is a consequence of Lemma \ref{tech99} (iii). Therefore, in this case $|\Xi^{\mu,\nu}(\xi,\eta)|\geq 2^{-2D}$ for all $(\xi,\eta)\in\mathcal{L}_{k,k_1,k_2}$, which provides the contradiction.

\medskip

{\bf{Step 2.}} Assume that
\begin{equation}\label{yut3}
\Phi^{\sigma;\mu,\nu}\in\in\{\Phi^{i;e+,e-},\Phi^{i;b+,b-}\}.
\end{equation}
These phases are only in the set $\mathcal{T}_C$, and it suffices to prove that if $\mathcal{L}^{\sigma;\mu,\nu}_{k,k_1,k_2;\delta_1,\delta_2}\neq\emptyset$ then
\begin{equation}\label{yut4}
k\leq -D/2.
\end{equation}

Assume for contradiction that $k\geq -D/2$. Since $\lambda'_\sigma(0)=0$, $\sigma\in\{e,b\}$, the restriction $|\Xi^{\mu,\nu}(\xi,\eta)|\leq \delta_1$ shows that $\min(k_1,k_2)\geq -D$. On the other hand, if $\min(k,k_1,k_2)\geq -D$ then, for any $(\xi,\eta)\in\mathcal{L}_{k,k_1,k_2}$ we have
\begin{equation*}
|\Xi^{\mu,\nu}(\xi,\eta)|\gtrsim_{C_b,\varepsilon}\big|\lambda'_\sigma(|\eta-\xi|)-\lambda'_\sigma(|\eta|)\big|+\max[\lambda'_\sigma(|\eta-\xi|),\lambda'_\sigma(|\eta|)]\Big|\frac{\eta-\xi}{|\eta-\xi|}-\frac{\eta}{|\eta|}\Big|\geq 2^D\delta_1,
\end{equation*}
which provides a contradiction.

\medskip

{\bf{Step 3.}} Assume that
\begin{equation}\label{yut6}
\Phi^{\sigma;\mu,\nu}\in\in\{\Phi^{i;e+,b-},\Phi^{i;e-,b+}\}.
\end{equation}
These phases are in the sets $\mathcal{T}_C$ and $\mathcal{T}_A$, and we have to prove that if $\mathcal{L}^{\sigma;\mu,\nu}_{k,k_1,k_2;\delta_1,\delta_2}\neq\emptyset$ then
\begin{equation*}
\text{ either }\qquad k\leq -D/4\qquad\text{ or }\qquad -D/2\leq k,k_1,k_2\leq D/2.
\end{equation*}
This is equivalent to proving that
\begin{equation}\label{yut7}
\text{ if }\quad k\geq -D/4\quad\text{ and }\quad\mathcal{L}^{\sigma;\mu,\nu}_{k,k_1,k_2;\delta_1,\delta_2}\neq\emptyset\quad\text{ then }\quad -D/2\leq k,k_1,k_2\leq D/2.
\end{equation}

Since
\begin{equation}\label{yut9}
\lim_{r\to\infty}\lambda'_i(r)=1,\qquad\lim_{r\to\infty}\lambda'_e(r)=\sqrt{T/\varepsilon},\qquad\lim_{r\to\infty}\lambda'_b(r)=\sqrt{C_b/\varepsilon},
\end{equation}
it is easy to see that the condition $\mathcal{L}^{\sigma;\mu,\nu}_{k,k_1,k_2;\delta_1,\delta_2}\neq\emptyset$ implies that $\max(k_1,k_2,k_3)\leq D/4$. Since $k\geq -D/4$ it follows that $\max(k_1,k_2)\geq -D/4-10$. Recall that $\lambda'_\sigma(0)=0$ and $\lambda''_\sigma(r)\approx_{C_b,\varepsilon}(1+r^2)^{-3/2}$, $\sigma\in\{e,b\}$ (see Lemma \ref{tech99} (i)). Since $|\Xi^{\mu,\nu}(\xi,\eta)|\leq \delta_1$ for some $(\xi,\eta)\in\mathcal{L}_{k,k_1,k_2}$ it follows that $\min(k_1,k_2)\geq -D/2$ as desired.

\medskip

{\bf{Step 4.}} Assume that
\begin{equation}\label{yuk10}
\Phi^{\sigma;\mu,\nu}\in\in\{\Phi^{e;i-,e+},\Phi^{b;i-,b+}\}.
\end{equation}
These phases are only in the set $\mathcal{T}_B$, and we have to prove that if $\mathcal{L}^{\sigma;\mu,\nu}_{k,k_1,k_2;\delta_1,\delta_2}\neq\emptyset$ then
\begin{equation}\label{yuk11}
k\geq -D/4\,\,\text{ and }\,\,\min(k_1,k_2)\leq-D/3.
\end{equation}

It is easy to see that $\max(k_1,k_2,k)\geq -D/10$; otherwise $|\Xi^{\mu,\nu}(\xi,\eta)|\geq 2^{-2D}$ for all $(\xi,\eta)\in\mathcal{L}_{k,k_1,k_2}$, in view of the fact that $\lambda'_e(0)=\lambda'_b(0)=0$ and $\lambda'_i(0)\approx_{C_b,\varepsilon}1$. Therefore it remains to prove that if $\mathcal{L}^{\sigma;\mu,\nu}_{k,k_1,k_2;\delta_1,\delta_2}\neq\emptyset$ then
\begin{equation}\label{yuk12}
\min(k_1,k_2)\leq-D/3.
\end{equation}

Assume, for contradiction, that \eqref{yuk12} fails. We may assume, without loss of generality, that
\begin{equation*}
\Phi^{\sigma;\mu,\nu}(\xi,\eta)=\Phi^{\sigma;i-,\sigma+}(\xi,\eta)=\lambda_\sigma(|\xi|)+\lambda_i(|\xi-\eta|)-\lambda_\sigma(|\eta|),
\end{equation*}
for $\sigma\in\{e,b\}$. We argue as in the proof of Lemma \ref{StronglyEll}, $\Phi^{\sigma;\mu,\nu}=\Phi^{b;i-,e+}\in\mathcal{T}_{NR}$. The conditions $|\Phi^{\sigma;\mu,\nu}(\xi,\eta)|\leq 2^{-10D}$ and $k_1\geq -D/3$ show that $|\eta|\geq |\xi|$. Since
\begin{equation*}
\lambda'_i(r)\leq \frac{\sqrt{1+T}}{\sqrt{1+\varepsilon}}\qquad\text{ and }\qquad\lambda'_b(r)\geq\lambda'_e(r)\geq \frac{(1-\sqrt\varepsilon)Tr}{\sqrt{\varepsilon}\sqrt{1+Tr^2}},
\end{equation*}
for any $r\in[0,\infty)$, see Lemma \ref{tech99} (ii) and (iii), the restriction $|\Xi^{\mu,\nu}(\xi,\eta)|\leq 2^{-D}$ shows that $|\eta|\leq\sqrt{\varepsilon/T}$. Therefore $|\xi|\leq \sqrt{\varepsilon/T}$ and $|\xi-\eta|\leq 2\sqrt{\varepsilon/T}$. Since $r_\ast\geq T^{-1/2}$, see Lemma \ref{tech99} (i), it follows that $|\xi-\eta|\leq r_\ast/2$. Therefore $\lambda'_i$ is decreasing on the interval $[0,|\xi-\eta|]$ and we estimate, recalling that $2^{-10D}\geq \big|\Xi^{\mu,\nu}(\xi,\eta)\big|\geq\big|\lambda'_i(|\xi-\eta|)-\lambda'_\sigma(|\eta|)\big|$ and $|\xi-\eta|\geq 2^{-D/2}$, 
\begin{equation*}
\Phi^{\sigma;\mu,\nu}(\xi,\eta)=\int_0^{|\xi-\eta|}\lambda'_i(s)\,ds-\int_{|\xi|}^{|\eta|}\lambda'_\sigma(s)\,ds\geq 2^{-2D}+|\xi-\eta|\lambda'_i(|\xi-\eta|)-(|\eta|-|\xi|)\lambda'_\sigma(|\eta|)\geq 2^{-4D}.
\end{equation*}
This provides the contradiction.

\medskip 

{\bf{Step 5.}} Assume that
\begin{equation}\label{yut10}
\Phi^{\sigma;\mu,\nu}\in\in\{\Phi^{e;i+,e+},\Phi^{b;i+,b+}\}.
\end{equation}
These phases are in the sets $\mathcal{T}_B$ and $\mathcal{T}_A$, and we have to prove that if $\mathcal{L}^{\sigma;\mu,\nu}_{k,k_1,k_2;\delta_1,\delta_2}\neq\emptyset$ then
\begin{equation}\label{yut11}
\text{ either }\qquad k\geq -D/4,\,\min(k_1,k_2)\leq-D/3\qquad\text{ or }\qquad -D/2\leq k,k_1,k_2\leq D/2.
\end{equation}

It is easy to see that $\max(k_1,k_2,k)\geq -D/10$; otherwise $|\Xi^{\mu,\nu}(\xi,\eta)|\geq 2^{-2D}$ for all $(\xi,\eta)\in\mathcal{L}_{k,k_1,k_2}$, in view of the fact that $\lambda'_e(0)=\lambda'_b(0)=0$ and $\lambda'_i(0)\approx_{C_b,\varepsilon}1$. Therefore, for \eqref{yut11} it suffices to prove that
\begin{equation}\label{yut12}
\text{ if }\qquad\min(k_1,k_2)\geq-D/3\qquad\text{ then }\qquad -D/2\leq k,k_1,k_2\leq D/2.
\end{equation}

In view of \eqref{yut9}, it is clear that $\min(k_1,k_2)\leq D/10$; otherwise $|\Xi^{\mu,\nu}(\xi,\eta)|\geq 2^{-2D}$ for all $(\xi,\eta)\in\mathcal{L}_{k,k_1,k_2}$. On the other hand, if $k\geq D/4$ then $\max(k_1,k_2)\geq D/4-10$, and one can use \eqref{yut9} again to see easily that this in contradiction with the assumption $\mathcal{L}^{\sigma;\mu,\nu}_{k,k_1,k_2;\delta_1,\delta_2}\neq\emptyset$. Therefore $\max(k,k_1,k_2)\leq D/2$, as desired.

Finally, for \eqref{yut12} it remains to prove that $k\geq -D/2$. Assuming, for contradiction, that $k\leq -D/2$ and recalling that $\max(k_1,k_2,k)\geq -D/10$, it follows that $\max(k_1,k_2)\geq -D/10$, $|k_1-k_2|\leq 10$. Therefore $|\Phi^{\sigma;\mu,\nu}(\xi,\eta)|\gtrsim_{C_b,\varepsilon}2^{-2D}$, which provides a contradiction.

\medskip 

{\bf{Step 6.}} Assume that
\begin{equation}\label{yut20}
\Phi^{\sigma;\mu,\nu}\in\in\{\Phi^{e;e-,b+},\Phi^{b;e+,e+},\Phi^{b;e+,b+},\Phi^{b;e-,b+}\}.
\end{equation}
These phases are only in the set $\mathcal{T}_A$, and we have to prove that if $\mathcal{L}^{\sigma;\mu,\nu}_{k,k_1,k_2;\delta_1,\delta_2}\neq\emptyset$ then
\begin{equation}\label{yut21}
-D/2\leq k,k_1,k_2
\end{equation}
and
\begin{equation}\label{yut22}
k,k_1,k_2\leq D/2.
\end{equation}

We prove first \eqref{yut21}. We notice that $\max(k_1,k_2)\geq -D/10$; otherwise $|\Phi^{\sigma;\mu,\nu}(\xi,\eta)|\gtrsim_{C_b,\varepsilon}1$ for any $(\xi,\eta)\in\mathcal{L}_{k,k_1,k_2}$, since $\lambda_e(0)=\lambda_b(0)=\sqrt{1+\varepsilon^{-1}}$. This implies that $\min(k_1,k_2)\geq -D/4$; otherwise $|\Xi^{\mu,\nu}(\xi,\eta)|\gtrsim_{C_b,\varepsilon}2^{-D}$ for any $(\xi,\eta)\in\mathcal{L}_{k,k_1,k_2}$, since $\lambda'_e(r)\approx_{C_b,\varepsilon}\min(r,1)$ and $\lambda'_b(r)\approx_{C_b,\varepsilon}\min(r,1)$. 

To complete the proof of \eqref{yut21}, assume, for contradiction, that $k\leq -D/2$, therefore $\max(k_1,k_2)\geq -D/10$, $|k_1-k_2|\leq 10$. If $\Phi^{\sigma;\mu,\nu}\in\in\{\Phi^{b;e+,e+},\Phi^{b;e+,b+}\}$ then $|\Phi^{\sigma;\mu,\nu}(\xi,\eta)|\gtrsim_{C_b,\varepsilon}1$ for any $(\xi,\eta)\in\mathcal{L}_{k,k_1,k_2}$, in contradiction with the assumption $\mathcal{L}^{\sigma;\mu,\nu}_{k,k_1,k_2;\delta_1,\delta_2}\neq\emptyset$. On the other hand, if $\Phi^{\sigma;\mu,\nu}\in\in\{\Phi^{e;e-,b+},\Phi^{b;e-,b+}\}$ then $|\Xi^{\mu,\nu}(\xi,\eta)|\gtrsim_{C_b,\varepsilon}2^{-2D}$ for any $(\xi,\eta)\in\mathcal{L}_{k,k_1,k_2}$, which is again in contradiction with the assumption $\mathcal{L}^{\sigma;\mu,\nu}_{k,k_1,k_2;\delta_1,\delta_2}\neq\emptyset$. This last bound is a consequence of the estimate
\begin{equation}\label{yut23}
\lambda'_b(r)-\lambda'_e(r)\gtrsim_{C_b,\varepsilon}\min(1,r),\qquad\text{ for any }r\geq 0,
\end{equation}
which follows from Lemma \ref{tech99} (ii). This completes the proof of \eqref{yut21}.

We prove now \eqref{yut22}. We notice first that $\min(k,k_1,k_2)\leq D/10$; otherwise either $|\Xi^{\mu,\nu}(\xi,\eta)|\gtrsim_{C_b,\varepsilon}1$ or $|\Phi^{\sigma;\mu,\nu}(\xi,\eta)|\gtrsim_{C_b,\varepsilon}1$ for any $(\xi,\eta)\in\mathcal{L}_{k,k_1,k_2}$, using \eqref{yut9}. Assuming, for contradiction, that \eqref{yut22} fails, we need to consider two cases:
\begin{equation}\label{yut25}
\begin{split}
\text{ either }\qquad &k\leq D/10,\quad\max(k_1,k_2)\geq D/2,\quad |k_1-k_2|\leq 10,\\
\text{ or }\qquad &\min(k_1,k_2)\leq D/10,\quad\max(k,k_1,k_2)\geq D/2,\quad |k-\max(k_1,k_2)|\leq 10.
\end{split}
\end{equation}
In the first case, we use \eqref{yut9} to see that $|\Xi^{\mu,\nu}(\xi,\eta)|\gtrsim_{C_b,\varepsilon}1$ for any $(\xi,\eta)\in\mathcal{L}_{k,k_1,k_2}$ if $\Phi^{\sigma;\mu,\nu}\in\in\{\Phi^{e;e-,b+},\Phi^{b;e+,b+},\Phi^{b;e-,b+}\}$. We also notice that $|\Xi^{\mu,\nu}(\xi,\eta)|\gtrsim_{C_b,\varepsilon}1$ for any $(\xi,\eta)\in\mathcal{L}_{k,k_1,k_2}$ if $\Phi^{\sigma;\mu,\nu}\in\in\{\Phi^{b;e+,e+}\}$, which completes the contradiction in this case.

Assume now that the inequalities in the second line of \eqref{yut25} hold. By symmetry we may assume that $k_1=\min(k_1,k_2)$. In view of \eqref{yut9} it is clear that $|\Xi^{\mu,\nu}(\xi,\eta)|\gtrsim_{C_b,\varepsilon}1$ if $(\xi,\eta)\in\mathcal{L}_{k,k_1,k_2}$ and $\Phi^{\sigma;\mu,\nu}\in\{\Phi^{e;e-,b+},\Phi^{b;e+,b+},\Phi^{b;e-,b+}\}$. Also, using \eqref{nbc2} it is clear that $|\Phi^{\sigma;\mu,\nu}(\xi,\eta)|\gtrsim_{C_b,\varepsilon}1$ if $(\xi,\eta)\in\mathcal{L}_{k,k_1,k_2}$ and $\Phi^{\sigma;\mu,\nu}\in\{\Phi^{e;,b+,e-},\Phi^{b;e+,e+},\Phi^{b;b+,e+},\Phi^{b;b+,e-}\}$. This completes the contradiction in this case as well, and the desired bound \eqref{yut22} follows.

\medskip

{\bf{Step 7.}} Assume that
\begin{equation}\label{yut30}
\Phi^{\sigma;\mu,\nu}\in\in\{\Phi^{i;i-,e+}, \Phi^{i;i-,b+}, \Phi^{e;i+,b+}, \Phi^{e; i-,b+}, \Phi^{b;i+,e+}\}.
\end{equation}
These phases are only in the set $\mathcal{T}_A$, and we have to prove that if $\mathcal{L}^{\sigma;\mu,\nu}_{k,k_1,k_2;\delta_1,\delta_2}\neq\emptyset$ then
\begin{equation}\label{yut31}
-D/2\leq k,k_1,k_2
\end{equation}
and
\begin{equation}\label{yut32}
k,k_1,k_2\leq D/2.
\end{equation}

To prove \eqref{yut32}, assume, for contradiction, that $\max(k_1,k_2)\geq D/4$. By symmetry, we may assume also $k_1\leq k_2$. Using \eqref{yut9} it is easy to see that $|\Xi^{\mu,\nu}(\xi,\eta)|\gtrsim_{C_b,\varepsilon}1$ if $(\xi,\eta)\in\mathcal{L}_{k,k_1,k_2}$ and $\Phi^{\sigma;\mu,\nu}\in\{\Phi^{i;i-,e+}, \Phi^{i;i-,b+}, \Phi^{e;i+,b+}, \Phi^{e; i-,b+}, \Phi^{b;i+,e+}\}$. On the other hand, if 
\begin{equation*}
\Phi^{\sigma;\mu,\nu}\in\{\Phi^{i;,e+,i-}, \Phi^{i;b+,i-}, \Phi^{e;b+,i+}, \Phi^{e;b+,i-}, \Phi^{b;e+,i+}\}
\end{equation*}
then, using again \eqref{yut9} and the smallness of $|\Xi^{\mu,\nu}(\xi,\eta)|$, we necessarily have $k_1\leq D/10$, $|k-k_2|\leq 10$. In this case, however, $|\Phi^{\sigma;\mu,\nu}(\xi,\eta)|\gtrsim_{C_b,\varepsilon}1$, as a consequence of \eqref{yut9}. This completes the proof of the contradiction.

We prove now \eqref{yut31}. We notice that $\max(k_1,k_2)\geq -D/10$; otherwise $|\Xi^{\mu,\nu}(\xi,\eta)|\gtrsim_{C_b,\varepsilon}1$ for any $(\xi,\eta)\in\mathcal{L}_{k,k_1,k_2}$, since $\lambda'_e(0)=\lambda'_b(0)=0,\,\lambda'_i(0)\approx_{C_b,\varepsilon}1$. Assume, for contradiction, that \eqref{yut31} fails. We may assume by symmetry that $k_1\leq k_2$ and need to consider two cases:
\begin{equation}\label{yut35}
\begin{split}
\text{ either }\qquad &k_1\leq -D/2,\quad k_2\geq -D/10,\quad |k-k_2|\leq 10,\\
\text{ or }\qquad &k\leq -D/2,\quad k_2\geq -D/10,\quad |k_1-k_2|\leq 10.
\end{split}
\end{equation}

Assume first that the inequalities in the first line of \eqref{yut35} hold. Since $\lambda'_i(r)\approx_{C_b,\varepsilon}1$ and $\lambda'_b(0)=\lambda'_e(0)=0$, it is easy to see that $|\Xi^{\mu,\nu}(\xi,\eta)|\gtrsim_{C_b,\varepsilon}2^{-2D}$ if $(\xi,\eta)\in\mathcal{L}_{k,k_1,k_2}$ and 
\begin{equation*}
\Phi^{\sigma;\mu,\nu}\in\{\Phi^{i;e+,i-}, \Phi^{i;b+,i-}, \Phi^{e;b+,i+}, \Phi^{e; b+,i-}, \Phi^{b;e+,i+},\}.
\end{equation*}
On the other hand, using \eqref{nbc2}, $|\Phi^{\sigma;\mu,\nu}(\xi,\eta)|\gtrsim_{C_b,\varepsilon}2^{-2D}$ if $(\xi,\eta)\in\mathcal{L}_{k,k_1,k_2}$ and 
\begin{equation*}
\Phi^{\sigma;\mu,\nu}\in\{\Phi^{i;i-,e+}, \Phi^{i;i-,b+}, \Phi^{e;i+,b+}, \Phi^{e; i-,b+}, \Phi^{b;i+,e+},\}.
\end{equation*}
The desired contradiction follows in this case.

Finally, assume that the inequalities in the second line of \eqref{yut35} hold. Using \eqref{nbc2} it is easy to see that $|\Phi^{\sigma;\mu,\nu}(\xi,\eta)|\gtrsim_{C_b,\varepsilon}2^{-2D}$ if $(\xi,\eta)\in\mathcal{L}_{k,k_1,k_2}$ and $\Phi^{\sigma;\mu,\nu}\in\in\{\Phi^{i;i-,e+}, \Phi^{i;i-,b+}, \Phi^{e;i+,b+}, \Phi^{b;i+,e+}\}$. On the other hand, if $\Phi^{\sigma;\mu,\nu}\in\in\{\Phi^{e; i-,b+}\}$, then the contradiction follows by the same argument as in Step 4. This completes the proof of the proposition.
\end{proof}

\section{The multipliers $m_{e;\mu,\nu}$, $m_{i;\mu,\nu}$, and $m_{b;\mu,\nu}$} \label{multiex}

For the sake of completeness, in this section we compute explicitly the multipliers $m_{e;\mu,\nu}$, $m_{i;\mu,\nu}$, and $m_{b,\alpha;\mu,\nu}$. The precise  formulas are not used at any other place in the paper; the information we use about these multipliers is summarized in Lemma \ref{tech1.3}.

Let
\begin{equation}\label{pla21}
\begin{split}
&U_{e+}:=U_e,\quad U_{e-}:=\overline{U_e},\quad U_{i+}:=U_i,\quad U_{i-}:=\overline{U_i},\quad [U_{b+}]_\alpha=[U_{b}]_\alpha,\quad [U_{b-}]_\alpha=\overline{[U_b]_\alpha}\\
&\widetilde{U_\mu}:=(1+R^2)^{-1/2}U_{\mu},\quad \mu\in\{e+,e-,i+,i-\}.
\end{split}
\end{equation}
It follows from \eqref{pla20} that
\begin{equation*}
\begin{split}
\varepsilon nR_\al h=&-i\frac{|\nabla|}{\Lambda_e}\widetilde{U_{e+}}\cdot R_\al\widetilde{U_{e+}}+i\Big(\frac{|\nabla|}{\Lambda_e}\widetilde{U_{e+}}\cdot R_\al\widetilde{U_{e-}}-R_\al\widetilde{U_{e+}}\cdot \frac{|\nabla|}{\Lambda_e}\widetilde{U_{e-}}\Big)+i\frac{|\nabla|}{\Lambda_e}\widetilde{U_{e-}}\cdot R_\al\widetilde{U_{e-}}\\
&-i\frac{|\nabla|R}{\Lambda_i}\widetilde{U_{i+}}\cdot R_\al R\widetilde{U_{i+}}+i\Big(\frac{|\nabla|R}{\Lambda_i}\widetilde{U_{i+}}\cdot R_\al R\widetilde{U_{i-}}-R_\al R\widetilde{U_{i+}}\cdot \frac{|\nabla|R}{\Lambda_i}\widetilde{U_{i-}}\Big)+i\frac{|\nabla|R}{\Lambda_i}\widetilde{U_{i-}}\cdot R_\al R\widetilde{U_{i-}}\\
&+i\Big(\frac{|\nabla|}{\Lambda_e}\widetilde{U_{e+}}\cdot R_\al R\widetilde{U_{i+}}+R_\al\widetilde{U_{e+}}\cdot\frac{|\nabla|R}{\Lambda_i}\widetilde{U_{i+}}\Big)+i\Big(-\frac{|\nabla|}{\Lambda_e}\widetilde{U_{e+}}\cdot R_\al R\widetilde{U_{i-}}+R_\al\widetilde{U_{e+}}\cdot\frac{|\nabla|R}{\Lambda_i}\widetilde{U_{i-}}\Big)\\
&+i\Big(\frac{|\nabla|}{\Lambda_e}\widetilde{U_{e-}}\cdot R_\al R\widetilde{U_{i+}}-R_\al\widetilde{U_{e-}}\cdot\frac{|\nabla|R}{\Lambda_i}\widetilde{U_{i+}}\Big)-i\Big(\frac{|\nabla|}{\Lambda_e}\widetilde{U_{e-}}\cdot R_\al R\widetilde{U_{i-}}+R_\al\widetilde{U_{e-}}\cdot\frac{|\nabla|R}{\Lambda_i}\widetilde{U_{i-}}\Big),
\end{split}
\end{equation*}
and
\begin{equation*}
\begin{split}
\varepsilon R_\al hR_\al h=&-R_\al \widetilde{U_{e+}}\cdot R_\al \widetilde{U_{e+}}+2R_\al \widetilde{U_{e+}}\cdot R_\al \widetilde{U_{e-}}-R_\al \widetilde{U_{e-}}\cdot R_\al \widetilde{U_{e-}}\\
&-R_\al R\widetilde{U_{i+}}\cdot R_\al R\widetilde{U_{i+}}+2R_\al R\widetilde{U_{i+}}\cdot R_\al R\widetilde{U_{i-}}-R_\al R\widetilde{U_{i-}}\cdot R_\al R\widetilde{U_{i-}}\\
&+2R_\al \widetilde{U_{e+}}\cdot R_\al R\widetilde{U_{i+}}-2R_\al \widetilde{U_{e+}}\cdot R_\al R\widetilde{U_{i-}}-2R_\al \widetilde{U_{e-}}\cdot R_\al R\widetilde{U_{i+}}+2R_\al \widetilde{U_{e-}}\cdot R_\al R\widetilde{U_{i-}}.
\end{split}
\end{equation*}
Moreover
\begin{equation*}
\begin{split}
\rho R_\al g=&-i\frac{|\nabla|R}{\Lambda_e}\widetilde{U_{e+}}\cdot R_\al R\widetilde{U_{e+}}+i\Big(\frac{|\nabla|R}{\Lambda_e}\widetilde{U_{e+}}\cdot R_\al R\widetilde{U_{e-}}-R_\al R\widetilde{U_{e+}}\cdot \frac{|\nabla|R}{\Lambda_e}\widetilde{U_{e-}}\Big)+i\frac{|\nabla|R}{\Lambda_e}\widetilde{U_{e-}}\cdot R_\al R\widetilde{U_{e-}}\\
&-i\frac{|\nabla|}{\Lambda_i}\widetilde{U_{i+}}\cdot R_\al \widetilde{U_{i+}}+i\Big(\frac{|\nabla|}{\Lambda_i}\widetilde{U_{i+}}\cdot R_\al \widetilde{U_{i-}}-R_\al \widetilde{U_{i+}}\cdot \frac{|\nabla|}{\Lambda_i}\widetilde{U_{i-}}\Big)+i\frac{|\nabla|}{\Lambda_i}\widetilde{U_{i-}}\cdot R_\al \widetilde{U_{i-}}\\
&-i\Big(\frac{|\nabla|R}{\Lambda_e}\widetilde{U_{e+}}\cdot R_\al \widetilde{U_{i+}}+R_\al R\widetilde{U_{e+}}\cdot\frac{|\nabla|}{\Lambda_i}\widetilde{U_{i+}}\Big)+i\Big(\frac{|\nabla|R}{\Lambda_e}\widetilde{U_{e+}}\cdot R_\al \widetilde{U_{i-}}-R_\al R\widetilde{U_{e+}}\cdot\frac{|\nabla|}{\Lambda_i}\widetilde{U_{i-}}\Big)\\
&+i\Big(-\frac{|\nabla|R}{\Lambda_e}\widetilde{U_{e-}}\cdot R_\al \widetilde{U_{i+}}+R_\al R\widetilde{U_{e-}}\cdot\frac{|\nabla|}{\Lambda_i}\widetilde{U_{i+}}\Big)+i\Big(\frac{|\nabla|R}{\Lambda_e}\widetilde{U_{e-}}\cdot R_\al \widetilde{U_{i-}}+R_\al R\widetilde{U_{e-}}\cdot\frac{|\nabla|}{\Lambda_i}\widetilde{U_{i-}}\Big),
\end{split}
\end{equation*}
and
\begin{equation*}
\begin{split}
R_\al gR_\al g=&-R_\al R\widetilde{U_{e+}}\cdot R_\al R\widetilde{U_{e+}}+2R_\al R\widetilde{U_{e+}}\cdot R_\al R\widetilde{U_{e-}}-R_\al R\widetilde{U_{e-}}\cdot R_\al R\widetilde{U_{e-}}\\
&-R_\al \widetilde{U_{i+}}\cdot R_\al \widetilde{U_{i+}}+2R_\al \widetilde{U_{i+}}\cdot R_\al \widetilde{U_{i-}}-R_\al \widetilde{U_{i-}}\cdot R_\al \widetilde{U_{i-}}\\
&-2R_\al R\widetilde{U_{e+}}\cdot R_\al \widetilde{U_{i+}}+2R_\al R\widetilde{U_{e+}}\cdot R_\al \widetilde{U_{i-}}+2R_\al R\widetilde{U_{e-}}\cdot R_\al \widetilde{U_{i+}}-2R_\al R\widetilde{U_{e-}}\cdot R_\al \widetilde{U_{i-}}.
\end{split}
\end{equation*}

We also have the terms involving the magnetic field:
\begin{equation*}
\begin{split}
nA_\alpha&=-\varepsilon^{-1/2}\vert\nabla\vert\Lambda_e^{-1}\widetilde{U_{e+}}\cdot\Lambda_b^{-1}\left[U_{b+}\right]_\alpha-\varepsilon^{-1/2}\vert\nabla\vert\Lambda_e^{-1}\widetilde{U_{e+}}\cdot\Lambda_b^{-1}\left[U_{b-}\right]_\alpha\\
&-\varepsilon^{-1/2}\vert\nabla\vert\Lambda_e^{-1}\widetilde{U_{e-}}\cdot\Lambda_b^{-1}\left[U_{b+}\right]_\alpha-\varepsilon^{-1/2}\vert\nabla\vert\Lambda_e^{-1}\widetilde{U_{e-}}\cdot\Lambda_b^{-1}\left[U_{b-}\right]_\alpha\\
&+\varepsilon^{-1/2}\vert\nabla\vert\Lambda_i^{-1}R\widetilde{U_{i+}}\cdot\Lambda_b^{-1}\left[U_{b+}\right]_\alpha+\varepsilon^{-1/2}\vert\nabla\vert\Lambda_i^{-1}R\widetilde{U_{i+}}\cdot\Lambda_b^{-1}\left[U_{b-}\right]_\alpha\\
&+\varepsilon^{-1/2}\vert\nabla\vert\Lambda_i^{-1}R\widetilde{U_{i-}}\cdot\Lambda_b^{-1}\left[U_{b+}\right]_\alpha+\varepsilon^{-1/2}\vert\nabla\vert\Lambda_i^{-1}R\widetilde{U_{i-}}\cdot\Lambda_b^{-1}\left[U_{b-}\right]_\alpha,
\end{split}
\end{equation*}
and
\begin{equation*}
\begin{split}
\rho A_\alpha&=\vert\nabla\vert\Lambda_e^{-1}R\widetilde{U_{e+}}\cdot\Lambda_b^{-1}\left[U_{b+}\right]_\alpha+\vert\nabla\vert\Lambda_e^{-1}R\widetilde{U_{e+}}\cdot\Lambda_b^{-1}\left[U_{b-}\right]_\alpha\\
&+\vert\nabla\vert\Lambda_e^{-1}R\widetilde{U_{e-}}\cdot\Lambda_b^{-1}\left[U_{b+}\right]_\alpha+\vert\nabla\vert\Lambda_e^{-1}R\widetilde{U_{e-}}\cdot\Lambda_b^{-1}\left[U_{b-}\right]_\alpha\\
&+\vert\nabla\vert\Lambda_i^{-1}\widetilde{U_{i+}}\cdot\Lambda_b^{-1}\left[U_{b+}\right]_\alpha+\vert\nabla\vert\Lambda_i^{-1}\widetilde{U_{i+}}\cdot\Lambda_b^{-1}\left[U_{b-}\right]_\alpha\\
&+\vert\nabla\vert\Lambda_i^{-1}\widetilde{U_{i-}}\cdot\Lambda_b^{-1}\left[U_{b+}\right]_\alpha+\vert\nabla\vert\Lambda_i^{-1}\widetilde{U_{i-}}\cdot\Lambda_b^{-1}\left[U_{b-}\right]_\alpha,
\end{split}
\end{equation*}
and
\begin{equation*}
\begin{split}
(R_\alpha h)\cdot A_\alpha=&i\varepsilon^{-1/2}R_\alpha\widetilde{U_{e+}}\cdot\Lambda_b^{-1}\left[U_{b+}\right]_\alpha+i\varepsilon^{-1/2}R_\alpha\widetilde{U_{e+}}\cdot\Lambda_b^{-1}\left[U_{b-}\right]_\alpha\\
&-i\varepsilon^{-1/2}R_\alpha\widetilde{U_{e-}}\cdot\Lambda_b^{-1}\left[U_{b+}\right]_\alpha-i\varepsilon^{-1/2}R_\alpha\widetilde{U_{e-}}\cdot\Lambda_b^{-1}\left[U_{b-}\right]_\alpha\\
&-i\varepsilon^{-1/2}R R_\alpha\widetilde{U_{i+}}\cdot\Lambda_b^{-1}\left[U_{b+}\right]_\alpha-i\varepsilon^{-1/2}R R_\alpha\widetilde{U_{i+}}\cdot\Lambda_b^{-1}\left[U_{b-}\right]_\alpha\\
&+i\varepsilon^{-1/2}R R_\alpha\widetilde{U_{i-}}\cdot\Lambda_b^{-1}\left[U_{b+}\right]_\alpha+i\varepsilon^{-1/2}R R_\alpha\widetilde{U_{i-}}\cdot\Lambda_b^{-1}\left[U_{b-}\right]_\alpha,\\
\end{split}
\end{equation*}
while
\begin{equation*}
\begin{split}
(R_\alpha g)\cdot A_\alpha=&-iRR_\alpha\widetilde{U_{e+}}\cdot\Lambda_b^{-1}\left[U_{b+}\right]_\alpha-iRR_\alpha\widetilde{U_{e+}}\cdot\Lambda_b^{-1}\left[U_{b-}\right]_\alpha\\
&+iR R_\alpha\widetilde{U_{e-}}\cdot\Lambda_b^{-1}\left[U_{b+}\right]_\alpha+iRR_\alpha\widetilde{U_{e-}}\cdot\Lambda_b^{-1}\left[U_{b-}\right]_\alpha\\
&-iR_\alpha\widetilde{U_{i+}}\cdot\Lambda_b^{-1}\left[U_{b+}\right]_\alpha-iR_\alpha\widetilde{U_{i+}}\cdot\Lambda_b^{-1}\left[U_{b-}\right]_\alpha\\
&+iR_\alpha\widetilde{U_{i-}}\cdot\Lambda_b^{-1}\left[U_{b+}\right]_\alpha+iR_\alpha\widetilde{U_{i-}}\cdot\Lambda_b^{-1}\left[U_{b-}\right]_\alpha.
\end{split}
\end{equation*}
Finally,
\begin{equation*}
\begin{split}
A_\alpha\cdot A_\alpha&=\Lambda_b^{-1}\left[U_{b+}\right]_\alpha\cdot\Lambda_b^{-1}\left[U_{b+}\right]_\alpha+2\Lambda_b^{-1}\left[U_{b+}\right]_\alpha\cdot\Lambda_b^{-1}\left[U_{b-}\right]_\alpha+\Lambda_b^{-1}\left[U_{b-}\right]_\alpha\cdot\Lambda_b^{-1}\left[U_{b-}\right]_\alpha.
\end{split}
\end{equation*}

Let
\begin{equation}\label{pla22}
\begin{split}
\mathcal{I}:=&\big\{(e+,e+),(e+,e-),(e-,e-),(i+,i+),(i+,i-),(i-,i-),(b+\alpha,b+\beta),(b+\alpha,b-\beta),(b-\alpha,b-\beta)\\
 &(e+,i+),(e+,i-),(e-,i+),(e-,i-),(e+,b+\alpha),(e+,b-\alpha),(e-,b+\alpha),(e-,b-\alpha),\\
 &(i+,b+\alpha),(i+,b-\alpha),(i-,b+\alpha),(i-,b-\alpha),\,1\le\alpha,\beta\le 3\big\}.
\end{split}
\end{equation}
The nonlinearities $\mathcal{N}_e,\mathcal{N}_i,\mathcal{N}_b$ can be written in the form
\begin{equation}\label{pla23}
\begin{split}
&\mathcal{F}(\mathcal{N}_e)(\xi,t)=c\sum_{(\mu,\nu)\in\mathcal{I}}\int_{\mathbb{R}^3}m_{e;\mu,\nu}(\xi,\eta)\widehat{U_\mu}(\xi-\eta,t)\widehat{U_\nu}(\eta,t)\,d\eta,\\
&\mathcal{F}(\mathcal{N}_i)(\xi,t)=c\sum_{(\mu,\nu)\in\mathcal{I}}\int_{\mathbb{R}^3}m_{i;\mu,\nu}(\xi,\eta)\widehat{U_\mu}(\xi-\eta,t)\widehat{U_\nu}(\eta,t)\,d\eta\\
&\mathcal{F}(\mathcal{N}_{b,\alpha})(\xi,t)=c\sum_{(\mu,\nu)\in\mathcal{I}}\int_{\mathbb{R}^3}m_{b,\alpha;\mu,\nu}(\xi,\eta)\widehat{U_\mu}(\xi-\eta,t)\widehat{U_\nu}(\eta,t)\,d\eta.
\end{split}
\end{equation}
Letting
\begin{equation}\label{mult0}
\begin{split}
&T_{e;e,e}(\xi,\eta):=\frac{i[\varepsilon^{-1/2}-R(\xi)R(\xi-\eta)R(\eta)]}{\sqrt{(1+R(\xi)^2)(1+R(\xi-\eta)^2)(1+R(\eta)^2)}},\\
&T_{e;i,i}(\xi,\eta):=\frac{i[\varepsilon^{-1/2}R(\xi-\eta)R(\eta)-R(\xi)]}{\sqrt{(1+R(\xi)^2)(1+R(\xi-\eta)^2)(1+R(\eta)^2)}},\\
&T_{e;e,i}(\xi,\eta):=\frac{i[\varepsilon^{-1/2}R(\eta)+R(\xi)R(\xi-\eta)]}{\sqrt{(1+R(\xi)^2)(1+R(\xi-\eta)^2)(1+R(\eta)^2)}},\\
&T_{e;e,b}(\xi,\eta):=\frac{i[-\varepsilon^{-1}+R(\xi)R(\xi-\eta)]}{\sqrt{(1+R(\xi)^2)(1+R(\xi-\eta)^2)}\Lambda_b(\eta)},\\
&T_{e;i,b}(\xi,\eta):=\frac{i[\varepsilon^{-1}R(\xi-\eta)+R(\xi)]}{\sqrt{(1+R(\xi)^2)(1+R(\xi-\eta)^2)}\Lambda_b(\eta)},\\
\end{split}
\end{equation}
the multipliers $m_{e;\mu,\nu}$ are given by
\begin{equation}\label{mult1}
\begin{split}
&m_{e;e+,e+}(\xi,\eta)=T_{e;e,e}(\xi,\eta)\Big[+\frac{\Lambda_e(\xi)|\xi-\eta|(\xi\cdot\eta)}{2|\xi||\eta|\Lambda_e(\xi-\eta)}+\frac{|\xi|((\xi-\eta)\cdot\eta)}{4|\xi-\eta||\eta|}\Big],\\
&m_{e;e+,e-}(\xi,\eta)=T_{e;e,e}(\xi,\eta)\Big[-\frac{\Lambda_e(\xi)|\xi-\eta|(\xi\cdot\eta)}{2|\xi||\eta|\Lambda_e(\xi-\eta)}+\frac{\Lambda_e(\xi)|\eta|(\xi\cdot (\xi-\eta))}{2|\xi||\xi-\eta|\Lambda_e(\eta)}-\frac{|\xi|((\xi-\eta)\cdot\eta)}{2|\xi-\eta||\eta|}\Big],\\
&m_{e;e-,e-}(\xi,\eta)=T_{e;e,e}(\xi,\eta)\Big[-\frac{\Lambda_e(\xi)|\xi-\eta|(\xi\cdot\eta)}{2|\xi||\eta|\Lambda_e(\xi-\eta)}+\frac{|\xi|((\xi-\eta)\cdot\eta)}{4|\xi-\eta||\eta|}\Big],\\
&m_{e;i+,i+}(\xi,\eta)=T_{e;i,i}(\xi,\eta)\Big[+\frac{\Lambda_e(\xi)|\xi-\eta|(\xi\cdot\eta)}{2|\xi||\eta|\Lambda_i(\xi-\eta)}+\frac{|\xi|((\xi-\eta)\cdot\eta)}{4|\xi-\eta||\eta|}\Big],\\
&m_{e;i+,i-}(\xi,\eta)=T_{e;i,i}(\xi,\eta)\Big[-\frac{\Lambda_e(\xi)|\xi-\eta|(\xi\cdot\eta)}{2|\xi||\eta|\Lambda_i(\xi-\eta)}+\frac{\Lambda_e(\xi)|\eta|(\xi\cdot(\xi-\eta))}{2|\xi||\xi-\eta|\Lambda_i(\eta)}-\frac{|\xi|((\xi-\eta)\cdot\eta)}{2|\xi-\eta||\eta|}\Big],\\
&m_{e;i-,i-}(\xi,\eta)=T_{e;i,i}(\xi,\eta)\Big[-\frac{\Lambda_e(\xi)|\xi-\eta|(\xi\cdot\eta)}{2|\xi||\eta|\Lambda_i(\xi-\eta)}+\frac{|\xi|((\xi-\eta)\cdot\eta)}{4|\xi-\eta||\eta|}\Big],\\
&m_{e;e+,i+}(\xi,\eta)=T_{e;e,i}(\xi,\eta)\Big[-\frac{\Lambda_e(\xi)|\xi-\eta|(\xi\cdot\eta)}{2|\xi||\eta|\Lambda_e(\xi-\eta)}-\frac{\Lambda_e(\xi)|\eta|(\xi\cdot(\xi-\eta))}{2|\xi||\xi-\eta|\Lambda_i(\eta)}-\frac{|\xi|((\xi-\eta)\cdot\eta)}{2|\xi-\eta||\eta|}\Big],\\
&m_{e;e+,i-}(\xi,\eta)=T_{e;e,i}(\xi,\eta)\Big[+\frac{\Lambda_e(\xi)|\xi-\eta|(\xi\cdot\eta)}{2|\xi||\eta|\Lambda_e(\xi-\eta)}-\frac{\Lambda_e(\xi)|\eta|(\xi\cdot(\xi-\eta))}{2|\xi||\xi-\eta|\Lambda_i(\eta)}+\frac{|\xi|((\xi-\eta)\cdot\eta)}{2|\xi-\eta||\eta|}\Big],\\
&m_{e;e-,i+}(\xi,\eta)=T_{e;e,i}(\xi,\eta)\Big[-\frac{\Lambda_e(\xi)|\xi-\eta|(\xi\cdot\eta)}{2|\xi||\eta|\Lambda_e(\xi-\eta)}+\frac{\Lambda_e(\xi)|\eta|(\xi\cdot(\xi-\eta))}{2|\xi||\xi-\eta|\Lambda_i(\eta)}+\frac{|\xi|((\xi-\eta)\cdot\eta)}{2|\xi-\eta||\eta|}\Big],\\
&m_{e;e-,i-}(\xi,\eta)=T_{e;e,i}(\xi,\eta)\Big[+\frac{\Lambda_e(\xi)|\xi-\eta|(\xi\cdot\eta)}{2|\xi||\eta|\Lambda_e(\xi-\eta)}+\frac{\Lambda_e(\xi)|\eta|(\xi\cdot(\xi-\eta))}{2|\xi||\xi-\eta|\Lambda_i(\eta)}-\frac{|\xi|((\xi-\eta)\cdot\eta)}{2|\xi-\eta||\eta|}\Big].\\
\end{split}
\end{equation}
We also have the magnetic terms
\begin{equation}\label{mult1.1}
\begin{split}
&m_{e;e+,b\pm,\alpha}(\xi,\eta)=T_{e;e,b}(\xi,\eta)\Big[\frac{1}{2}\frac{\Lambda_e(\xi)\vert\xi-\eta\vert}{\Lambda_e(\xi-\eta)}\frac{\xi_\alpha}{\vert\xi\vert}+\frac{1}{2}\vert\xi\vert\frac{(\xi-\eta)_\alpha}{\vert\xi-\eta\vert}\Big]\\
&m_{e;e-,b\pm,\alpha}(\xi,\eta)=T_{e;e,b}(\xi,\eta)\Big[\frac{1}{2}\frac{\Lambda_e(\xi)\vert\xi-\eta\vert}{\Lambda_e(\xi-\eta)}\frac{\xi_\alpha}{\vert\xi\vert}-\frac{1}{2}\vert\xi\vert\frac{(\xi-\eta)_\alpha}{\vert\xi-\eta\vert}\Big]\\
&m_{e;i+,b\pm,\alpha}(\xi,\eta)=T_{e;i,b}(\xi,\eta)\Big[\frac{1}{2}\frac{\Lambda_e(\xi)\vert\xi-\eta\vert}{\Lambda_i(\xi-\eta)}\frac{\xi_\alpha}{\vert\xi\vert}+\frac{1}{2}\vert\xi\vert\frac{(\xi-\eta)_\alpha}{\vert\xi-\eta\vert}\Big]\\
&m_{e;i-,b\pm,\alpha}(\xi,\eta)=T_{e;i,b}(\xi,\eta)\Big[\frac{1}{2}\frac{\Lambda_e(\xi)\vert\xi-\eta\vert}{\Lambda_i(\xi-\eta)}\frac{\xi_\alpha}{\vert\xi\vert}-\frac{1}{2}\vert\xi\vert\frac{(\xi-\eta)_\alpha}{\vert\xi-\eta\vert}\Big]\\
&m_{e;b+,\alpha,b+\beta}(\xi,\eta)=m_{e;b-,\alpha,b-,\beta}(\xi,\eta)=\frac{1}{2}m_{e;b+,\alpha,b-,\beta}(\xi,\eta)=\frac{i[\varepsilon^{-3/2}-R(\xi)]}{\sqrt{1+R(\xi)^2}}\frac{\vert\xi\vert}{4\Lambda_b(\xi-\eta)\Lambda_b(\eta)}\delta_{\alpha\beta}.
\end{split}
\end{equation}
Similarly, letting
\begin{equation}\label{mult1.5}
\begin{split}
&T_{i;e,e}(\xi,\eta):=\frac{i[\varepsilon^{-1/2}R(\xi)+R(\xi-\eta)R(\eta)]}{\sqrt{(1+R(\xi)^2)(1+R(\eta)^2)(1+R(\xi-\eta)^2)}},\\
&T_{i;i,i}(\xi,\eta):=\frac{i[\varepsilon^{-1/2}R(\xi)R(\xi-\eta)R(\eta)+1]}{\sqrt{(1+R(\xi)^2)(1+R(\eta)^2)(1+R(\xi-\eta)^2)}},\\
&T_{i;e,i}(\xi,\eta):=\frac{i[\varepsilon^{-1/2}R(\xi)R(\eta)-R(\xi-\eta)]}{\sqrt{(1+R(\xi)^2)(1+R(\eta)^2)(1+R(\xi-\eta)^2)}},\\
&T_{i;e,b}(\xi,\eta):=\frac{i[\varepsilon^{-1}R(\xi)+R(\xi-\eta)]}{\sqrt{(1+R(\xi)^2)(1+R(\xi-\eta)^2}\Lambda_b(\eta)},\\
&T_{i;i,b}(\xi,\eta):=\frac{-i[\varepsilon^{-1}R(\xi)R(\xi-\eta)-1]}{\sqrt{(1+R(\xi)^2)(1+R(\xi-\eta)^2)}\Lambda_b(\eta)},\\
&T_{i;b,b}(\xi,\eta):=\frac{-i[\varepsilon^{-3/2}R(\xi)+1]}{\sqrt{1+R(\xi)^2}\Lambda_b(\xi-\eta)\Lambda_b(\eta)},
\end{split}
\end{equation}
the multipliers $m_{i;\mu,\nu}$ are given by
\begin{equation}\label{mult2}
\begin{split}
&m_{i;e+,e+}(\xi,\eta)=T_{i;e,e}(\xi,\eta)\Big[-\frac{\Lambda_i(\xi)|\xi-\eta|(\xi\cdot\eta)}{2|\xi||\eta|\Lambda_e(\xi-\eta)}-\frac{|\xi|((\xi-\eta)\cdot\eta)}{4|\xi-\eta||\eta|}\Big],\\
&m_{i;e+,e-}(\xi,\eta)=T_{i;e,e}(\xi,\eta)\Big[+\frac{\Lambda_i(\xi)|\xi-\eta|(\xi\cdot\eta)}{2|\xi||\eta|\Lambda_e(\xi-\eta)}-\frac{\Lambda_i(\xi)|\eta|(\xi\cdot (\xi-\eta))}{2|\xi||\xi-\eta|\Lambda_e(\eta)}+\frac{|\xi|((\xi-\eta)\cdot\eta)}{2|\xi-\eta||\eta|}\Big],\\
&m_{i;e-,e-}(\xi,\eta)=T_{i;e,e}(\xi,\eta)\Big[+\frac{\Lambda_i(\xi)|\xi-\eta|(\xi\cdot\eta)}{2|\xi||\eta|\Lambda_e(\xi-\eta)}-\frac{|\xi|((\xi-\eta)\cdot\eta)}{4|\xi-\eta||\eta|}\Big],\\
&m_{i;i+,i+}(\xi,\eta)=T_{i;i,i}(\xi,\eta)\Big[-\frac{\Lambda_i(\xi)|\xi-\eta|(\xi\cdot\eta)}{2|\xi||\eta|\Lambda_i(\xi-\eta)}-\frac{|\xi|((\xi-\eta)\cdot\eta)}{4|\xi-\eta||\eta|}\Big],\\
&m_{i;i+,i-}(\xi,\eta)=T_{i;i,i}(\xi,\eta)\Big[+\frac{\Lambda_i(\xi)|\xi-\eta|(\xi\cdot\eta)}{2|\xi||\eta|\Lambda_i(\xi-\eta)}-\frac{\Lambda_i(\xi)|\eta|(\xi\cdot (\xi-\eta))}{2|\xi||\xi-\eta|\Lambda_i(\eta)}+\frac{|\xi|((\xi-\eta)\cdot\eta)}{2|\xi-\eta||\eta|}\Big],\\
&m_{i;i-,i-}(\xi,\eta)=T_{i;i,i}(\xi,\eta)\Big[+\frac{\Lambda_i(\xi)|\xi-\eta|(\xi\cdot\eta)}{2|\xi||\eta|\Lambda_i(\xi-\eta)}-\frac{|\xi|((\xi-\eta)\cdot\eta)}{4|\xi-\eta||\eta|}\Big],\\
&m_{i;e+,i+}(\xi,\eta)=T_{i;e,i}(\xi,\eta)\Big[+\frac{\Lambda_i(\xi)|\xi-\eta|(\xi\cdot\eta)}{2|\xi||\eta|\Lambda_e(\xi-\eta)}+\frac{\Lambda_i(\xi)|\eta|(\xi\cdot (\xi-\eta))}{2|\xi||\xi-\eta|\Lambda_i(\eta)}+\frac{|\xi|((\xi-\eta)\cdot\eta)}{2|\xi-\eta||\eta|}\Big],\\
&m_{i;e+,i-}(\xi,\eta)=T_{i;e,i}(\xi,\eta)\Big[-\frac{\Lambda_i(\xi)|\xi-\eta|(\xi\cdot\eta)}{2|\xi||\eta|\Lambda_e(\xi-\eta)}+\frac{\Lambda_i(\xi)|\eta|(\xi\cdot (\xi-\eta))}{2|\xi||\xi-\eta|\Lambda_i(\eta)}-\frac{|\xi|((\xi-\eta)\cdot\eta)}{2|\xi-\eta||\eta|}\Big],\\
&m_{i;e-,i+}(\xi,\eta)=T_{i;e,i}(\xi,\eta)\Big[+\frac{\Lambda_i(\xi)|\xi-\eta|(\xi\cdot\eta)}{2|\xi||\eta|\Lambda_e(\xi-\eta)}-\frac{\Lambda_i(\xi)|\eta|(\xi\cdot (\xi-\eta))}{2|\xi||\xi-\eta|\Lambda_i(\eta)}-\frac{|\xi|((\xi-\eta)\cdot\eta)}{2|\xi-\eta||\eta|}\Big],\\
&m_{i;e-,i-}(\xi,\eta)=T_{i;e,i}(\xi,\eta)\Big[-\frac{\Lambda_i(\xi)|\xi-\eta|(\xi\cdot\eta)}{2|\xi||\eta|\Lambda_e(\xi-\eta)}-\frac{\Lambda_i(\xi)|\eta|(\xi\cdot (\xi-\eta))}{2|\xi||\xi-\eta|\Lambda_i(\eta)}+\frac{|\xi|((\xi-\eta)\cdot\eta)}{2|\xi-\eta||\eta|}\Big].
\end{split}
\end{equation}
Again, the magnetic terms give
\begin{equation}\label{mult2.5}
\begin{split}
&m_{i;e+,b\pm,\alpha}(\xi,\eta)=T_{i;e,b}(\xi,\eta)\Big[\frac{1}{2}\frac{\Lambda_i(\xi)\vert\xi-\eta\vert}{\Lambda_e(\xi-\eta)}\frac{\xi_\alpha}{\vert\xi\vert}+\frac{1}{2}\vert\xi\vert\frac{(\xi-\eta)_\alpha}{\vert\xi-\eta\vert}\Big],\\
&m_{i;e-,b\pm,\alpha}(\xi,\eta)=T_{i;e,b}(\xi,\eta)\Big[\frac{1}{2}\frac{\Lambda_i(\xi)\vert\xi-\eta\vert}{\Lambda_e(\xi-\eta)}\frac{\xi_\alpha}{\vert\xi\vert}-\frac{1}{2}\vert\xi\vert\frac{(\xi-\eta)_\alpha}{\vert\xi-\eta\vert}\Big],\\
&m_{i;i+,b\pm,\alpha}(\xi,\eta)=T_{i;i,b}(\xi,\eta)\Big[\frac{1}{2}\frac{\Lambda_i(\xi)\vert\xi-\eta\vert}{\Lambda_i(\xi-\eta)}\frac{\xi_\alpha}{\vert\xi\vert}+\frac{1}{2}\vert\xi\vert\frac{(\xi-\eta)_\alpha}{\vert\xi-\eta\vert}\Big],\\
&m_{i;i-,b\pm,\alpha}(\xi,\eta)=T_{i;i,b}(\xi,\eta)\Big[\frac{1}{2}\frac{\Lambda_i(\xi)\vert\xi-\eta\vert}{\Lambda_i(\xi-\eta)}\frac{\xi_\alpha}{\vert\xi\vert}-\frac{1}{2}\vert\xi\vert\frac{(\xi-\eta)_\alpha}{\vert\xi-\eta\vert}\Big],\\
&m_{i;b+,\alpha,b+,\beta}(\xi,\eta)=m_{i;b-,\alpha,b-,\beta}(\xi,\eta)=\frac{1}{2}m_{i;b+,\alpha,b-,\beta}(\xi,\eta)=T_{i;b,b}(\xi,\eta)\frac{\vert\xi\vert}{4}\delta_{\alpha\beta}.
\end{split}
\end{equation}

Finally, we can decompose $\mathcal{N}_b$. Letting
\begin{equation}\label{mult3}
\begin{split}
&T_{b;e,e}^{\alpha,\beta}(\xi,\eta):=\frac{iQ^2_{\alpha,\beta}(\xi)\left[-\varepsilon^{-1}+R(\xi-\eta)R(\eta)\right]}{\sqrt{(1+R(\xi-\eta)^2)(1+R(\eta)^2)}}\\
&T_{b;i,i}^{\alpha,\beta}(\xi,\eta):=\frac{iQ^2_{\alpha,\beta}(\xi)\left[-\varepsilon^{-1}R(\xi-\eta)R(\eta)+1\right]}{\sqrt{(1+R(\xi-\eta)^2)(1+R(\eta)^2)}}\\
&T_{b;e,i}^{\alpha,\beta}(\xi,\eta):=\frac{iQ^2_{\alpha,\beta}(\xi)\left[\varepsilon^{-1}R(\eta)+R(\xi-\eta)\right]}{\sqrt{(1+R(\xi-\eta)^2)(1+R(\eta)^2)}}\\
\end{split}
\end{equation}
we get that
\begin{equation}\label{mult3.1}
\begin{split}
&m_{b,\alpha;e+,e+}(\xi,\eta)=T^{\alpha,\beta}_{b;e,e}(\xi,\eta)\frac{1}{2}\frac{\vert\xi-\eta\vert}{\Lambda_e(\xi-\eta)}\frac{\eta_\beta}{\vert\eta\vert}\\
&m_{b,\alpha;e+,e-}(\xi,\eta)=T^{\alpha,\beta}_{b;e,e}(\xi,\eta)\frac{1}{2}\Big[-\frac{\vert\xi-\eta\vert}{\Lambda_e(\xi-\eta)}\frac{\eta_\beta}{\vert\eta\vert}+\frac{\vert\eta\vert}{\Lambda_e(\eta)}\frac{(\xi-\eta)_\beta}{\vert\xi-\eta\vert}\Big]\\
&m_{b,\alpha;e-,e-}(\xi,\eta)=-T^{\alpha,\beta}_{b;e,e}(\xi,\eta)\frac{1}{2}\frac{\vert\xi-\eta\vert}{\Lambda_i(\xi-\eta)}\frac{\eta_\beta}{\vert\eta\vert}\\
&m_{b,\alpha;i+,i+}(\xi,\eta)=T^{\alpha,\beta}_{b;i,i}(\xi,\eta)\frac{1}{2}\frac{\vert\xi-\eta\vert}{\Lambda_i(\xi-\eta)}\frac{\eta_\beta}{\vert\eta\vert}\\
&m_{b,\alpha;i+,i-}(\xi,\eta)=-T^{\alpha,\beta}_{b;i,i}(\xi,\eta)\frac{1}{2}\Big[\frac{\vert\xi-\eta\vert}{\Lambda_i(\xi\eta)}\frac{\eta_\beta}{\vert\eta\vert}-\frac{\vert\eta\vert}{\Lambda_i(\eta)}\frac{(\xi-\eta)_\beta}{\vert\xi-\eta\vert}\Big]\\
&m_{b,\alpha;i-,i-}(\xi,\eta)=-T^{\alpha,\beta}_{b;i,i}(\xi,\eta)\frac{1}{2}\frac{\vert\xi-\eta\vert}{\Lambda_i(\xi-\eta)}\frac{\eta_\beta}{\vert\eta\vert}\\
&m_{b,\alpha;e+,i+}(\xi,\eta)=T^{\alpha,\beta}_{b;i,e}(\xi,\eta)\frac{1}{2}\Big[\frac{\vert\xi-\eta\vert}{\Lambda_e(\xi-\eta)}\frac{\eta_\beta}{\vert\eta\vert}+\frac{\vert\eta\vert}{\Lambda_i(\eta)}\frac{(\xi-\eta)_\beta}{\vert\xi-\eta\vert}\Big]\\
&m_{b,\alpha;e+,i-}(\xi,\eta)=T^{\alpha,\beta}_{b;i,e}(\xi,\eta)\frac{1}{2}\Big[-\frac{\vert\xi-\eta\vert}{\Lambda_e(\xi-\eta)}\frac{\eta_\beta}{\vert\eta\vert}+\frac{\vert\eta\vert}{\Lambda_i(\eta)}\frac{(\xi-\eta)_\beta}{\vert\xi-\eta\vert}\Big]\\
&m_{b,\alpha;e-,i+}(\xi,\eta)=T^{\alpha,\beta}_{b;i,e}(\xi,\eta)\frac{1}{2}\Big[\frac{\vert\xi-\eta\vert}{\Lambda_e(\xi-\eta)}\frac{\eta_\beta}{\vert\eta\vert}-\frac{\vert\eta\vert}{\Lambda_i(\eta)}\frac{(\xi-\eta)_\beta}{\vert\xi-\eta\vert}\Big]\\
&m_{b,\alpha;e-,i-}(\xi,\eta)=-T^{\alpha,\beta}_{b;i,e}(\xi,\eta)\frac{1}{2}\Big[\frac{\vert\xi-\eta\vert}{\Lambda_e(\xi-\eta)}\frac{\eta_\beta}{\vert\eta\vert}+\frac{\vert\eta\vert}{\Lambda_i(\eta)}\frac{(\xi-\eta)_\beta}{\vert\xi-\eta\vert}\Big].
\end{split}
\end{equation}
Finally, we add the magnetic terms
\begin{equation}\label{mult3.5}
\begin{split}
&m_{b,\alpha;e\pm,b\pm,\beta}(\xi,\eta)=\frac{iQ^2_{\alpha,\beta}(\xi)\left[\varepsilon^{-3/2}-R(\xi-\eta)\right]}{2\sqrt{1+R(\xi-\eta)^2}}\frac{\vert\xi-\eta\vert}{\Lambda_e(\xi-\eta)\Lambda_b(\eta)}\\
&m_{b,\alpha;i\pm,b\pm,\beta}(\xi,\eta)=-\frac{iQ^2_{\alpha,\beta}(\xi)\left[\varepsilon^{-3/2}R(\xi-\eta)+1\right]}{2\sqrt{1+R(\xi-\eta)^2}}\frac{\vert\xi-\eta\vert}{\Lambda_i(\xi-\eta)\Lambda_b(\eta)}.
\end{split}
\end{equation}

\section{Other models}\label{AEP}

The purpose of this section is to show how the results in the main body of this work extend to several other models. In particular we focus on two variants of \eqref{EuMa1} which are invariant under appropriate change of frame: the {\it Euler-Poisson} is invariant under Galilean transformation, while the {\it relativistic model} is invariant under Lorentz transformations.

\subsection{The Euler-Poisson model}

Formally taking $c\rightarrow \infty$ and $B\equiv 0$ in \eqref{EuMa1}, we derive the Euler-Poisson system in the two-fluid theory, which is relevant when the magnetic forces can be neglected:
\begin{equation}  \label{EP}
\begin{split}
\partial _{t}n_{e}+\hbox{div}(n_{e}v_{e}) &=0, \\
n_{e}m_{e}\left[ \partial _{t}v_{e}+v_{e}\cdot \nabla v_{e}\right] +\nabla
p_{e} &=n_{e}e\nabla \phi , \\
\partial _{t}n_{i}+\hbox{div}(n_{i}v_{i}) &=0, \\
n_{i}M_{i}\left[ \partial _{t}v_{i}+v_{i}\cdot \nabla v_{i}\right] +\nabla
p_{i} &=-Zn_{i}e\nabla \phi , \\
-\Delta \phi &=4\pi e(Zn_{i}-n_{e}).
\end{split}
\end{equation}
Here the electromagnetic fields are given by
\begin{equation*}
E=-\nabla\phi,\quad B=0.
\end{equation*}

{\bf Galilean invariance.} This system is left invariant by Galilean change of unknowns. More precisely, let $V\in\mathbb{R}^3$ be a fixed vector and let
\begin{equation*}
\begin{split}
&x\to x^\prime:= x+Vt,\quad t\to t^\prime:=t\\
&n_e\to n_e^\prime:=n_e,\quad v_e\to v_e^\prime:=v_e-V,\quad n_i\to n_i^\prime:=n_i,\quad v_i\to v_i^\prime:=v_i-V,\quad \phi\to \phi^\prime:=\phi,
\end{split}
\end{equation*}
then, we observe that $(n_e,v_e,n_i,v_i,\phi)(x,t)$ solves \eqref{EP} if and only if $(n_e^\prime,v_e^\prime,n_i^\prime,v_i^\prime,\phi^\prime)(x^\prime,t^\prime)$ does.

\medskip

{\bf Arbitrary pressures.}\footnote{The analysis of the arbitrary pressure here carries over directly to the Euler-Maxwell case in \eqref{EuMa1}.} We consider arbitrary smooth pressure laws of the form
\begin{equation}\label{ArbitraryPressure}
p_i=p_i(n_i),\quad p_e=p_e(n_e),\quad\text{ with}\quad p_e^\prime(n_0)>0\hbox{ and }p_i^\prime(n_0/Z)>0.
\end{equation}
Defining
\begin{equation*}
P_i=Zp_i^\prime(n_0/Z)/n_0,\quad P_e=p_e^\prime(n_0)/n_0,
\end{equation*}
and using the rescaling \eqref{AddedResc}, we have the Taylor expansions\footnote{The {\it enthalpies} given by $h_e^\prime(x)=p_e^\prime(x)/x$ and $h_i^\prime(x)=p_i^\prime(x)/x$ are also important for the relativistic model later on. The functions $q_i$ and $q_e$ are essentially the {\it internal specific energies}, $\epsilon_i=n_ih_i-p_i$ and $\epsilon_e=n_eh_e-p_e$.}
\begin{equation}\label{EnthalpyEP}
\frac{p_i^\prime(n_i)}{n_i}=P_i\left[1+P_i^1\rho+\rho^2\cdot c_i(\rho)\right]:=P_iq_i^\prime(\rho)\nabla\rho,\quad \frac{p_e^\prime(n_e)}{n_e}=P_i\left[T+P_e^1n+n^2\cdot c_e(n)\right]:=P_iq_e^\prime(n)\nabla n,
\end{equation}
for some constants $P_i^1, P_e^1$ and some smooth functions $c_i$ and $c_e$. After rescaling as in \eqref{AddedResc}, this gives the new system
\begin{align}
&\partial_tn+\hbox{div}((n+1)v)&&&=&0,\label{NewSysEP1}\\
&\varepsilon\left(\partial_tv+v\cdot\nabla v\right)+T\nabla n-\nabla\phi & +&P_e^1n\nabla n+n^2c_e(n)\nabla n&=&0,\\
&\partial_t\rho+\hbox{div}((\rho+1) u)&&&=&0,\\
&\left(\partial_tu+u\cdot\nabla u\right)+\nabla \rho+\nabla\phi&+&P_i^1\rho\nabla\rho+\rho^2c_i(\rho)\nabla \rho&=&0,\\
&-\Delta\phi-\rho+n&&&=&0\label{NewSysEP2},
\end{align}

The analysis developed in the main body of this work applies directly to the Euler-Poisson system and gives:
\begin{theorem}
\label{ThmEP}
Assume $\varepsilon\le 10^{-3}$ and $T\in[1/10,100]$.  Let $N_{0}=10^{4}$
and assume that 
\begin{equation}
\begin{split}
& \Vert (n^{0},v^{0},\rho ^{0},u^{0})\Vert
_{H^{N_{0}}}+\Vert (n^{0}, v^{0},\rho^{0}, u^{0})\Vert _{Z}=\delta _{0}\leq \overline{\delta }
, \\
&\nabla \times v^{0}=\nabla \times u^{0}=0,
\end{split}
\label{maincondEP}
\end{equation}
where $\overline{\delta }=\overline{\delta }(T,\varepsilon )>0$ is sufficiently small, and the $Z$ norm is defined in Definition \ref{MainDef}. Then there exists a unique global solution $(n,v,\rho ,u)\in C([0,\infty ):{H^{N_{0}}})$ of the system \eqref{NewSysEP1}--\eqref{NewSysEP2} with initial data $(n(0),v(0),\rho (0),u(0))=(n^{0},v^{0},\rho ^{0},u^{0})$. Moreover, for any $t\in \lbrack 0,\infty )$, 
\begin{equation}
\nabla \times v(t)=\nabla \times u(t)\equiv 0,\text{ (irrotationality) }  \label{mainconcl2EP}
\end{equation}
and, with $\beta :=1/100$, 
\begin{equation}
\begin{split}
&\Vert (n(t),v(t),\rho(t),u(t))\Vert_{H^{N_{0}}}+\sup_{|\alpha |\leq 4}(1+t)^{1+\beta /2}\Vert D_{x}^{\alpha }(n(t),v(t),\rho (t),u(t))\Vert _{L^{\infty }}\lesssim \delta_{0}.
\end{split}
\label{mainconcl2.1EP}
\end{equation}
\end{theorem}

\subsubsection{Derivation and main steps of the proof.}

The proof of this theorem is very similar to the proof of Theorem \ref{MainThm}. Proposition \ref{Localexistence} remains essentially unchanged using the energy
\begin{equation*}
\mathcal{E}_N:=\sum_{|\gamma|\leq N}\int_{\mathbb{R}^3}
\big[q_e^\prime(n)|D_x^\gamma n|^2+\varepsilon(1+n)|D_x^\gamma v|^2
+q_i^\prime(\rho)\vert D_x^\gamma\rho\vert^2+(\rho+1)\vert D_x^\gamma u\vert^2+|D_x^\gamma \nabla\phi|^2\big]\,dx.
\end{equation*}
Note that $q_e^\prime>0$ and $q_i^\prime>0$ for small perturbations, as follows from \eqref{ArbitraryPressure} and \eqref{EnthalpyEP} and that $q_e^\prime(0)=T$, $q_i^\prime(0)=1$.
This gives local existence of smooth solutions.
Using that
\begin{equation*}
\begin{split}
\partial_t\left[\nabla\times v\right]&=\nabla\times\left[v\times\left[\nabla\times v\right]\right]\\
\partial_t\left[\nabla\times u\right]&=\nabla\times\left[u\times\left[\nabla\times u\right]\right]
\end{split}
\end{equation*}
we see that smooth solutions with irrotational initial data remain irrotational for all time. We can therefore write
\begin{equation*}
v_\alpha=R_\alpha h,\,\, u_\alpha=R_\alpha g,\quad\hbox{or}\quad h:=-\vert\nabla\vert^{-1}\hbox{div}(v),\,\, g:=-\vert\nabla\vert^{-1}\hbox{div}(u)
\end{equation*}
and we then see that \eqref{NewSysEP1}-\eqref{NewSysEP2} becomes
\begin{equation}\label{pla0AP}
\begin{split}
\partial_t n-\vert\nabla\vert h&=-\partial_\alpha\left[n R_\alpha h\right],\\
\partial_t\rho-\vert\nabla\vert g&=-\partial_\alpha\left[\rho R_\alpha g\right],\\
\partial_th+\vert\nabla\vert^{-1}H_\varepsilon^2n-\varepsilon^{-1}\vert\nabla\vert^{-1}\rho&=-(1/2)\vert\nabla\vert\left[R_\alpha h R_\alpha h\right]-\vert\nabla\vert\left[(P_e^1/2) n^2+n^2C_e(n)\right],\\
\partial_t g-\vert\nabla\vert^{-1}n+\vert\nabla\vert^{-1}H_1^2\rho&=-(1/2)\vert\nabla\vert\left[R_\alpha g R_\alpha g\right]-\vert\nabla\vert\left[(P_i^1/2)\rho^2+\rho^2C_i(\rho)\right],\\
\end{split}
\end{equation}
where $H_1$ and $H_\varepsilon$ are given in \eqref{pla1} and $C_i$ and $C_e$ are smooth functions related to $c_i$ and $c_e$ by
\begin{equation*}
C_i(t)=\frac{1}{t^2}\int_0^ts^2c_i(s)ds,\quad C_e(t)=\frac{1}{t^2}\int_0^ts^2c_e(s)ds.
\end{equation*}
In particular, they satisfy that $C_i(0)=C_e(0)=0$.

\medskip

Defining $U_e$ and $U_i$ as in \eqref{pla7}, we can recast this as
\begin{equation*}
\begin{cases}
\left(\partial_t+i\Lambda_e\right)U_e&=\mathcal{N}_e+\mathcal{N}_e^\prime\\
\left(\partial_t+i\Lambda_i\right)U_i&=\mathcal{N}_i+\mathcal{N}_i^\prime
\end{cases}
\end{equation*}
where (recall that $R$ is defined in \eqref{pla1})
\begin{equation}\label{EPpla9}
\begin{split}
&\mathcal{N}_e=\frac{\Lambda_eR_\al}{2\sqrt{1+R^2}}\varepsilon^{1/2}(nR_\al h)+i\left\{\frac{|\nabla|}{4\sqrt{1+R^2}}\varepsilon^{1/2}R_\al h\cdot R_\al h-R[R_\al g\cdot R_\al g]\right\},\\
&\mathcal{N}_i=\frac{-\Lambda_iR_\al}{2\sqrt{1+R^2}}\big[\varepsilon^{1/2}R(nR_\al h)+(\rho R_\al g)+i\left\{\frac{-|\nabla|}{4\sqrt{1+R^2}}\big[\varepsilon^{1/2}R[ R_\al h\cdot  R_\al h]+R_\al g\cdot R_\al g\right\}.\\
\end{split}
\end{equation}
contains nonlinear terms already present in \eqref{EuMa1} and 
\begin{equation}\label{EPpla10}
\begin{split}
\mathcal{N}_i^\prime&=-i\frac{\vert\nabla\vert}{4\sqrt{1+R^2}}\left[\varepsilon^\frac{1}{2}P_e^1R(n\cdot n)+P_i^1\rho\cdot\rho\right]-i\frac{\vert\nabla\vert}{2\sqrt{1+R^2}}\left[\varepsilon^\frac{1}{2}R[n\cdot n\cdot C_e(n)]+\rho\cdot \rho\cdot C_i(\rho)\right] ,\\
\mathcal{N}_e^\prime&=i\frac{\vert\nabla\vert}{4\sqrt{1+R^2}}\left[\varepsilon^\frac{1}{2}P_e^1n\cdot n-P_i^1R(\rho\cdot\rho)\right]+i\frac{\vert\nabla\vert}{2\sqrt{1+R^2}}\left[\varepsilon^\frac{1}{2}n\cdot n\cdot C_e(n)-R[\rho\cdot \rho\cdot C_i(\rho)]\right],\\
\end{split}
\end{equation}
is a new nonlinearities coming from the new pressure.

Since the formulas in \eqref{EPpla10} involve Fourier multipliers with $\mathcal{S}^{100}$ Fourier symbols, we see that all the additional quadratic symbols are of the form given in \eqref{mk5}. In addition, we have that the cubic terms are of the form
\begin{equation*}
\mathcal{F}\mathcal{N}_\sigma^3(\xi,t)=\sum_{\mu,\nu\in\mathcal{I}_0,j\in\{1,2\}}\int_{\mathbb{R}^3}m^3_{\sigma;\mu,\nu,j}(\xi-\eta,\eta,\chi)\widehat{U}_\mu(\xi-\eta,t)\widehat{U}_\nu(\eta-\chi,t)\widehat{C_j}(\chi,t)d\eta d\chi
\end{equation*}
where
\begin{equation*}
m^3_{\sigma;\mu,\nu,j}(\xi,\eta,\chi)=\vert\xi\vert\cdot m(\xi-\eta,\eta,\chi)
\end{equation*}
for some $m(\xi_1,\xi_2,\xi_3)$ a linear combination of functions of the form \eqref{StructureCubMult} and
\begin{equation*}
C_i=C_i(\rho),\quad C_e=C_e(\rho)
\end{equation*}
satisfy \eqref{AssumpHCubic} in view of Lemma \ref{LemCubComp} below. Thus, the new terms in the nonlinearity fall under the scope of Proposition \ref{reduced1} or under the scope of Proposition \ref{PropCub} and can be handled.

\subsection{Relativistic model}

We refer to \cite{Ch2,Gou,GuoTah} for other references about relativistic fluids.

\subsubsection{Relativistic Euler-Maxwell model for the electron}

As a starter, we introduce a simpler {\it one-fluid} relativistic Euler-Maxwell model, namely the model for the electrons. This is the relativistic counterpart of the (classical) Euler-Maxwell model for electrons already discussed in \cite{GeMa,IoPa2}.

We consider Minkowski space $(\mathbb{R}^{1+3},g_{\alpha\beta})$ with $g_{00}=-c^2$, $g_{ij}=\delta_{ij}$ and $g_{0j}=g_{j0}=0$. Its inverse is denoted $g^{\mu\nu}$ where $g^{00}=-c^{-2}$, $g^{ij}=\delta_{ij}$ and $g^{0j}=0=g^{j0}$. We use the Einstein convention that repeated up-down indices be summed and we raise and lower indices using the metric. Latin indices $i,j\dots$ vary from $1$ to $3$, while greek indices $\mu,\nu\dots$ vary from $0$ to $3$.

We denote $\mathcal{T}^d(\mathbb{R}^{1+3})$ the set of contravariant $d$-tensors on the Minkowski space.
We model the electron fluid by a scalar function $n_e\in\mathcal{T}^0(\mathbb{R}^{1+3})$ and a velocity function $u=(u^\alpha)_{0\le\alpha\le 3}\in\mathcal{T}^1(\mathbb{R}^{1+3})$ normalized by\footnote{Some authors prefer to normalize the velocity of a particle of {\it rest-mass} $m$ by asking that $u^\alpha u_\alpha=-mc^2$. Here, we prefer to incorporate the mass separately.}
\begin{equation}\label{NormREMFluids}
u^\alpha u_\alpha=-c^2.
\end{equation}
Below, we will always note $\gamma_e=u^0$ so that $u^\nu=(\gamma_e,v^1,v^2,v^3)$ with $\gamma_e=\sqrt{1+c^{-2}\vert v\vert^2}$. We will repeatedly switch between $u$ and $v$ as one uniquely defines the other thanks to \eqref{NormREMFluids}.

In addition, we also consider an electromagnetic field $F=\{F^{\mu\nu}\}_{0\le\mu,\nu\le 3}\in \mathcal{T}^2(\mathbb{R}^{1+3})$. We assume that this field is skew-symmetric: $F^{\mu\nu}=-F^{\nu\mu}$. Finally, we also assume the presence of a uniform flat positively charged background of density of charge $n_0e$ and {\it four-velocity} $\partial_t=(1,0,0,0)\in\mathcal{T}^1(\mathbb{R}^{1+3})$.

We can introduce the energy-momentum tensor associated to the fluid under consideration:
\begin{equation*}
T^{\mu\nu}=nh\frac{u^\mu u^\nu}{c^2}+pg^{\mu\nu}\in\mathcal{T}^2(\mathbb{R}^{1+3}),
\end{equation*}
where $p=p(n)$ is a function satisfying assumptions as in \eqref{ArbitraryPressure} and $h=h(n)$, the {\it specific enthalpy} is a function satisfying
\begin{equation}\label{Enth}
h^\prime(x)=\frac{p^\prime(x)}{x},\quad h(0)=m_ec^2,
\end{equation}
where $m_e$ is the rest-mass of an electron. We can also consider the energy-momentum tensor of the electromagnetic field:
\begin{equation*}
\mathcal{E}^{\mu\nu}=-(4\pi)^{-1}\left[F^{\mu\alpha}F^{\beta\nu}g_{\alpha\beta}+\frac{1}{4}F^{\alpha\beta}F_{\alpha\beta}g^{\mu\nu} \right]\in\mathcal{T}^2(\mathbb{R}^{1+3}).
\end{equation*}

The dynamics are then given by three equations: the Maxwell equations,  the continuity of matter and the balance of energy-momentum. The Maxwell equations give:
\begin{equation}\label{Max}
\begin{split}
\partial_\mu F^{\mu\nu}=\frac{4\pi}{c}J^\nu\qquad\partial_\alpha F_{\beta\gamma}+\partial_\beta F_{\gamma\alpha}+\partial_\gamma F_{\alpha\beta}=0.
\end{split}
\end{equation}
where the total relativistic current is defined by
\begin{equation}\label{RelCur}
J^\nu=en_0\partial_t^\nu-en u^\nu.
\end{equation}
The continuity of matter gives
\begin{equation}\label{ContMatt}
\partial_\nu(n u^\nu)=0.
\end{equation}
The balance of Energy-momentum is then
\begin{equation*}\label{Bianchi}
\begin{split}
\partial_\nu T^{\mu\nu}=\frac{en}{c}u_\alpha F^{\mu\alpha}.
\end{split}
\end{equation*}
After using \eqref{ContMatt}, this reduces to
\begin{equation}\label{Bianchi2}
\frac{nu^\nu}{c^2}\partial_\nu\left[hu^\mu\right]+g^{\mu\nu}\partial_\nu p=\frac{en}{c}u_\alpha F^{\mu\alpha}.
\end{equation}
Projecting onto the direction of $u$ gives
\begin{equation*}\label{DirectionOfU}
-nu^\nu\partial_\nu h+u^\nu\partial_\nu p=0
\end{equation*}
which is always satisfied. Therefore, \eqref{Bianchi2} only contains three nontrivial equations which can be obtained by projecting onto $\mathbb{R}^3$, the orthogonal of $\partial_t$.

\medskip

{\bf Lorentz Covariance.} Consider a Lorentz-transformation $L$, i.e. a (fixed) $2$-tensor $L$ satisfying $L_{\alpha\beta}L^{\alpha\gamma}=\delta_\beta^\gamma$ and define
\begin{equation*}
\begin{split}
&(X^\prime)^\alpha=L^{\alpha\beta}X_\beta,\quad n^\prime(X^\prime)=n(X),\quad (u^\prime)^\alpha(X^\prime)=L^{\alpha\beta}u_\beta(X),\quad (F^\prime)^{\alpha\beta}(X^\prime)=L^{\alpha\gamma}L^{\beta\delta}F_{\gamma\delta}(X),\\
 &(J^\prime)^\alpha(X^\prime)=L^{\alpha\beta}J_\beta(X)
\end{split}
\end{equation*}
Then, we see that $(n,u,J,F)$ satisfy \eqref{Max}-\eqref{ContMatt}-\eqref{Bianchi} if and only if $(n^\prime,u^\prime,J^\prime,F^\prime)$ does.

\medskip

{\bf Irrotational flows.} We introduce the (generalized) vorticity defined by
\begin{equation*}
\omega_{\alpha\beta}=\partial_\alpha(hu_\beta)-\partial_\beta(hu_\alpha)+ecF_{\alpha\beta}.
\end{equation*}
This is transported by the flow in the following sense:
\begin{equation}\label{GV1FC2}
u^\nu\partial_\nu\omega_{\alpha\beta}=(\partial_\alpha u^\nu)\omega_{\beta\nu}-(\partial_\beta u^\nu)\omega_{\alpha\nu}.
\end{equation}
Indeed, we may simply compute
\begin{equation*}
\begin{split}
u^\nu\partial_\nu\omega_{\alpha\beta}&=\partial_\alpha(u^\nu\partial_\nu(hu_\beta))-\partial_\nu(hu_\beta)\partial_\alpha u^\nu-\partial_\beta(u^\nu\partial_\nu(hu_\alpha))+\partial_\nu(hu_\alpha)\partial_\beta u^\nu+ec u^\nu\partial_\nu F_{\alpha\beta}\\
&=-\partial_\alpha(\frac{c^2}{n}\partial_\beta p-ceu_\theta F_{\beta\gamma}g^{\gamma\theta})+\partial_\beta(\frac{c^2}{n}\partial_\alpha p-ceu_\theta F_{\alpha\gamma}g^{\gamma\theta})-ecu^\nu(\partial_\alpha F_{\beta\nu}+\partial_\beta F_{\nu\alpha})\\
&\quad-(\partial_\alpha u^\nu)\omega_{\nu\beta}-(\partial_\alpha u^\nu)\partial_\beta(hu_\nu)+ec(\partial_\alpha u^\nu)F_{\nu\beta}+(\partial_\beta u^\nu)\omega_{\nu\alpha}+(\partial_\beta u^\nu)\partial_\alpha(hu_\nu)-ec(\partial_\beta u^\nu)F_{\nu\alpha}\\
&=ec\{\partial_\alpha(u^\theta F_{\beta\theta})-\partial_\beta(u^\theta F_{\alpha\theta})-u^\nu\partial_\alpha F_{\beta\nu}-u^\nu\partial_\beta F_{\nu\alpha}+(\partial_\alpha u^\nu)F_{\nu\beta}-(\partial_\beta u^\nu)F_{\nu\alpha}\}\\
&\quad-(\partial_\alpha u^\nu)\omega_{\nu\beta}+(\partial_\beta u^\nu)\omega_{\nu\alpha}-\{(\partial_\alpha u^\nu)\partial_\beta(hu_\nu)-(\partial_\beta u^\nu)\partial_\alpha(hu_\nu)\}\\
&=-(\partial_\alpha u^\nu)\omega_{\nu\beta}+(\partial_\beta u^\nu)\omega_{\nu\alpha},\\
\end{split}
\end{equation*}
and hence as long as the solution is smooth irrotational initial data lead to solutions which remain irrotational.

\medskip

{\bf New Unknowns.} We assume that the data are irrotational and we consider the unknowns\footnote{This choice of unknown is motivated from the choice of unknowns in the non relativistic case \cite{IoPa2}.}
\begin{equation}\label{NewU1}
\begin{split}
\mu_e^j=c^{-2}hu^j=c^{-2}hv^j,\quad E^j=ecF^{j0}.
\end{split}
\end{equation}
All the other unknowns can be recovered from the formula
\begin{equation*}
ec^{-1}F^{jk}=\partial_k\mu_e^j-\partial_j\mu_e^k,
\end{equation*}
the first equation in \eqref{Max}
\begin{equation*}
n\gamma_e=n_0-\frac{1}{4\pi e^2}\partial_jE^j,
\end{equation*}
and from the fact that the mapping
\begin{equation*}\label{RCoV}
D:(n,v)\mapsto (n\gamma_e,\mu_e)=(n\sqrt{1+c^{-2}\vert v\vert^2},c^{-2}h(n)v)
\end{equation*}
is invertible in an $L^\infty$-neighborhood of $v\equiv 0$ and $n\equiv n_0$, where $h\simeq h_0:=h(n_0)$. This latter point is easily seen from the Jacobian matrix
\begin{equation*}
\nabla D=\begin{pmatrix}\gamma_e&\gamma_e^{-1}c^{-2}n v^T\\ c^{-2}h^\prime(n)v&c^{-2}h(n)I_3\end{pmatrix},
\end{equation*}
which also in particular implies that
\begin{equation}\label{useful1}
\begin{split}
\partial_jh=-\frac{h^\prime(n_0)}{4\pi e^2}\partial_j\partial_kE^k+h.o.t.
\end{split}
\end{equation}
The dynamical equations then reduce to the following (from \eqref{Bianchi2} and the first equation of \eqref{Max})
\begin{equation}\label{NewSysREM1F}
\begin{split}
\partial_t\mu_e^j+E^{j}+\frac{1}{\gamma}\partial_jh+\frac{c^2}{h\gamma_e}\partial_j\frac{\vert \mu_e\vert^2}{2}&=0\\
\partial_tE^j-c^2(\hbox{curl}\,\hbox{curl}(\mu_e))^j-4\pi e^2c^2\frac{n}{ h}\mu_e^j&=0,
\end{split}
\end{equation}
where
\begin{equation*}
(\hbox{curl}(v))^i:=\in^{ijk}\partial_jv_k.
\end{equation*}

We can now choose scales appropriately so as to minimize the number of parameters in the linear system. Define\footnote{Here $\omega_e$ denotes the (nonrelativist) electron plasma frequency and $c_s$ denotes the (non relativist) sound velocity.}
\begin{equation*}
\begin{split}
&\lambda=\sqrt{\frac{4\pi e^2n_0}{h_0}}\approx c\omega_e,\qquad T=\frac{p_0^\prime}{h_0}\approx\left(\frac{c_s}{c}\right)^2,\quad p_0^\prime=p^\prime(n_0),\quad \beta=\lambda c\\
&\mu_e(x,t)=\widetilde{\mu}(\lambda x,\beta t),\quad E(x,t)=\beta \widetilde{E}(\lambda x,\beta t),\\
&n=n_0(1+\widetilde{n}(\lambda x,\beta t)),\quad h_0\widetilde{h}=h(\widetilde{n}),\quad\widetilde{\gamma}=\sqrt{1+(c/h_0)^{2}\widetilde{h}^{-2}\vert\widetilde{\mu}\vert^2}\\
\end{split}
\end{equation*}
and introduce
\begin{equation*}
Q=\vert\nabla\vert^{-1}\hbox{curl},\quad P=-\nabla(-\Delta)^{-1}\hbox{div},\quad P^2+Q^2=Id,\quad PQ=0,\quad P^2=P,\quad Q^3=Q.
\end{equation*}
We can recast \eqref{NewSysREM1F} as
\begin{equation*}
\begin{split}
\partial_t\widetilde{\mu}+\widetilde{E}-T\Delta P\widetilde{E}&=\mathcal{N}_1\\
\partial_t\widetilde{E}+\Delta Q^2\widetilde{\mu}-\widetilde{\mu}&=\mathcal{N}_2,
\end{split}
\end{equation*}
where
\begin{equation*}
\begin{split}
-\mathcal{N}_1&=\frac{c}{h_0}\nabla\frac{\vert\widetilde{\mu}\vert^2}{2}+\left\{T\Delta P\widetilde{E}+\frac{h_0}{\widetilde{\gamma}}\nabla \widetilde{h}\right\}+\frac{c}{h_0}\left\{\frac{1}{\widetilde{h}\widetilde{\gamma}}-1\right\}\nabla\frac{\vert \widetilde{\mu}\vert^2}{2}\\
-\mathcal{N}_2&=\left\{1-\frac{1+\widetilde{n}}{\widetilde{ h}}\right\}\widetilde{\mu}^j.
\end{split}
\end{equation*}

We define
\begin{equation*}
\begin{split}
\Lambda_e^2:=1-T\Delta,\qquad\Lambda_b^2&:=1-\Delta
\end{split}
\end{equation*}
and we introduce the dispersive unknowns
\begin{equation}\label{NewUREM1F}
\begin{split}
U_e&:=P\widetilde{\mu}-i\Lambda_eP\widetilde{E}\\
U_b&:=Q\widetilde{\mu}-i\Lambda_b^{-1} Q\widetilde{E}
\end{split}
\end{equation}
which satisfy the system
\begin{equation}\label{NewSysREM1F2}
\begin{split}
\left(\partial_t+i\Lambda_e\right)U_e&=P\mathcal{N}_1-i\Lambda_eP\mathcal{N}_2\\
\left(\partial_t+i\Lambda_b\right)U_b&=Q\mathcal{N}_1-i\Lambda_b^{-1}Q\mathcal{N}_2.
\end{split}
\end{equation}
Now, it suffices to remark that $(\widetilde{E},\widetilde{h},\widetilde{n},\widetilde{\mu},\widetilde{\gamma})$ are smooth invertible functions of $U_e$ and $U_b$ to see that the nonlinearities in \eqref{NewSysREM1F2} correspond to a quadratic form that can be treated using the same techniques as in \cite{IoPa2} plus a cubic term that can be treated using techniques similar to those of Subsection \ref{Cubic}. This yields the following theorem:

\begin{theorem}\label{MainThmREM1F}
Fix $h$, $c$, $T$. Let $N_0=10^4$ and consider $\nu_0=n_0(1+\rho_0)$ with $(\rho_0,v_0,F_0)\in H^{N_0+1}$ satisfy
\begin{equation*}
\partial_j F^{j0}_0=-4\pi c^{-1} \nu_0\sqrt{1+c^{-2}\vert v_0\vert^2}-n_0,\qquad ec^{-1}F^{jk}_0=\partial_k(h(\nu_0)v_0^j)-\partial_j(h(\nu_0)v_0^k),
\end{equation*}
and
\begin{equation*}\label{maincond2REM}
\begin{split}
&\|(v_0,F_0)\|_{H^{N_0+2}}+\sum_{j=1}^3\|((1-\Delta)v_0^j,(1-\Delta)F_0^{j0})\|_Z=\varepsilon_0\leq\overline{\varepsilon},\\
&
\end{split}
\end{equation*}
where $\overline{\varepsilon}>0$ is sufficiently small, and the $Z$ norm is defined in Definition \ref{MainDef}. Then there exists a unique global solution $(n,u,F)$ with $(n-n_0,v,F)\in C([0,\infty):H^{N_0+1})$ of the system \eqref{Max}-\eqref{RelCur}-\eqref{ContMatt}-\eqref{Bianchi2} with initial data $(n(0),v(0),F(0))=(\nu_0,v_0,F_0)$. Moreover,
\begin{equation*}\label{mainconcl2REM}
\partial_j F^{j0}(t)=-4\pi c^{-1} [n(t)\sqrt{1+c^{-2}\vert v(t)\vert^2}-n_0],\quad ec^{-1}F_{jk}=\partial_k(hv_j)-\partial_j(hv_k),\quad\text{ for any }t\in[0,\infty),
\end{equation*}
and, with $\beta=1/100$,
\begin{equation*}\label{mainconcl2.1REM}
\sup_{t\in[0,\infty)}\|(v(t),F(t))\|_{H^{N_0+1}}+\sup_{t\in[0,\infty)}\sup_{|\rho|\leq 4}(1+t)^{1+\beta}\|(D^\rho_x v(t),D^\rho_x F(t))\|_{L^\infty}\lesssim \varepsilon_0.
\end{equation*}
\end{theorem}

The proof follows the same strategy as in \cite{IoPa2}. To find the high-order energies, we start from the conservation of energy:
\begin{equation*}
\partial_\nu\left[T^{0\nu}+\mathcal{E}^{0\nu}\right]=0
\end{equation*}
which implies that
\begin{equation*}
\mathcal{E}_0=\int_{\mathbb{R}^3}\left[T^{00}+\mathcal{E}^{00}\right]dx
\end{equation*}
is conserved. We then define the high order energies by
\begin{equation*}
\mathcal{E}_N=\int_{\mathbb{R}^3}\left[h^\prime(n)\vert D^N n\vert^2+\frac{nh}{c^2}\sum_{j=1}^3\vert D^Nv^j\vert^2+\frac{1}{4\pi e^2}\sum_{j=1}^3\vert D^NE^{j}\vert^2+\frac{c^2}{8\pi}\sum_{j,k=1}^3\vert D^NF^{jk}\vert^2\right]dx,
\end{equation*}
from the Hessian of $\mathcal{E}_{0}$.

The semilinear analysis is a direct adaptation of \cite[Section 3 and 4]{IoPa2}. Once we remark that $U_e(0)$ and $U_b(0)$ defined in terms of $(v_0,F_0)$ and $h$ as above satisfy
\begin{equation*}
\|(U_e(0),U_b(0))\|_{H^{N_0+1}}+\|(U_e(0),U_b(0))\|_Z\le \varepsilon_0^\frac{3}{4},
\end{equation*}
and \eqref{NewSysREM1F2}.

\subsubsection{Relativistic Euler-Maxwell model for two fluids} Now we consider the relativistic analogue of \eqref{EuMa1}. We consider again the standard Minkowski space $(\mathbb{R}^{1+3},g)$ defined as in the previous subsection.

The main unknowns are two densities $n_i$ and $n_e$, two velocity fields $v_i$ and $v_e$ (both of which satisfy \eqref{NormREMFluids}) and an electromagnetic field $F$. We are also given pressure laws $p_i$ and $p_e$ and enthalpies $h_i$ and $h_e$ satisfying \eqref{Enth}, with $M_i$, the rest-mass of an ion instead of $m_e$ for $p_i$, $h_i$.

\medskip

The Maxwell equations \eqref{Max} remain the same, with the relativistic current now defined as
\begin{equation}\label{RelCur2}
J^\nu=Ze n_iu_i^\nu-en_eu_e^\nu
\end{equation}
instead of \eqref{RelCur}. Both species are independently conserved so that
\begin{equation}\label{ContMatt2}
\partial_\nu(n_iu_i^\nu)=0=\partial_\nu(n_eu_e^\nu)
\end{equation}
and we have two forms of balance of momentum:
\begin{equation}\label{Bianchi2F}
\begin{split}
\frac{n_iu_i^\nu}{c^2}\partial_\nu\left[h_iu_i^\mu\right]+g^{\mu\nu}\partial_\nu p_i=-Z\frac{en_i}{c}(u_i)_\alpha F^{\mu\alpha}\\
\frac{n_eu_e^\nu}{c^2}\partial_\nu\left[h_eu_e^\mu\right]+g^{\mu\nu}\partial_\nu p_e=\frac{en_e}{c}(u_e)_\alpha F^{\mu\alpha}.
\end{split}
\end{equation}
In particular, we recover the fact that the stress-energy tensor is divergence free\footnote{In general relativity, the stress energy tensor should be equated to a multiple of the Einstein tensor of space time. The fact that the stress energy tensor is divergence free is then a consequence of the Bianchi identity.}
\begin{equation*}
\begin{split}
&\partial_\nu\left[T_i^{\mu\nu}+T_e^{\mu\nu}+\mathcal{E}^{\mu\nu}\right]=0,\\
&T_i^{\mu\nu}=n_ih_i\frac{u^\mu_i u^\nu_i}{c^2}+p_ig^{\mu\nu},\qquad T_e^{\mu\nu}=n_eh_e\frac{u^\mu_e u^\nu_e}{c^2}+p_eg^{\mu\nu}.
\end{split}
\end{equation*}

Again, we have two naturally transported (generalized) vorticities:
\begin{equation*}
\begin{split}
\omega^i_{\alpha\beta}&=\partial_\alpha\left[h_i(u_i)_\beta\right]-\partial_\beta\left[h_i(u_i)_\alpha\right]-ZecF_{\alpha\beta},\\
\omega^e_{\alpha\beta}&=\partial_\alpha\left[h_e(u_e)_\beta\right]-\partial_\beta\left[h_e(u_e)_\alpha\right]+ecF_{\alpha\beta},\\
\end{split}
\end{equation*}
which satisfy that
\begin{equation}\label{VortTrans2F}
\begin{split}
u_i^\nu\partial_\nu\omega^i_{\alpha\beta}&=-(\partial_\alpha u_i^\nu)\omega^i_{\nu\beta}+(\partial_\beta u_i^\nu)\omega^i_{\nu\alpha},\\
u_e^\nu\partial_\nu\omega^e_{\alpha\beta}&=-(\partial_\alpha u_e^\nu)\omega^e_{\nu\beta}+(\partial_\beta u_e^\nu)\omega^e_{\nu\alpha}.
\end{split}
\end{equation}
We thus see that irrotational flows are well-defined and remain irrotational along the flow.

\medskip

We can easily see from \eqref{Enth} that the component of \eqref{Bianchi2F} parallel to the fluids under consideration are automatically satisfied. Thus to verify \eqref{Bianchi2F}, it suffices to verify it when $\mu=j$ varies between $1$ and $3$.

\medskip

We now define the unknowns
\begin{equation*}
\begin{split}
&\mu_i^j=c^{-2}h_iu_i^j,\quad \mu_e^j=c^{-2}h_eu_e^j,\,\, 1\le j\le 3,\\
&E^j=ecF^{j0},\quad 2B^j=-ec^{-1}\in^{jkl}F_{kl},\quad F^{jk}=-ce^{-1}\in^{jkl}B^l.
\end{split}
\end{equation*}
Since we only consider irrotational flows,
\begin{equation}\label{IrrB}
B=\nabla\times\mu_e=-Z^{-1}\nabla\times\mu_i.
\end{equation}
at all times and we see that $B$ is uniquely determined by $\mu_i$ or $\mu_e$.

Now, we can rewrite our evolution system as
\begin{equation*}
\begin{split}
\partial_t(n_i\gamma_i)+c^2\hbox{div}(\frac{n_i}{h_i}\mu_i)&=0,\\
\partial_t\mu_i+\frac{1}{\gamma_i}\partial_jh_i-ZE^j+\frac{c^2}{\gamma_i h_i}\partial_j\frac{\vert \mu_i\vert^2}{2}&=0,\\
\partial_t(n_e\gamma_e)+c^2\hbox{div}(\frac{n_e}{h_e}\mu_e)&=0,\\
\partial_t\mu_e^j+\frac{1}{\gamma_e}\partial_jh_e+E^j+\frac{c^2}{h_e\gamma_e}\partial_j\frac{\vert\mu_e\vert^2}{2}&=0,\\
\partial_tE-c^2\hbox{curl}(B)+4\pi e^2c^2[Z\frac{n_i}{h_i}\mu_i-\frac{n_e}{h_e}\mu_e]&=0,\\
\partial_tB+\hbox{curl}(E)&=0.\\
\end{split}
\end{equation*}
We now set
\begin{equation*}
\begin{split}
H_i=h_i(n_0/Z),\quad n_0P_iZ=p_i^\prime(n_0/Z),\quad H_e=h_e(n_0),\quad n_0P_e= p_e^\prime(n_0)\\
\beta:=\sqrt{\frac{4\pi n_0Ze^2c^2}{H_i}},\qquad\lambda:=\sqrt{\frac{4\pi e^2}{P_i}},\qquad\mu:=\frac{\sqrt{n_0ZP_iH_i}}{c}
\end{split}
\end{equation*}
and
\begin{equation}\label{NewEpsTC}
\varepsilon=\frac{ZH_e}{H_i},\qquad T=\frac{P_e}{P_i},\qquad C_b=\frac{H_e}{n_0P_i}
\end{equation}
and use the rescaling
\begin{equation*}
\begin{split}
\gamma_i(x,t)=\widetilde{\gamma}_i(\lambda x,\beta t),\qquad&\gamma_e(x,t)=\widetilde{\gamma}_e(\lambda x,\beta t),\\
n_i(x,t)\gamma_i(x,t)=(n_0/Z)[\rho(\lambda x,\beta t)+1],\qquad&n_e(x,t)\gamma_e(x,t)=n_0[n(\lambda x,\beta t)+1]\\
\mu_i(x,t)= \mu u(\lambda x,\beta t),\qquad&\mu_e(x,t)=(\varepsilon\mu/Z)v(\lambda x,\beta t)\\
E(x,t)=n_0\lambda P_i\widetilde{E}(\lambda x,\beta t),\qquad&B(x,t)=(\lambda\mu/Z)\widetilde{B}(\lambda x,\beta t)\\
h_i(n_i(x,t))=H_i\widetilde{h}_i(\widetilde{\rho}(\lambda x,\beta t)),\qquad &h_i^\prime(n_i)=Z^2P_iq_i^\prime(\widetilde{\rho}),\qquad \widetilde{h}_i(0)=1=q_i^\prime(0)\\
h_e(n_e(x,t))=H_e\widetilde{h}_e(\widetilde{n}(\lambda x,\beta t)),\qquad&h_e^\prime(n_e)=P_eq_e^\prime(\widetilde{n}),\qquad\widetilde{h}_e(0)=1=q_e^\prime(0)\\
\end{split}
\end{equation*}
to obtain the system
\begin{equation}\label{NewSysREM2F}
\begin{split}
\partial_t\rho+\hbox{div}[\frac{1+\rho}{\widetilde{\gamma}_i\widetilde{h}_i}u]&=0,\\
\partial_tu^j-\widetilde{E}^j+\frac{1}{\widetilde{\gamma}_i}\partial_jq_i+\frac{1}{\widetilde{\gamma}_i\widetilde{h}_i}\partial_j\frac{\vert u\vert^2}{2}&=0,\\
\partial_tn+\hbox{div}[\frac{1+n}{\widetilde{\gamma}_e\widetilde{h}_e}v]&=0,\\
\varepsilon\{\partial_tv+\frac{1}{\widetilde{h}_e\widetilde{\gamma}_e}\partial_j\frac{\vert v\vert^2}{2}\}+\widetilde{E}^j+\frac{T}{\widetilde{\gamma}_e}\partial_jq_e&=0,\\
\partial_t\widetilde{E}^j-\frac{C_b}{\varepsilon}\hbox{curl}(\widetilde{B})+[\frac{1+\rho}{\widetilde{\gamma}_i\widetilde{h}_i}u-\frac{1+n}{\widetilde{\gamma}_e\widetilde{h}_e}v]&=0\\
\partial_t\widetilde{B}+\hbox{curl}(\widetilde{E})&=0\\
\end{split}
\end{equation}
which has a similar structure to \eqref{EuMa1}. Indeed, we may Taylor expand to get
\begin{equation*}
\begin{split}
\frac{1+\rho}{\widetilde{\gamma}_i\widetilde{h}_i}=1+r_1\rho+g_1,\quad& q_i=\rho+r_2\rho^2+r_2^\prime\vert u\vert^2+h_2,\\
\frac{1+n}{\widetilde{\gamma}_e\widetilde{h}_e}=1+r_3n+g_3,\quad& q_e=n+r_4n^2+r_4^\prime\vert v\vert^2+h_4,\\
\end{split}
\end{equation*}
where $r_1$, $r_2$, $r_2^\prime$, $r_3$, $r_4$ and $r_4^\prime$ are constants and $g_1,g_3$ are smooth functions of $(\rho,u,n,v)$ which vanish at the origin $(0,0,0,0)$ together with their gradient, and $h_2,h_4$ are smooth functions of $(\rho,u,n,v)$ which vanish at the origin $(0,0,0,0)$ together with their first and second derivatives.

We may thus rewrite \eqref{NewSysREM2F} as
\begin{equation}\label{NewSysREM2F2}
\begin{split}
\partial_t\rho+\hbox{div}[u]+r_1\hbox{div}[\rho u]&=-\hbox{div}[g_1],\\
\partial_tu^j-\widetilde{E}^j+\partial_j\rho+r_2\partial_j(\rho^2)+(r_2^\prime+\frac{1}{2})\partial_j\vert u\vert^2&=T_2,\\
\partial_tn+\hbox{div}[v]+r_3\hbox{div}[nv]&=-\hbox{div}[g_3],\\
\partial_tv+\varepsilon^{-1}\widetilde{E}^j+\varepsilon^{-1}T\partial_jn+\varepsilon^{-1}Tr_4\partial_j(n^2)+(\varepsilon^{-1}Tr_r^\prime+\frac{1}{2})\partial_j\vert v\vert^2&=T_4,\\
\partial_t\widetilde{B}+\hbox{curl}(\widetilde{E})&=0,\\
\partial_t\widetilde{E}-\frac{C_b}{\varepsilon}\hbox{curl}(\widetilde{B})+u-v+[r_1\rho u-r_3nv]&=-g_1u+g_3v,
\end{split}
\end{equation}
where
\begin{equation*}
\begin{split}
T_2&=-\{\widetilde{\gamma}_i^{-1}-1\}\partial_jq_i-\widetilde{\gamma}_i^{-1}\partial_jh_2-((\widetilde{\gamma}_i\widetilde{h}_i)^{-1}-1)/2\cdot\partial_j\vert u\vert^2\\
T_4&=-\{\varepsilon^{-1}T\widetilde{\gamma}_e^{-1}-1\}\partial_jq_e-\varepsilon^{-1}\widetilde{\gamma}_e^{-1}\partial_jh_4-((\widetilde{h}_e\widetilde{\gamma}_e)^{-1}-1)/2\cdot\partial_j\vert v\vert^2
\end{split}
\end{equation*}
are simply smooth cubic (or higher order) terms in $(\rho,n,u,v)$ with no particle structure. We can directly observe that the linearization of \eqref{NewSysREM2F2} coincides with the linearization of \eqref{EuMa1}. Therefore we consider the same dispersion relations \eqref{operators1} and we define the dispersive unknowns as in \eqref{pla7}:
\begin{equation}\label{REM2FDUnknowns}
\begin{split}
&U_i:=\frac{1}{2\sqrt{1+R^2}}\big[\varepsilon^{1/2}R|\nabla|^{-1}\Lambda_i n+|\nabla|^{-1}\Lambda_i\rho-i\varepsilon^{1/2}RR_jv^j-iR_ju^j\big],\\
&U_e:=\frac{1}{2\sqrt{1+R^2}}\big[-\varepsilon^{1/2}|\nabla|^{-1}\Lambda_e n+R|\nabla|^{-1}\Lambda_e\rho+i\varepsilon^{1/2}R_jv^j-iRR_ju^j\big],\\
&2U_b:=\Lambda_b\vert\nabla\vert^{-1}Q\tilde{B}-iQ^2\tilde{E}
 \end{split}
 \end{equation}
with inverse transformation given by
\begin{equation}\label{pla20REM2F}
\begin{split}
&n=\frac{-|\nabla|\varepsilon^{-1/2}}{\sqrt{1+R^2}\Lambda_e}(U_e+\overline{U_e})+\frac{|\nabla|\varepsilon^{-1/2}R}{\sqrt{1+R^2}\Lambda_i}(U_i+\overline{U_i}),\\
&\rho=\frac{|\nabla|R}{\sqrt{1+R^2}\Lambda_e}(U_e+\overline{U_e})+\frac{|\nabla|}{\sqrt{1+R^2}\Lambda_i}(U_i+\overline{U_i}),\\
&v^j=R_j\left\{\frac{i\varepsilon^{-1/2}}{\sqrt{1+R^2}}(U_e-\overline{U_e})+\frac{-i\varepsilon^{-1/2}R}{\sqrt{1+R^2}}(U_i-\overline{U_i})\right\}+\frac{2}{\varepsilon}\Lambda_b^{-1}\hbox{Re}(U_b^j),\\
&u^j=R_j\left\{\frac{-iR}{\sqrt{1+R^2}}(U_e-\overline{U_e})+\frac{-i}{\sqrt{1+R^2}}(U_i-\overline{U_i})\right\}-2\Lambda_b^{-1}\hbox{Re}(U_b^j).\\
\end{split}
\end{equation}

Using these formulas (and in particular the fact that $\partial_tn$ and $\partial_t\rho$ are exact spatial derivatives so as to counteract the singular relation at $0$ frequency in the definition of $U_e$ and to keep the derivative structure in the quadratic part of the nonlinearity $\mathcal{N}_i$), one quickly sees that $(U_i,U_e,U_b)$ satisfy \eqref{pla51} with quadratic nonlinear terms of the form given in \eqref{pla52} and Lemma \ref{tech1.3} and cubic nonlinear terms of the form handled by Proposition \ref{PropCub}.

Following the analysis in Section \ref{normproof}-\ref{normproof4}, we may obtain
\begin{theorem}
\label{REM2FThm}
Fix $h_i$, $h_e$, $c$, define $\varepsilon$, $T$, $C_b$ by \eqref{NewEpsTC} and assume \eqref{condTeps}. Let $N_{0}=10^{4}$ and assume that 
\begin{equation*}
\begin{split}
& \Vert (n^{0}-n_0,v^0,\rho^0-Z^{-1}n_0,u^0,F^{0})\Vert
_{H^{N_{0+2}}}+\Vert (1-\Delta)(n^{0},v^{0},\rho
^{0},u^{0},F)\Vert _{Z}=\delta _{0}\leq \overline{\delta }
, \\
& c\partial_j(F^0)^{j0}+4\pi e[n^{0}\sqrt{1+c^{-2}\vert v^0\vert^2}-Z\rho ^{0}\sqrt{1+c^{-2}\vert u^0\vert^2}]=0,\\
&ecF^0_{jk}=Z^{-1}\{\partial_j[h_i(\rho^0)u^0_k]-\partial_k[h_i(\rho^0)u^0_j]\}=\partial_k[h_e(n^0)v^0_j]-\partial_j[h_e(n^0)v^0_k],
\end{split}
\end{equation*}
where $\overline{\delta }>0$ is sufficiently small, and the $Z$ norm is defined in Definition \ref{MainDef}. Then there exists a unique global solution $(n_e-n_0,u_e,n_i-Z^{-1}n_0 ,u_i,F)\in C([0,\infty ):{H^{N_{0}}})$ of the system \eqref{Max}-\eqref{RelCur2}-\eqref{ContMatt2}-\eqref{Bianchi2F} with initial data $(n_e(0),v_e(0),n_i(0),v_i(0),F(0))=(n^{0},v^{0},\rho ^{0},u^{0},F^0)$. Moreover, for any $t\in \lbrack 0,\infty )$, 
\begin{equation*}
\begin{split}
&ecF^0_{jk}=Z^{-1}\{\partial_j[h_i(v_i)_k]-\partial_k[h_i(v_i)_j]\}=\partial_k[h_e(v_e)_j]-\partial_j[h_e(v_e)_k],
\end{split}
\end{equation*}
and, with $\beta :=1/100$, 
\begin{equation*}
\begin{split}
&\Vert (n_e(t)-n_0,v_e(t),n_i(t)-Z^{-1}n_0,v_i(t),F(t))\Vert
_{H^{N_{0}}}\\
&+\sup_{|\alpha |\leq 4}(1+t)^{1+\beta /2}\Vert
D_{x}^{\alpha }(n_e(t)-n_0,v_e(t),n_i(t)-Z^{-1}n_0,v_i(t),F(t))\Vert _{L^{\infty
}}\lesssim \delta
_{0}.
\end{split}
\label{mainconcl2.1REM2F}
\end{equation*}
\end{theorem}

\subsection{Cubic nonlinearities}\label{Cubic}

We consider an operator of the type
\begin{equation}\label{CubicOp}
\begin{split}
&\tilde{T}^{\sigma;\mu,\nu}_m[f,g;h](x)=\int_{\mathbb{R}}\int_{\mathbb{R}^9}q_m(s)m(\xi_1,\xi_2,\xi_3)e^{i\Phi(s,\xi_1,\xi_2,\xi_3)}\widehat{f}(\xi_1,s)\widehat{g}(\xi_2,s)\widehat{h}(\xi_3,s)dsd\xi_1 d\xi_2 d\xi_3,\\
&\Phi(s,\xi_1,\xi_2,\xi_3)=x\cdot(\xi_1+\xi_2+\xi_3)+s\left[\Lambda_\sigma(\xi_1+\xi_2+\xi_3)-\widetilde{\Lambda}_{\mu}(\xi_1)-\widetilde{\Lambda}_\nu(\xi_2)\right]
\end{split}
\end{equation}
where
\begin{equation}\label{StructureCubMult}
m(\xi_1,\xi_2,\xi_3)=q_0(\xi_1+\xi_2+\xi_3)q_1(\xi_1)q_2(\xi_2)q_3(\xi_3),\quad\sup_{0\le n\le 3}\Vert q_n\Vert_{\mathcal{S}^{100}}\le 1.
\end{equation}
In particular,
\begin{equation*}
\begin{split}
&\mathcal{F}\tilde{T}^{\sigma;\mu,\nu}_m[f,g;h](\xi)\\
&=\int_{\mathbb{R}}\int_{\mathbb{R}^6}q_m(s)m(\xi-\eta,\eta-\theta,\theta)e^{is\left[\Lambda_\sigma(\xi)-\widetilde{\Lambda}_{\mu}(\xi-\eta)-\widetilde{\Lambda}_\nu(\eta-\theta)\right]}\widehat{f}(\xi-\eta,s)\widehat{g}(\eta-\theta,s)\widehat{h}(\theta,s)dsd\eta d\theta.\\
\end{split}
\end{equation*}

We will prove the following proposition
\begin{proposition}\label{PropCub}
Assume that
\begin{equation*}
\Vert f^\mu\Vert_{H^{N_0}\cap Z}+\Vert f^\nu\Vert_{H^{N_0}\cap Z}\le 2
\end{equation*}
and that $h$ satisfies for any $s\in [0,T_0]$, and any $k\in\mathbb{Z}$
\begin{equation}\label{AssumpHCubic}
\begin{split}
\Vert P_kh(s)\Vert_{L^2}&\lesssim \min(2^{-N_0k},2^{(1+\beta-\alpha)k}),\\
\Vert P_kh(s)\Vert_{L^\infty}&\lesssim \min(2^{-6k},2^{(1/2-\beta-\alpha)k})(1+s)^{-1-\beta}
\end{split}
\end{equation}
then  there holds that
\begin{equation}\label{CubicBound}
\sum_{(k_1,j_1),(k_2,j_2)\in\mathcal{J},k_3}(1+2^{k_1}+2^{k_2}+2^{k_3})\Vert \tilde{\varphi}^{(k)}_j\cdot P_{k_0}\tilde{T}^{\sigma;\mu,\nu}_m[f^{\mu}_{k_1,j_1},f^\nu_{k_2,j_2};P_{k_3}h]\Vert_{B^1_{k,j}}\lesssim 2^{-\beta^4m}
\end{equation}
for any choice of
\begin{equation*}
\sigma\in\{i,e,b\},\quad\mu,\nu\in\mathcal{I}_0,\quad(k,j)\in\mathcal{J},\quad m\in\{0,\dots,L+1\}
\end{equation*}

\end{proposition}

\begin{proof}[Proof of Proposition \ref{CubicOp}]

Proceeding as for Proposition \ref{reduced1}, from \eqref{StructureCubMult}, we may assume that $m(\xi_1,\xi_2,\xi_3)=1$. Then, Plancherel theorem gives that
\begin{equation}\label{BasicPlaCub}
\begin{split}
&\Vert \tilde{T}^{\sigma;\mu,\nu}_{m}[P_{k_1}f,P_{k_2}g;P_{k_3}h]\Vert_{L^2}\\
&\lesssim 2^m\sup_{s\in[2^{m-4},2^{m+4}]}\min\{\Vert P_{k_1}f(s)\Vert_{L^2}\Vert e^{is\Lambda_\nu}P_{k_2}g(s)\Vert_{L^\infty}\Vert P_{k_3}h(s)\Vert_{L^\infty},\\
&\Vert e^{is\Lambda_\mu}P_{k_1}f(s)\Vert_{L^\infty}\Vert P_{k_2}g(s)\Vert_{L^2}\Vert P_{k_3}h(s)\Vert_{L^\infty}, \Vert e^{is\Lambda_\mu}P_{k_1}f(s)\Vert_{L^\infty}\Vert e^{is\Lambda_\nu} P_{k_2}g(s)\Vert_{L^\infty}\Vert P_{k_3}h(s)\Vert_{L^2}\},
\end{split}
\end{equation}
which will be used repeatedly.

\medskip

Let $k_{++}=\max(0,k_1,k_2,k_3)$. We first claim that if
\begin{equation*}
j\le m/2+N_0^\prime k_{++}+D^2
\end{equation*}
where $N_0^\prime$ is as in \eqref{nh7}, then
\eqref{CubicBound} holds.

We proceed as for Lemma \ref{BigBound1}. It suffices to show that
\begin{equation*}
\sum_{(k_1,j_1),(k_2,j_2)\in\mathcal{J},k_3}2^{k_{++}}(2^{\alpha k}+2^{10k})2^{3j/2}2^{(1/2-\beta)\tilde{k}}\Vert \tilde{\varphi}^{(k)}_j\cdot P_{k}\tilde{T}^{\sigma;\mu,\nu}_{m}[f^{\mu}_{k_1,j_1},f^\nu_{k_2,j_2};P_{k_3}h]\Vert_{L^2}\lesssim 2^{-\beta^4m}.
\end{equation*}
Using \eqref{nh9}, we first observe that
\begin{equation*}
\begin{split}
\Vert \tilde{T}^{\sigma;\mu,\nu}_{m}[f^{\mu}_{k_1,j_1},f^\nu_{k_2,j_2};P_{k_3}h]\Vert_{L^2}
&\lesssim \int_{\mathbb{R}}q_m(s)\Vert Ef^{\mu}_{k_1,j_1}(s)\Vert_{L^\infty}\Vert Ef^\nu_{k_2,j_2}(s)\Vert_{L^\infty}\Vert P_{k_3}h(s)\Vert_{L^2}ds\\
&\lesssim 2^{-\beta (\vert k_1\vert+\vert k_2\vert)}2^{-m}2^{-N_0k_3}
\end{split}
\end{equation*}
and therefore
\begin{equation*}
\sum_{(k_1,j_1),(k_2,j_2)\in\mathcal{J},k_3\ge k_{++}-10}2^{k_{++}}(2^{\alpha k}+2^{10k})2^{3j/2}2^{(1/2-\beta)\tilde{k}}\Vert \tilde{\varphi}^{(k)}_j\cdot P_{k}\tilde{T}^{\sigma;\mu,\nu}_{m}[f^{\mu}_{k_1,j_1},f^\nu_{k_2,j_2};P_{k_3}h]\Vert_{L^2}\lesssim 2^{-\beta^4m}.
\end{equation*}
On the other hand, if $k_3\le k_{++}-10$, then $\max(0,k_1,k_2)\ge k_{++}$ and we may proceed as in the proof of Lemma \ref{BigBound1} using \eqref{BasicPlaCub}.

\medskip

We now define three sets
\begin{equation*}
\begin{split}
&\mathcal{S}_1=\{(k_1,j_1),(k_2,j_2)\in\mathcal{J}:\,\max(k_1,k_2,k_3)\ge j/N_0^\prime\},\\
&\mathcal{S}_2=\{(k_1,j_1),(k_2,j_2)\in\mathcal{J}:\,\min(k_1,k_2,k_3)\le -10j\},\quad\mathcal{S}_3=\{(k_1,j_1),(k_2,j_2)\in\mathcal{J}:\,\max(j_1,j_2)\ge 10j\}
\end{split}
\end{equation*}
and claim that if
\begin{equation*}
j\ge m/2+N_0^\prime k_++D^2
\end{equation*}
then, for $p=1,2,3$, we have that
\begin{equation}\label{CubicBound1.1}
\sum_{(k_1,j_1),(k_2,j_2)\in\mathcal{S}_p,k_3}2^{k_{++}}\Vert \tilde{\varphi}^{(k)}_j\cdot P_{k}\tilde{T}^{\sigma;\mu,\nu}_{m}[f^{\mu}_{k_1,j_1},f^\nu_{k_2,j_2};P_{k_3}h]\Vert_{B^1_{k,j}}\lesssim 2^{-\beta^2m}.
\end{equation}

This is done similarly to the proof of Lemma \ref{BigBound2}.

\medskip

We can now consider each term in the remaining sum separately. More precisely, it suffices to prove that
\begin{equation}\label{CubicBound2}
2^{k_{++}}\Vert \tilde{\varphi}^{(k)}_j\cdot P_{k}\tilde{T}^{\sigma;\mu,\nu}_m[f^{\mu}_{k_1,j_1},f^\nu_{k_2,j_2};P_{k_3}h]\Vert_{B^1_{k,j}}\lesssim 2^{-\beta^2m}
\end{equation}
provided that
\begin{equation}\label{CondCubicBound2}
j\ge m/2+N_0^\prime k_{++}+D^2,\quad -10j\le k_1,k_2,k_3\le j/N_0^\prime,\quad\max(j_1,j_2)\le 10j.
\end{equation}

This is similar to the proofs of Lemma \ref{BigBound3} and \ref{BigBound4} but in an easier situation.
We first assume that
\begin{equation}\label{AssCCB2-1}
j+4k/3\le 0.
\end{equation}
In this case, as in Lemma \ref{BigBound4}, it suffices to prove that
\begin{equation*}
2^{k_{++}}2^{\alpha k}2^{(1+\beta)j}2^{3k/2}\Vert \mathcal{F}P_{k}\tilde{T}^{\sigma;\mu,\nu}_m[f^{\mu}_{k_1,j_1},f^\nu_{k_2,j_2};P_{k_3}h]\Vert_{L^\infty}\lesssim 2^{-\beta^4(m+j)}.
\end{equation*}
On the other hand, Cauchy-Schwartz inequality gives that, for any $\xi$,
\begin{equation*}
\begin{split}
\left\vert\mathcal{F}P_{k}\tilde{T}^{\sigma;\mu,\nu}_m[f^{\mu}_{k_1,j_1},f^\nu_{k_2,j_2};P_{k_3}h](\xi)\right\vert
&\lesssim \int_{\mathbb{R}}q_m(s)\Vert f^{\mu}_{k_1,j_1}(s)\Vert_{L^2}\Vert (P_{k_3}h(s))\cdot f^\nu_{k_2,j_2}(s)\Vert_{L^2}ds\\
&\lesssim 2^{-\max(0,k_3)}2^{-\beta m}\sup_{s\in [2^{m-2},2^{m+2}]}\Vert f^{\mu}_{k_1,j_1}(s)\Vert_{L^2}\Vert f^\nu_{k_2,j_2}(s)\Vert_{L^2}\\
&\lesssim 2^{-k_{++}}2^{-\beta m}.
\end{split}
\end{equation*}
which suffices given \eqref{AssCCB2-1}.

\medskip

Recall Definition \ref{MainDef}. We now show that, whenever
\begin{equation*}
j\ge m/2+N_0^\prime k_{++}+D^2,\quad -10j\le k_1,k_2,k_3\le j/N_0^\prime,\quad\max(j_1,j_2)\le 10j,
\end{equation*}
there holds that
\begin{equation}\label{LinftyCunic}
2^{k_{++}}(2^{\alpha k}+2^{10k})\Vert\mathcal{F}\left\{ \tilde{\varphi}^{(k)}_j\cdot P_{k}\tilde{T}^{\sigma;\mu,\nu}_m[f^{\mu}_{k_1,j_1},f^\nu_{k_2,j_2};P_{k_3}h]\right\}\Vert_{L^\infty}\lesssim 2^{-\beta^2m}.
\end{equation}
Indeed, for any $\xi$,
\begin{equation*}
\begin{split}
\left\vert\tilde{T}^{\sigma;\mu,\nu}_m[f^{\mu}_{k_1,j_1},f^\nu_{k_2,j_2};P_{k_3}h](\xi)\right\vert
&\lesssim \int_{\mathbb{R}}q_m(s)\Vert \mathcal{F}\{Ef^{\mu}_{k_1,j_1}(s)\cdot f^\nu_{k_2,j_2}(s)\}\Vert_{L^2}\Vert P_{k_3}h(s)\Vert_{L^2}ds\\
&\lesssim 2^{-N_0k_{++}}\int_{\mathbb{R}}q_m(s)\Vert Ef^{\mu}_{k_1,j_1}(s)\Vert_{L^\infty}\Vert f^\nu_{k_2,j_2}(s)\Vert_{L^2}ds\\
&\lesssim 2^{-N_0k_{++}}2^{-\beta m},
\end{split}
\end{equation*}
if $\max(0,k_3)\ge k_{++}$, and
\begin{equation*}
\begin{split}
\left\vert\tilde{T}^{\sigma;\mu,\nu}_m[f^{\mu}_{k_1,j_1},f^\nu_{k_2,j_2};P_{k_3}h](\xi)\right\vert
&\lesssim \int_{\mathbb{R}}q_m(s)\Vert \mathcal{F}\{Ef^{\mu}_{k_1,j_1}(s)\cdot P_{k_3}h(s)\}\Vert_{L^2}\Vert f^\nu_{k_2,j_2}\Vert_{L^2}ds\\
&\lesssim 2^{-N_0k_{++}}2^{-\beta m},
\end{split}
\end{equation*}
if $k_2\ge k_{++}$, and similarly if $k_1\ge k_{++}$. In every case, we easily deduce \eqref{LinftyCunic}.

\medskip

We now consider the $L^2$-bound in Defintion \ref{MainDef}. We first assume that
\begin{equation}\label{AssCCB2-2}
k\ge-3/4j,\quad j\ge m, \quad\min(j_1,j_2)\le (1-\beta^2)j.
\end{equation}
We may also assume that $j_1\le j_2$ and $j_1\ge-k_1\ge-j/N_0^\prime$. In this case, using that
\begin{equation*}
\left\vert\nabla_{\xi_1}\left[x\cdot(\xi_1+\xi_2+\xi_3)+s\Lambda_{\sigma}(\xi_1+\xi_2+\xi_3)-s\widetilde{\Lambda}_\mu(\xi_1)-s\widetilde{\Lambda}(\xi_2)\right]\right\vert\ge\vert x\vert-C2^{m}\ge 2^{j-10}
\end{equation*}
we may use Lemma \ref{tech5} in \eqref{CubicOp} with $K=2^{j}$ and $\epsilon=\min(2^k,2^{-j_1})$ to see that
\begin{equation*}
\vert \tilde{\varphi}^{(k)}_j(x)\cdot P_{k}\tilde{T}^{\sigma;\mu,\nu}_m[f^{\mu}_{k_1,j_1},f^\nu_{k_2,j_2};h](x)\vert\lesssim 2^{-100j}
\end{equation*}
from which \eqref{CubicBound1.1} follows easily.

\medskip

Assume now that
\begin{equation}\label{AssCCB2-3}
j/N_0^\prime\ge k\ge-3j/4,\quad j\ge m, \quad\min(j_1,j_2)\ge (1-\beta^2)j. 
\end{equation}
In this situation, using \eqref{nh9.1}, we compute that, when $k_1\le k_2$,
\begin{equation*}
\begin{split}
\Vert P_{k}\tilde{T}^{\sigma;\mu,\nu}_m[f^{\mu}_{k_1,j_1},f^\nu_{k_2,j_2};P_{k_3}h]\Vert_{L^2}
&\lesssim \int_{\mathbb{R}}q_m(s)\Vert \widehat{f^{\mu}_{k_1,j_1}}(s)\Vert_{L^1}\Vert \mathcal{F}\left\{f^{\nu}_{k_2,j_2}(s)\cdot (P_{k_3}h(s))\right\}\Vert_{L^2}ds\\
&\lesssim 2^{3k_1/2}\int_{\mathbb{R}}q_m(s)\Vert f^{\mu}_{k_1,j_1}(s)\Vert_{L^2}\Vert f^{\nu}_{k_2,j_2}(s)\Vert_{L^2}\Vert P_{k_3}h(s)\Vert_{L^\infty}ds\\
&\lesssim 2^{-\beta m}\cdot 2^{-k_{++}}\cdot2^{3k_1/2}\min(1,2^{-10k_1})2^{-(1-\beta)(j_1+j_2)},
\end{split}
\end{equation*}
from which we deduce that
\begin{equation*}
2^{k_{++}}(2^{\alpha k}+2^{10k})2^{(1+\beta)j}\Vert P_{k}\tilde{T}^{\sigma;\mu,\nu}_m[f^{\mu}_{k_1,j_1},f^\nu_{k_2,j_2};P_{k_3}h]\Vert_{L^2}\lesssim 2^{-\beta j}.
\end{equation*}
Together with \eqref{LinftyCunic}, this finishes the proof in this case.

\medskip

Finally, we may assume that
\begin{equation*}
m/2+N_0^\prime k_{++}+D^2\le j\le m+D,\quad -10j\le k_1,k_2\le j/N_0^\prime,\quad\max(j_1,j_2)\le 10j.
\end{equation*}

In this case, we may simply use \eqref{BasicPlaCub} to conclude that
\begin{equation*}
\begin{split}
\Vert P_{k}\tilde{T}^{\sigma;\mu,\nu}_m[f^{\mu}_{k_1,j_1},f^\nu_{k_2,j_2};P_{k_3}h]\Vert_{L^2}
&\lesssim 2^{-(1+2\beta)m}2^{-(N_0-1)k_{++}}
\end{split}
\end{equation*}
and therefore
\begin{equation*}
2^{k_{++}}(2^{\alpha k}+2^{10k})2^{(1+\beta)j}\Vert \widetilde{\varphi}^{(k)}_j\cdot P_{k}\tilde{T}^{\sigma;\mu,\nu}_m[f^{\mu}_{k_1,j_1},f^\nu_{k_2,j_2};P_{k_3}h]\Vert_{L^2}\lesssim 2^{-\beta m}
\end{equation*}
which suffices when combined with \eqref{LinftyCunic}.
\end{proof}

\begin{lemma}\label{LemCubComp}

Assume that $t\in\mathbb{R}$ and that $f\in C^\infty$ satisfies $f(0)=0$ and
\begin{equation}\label{FLip}
\sup_{\vert x\vert\le 1,\,\,j\le N_0+3}\vert f^{(j)}(x)\vert\le A.
\end{equation}
Let $n\in H^{N_0}$ be such that
\begin{equation}\label{AssNormD3}
\Vert n\Vert_{H^{N_0}}+\Vert e^{-it\Lambda_\sigma}n\Vert_{Z}\le 1
\end{equation}
for some $\sigma\in\{i,e,b\}$. Then
\begin{equation*}
\begin{split}
\sup_{k\in\mathbb{Z}}(2^{-(1+\alpha-\beta)k}+2^{N_0k})\Vert P_k(f(n))\Vert_{L^2}&\lesssim_{A,N_0}\Vert n\Vert_{H^{N_0}}+\Vert e^{-it\Lambda_\sigma}n\Vert_{Z},\\
\sup_{k\in\mathbb{Z}}(2^{-(1/2-\beta-\alpha)k}+ 2^{6k})(1+t)^{1+\beta}\Vert P_k(f(n))\Vert_{L^\infty}&\lesssim_{A,N_0}\Vert n\Vert_{H^{N_0}}+\Vert e^{-it\Lambda_\sigma}n\Vert_{Z}.
\end{split}
\end{equation*}
\end{lemma}

\begin{proof}

We first assume that
\begin{equation}\label{AssNormD32}
\Vert n\Vert_{H^{N_0}}+\Vert e^{-it\Lambda_\sigma}n\Vert_{Z} = 1.
\end{equation}
When $k\ge 0$, the lemma essentially follows from the following easy inequalities,
\begin{equation*}
\begin{split}
\Vert  f(u)\Vert_{H^{N_0}}&\lesssim_{N_0} \Vert u\Vert_{H^{N_0}},\quad\hbox{ and }\quad\Vert \partial^\alpha_x f(u)\Vert_{L^\infty}\lesssim \sum_{\beta\le\alpha}\Vert \partial^\beta u\Vert_{L^\infty}.
\end{split}
\end{equation*}
This is sufficient for the first inequality; for the second, we first observe that
\begin{equation*}
\begin{split}
2^{6k}\Vert P_k(f(P_{\le k}n))\Vert_{L^\infty}&\lesssim 2^{-k}\sum_{\vert\alpha\vert\le 7}\Vert \partial^\alpha P_k(f(P_{\le k}n))\Vert_{L^\infty}\lesssim 2^{-k}\sum_{\vert\alpha\vert\le 7}\Vert \partial^\alpha f(P_{\le k}n)\Vert_{L^\infty}\lesssim (1+s)^{-1-\beta},
\end{split}
\end{equation*}
while using that
\begin{equation*}
f(n)-f(P_{\le k}n)=\int_0^1 f^\prime(P_{\le k}n+sP_{\ge k+1}n)ds\cdot P_{\ge k+1}n
\end{equation*}
we see that
\begin{equation*}
2^{6k}\Vert P_k[f(n)-f(P_{\le k}n)]\Vert_{L^\infty}\lesssim 2^{6k}\Vert P_{\ge k+1}n\Vert_{L^\infty}\lesssim (1+s)^{-1-\beta}
\end{equation*}
which is what we wanted.
If $k\le 0$, since $f$ is smooth, we can write
\begin{equation*}
f(n)=C_1n+g(n),\quad C_1\in\mathbb{R},\,\,\lim_{n\to 0}\frac{g(n)}{n^2}\in\mathbb{R}.
\end{equation*}
The linear part clearly verifies the appropriate bounds by \eqref{nh9}. For the nonlinear part, we use that
\begin{equation*}
\Vert P_kg(n)\Vert_{L^2}\lesssim 2^{3k/2}\Vert\widehat{g(n)}\Vert_{L^\infty}\lesssim 2^{3k/2}\Vert g(n)\Vert_{L^1}\lesssim 2^{3k/2}
\end{equation*}
and
\begin{equation*}
\begin{split}
\Vert P_kg(n)\Vert_{L^\infty}&\lesssim \Vert n\Vert_{L^\infty}^2\lesssim (1+s)^{-2(1+\beta)}\\
\Vert P_kg(n)\Vert_{L^\infty}&\lesssim 2^{3k}\Vert\widehat{g(n)}\Vert_{L^\infty}\lesssim 2^{3k}
\end{split}
\end{equation*}
and therefore, we can easily get the bounds when $k\le 0$.

In the general case of \eqref{AssNormD3}, we may simply write $n=\delta\tilde{n}$ where $\tilde{n}$ satisfy \eqref{AssNormD32}. Then the Lemma follow by changing
$f(x)\to \delta^{-1} f(\delta x)$ which still satisfies \eqref{FLip}.
\end{proof}

\end{document}